\documentclass[leqno,twoside]{article}[10 pt]

\usepackage[english]{babel}

\usepackage[letterpaper,top=2cm,bottom=2cm,left=2cm,right=2cm,marginparwidth=1.75cm]{geometry}

\usepackage[utf8]{inputenc}
\usepackage[T1]{fontenc}
\usepackage{mathrsfs}
\usepackage{fancyhdr}
\usepackage{lastpage}
\usepackage{extramarks}
\usepackage[inline]{enumitem}
\usepackage{amsmath,amssymb,latexsym,amsfonts, amsthm}
\usepackage{comment}
\usepackage[square,sort,comma,numbers]{natbib}
\usepackage{mathtools}
\usepackage{extarrows}
\usepackage{titlesec}
\usepackage{stmaryrd}
\usepackage{upgreek}
\usepackage{url}
\SetSymbolFont{stmry}{bold}{U}{stmry}{m}{n}

\usepackage[dvipsnames]{xcolor}
\definecolor{tiffanyblue}{rgb}{0.04, 0.73, 0.71}
\definecolor{darkscarlet}{rgb}{0.34, 0.01, 0.1}
\definecolor{falured}{rgb}{0.5, 0.09, 0.09}
\usepackage[linktocpage,
  colorlinks,
            linkcolor=darkscarlet,
            anchorcolor=darkscarlet,
            citecolor=tiffanyblue,
            urlcolor=falured
            ]{hyperref}

\PassOptionsToPackage{hyphens}{url}\usepackage{hyperref}

\usepackage{url}

\usepackage[all]{xy}

\usepackage{tikz}
\usepackage{tikz-cd}
\usetikzlibrary{arrows}
\usetikzlibrary{matrix}

\newtheorem{mainthm}{Theorem}
\newtheorem{maincor}[mainthm]{Corollary}
\newtheorem*{mainrem}{Remark}

\usepackage{fancyhdr}

\makeatletter
\def\@pnumwidth{2em}
\makeatother

\makeatletter
\newcommand{\fundingnote}[1]{%
  \begingroup
  \renewcommand{\@makefnmark}{}%
  \footnotetext[0]{#1}%
  \endgroup
}
\makeatother

\usepackage{abstract}

\renewenvironment{abstract}
 {\small
  \list{}{%
    \setlength{\leftmargin}{5mm}% <---------- CHANGE HERE
    \setlength{\rightmargin}{\leftmargin}%
  }%
  \item\relax}
 {\endlist}

\titleformat*{\section}{\large\centering\scshape}
\titleformat*{\subsection}{\bfseries}
\titlelabel{\thetitle.\quad}

\usepackage{titletoc}% http://ctan.org/pkg/titletoc
\titlecontents{section}[0pt]
{\addvspace{2pt}}
{\contentspush{\thecontentslabel.\quad}\normalsize}
{\Large}
{\hspace{1em plus 1fill}\large\contentspage}

\newcommand{\R}{\mathbb{R}}
\newcommand{\Z}{\mathbb{Z}}

\newcommand{\C}{\mathbb{C}}

\newcommand{\Q}{\mathbb{Q}}

\newcommand{\la}{\langle}
\newcommand{\ra}{\rangle}

\newcommand{\Spec}{\mathrm{Spec}}

\newcommand{\Tr}{\mathrm{Tr}}
\newcommand{\Norm}{\mathrm{N}}

\newcommand{\vol}{\mathrm{vol}}

\newcommand{\GL}{\mathrm{GL}}

\def\map#1#2{\def\mapname{#1}\def\mapdomain{#2}\def\maparrow{\longrightarrow}\futurelet\next\mapAR}
\def\mapAR{\ifx\next[\expandafter\mapar\else\expandafter\mapAux\fi}
\def\mapar[#1]{\def\maparrow{#1}\expandafter\mapAux}
\def\mapAux#1#2{\def\mapcodomain{#1}\def\mapelement{#2}\def\mapmaparrow{\longmapsto}\futurelet\next\mapMAR}
\def\mapMAR{\ifx\next[\expandafter\mapmar\else\expandafter\mapAuxAux\fi}
\def\mapmar[#1]{\def\mapmaparrow{#1}\expandafter\mapAuxAux}
\def\mapAuxAux#1{\def\mapcoelemnt{#1}\futurelet\next\mapAuxAuxAux}
\def\mapAuxAuxAux{\def\nothing{}%
\ifx\mapname\nothing%
  \mapdomain\maparrow\mapcodomain\qquad\mapelement\mapmaparrow\mapcoelemnt%%
\else%
  \mapname:\mapdomain\maparrow\mapcodomain\qquad\mapelement\mapmaparrow\mapcoelemnt%
\fi%
}
\def\Map#1#2{\def\Mapname{#1}\def\Mapdomain{#2}\def\Maparrow{\longrightarrow}\futurelet\next\MapAR}
\def\MapAR{\ifx\next[\expandafter\Mapar\else\expandafter\MapAux\fi}
\def\Mapar[#1]{\def\Maparrow{#1}\expandafter\MapAux}
\def\MapAux#1#2{\def\Mapcodomain{#1}\def\Mapelement{#2}\def\Mapmaparrow{\longmapsto}\futurelet\next\MapMAR}
\def\MapMAR{\ifx\next[\expandafter\Mapmar\else\expandafter\MapAuxAux\fi}
\def\Mapmar[#1]{\def\Mapmaparrow{#1}\expandafter\MapAuxAux}
\def\MapAuxAux#1{\def\Mapcoelemnt{#1}\futurelet\next\MapAuxAuxAux}
\def\MapAuxAuxAux{\def\nothing{}%
\ifx\Mapname\nothing%
  \begin{array}{rcl}{}
  \Mapdomain&\Maparrow&\Mapcodomain\\{}
  \Mapelement&\Mapmaparrow&\Mapcoelemnt
  \end{array}
\else%
  \begin{array}{rrcl}{}
  \Mapname:&\Mapdomain&\Maparrow&\Mapcodomain\\{}
  &\Mapelement&\Mapmaparrow&\Mapcoelemnt
  \end{array}
\fi%
}

\renewcommand{\subset}{\subseteq}

\newtheorem{thm}{Theorem}[section]
\newtheorem{notation}[thm]{Notation}
\newtheorem{lem}[thm]{Lemma}
\newtheorem{prop}[thm]{Proposition}
\newtheorem{cor}[thm]{Corollary}

\newtheorem{defn}[thm]{Definition}
\newtheorem{rem}[thm]{Remark}
\newtheorem{exmp}[thm]{Example}
\newtheorem{constr}[thm]{Construction}
\numberwithin{equation}{section}

\usetikzlibrary{calc}
\usetikzlibrary{arrows.meta}

\usepackage{mathtools}
\usetikzlibrary{arrows}
\usepackage{xparse}
\NewDocumentCommand{\xrightleftarrow}{ O{}O{} }{%
  \mathrel{%
    \vcenter{
      \hbox{%
        \begin{tikzpicture}
          \node[minimum width=0.75cm,align=center] (a){\text{\vphantom{hg}#1}\\[0.3ex]\vphantom{hg}#2};
          \draw[-stealth] ([yshift=0.4ex]a.west) -- ([yshift=0.4ex]a.east);
          \draw[stealth-] ([yshift=-0.4ex]a.west) -- ([yshift=-0.4ex]a.east);
        \end{tikzpicture}
      }
    }%
  }\!\!%
}

\title{\Large\bfseries\boldmath
Algebraic Eisenstein classes for \texorpdfstring{$\mathrm{GL}_n$}{GLn} with applications to the Bianchi case}
\author{ZHENGHANG DU}
\date{}

\begin{document}
\maketitle
\fundingnote{This research was supported by the Deutsche Forschungsgemeinschaft (DFG) through SFB 1085 ``Higher Invariants''.}
\vspace{-7mm}
\begin{abstract}\small
  \par\noindent\textsc{\abstractname.}
  Let $L$ be a number field containing a CM field $K$.
  We compute an explicit current representing the
  $\GL_{\mathcal O_K}(\mathcal O_L)$-equivariant algebraic
  Eisenstein--Kronecker classes constructed by Kings--Sprang.
  In the Bianchi case, where $K$ is imaginary quadratic and $[L:K]=2$,
  this yields a purely algebraic construction of Eisenstein cohomology
  classes, even though the associated locally symmetric space is not
  an algebraic variety.
  We express their constant terms at the cusp in terms of partial
  Hecke $L$-values of $K$.
\end{abstract}

\setcounter{tocdepth}{1}
\tableofcontents

\section*{Introduction}

Eisenstein cohomology studies the part of the cohomology of an arithmetic
locally symmetric space that is detected on the Borel--Serre boundary.
More precisely, if $\Gamma$ is an arithmetic group with associated locally
symmetric space $Y_\Gamma$, and $\mathcal V_E$ is an
algebraic local system on $Y_\Gamma$ defined over a number field $E$, the
basic map is the boundary restriction
\[
  \operatorname{res}_{\partial}:
  H^\bullet(Y_\Gamma,\mathcal V_E)
  \longrightarrow
  H^\bullet(\partial\overline{Y}_\Gamma^{\,\mathrm{BS}},\mathcal V_E).
\]
Over $\C$, the Eisenstein operator gives a section of this map on its
image,
\[
  \operatorname{Eis}:
  \operatorname{im}(\operatorname{res}_{\partial})\otimes_E\C
  \longrightarrow
  H^\bullet(Y_\Gamma,\mathcal V_E)\otimes_E\C,
  \qquad
  \operatorname{res}_{\partial}\circ\operatorname{Eis}=\mathrm{id}.
\]
In the classical analytic approach, suitable boundary classes are
lifted using Eisenstein series, whose constant terms encode special
values of $L$-functions.
Since the Eisenstein operator is defined by analytic summation and,
in general, meromorphic continuation, compatibility with the natural
rational structures does not follow formally from the construction.
Harder's rationality proof, for example, compares Galois-conjugate
lifts and uses multiplicity one to exclude a cuspidal discrepancy
\cite{Harder87}.

Our approach starts instead from the equivariant coherent
Eisenstein--Kronecker classes of Kings--Sprang~\cite{KS_2025} and gives explicit
representatives for these algebraically defined classes for $\GL_n$ over
a general CM field $K$ without assuming that the acting arithmetic subgroup is commutative. The algebraic structure comes
from the abelian scheme and its group action, rather than from an
algebraic model of the locally symmetric space, which is generally
non-Hermitian. Using Levin's method of polylogarithmic currents~\cite{Levin_2000}, we compute
their complex realizations, thereby obtaining explicit representatives
of algebraically defined Eisenstein--Kronecker classes in arithmetic-group
cohomology. In the Bianchi case, where $n=2$
and $K$ is imaginary quadratic, we compute the boundary restrictions as
Fourier constant terms at the cusps, which for suitable arithmetic data
are expressed in terms of partial Hecke $L$-values of $K$. The
algebraicity of the corresponding regularized partial values, after
normalization by the appropriate CM periods, then follows directly from
the algebraic construction, without a separate rationality argument for
the Eisenstein operator.

\subsection*{Main results}

Let $L$ be a number field containing a CM field, let $K$ be its maximal
CM subfield, and write $[K:\Q]=2g$, $[L:K]=n$, and $d=gn$. For an
abelian scheme $\mathcal A_0/S$ of relative dimension $g$ with complex
multiplication by $\mathcal O_K$, the Serre tensor construction
$\mathcal A=\mathcal O_L\otimes_{\mathcal O_K}\mathcal A_0$ carries the
natural action of $\GL_{\mathcal O_K}(\mathcal O_L)$. Fix a subgroup
$\Gamma\subseteq\GL_{\mathcal O_K}(\mathcal O_L)$, a $\Gamma$-stable
finite étale torsion subscheme $\mathcal D\subset\mathcal A$ satisfying
the hypotheses of the construction~\ref{notation: cond of D}, and a trace-zero invariant
function $f\in\mathcal O_S[\mathcal D]^{0,\Gamma}$. Writing
$j:\mathcal U_{\mathcal D}=\mathcal A\setminus\mathcal D\hookrightarrow
\mathcal A$, one obtains
\[
  EK_{\Gamma,\mathcal A}(f)\in H^{d-1}\!\left(\mathcal U_{\mathcal D},\Gamma;j^{-1}\!\left(\widehat{\mathcal P}\otimes\Omega_{\mathcal A/S}^d\right)\right).
\]
Here $\widehat{\mathcal P}$ is the completed Poincaré bundle.
Write $EK_{\Gamma,\mathcal A}(f)_\C$ for the complex realization and $S_\infty:=(\GL_n(\C)/U(n))^g$ for the symmetric space attached to $\operatorname{Res}_{K/\Q}\GL_n$.

\begin{mainthm}[Explicit realization of the Eisenstein--Kronecker class;
  see Theorem~\ref{prop:functional equation for currents}]
  \label{thm:intro explicit EK current}
  There is an explicit $\Gamma$-equivariant current $\phi_{g,f}$ on
  $\mathcal A(\C)\times S_\infty$, with values in the pullback of
  $\widehat{\mathcal P}^{an}$, whose coefficients are given explicitly
  by generalized Eisenstein--Kronecker series. Its restriction to
  $\mathcal U_{\mathcal D}(\C)\times S_\infty$ is smooth and closed in
  the equivariant Dolbeault complex, and its cohomology class satisfies
  \[
    [\phi_{g,f}]_{\mathrm{coh}}
    =
    EK_{\Gamma,\mathcal A}(f)_\C\in H^{d-1}\left(
    \mathcal U_{\mathcal D}(\C),\Gamma;
    j^{-1}\bigl(
      \widehat{\mathcal P}^{\mathrm{an}}
      \otimes\Omega^d_{\mathcal A(\C)}
    \bigr)
  \right).
  \]
\end{mainthm}

Over an affine base $S=\Spec R$, specialization at suitable
$\Gamma$-invariant torsion sections and application of the moment maps
give group-cohomology classes with algebraic Hodge coefficients. For
arithmetic $\Gamma$, their restrictions to parabolic stabilizers are
therefore defined over the same coefficient ring. Thus
Theorem~\ref{thm:intro explicit EK current} gives explicit complex
representatives of global $\GL_n$ classes whose algebraic structure is
fixed before passing to $\C$. In Harder's analytic Eisenstein
construction, by contrast, compatibility with rational cohomology has to
be established separately \cite[Section~4.3]{Harder87}.

In the Bianchi case, where $K$ is imaginary quadratic and $n=2$, the
boundary class at each cusp is represented explicitly by the zeroth
horospherical Fourier coefficient of the Eisenstein--Kronecker form.
For suitable arithmetic torsion data, its two rank-one components give
regularized partial Hecke $L$-values of $K$. Since the boundary class is
algebraic, this gives a direct route to the algebraicity of the
corresponding period-normalized values; summing over ray classes yields
the following consequence for the full imprimitive $L$-values.

\begin{maincor}[Algebraicity of critical Hecke $L$-values,
  see Theorem~\ref{thm: algebraicity of critical Hecke L values}]
  \label{cor:intro algebraicity of critical Hecke L values}
  Let $b\in\mathbb N$, and let $\chi_A$ and $\chi_B$ be algebraic Hecke
  characters of $K$, of conductor dividing $\mathfrak f$, characterized by
  \(
    \chi_A((\lambda))
    =
    \frac{\overline{\lambda}^{\,b}}{\lambda^2}\) and \(
    \chi_B((\lambda))
    =
    \frac{\overline{\lambda}^{\,b+1}}{\lambda}
  \)
  for $\lambda$ satisfying the ray congruence modulo $\mathfrak f$.
  After enlarging the field of definition if necessary, there exists a
  number field $k\subset\C$ containing the coefficient field and $R$ such
  that
  \[
    \frac{(2\pi i)^b}
    {(\Omega_K^\vee)^b\Omega_K^2}
    L_{\mathfrak f}(\chi_A,0)
    \in k,
    \qquad
    \frac{(2\pi i)^{b+1}}
    {(\Omega_K^\vee)^b\Omega_K^2\operatorname{vol}(\mathcal F)}
    L_{\mathfrak f}(\chi_B,0)
    \in k.
  \]
  Here $\Omega_K$ and $\Omega_K^\vee$ are the fixed CM periods, and
  $\mathcal F$ is a fundamental domain for the unipotent lattice at the
  standard cusp.
\end{maincor}

This recovers a case of the algebraicity theorem of Kings--Sprang~\cite[Theorem~4.10]{KS_2025} through the explicit Borel--Serre boundary realization of the equivariant Eisenstein--Kronecker class.

\begin{mainrem}[Towards explicit boundary realizations in the general $\GL_n$-CM case]\rm
The equivariant Eisenstein--Kronecker classes and their explicit current
representatives constructed above are available for general $g$ and $n$.
In the Bianchi case, we describe their full Borel--Serre boundary
restriction by the zeroth horospherical Fourier coefficient at each cusp.
A comparably explicit description beyond this case must also account
for the cohomology of Levi quotients and, in higher rank, the
compatibility of constant terms across different boundary strata
\cite[Chapter~4]{HarderRaghuram2020}.
\end{mainrem}

These results belong to the geometric approach to Eisenstein classes through
specializations of polylogarithms. The elliptic polylogarithm of
Beilinson--Levin, and its generalization to families of commutative group
schemes, give rise to such classes by specialization at torsion sections
\cite{BeilinsonLevin,HK_2018}. For families of abelian varieties, Levin
constructed explicit polylogarithmic currents whose Fourier description
provides the analytic framework for Eisenstein--Kronecker series
\cite{Levin_2000}. Related topological constructions go back to Nori and Sczech
\cite{Nori,Sczech}. Beilinson--Kings--Levin developed the topological
polylogarithm as an integral refinement of Nori's construction
\cite{BKL2018}, and Graf used this framework to construct Eisenstein classes
for $\GL_2$ over totally real fields and compare them with Harder's classes
\cite{graf2016polylogarithmsgl2totallyreal}. Related Eisenstein cocycles and
their integral, imaginary-quadratic, and motivic variants have been studied
in \cite{Charollois-Dasgupta,BCG2,Fl_rez_2019,Sharifi-Venkatesh}. Most directly related to the present work, Kings--Sprang construct
equivariant coherent Eisenstein--Kronecker classes for abelian schemes
with automorphisms and compute their complex realizations explicitly
for finite-index subgroups of $\mathcal O_L^\times$~\cite{KS_2025}.

\subsection*{Overview of the approach}

Starting from the Poincaré bundle on
$\mathcal A\times_S\mathcal A^\vee$, we complete along
$\mathcal A\times_S e^\vee(S)$ to obtain $\widehat{\mathcal P}$.
The vanishing theorem for
$\widehat{\mathcal P}\otimes\Omega_{\mathcal A/S}^d$, together with
equivariant cohomology with support and the localization sequence
associated with
\[
  \begin{tikzcd}[column sep=large]
    \mathcal D
      \arrow[r, hook, "i"]
      \arrow[dr, "\pi_{\mathcal D}"']
    &
    \mathcal A
      \arrow[d, "\pi"]
    &
    \mathcal U_{\mathcal D}
      \arrow[l, hook', "j"']
      \arrow[dl, "\pi_{\mathcal U_{\mathcal D}}"]
    \\
    & S &
  \end{tikzcd}
\]
produces the canonical class $EK_{\Gamma,\mathcal A}(f)$ from the
trace-zero invariant function $f$. Canonical splittings at suitable
torsion sections allow us to pull back along $x$ and apply the
degree-$b$ moment map, obtaining $EK_{\Gamma,\mathcal A}^b(f,x)$.
Both classes are therefore defined algebraically, before any analytic
representative is introduced.

After base change to $\C$, we work on the product of $\mathcal A(\C)$
with the symmetric space $S_\infty$ associated with
$\operatorname{Res}_{K/\Q}\GL_n$, retaining the $\Gamma$-action on
both factors. Following Levin's construction of polylogarithmic
currents \cite[Section~3.4]{Levin_2000} and its equivariant adaptation
by Kings--Sprang \cite[Section~3]{KS_2025}, we construct an explicit
current with values in $\widehat{\mathcal P}$ by solving its
differential equation through a Fourier expansion. Its coefficients
are expressed in terms of generalized Eisenstein--Kronecker series.
The differential equation gives precisely the residue that
characterizes the original algebraic class, and thus identifies
the current as a representative of the complex realization of
$EK_{\Gamma,\mathcal A}(f)$.
Specializing at $x$ and applying the moment map then gives explicit
differential forms on $S_\infty$ representing the corresponding
algebraic moment classes.

In the Bianchi case $g=1$ and $n=2$, the resulting forms are closed
one-forms on $S_\infty\simeq\mathbb H^3\times\R_{>0}$.
To compute their boundary classes, we first restrict to a horosphere
at the standard cusp. A suitable change of basis makes the scalar
coefficients periodic, allowing us to expand them in Fourier series
and compute the constant term explicitly. We show that the
nonconstant terms are exact, so the restricted form and its zeroth
Fourier coefficient represent the same boundary class.
The calculation at the standard cusp extends to every cusp,
giving the full Borel--Serre boundary restriction.

We next express the constant term in the chosen algebraic Hodge
bases. The period factors arise from the de Rham pairing when
comparing these bases with those used in the analytic calculation.
This comparison identifies the computed constant term with the
complex realization of the algebraic boundary restriction.
The multiplication action of $L^\times$ on the $K$-vector space $L$
defines an embedding
\(
  \operatorname{Res}_{L/\Q}\mathbb G_m
  \hookrightarrow
  \operatorname{Res}_{K/\Q}\GL_2.
\)
Integrating the same differential form over the associated torus
cycle recovers the dimension-two lattice sum of Kings--Sprang
\cite[Corollary~3.28]{KS_2025}.

The arithmetic application instead uses two scalar lattice sums
obtained by projecting the constant term at the standard cusp.
For the chosen torsion data, a common distribution relation
expresses both sums as linear combinations of partial Hecke
$L$-values. The original class, its boundary restriction, and
the projections are defined over a number field, so the
coordinates of the projected class in algebraic bases lie
in that field. Comparing these coordinates with the analytic
formulas proves that the corresponding $L$-value expressions
are algebraic after division by the appropriate periods.
Working with integral coefficients and keeping track of
denominators gives the stated integrality results. Summing
over ray classes then yields the corresponding statements
for full Hecke $L$-values.

\subsection*{Outline of the paper}

Section~1 sets up the CM and Hodge-theoretic framework for the Serre
tensor construction. It fixes the resulting group actions, coefficient
systems, and algebraic Hodge bases, together with the decompositions
and pairings needed to compare algebraic and analytic realizations.

Section~2 constructs the equivariant Eisenstein--Kronecker classes
from the completed Poincaré bundle and equivariant localization,
and establishes their moment specializations and functoriality.
It also introduces the two rank-one projections arising from
unipotent cohomology in the rank-two setting.

Section~3 develops the complex realization through the Borel model
and polylogarithmic currents, then specializes to the Bianchi setting
$g=1$, $n=2$. It computes the cuspwise Fourier constant terms and
identifies them with the Borel--Serre boundary restriction, the
nonzero modes being cohomologically trivial. Comparison with the
algebraic Hodge bases gives the algebraic normalization of the
boundary class. The quadratic-torus period calculation recovers
the dimension-two Kings--Sprang formula.

Section~4 applies the boundary calculation to critical values of
Hecke $L$-functions, after fixing the Hecke-character conventions
and CM periods for compatible Serre twists. The projections
$\mathrm{pr}_A$ and $\mathrm{pr}_B$, together with distribution
relations, give localized integrality statements for regularized,
period-normalized partial values. Summation over ray classes yields
algebraicity of the full imprimitive values, with the
trivial-conductor case treated using a second auxiliary ideal.

\subsection*{Acknowledgements}

I am grateful to my supervisor, Guido Kings, for his support,
valuable comments, and many helpful mathematical discussions.
I also thank Johannes Sprang for the productive discussions during
my visits to Essen. My thanks also go to Christoph Fronhöfer,
Han-Ung Kufner, and Chenying Lin for their helpful comments and
suggestions. This research was supported by the Deutsche
Forschungsgemeinschaft (DFG) through the Collaborative Research
Centre SFB 1085, \emph{Higher Invariants}.

\section{Abelian schemes with complex multiplication}

In this section, we establish the basic algebraic structure of the CM abelian schemes used throughout the paper and make explicit the CM-type decompositions governing their Hodge-theoretic realizations and compatible algebraic bases.

\subsection{Serre tensor constructions of CM abelian schemes}\label{Section: Serre tensor constructions of CM abelian schemes}

Throughout this paper, let $L$ be a number field containing a CM field,
and let $K$ be its maximal CM subfield. We write $[K:\Q]=2g$,
$[L:K]=n$, and $[L:\Q]=2d$, so that $d=gn$.

Let $S$ be a Noetherian scheme, and let
$\pi_0\colon\mathcal A_0\to S$ be an abelian scheme of relative
dimension $g$, with unit section $e_0\colon S\to\mathcal A_0$.
Its sheaf of translation-invariant relative differential $1$-forms is
\[
  \omega_{\mathcal A_0/S}
  :=
  e_0^*\Omega^1_{\mathcal A_0/S}
  \cong
  \pi_{0*}\Omega^1_{\mathcal A_0/S}
  \cong
  \mathcal{H}om_{\mathcal O_S}
  \bigl(
    \operatorname{Lie}(\mathcal A_0/S),
    \mathcal O_S
  \bigr),
\]
which is locally free of rank $g$. For $i\geq 0$, put
$\omega_{\mathcal A_0/S}^{i}
:=\bigwedge_{\mathcal O_S}^{i}\omega_{\mathcal A_0/S}$
and
$\omega_{\mathcal A_0/S}^{-i}
:=\mathcal{H}om_{\mathcal O_S}
(\omega_{\mathcal A_0/S}^{i},\mathcal O_S)
\cong\bigwedge_{\mathcal O_S}^{i}
\operatorname{Lie}(\mathcal A_0/S)$.

Write $\pi_0^\vee\colon\mathcal A_0^\vee\to S$ for the dual abelian
scheme and $e_0^\vee\colon S\to\mathcal A_0^\vee$ for its unit section,
and use analogous notation for its Lie algebra and invariant
differentials. The tangent-space description of
$\operatorname{Pic}^0_{\mathcal A_0/S}$ gives
$R^1\pi_{0*}\mathcal O_{\mathcal A_0}
\cong\operatorname{Lie}(\mathcal A_0^\vee/S)$.
Together with the cup-product isomorphisms
$\bigwedge_{\mathcal O_S}^{i}
R^1\pi_{0*}\mathcal O_{\mathcal A_0}
\xrightarrow{\sim}
R^i\pi_{0*}\mathcal O_{\mathcal A_0}$,
this yields
$R^i\pi_{0*}\mathcal O_{\mathcal A_0}
\cong\omega_{\mathcal A_0^\vee/S}^{-i}$
for all $i\geq 0$.

\begin{defn}[CM structure]\rm
  A \textit{CM structure by $\mathcal O_K$} on $\mathcal A_0/S$
  is an injective ring homomorphism
  $\iota_K\colon\mathcal O_K\hookrightarrow
  \operatorname{End}_S(\mathcal A_0)$
  such that $\operatorname{Lie}(\mathcal A_0/S)$, with its induced
  $\mathcal O_K$-action, is projective over
  $\mathcal O_K\otimes_{\Z}\mathcal O_S$.
  We then say that $\mathcal A_0/S$ has CM by $\mathcal O_K$. Restriction of $\iota_K$ gives an action of
  $\mathcal O_K^\times$ on $\mathcal A_0/S$.
\end{defn}

We use the Serre tensor construction in the following form.
Let $R$ be a ring, not necessarily commutative, which is finite free
over $\Z$. Suppose that $\mathcal A_0/S$ is equipped with a ring
homomorphism $R\to\operatorname{End}_S(\mathcal A_0)$, and let
$\mathfrak a$ be a finitely generated projective right $R$-module.
The Serre tensor construction gives an abelian scheme
$\mathcal A_0(\mathfrak a):=\mathfrak a\otimes_R\mathcal A_0$
representing the fppf sheaf
$T\mapsto\mathfrak a\otimes_R\mathcal A_0(T)$
on $S$-schemes.
It is compatible with relative Lie algebras and invariant differentials:
\[
  \begin{aligned}
    \operatorname{Lie}\bigl(\mathcal A_0(\mathfrak a)/S\bigr)
    \cong
    \mathfrak a\otimes_R\operatorname{Lie}(\mathcal A_0/S),
    \qquad
    \omega_{\mathcal A_0(\mathfrak a)/S}
    \cong
    \omega_{\mathcal A_0/S}\otimes_R\mathfrak a^\vee.
  \end{aligned}
\]
Here $\mathfrak a^\vee:=\operatorname{Hom}_R(\mathfrak a,R)$,
with $\operatorname{Hom}_R$ taken in the category of right
$R$-modules, is naturally a left $R$-module.
The right $R$-action on $\omega_{\mathcal A_0/S}$ is induced
by pullback along the endomorphisms of $\mathcal A_0$.

From now on, $\mathcal A_0/S$ is endowed with a fixed CM structure
by $\mathcal O_K$. If $\mathfrak a\subseteq\mathfrak b$ are
fractional $\mathcal O_K$-ideals, their inclusion induces an isogeny
$\mathcal A_0(\mathfrak a)\to\mathcal A_0(\mathfrak b)$
of degree $[\mathfrak b:\mathfrak a]$.
In particular, for a nonzero integral ideal
$\mathfrak c\subseteq\mathcal O_K$, the inclusion
$\mathcal O_K\hookrightarrow\mathfrak c^{-1}$ induces an isogeny
$[\mathfrak c]\colon\mathcal A_0\to
\mathcal A_0(\mathfrak c^{-1})$.
We write $\mathcal A_0[\mathfrak c]:=\ker[\mathfrak c]$;
its degree is
$\deg[\mathfrak c]=N_{K/\Q}(\mathfrak c)$.

\begin{constr}\label{example_A}\rm
  The abelian scheme used throughout the paper is
  $\mathcal A:=\mathcal O_L\otimes_{\mathcal O_K}\mathcal A_0$.
  Since $\mathcal O_L$ is projective of rank $n$ over
  $\mathcal O_K$, this is an abelian scheme over $S$ of relative
  dimension $d=gn$. Multiplication on the first tensor factor
  gives its $\mathcal O_L$-action. More generally, the group
  $\operatorname{GL}_{\mathcal O_K}(\mathcal O_L)
  :=\operatorname{Aut}_{\mathcal O_K}(\mathcal O_L)$
  acts on $\mathcal A$ through the first tensor factor,
  extending the multiplication action of $\mathcal O_L^\times$.

  Denote the structure morphism and unit section by
  $\pi\colon\mathcal A\to S$ and $e\colon S\to\mathcal A$,
  respectively. Likewise, write
  $\pi^\vee\colon\mathcal A^\vee\to S$ and
  $e^\vee\colon S\to\mathcal A^\vee$ for the corresponding
  morphisms of the dual abelian scheme.
  The Serre tensor construction gives canonical isomorphisms
  \[
    \begin{aligned}
      \operatorname{Lie}(\mathcal A/S)
      \cong
      \mathcal O_L\otimes_{\mathcal O_K}
      \operatorname{Lie}(\mathcal A_0/S),\qquad
      \omega_{\mathcal A/S}
      \cong
      \omega_{\mathcal A_0/S}\otimes_{\mathcal O_K}\mathcal O_L^\vee
      \cong
      \omega_{\mathcal A_0/S}\otimes_{\mathcal O_K}
      \mathfrak d_{L/K}^{-1}.
    \end{aligned}
  \]
  Here
  $\mathcal O_L^\vee
  :=\operatorname{Hom}_{\mathcal O_K}(\mathcal O_L,\mathcal O_K)$,
  and the trace pairing of $L/K$ identifies
  $\mathcal O_L^\vee$ with the relative inverse different
  $\mathfrak d_{L/K}^{-1}$.

  Let $\mathfrak c\subseteq\mathcal O_K$ be a nonzero integral
  ideal, and put
  $\mathfrak C:=\mathfrak c\mathcal O_L
  \cong\mathfrak c\otimes_{\mathcal O_K}\mathcal O_L$.
  The inclusion $\mathcal O_L\hookrightarrow\mathfrak C^{-1}$
  induces an isogeny
  $[\mathfrak C]\colon\mathcal A\to
  \mathcal A(\mathfrak C^{-1})$,
  where
  $\mathcal A(\mathfrak C^{-1})
  :=\mathfrak C^{-1}\otimes_{\mathcal O_L}\mathcal A$.
  Under the canonical identification
  \[
    \mathcal A(\mathfrak C^{-1})
    \cong
    \mathcal O_L\otimes_{\mathcal O_K}
    \mathcal A_0(\mathfrak c^{-1}),
  \]
  the isogeny $[\mathfrak C]$ identifies with
  $\operatorname{id}_{\mathcal O_L}\otimes[\mathfrak c]$.
  Thus it is equivariant for the action on the first tensor
  factor, and its degree is
  $\deg[\mathfrak C]
  =N_{L/\Q}(\mathfrak C)
  =N_{K/\Q}(\mathfrak c)^n$.
\end{constr}

The Serre tensor construction equips $\mathcal A$ with its
natural $\operatorname{GL}_{\mathcal O_K}(\mathcal O_L)$-action.

\subsection{CM types and their algebraic realizations}
\label{subsection: CM types and algebraic realizations}
\label{section:CM field and CM type}
\label{subsection: Paring and decompositions}

We keep the number fields $K\subset L$ and the notation fixed above.
Fix an algebraic closure $\overline{\mathbb Q}\subset\mathbb C$
containing $L$, and denote the discriminant of $K$ by $d_K$.
Put $J_K:=\operatorname{Hom}_{\mathbb Q}(K,\overline{\mathbb Q})$
and $I_K:=\mathbb Z[J_K]$, and use analogous notation for $L$.
A \textit{type} of $K$ is an element
$\mu=\sum_{\sigma\in J_K}\mu(\sigma)\sigma\in I_K$.
Its degree is $|\mu|:=\sum_{\sigma\in J_K}\mu(\sigma)$, and we write
$I_K^+:=\{\mu\in I_K\mid\mu(\sigma)\geq 0\text{ for all }\sigma\in J_K\}$.
For $k\in K^\times$, set
$k^\mu:=\prod_{\sigma\in J_K}\sigma(k)^{\mu(\sigma)}$.

\begin{defn}[CM type]\rm
  A \textit{CM type} of $K$ is a subset $\Sigma_K\subset J_K$
  such that $J_K=\Sigma_K\sqcup\overline{\Sigma}_K$, where
  $\overline{\Sigma}_K:=\{\overline\sigma\mid\sigma\in\Sigma_K\}$
  and $\overline\sigma$ denotes the complex conjugate embedding.
  Its lift to $L$ is
  \[
    \Sigma_L:=\{\tau\in J_L\mid\tau|_K\in\Sigma_K\},
  \]
  called a \textit{CM type of $L$ lifted from $K$}.
\end{defn}

Then $J_L=\Sigma_L\sqcup\overline{\Sigma}_L$, where
$\overline{\Sigma}_L=\{\tau\in J_L\mid\tau|_K\in\overline{\Sigma}_K\}$,
and $I_L=I_{\Sigma_L}\oplus I_{\overline{\Sigma}_L}$, with
$I_{\Sigma_L}:=\mathbb Z[\Sigma_L]$ and
$I_{\overline{\Sigma}_L}:=\mathbb Z[\overline{\Sigma}_L]$.
We use the analogous notation for $K$ and also regard a CM type
as the element obtained by summing its embeddings in $I_K$ or $I_L$.

We now realize these types on the Hodge-theoretic objects associated
with the Serre tensor construction. Let $K^{\mathrm{Gal}}$ be the
Galois closure of $K$ in $\overline{\mathbb Q}$. We work over
$S=\operatorname{Spec}R$, where $R$ is an algebra over
$\mathcal O_{K^{\mathrm{Gal}}}[1/d_K]$.
The trace pairing identifies
$\operatorname{Hom}_{\mathbb Z}(\mathcal O_K,\mathbb Z)$ with the
inverse different $\mathfrak d_K^{-1}$. Since $d_K$ is invertible
in $R$, the algebra $\mathcal O_K\otimes_{\mathbb Z}R$ is finite
\'{e}tale over $R$ and its trace pairing is perfect. Thus, for every
$\mathcal O_K\otimes_{\mathbb Z}R$-module $M$, there is a canonical
identification
\[
  \operatorname{Hom}_R(M,R)
  \cong
  \operatorname{Hom}_{\mathcal O_K\otimes_{\mathbb Z}R}
  (M,\mathcal O_K\otimes_{\mathbb Z}R).
\]
Here the module action on the $R$-linear dual is
$(a\cdot\varphi)(m):=\varphi(am)$. If $M$ is finite projective over
$\mathcal O_K\otimes_{\mathbb Z}R$, the resulting
$\mathcal O_K\otimes_{\mathbb Z}R$-valued pairing is perfect.

The $\mathcal O_K$-actions on the Hodge-theoretic modules below
induce algebraic actions of the split $R$-torus
$\mathbb T_{K,R}:=(\operatorname{Res}_{\mathcal O_K/\mathbb Z}\mathbb G_m)_R$.
Restriction to $\mathcal O_K^\times$ recovers the corresponding
unit actions. For an algebraic $\mathbb T_{K,R}$-module $M$ and
$J\subset J_K$, write
$M(J):=\bigoplus_{\sigma\in J}M(\sigma)$, where $M(\sigma)$ is the
$\sigma$-weight space. For $\mu\in I_K^+\cap\mathbb Z[J]$, set
\[
  \operatorname{TSym}_R^\mu(M(J))
  :=\bigotimes_{\sigma\in J}
  \operatorname{TSym}_R^{\mu(\sigma)}(M(\sigma)).
\]
If $m\in M(J)$ has components $m(\sigma)$, put
$m^{[\mu]}:=\bigotimes_{\sigma\in J}m(\sigma)^{[\mu(\sigma)]}$,
where $m(\sigma)^{[k]}$ is represented by
$m(\sigma)^{\otimes k}$ and $m(\sigma)^{[0]}:=1$.
The tensor-symmetric algebra carries divided powers satisfying
$m^k=k!\,m^{[k]}$. For every $r\geq 0$, the weight decomposition gives
\[
  \operatorname{TSym}_R^r(M(J))
  \cong
  \bigoplus_{\substack{\mu\in I_K^+\cap\mathbb Z[J]\\|\mu|=r}}
  \operatorname{TSym}_R^\mu(M(J)).
\]

Write
$H^1_{\mathrm{dR}}(\mathcal A_0/S):=
R^1\pi_{0*}\Omega^\bullet_{\mathcal A_0/S}$.
The first de Rham Chern class of the Poincaré bundle induces a
canonical perfect pairing between
$H^1_{\mathrm{dR}}(\mathcal A_0/S)$ and
$H^1_{\mathrm{dR}}(\mathcal A_0^\vee/S)$ with values in
$\mathcal O_S$, denoted by $\langle-,-\rangle_{\mathrm{dR}}$.
We use the same convention for $\mathcal A$.

\begin{defn}\label{def_2}\rm
  For an abelian scheme $\mathcal A/S$, put
  \[
    \mathcal H_{\mathcal A}
    :=\mathcal H om_{\mathcal O_S}
    \bigl(H^1_{\mathrm{dR}}(\mathcal A/S),\mathcal O_S\bigr)
    \cong H^1_{\mathrm{dR}}(\mathcal A^\vee/S).
  \]
  We write $\mathcal H_0:=\mathcal H_{\mathcal A_0}$ and, when no
  ambiguity can arise, $\mathcal H:=\mathcal H_{\mathcal A}$.
  The Hodge filtration gives an exact sequence
  \begin{equation}\label{eq:Hodge-decomp}
    0\longrightarrow\omega_{\mathcal A^\vee/S}
    \longrightarrow\mathcal H_{\mathcal A}
    \longrightarrow\operatorname{Lie}(\mathcal A/S)
    \longrightarrow 0.
  \end{equation}
  No splitting is implicit in this sequence.
\end{defn}

\begin{rem}\label{rem_2}\rm
  Over $\mathbb C$, the de Rham comparison identifies
  $\mathcal H_{\mathbb C}$ with
  $H_1(\mathcal A(\mathbb C),\mathbb C)$, whose Hodge decomposition is
  \[
    \mathcal H_{\mathbb C}
    =\mathcal H^{-1,0}\oplus\mathcal H^{0,-1}.
  \]
  It is compatible with \eqref{eq:Hodge-decomp}, with canonical
  identifications
  $\mathcal H^{-1,0}\cong\operatorname{Lie}(\mathcal A/\mathbb C)$
  and $\mathcal H^{0,-1}\cong\omega_{\mathcal A^\vee/\mathbb C}$.
\end{rem}

\begin{prop}[Compatibility of the CM type with the Hodge filtration]
\label{prop_5}
  Assume that $R\subset\mathbb C$. There exists a CM type
  $\Sigma_K\subset J_K$ such that
  \[
    \mathcal H_0
    \cong\mathcal H_0(\Sigma_K)\oplus
    \mathcal H_0(\overline{\Sigma}_K),
  \]
  and the inclusion and quotient maps in the Hodge filtration induce
  isomorphisms
  \[
    \omega_{\mathcal A_0^\vee/R}
    \cong\mathcal H_0(\overline{\Sigma}_K),
    \qquad
    \mathcal H_0(\Sigma_K)
    \cong\operatorname{Lie}(\mathcal A_0/R).
  \]
  In particular, the Hodge filtration splits as a sequence of
  $\mathbb T_{K,R}$-modules.
\end{prop}

\begin{proof}
  By the decomposition theorem in
  \cite[Expos\'{e}~I, Proposition~4.7.3]{SGA3_1}, the Hodge filtration
  decomposes into weight spaces. Since $R\subset\mathbb C$ is a
  domain, the ranks of these locally free weight spaces can be
  computed after extension to $\mathbb C$. We have
  \[
    \mathcal H_0\otimes_R\mathbb C
    \cong H_1(\mathcal A_0(\mathbb C),\mathbb C)
    \cong K\otimes_{\mathbb Q}\mathbb C
    \cong\bigoplus_{\sigma\in J_K}\mathbb C,
  \]
  where the summand indexed by $\sigma$ has weight $\sigma$.
  Thus every embedding occurs exactly once. Complex conjugation
  exchanges the two Hodge summands, so the weights of
  $\operatorname{Lie}(\mathcal A_0/\mathbb C)$ form a CM type
  $\Sigma_K$, while those of
  $\omega_{\mathcal A_0^\vee/\mathbb C}$ form
  $\overline{\Sigma}_K$. Applying the weight decomposition to
  \eqref{eq:Hodge-decomp} gives the asserted identifications and
  splitting over $R$.
\end{proof}

From now on, assume $R\subset\mathbb C$ and fix the CM type
$\Sigma_K$ supplied by Proposition~\ref{prop_5}, together with its
lift $\Sigma_L$. Over $\mathbb C$, the Hodge decomposition also
identifies $\overline{\omega_{\mathcal A_0^\vee/\mathbb C}}$
with $\operatorname{Lie}(\mathcal A_0/\mathbb C)$.
On $\omega_{\mathcal A_0/R}$, we use the left unit action associated
with the right pullback action fixed above, namely
$a\cdot\eta=(a^{-1})^*\eta$. Its torus action is contragredient to
that on $\operatorname{Lie}(\mathcal A_0/R)$. This convention is
separate from the underlying $\mathcal O_K$-module structure used
in trace duality and in the Serre tensor products.

\begin{cor}\label{cor_1}
  There is a decomposition of $\mathbb T_{K,R}$-modules
  \[
    \omega_{\mathcal A_0/R}
    \cong\bigoplus_{\sigma\in\Sigma_K}
    \omega_{\mathcal A_0/R}(\sigma^{-1}),
  \]
  where $\sigma^{-1}$ denotes the inverse torus character.
  Moreover, there is an isomorphism of $\mathcal O_K^\times$-modules
  \[
    \omega_{\mathcal A_0/R}^{g}
    \cong\omega_{\mathcal A_0^\vee/R}^{g}.
  \]
\end{cor}

\begin{proof}
  Dualizing Proposition~\ref{prop_5} gives the first assertion.
  For the second, Grothendieck duality and the identification
  $R^g\pi_{0*}\mathcal O_{\mathcal A_0}
  \cong\bigwedge_R^g\operatorname{Lie}(\mathcal A_0^\vee/R)$
  give an isomorphism of the underlying determinant lines.
  Their torus characters differ by $N_{K/\mathbb Q}$, which is
  trivial on $\mathcal O_K^\times$ because $K$ is a CM field.
  The resulting isomorphism is therefore
  $\mathcal O_K^\times$-equivariant.
\end{proof}

Over $\mathbb C$, integration defines the usual period pairing
$(\eta,\gamma)\longmapsto\int_\gamma\eta$ between
$H^1_{\mathrm{dR}}(\mathcal A_0/\mathbb C)$ and
$H_1(\mathcal A_0(\mathbb C),\mathbb C)$.
Trace duality gives canonical $K\otimes_{\mathbb Q}\mathbb C$-valued
refinements of this pairing and of the de Rham pairing, denoted
by $\langle-,-\rangle_{\mathcal A_0}$ and
$\langle-,-\rangle_{\mathrm{dR},\mathcal A_0}$, respectively.
Their compositions with the trace recover the usual scalar-valued
pairings.

\begin{prop}[Refined period pairings]
\label{prop: refined period pairings}
  Put $\mathbb C^{\Sigma_K}:=
  (K\otimes_{\mathbb Q}\mathbb C)(\Sigma_K)
  \cong\prod_{\sigma\in\Sigma_K}\mathbb C$, and define
  $\mathbb C^{\overline{\Sigma}_K}$,
  $\mathbb C^{\Sigma_L}$, and
  $\mathbb C^{\overline{\Sigma}_L}$ similarly.
  The refined pairings restrict to
  \[
    \langle-,-\rangle_{\mathcal A_0}\colon
    \omega_{\mathcal A_0/\mathbb C}
    \times H_1(\mathcal A_0(\mathbb C),\mathbb Z)
    \longrightarrow\mathbb C^{\Sigma_K}
  \]
  and
  \[
    \langle-,-\rangle_{\mathrm{dR},\mathcal A_0}\colon
    \overline{\omega_{\mathcal A_0/\mathbb C}}
    \times\omega_{\mathcal A_0^\vee/\mathbb C}
    \longrightarrow\mathbb C^{\overline{\Sigma}_K}.
  \]
  The analogous restrictions of the refined pairings
  $\langle-,-\rangle_{\mathcal A}$ and
  $\langle-,-\rangle_{\mathrm{dR},\mathcal A}$ for
  $\mathcal A=\mathcal O_L\otimes_{\mathcal O_K}\mathcal A_0$
  take values in $\mathbb C^{\Sigma_L}$ and
  $\mathbb C^{\overline{\Sigma}_L}$, respectively.
\end{prop}

\begin{proof}
  The Hodge decomposition, Proposition~\ref{prop_5}, and
  Corollary~\ref{cor_1} identify the indicated factors of the
  trace-valued pairings. The assertions for $\mathcal A$ follow
  from the Serre tensor construction and the definition of the
  lifted type $\Sigma_L$.
\end{proof}

The Serre tensor construction gives the compatible decompositions
\[
  \omega_{\mathcal A^\vee/R}
  \cong\bigoplus_{\sigma\in\Sigma_K}
  \bigl(\omega_{\mathcal A_0^\vee/R}(\overline\sigma)
  \otimes_{\mathcal O_K}\mathcal O_L\bigr)\qquad\text{and}\qquad
  \omega_{\mathcal A/R}
  \cong\bigoplus_{\sigma\in\Sigma_K}
  \bigl(\omega_{\mathcal A_0/R}(\sigma^{-1})
  \otimes_{\mathcal O_K}\mathcal O_L^\vee\bigr),
\]
where $\mathcal O_L^\vee$ is the relative trace dual fixed above.
The rank-one components admit generators after shrinking
$\operatorname{Spec}R$ or making a suitable base change if necessary.
In the indexed formulas below, $(m_1,\ldots,m_n)$ denotes an
$\mathcal O_K$-basis of $\mathcal O_L$, with dual basis
$(m_1^\vee,\ldots,m_n^\vee)$ of $\mathcal O_L^\vee$.
Since $\mathcal O_L$ need only be projective over $\mathcal O_K$,
these formulas are understood in local frames when a global
basis is unavailable.

\begin{defn}[Algebraic bases]\label{def: algebraic bases}\rm
  A pair of \textit{algebraic bases}
  $(\omega(\mathcal A_0),\omega(\mathcal A_0^\vee))$ consists of
  generators $\omega(\mathcal A_0)(\sigma^{-1})$ of
  $\omega_{\mathcal A_0/R}(\sigma^{-1})$ and
  $\omega(\mathcal A_0^\vee)(\bar{\sigma})$ of
  $\omega_{\mathcal A_0^\vee/R}(\bar{\sigma})$,
  for $\sigma\in\Sigma_K$.
  The induced Serre bases of $\omega_{\mathcal A/R}$ and
  $\omega_{\mathcal A^\vee/R}$ are, respectively,
  $\omega(\mathcal A)(\sigma^{-1},m_i^\vee)
  :=\omega(\mathcal A_0)(\sigma^{-1})\otimes_{\mathcal O_K}m_i^\vee$
  and
  $\omega(\mathcal A^\vee)(\bar{\sigma},m_i)
  :=\omega(\mathcal A_0^\vee)(\bar{\sigma})\otimes_{\mathcal O_K}m_i$,
  for $\sigma\in\Sigma_K$ and $1\leq i\leq n$.
\end{defn}

Dualizing the chosen basis of invariant differentials gives a
trivialization
\begin{equation}\label{eq_1}
  \omega(\mathcal A_0)\colon
  \operatorname{Lie}(\mathcal A_0/R)
  \xrightarrow{\sim}
  (\mathcal O_K\otimes_{\mathbb Z}R)(\Sigma_K).
\end{equation}
Tensoring with $\mathcal O_L$ over $\mathcal O_K$ gives
\begin{equation}\label{eq_2}
  \omega(\mathcal A)\colon
  \operatorname{Lie}(\mathcal A/R)
  \xrightarrow{\sim}
  \mathcal O_L\otimes_{\mathcal O_K}
  (\mathcal O_K\otimes_{\mathbb Z}R)(\Sigma_K).
\end{equation}
The choice of the $m_i$ supplies coordinates but is not needed
for the intrinsic trivialization \eqref{eq_2}. In particular,
this construction does not require
$\mathcal O_L\otimes_{\mathbb Z}R$ to split over $R$.
The analogous observation applies to
$\omega(\mathcal A_0^\vee)$ and $\omega(\mathcal A^\vee)$.
By Proposition~\ref{prop_5}, specifying the pair
$(\omega(\mathcal A_0),\omega(\mathcal A_0^\vee))$ is equivalent
to specifying a basis of $\mathcal H_0$ compatible with its
CM-type decomposition and Hodge filtration; the same holds for
$\mathcal A$.

After base change to $\mathbb C$, \eqref{eq_1} identifies
$\operatorname{Lie}(\mathcal A_0/\mathbb C)$ with
$\mathbb C^{\Sigma_K}$. The exponential sequence
\[
  0\longrightarrow H_1(\mathcal A_0(\mathbb C),\mathbb Z)
  \longrightarrow\operatorname{Lie}(\mathcal A_0/\mathbb C)
  \xrightarrow{\exp}\mathcal A_0(\mathbb C)
  \longrightarrow 0
\]
identifies the image of integral homology with a period lattice
$\Lambda_0\subset\mathbb C^{\Sigma_K}$, giving
\begin{equation}\label{eq_3}
  \mathcal A_0(\mathbb C)
  \cong\mathbb C^{\Sigma_K}/\Lambda_0.
\end{equation}
Similarly, the identifications
$\mathcal O_L\otimes_{\mathcal O_K,\sigma}\mathbb C
\cong\bigoplus_{\substack{\tau\in J_L\\\tau|_K=\sigma}}\mathbb C$
identify the target of \eqref{eq_2} over $\mathbb C$ with
$\mathbb C^{\Sigma_L}$. Let $\Lambda\subset\mathbb C^{\Sigma_L}$
be the image of $H_1(\mathcal A(\mathbb C),\mathbb Z)$ under the
resulting identification of $\operatorname{Lie}(\mathcal A/\mathbb C)$.
Then
\begin{equation}\label{eq_4}
  \mathcal A(\mathbb C)
  \cong\mathbb C^{\Sigma_L}/\Lambda.
\end{equation}

\section{Algebraic construction of equivariant Eisenstein--Kronecker classes}

This section develops the algebraic framework for the equivariant
Eisenstein--Kronecker classes used throughout the paper. We first recall
the construction and basic functoriality of the equivariant coherent
Eisenstein--Kronecker classes of Kings--Sprang. We then specialize to the
Serre tensor abelian scheme
\(
\mathcal A=\mathcal O_L\otimes_{\mathcal O_K}\mathcal A_0
\)
and, in the rank-two case \(g=1\), \(n=2\), describe the cuspwise
decomposition relevant to the Bianchi boundary. Finally, we construct
the specialization at \(\mathfrak F\)-torsion sections using isogeny
functoriality, thereby obtaining the algebraic classes that will be used
in the analytic computations of the next section.

\subsection{The equivariant coherent Eisenstein--Kronecker classes}
\label{Section: The equivariant coherent Eisenstein--Kronecker classes}

In this subsection, we briefly recall the construction of the equivariant coherent Eisenstein--Kronecker classes developed by Kings--Sprang in \cite[Section~2]{KS_2025}. We review only the notation and the basic properties of the construction that will be needed later, in a form suited to the equivariant setting considered in this paper.

Let $\pi:\mathcal A\to S$ be an abelian scheme of relative dimension $d$
over a Noetherian scheme $S$, with zero section $e:S\to\mathcal A$.
Let $\pi^\vee:\mathcal A^\vee\to S$ be its dual abelian scheme, with
zero section $e^\vee:S\to\mathcal A^\vee$. We denote by $\mathcal P$
the normalized Poincaré bundle on $\mathcal A\times_S\mathcal A^\vee$,
equipped with its canonical rigidifications
$(e\times\operatorname{id}_{\mathcal A^\vee})^*\mathcal P
\cong\mathcal O_{\mathcal A^\vee}$ and
$(\operatorname{id}_{\mathcal A}\times e^\vee)^*\mathcal P
\cong\mathcal O_{\mathcal A}$.

Let $\mathcal J^\vee\subset\mathcal O_{\mathcal A^\vee}$ be the ideal
sheaf of $e^\vee(S)$. For $n\geq0$, let $\mathcal A^{\vee(n)}$ be the
$n$-th infinitesimal neighborhood of $e^\vee(S)$, defined by
$(\mathcal J^\vee)^{n+1}$, and write
$\pi^{\vee(n)}:\mathcal A^{\vee(n)}\to S$ for its structure morphism.
We denote the formal completion of $\mathcal A^\vee$ along $e^\vee(S)$
by $\widehat{\mathcal A}^{\vee}$, with structure morphism
$\widehat\pi^\vee:\widehat{\mathcal A}^{\vee}\to S$.
The infinitesimal and completed Poincaré bundles are defined by
\[
  \mathcal P^{(n)}
  :=
  \bigl(\operatorname{id}_{\mathcal A}\times\pi^{\vee(n)}\bigr)_*
  \left(
    \left.\mathcal P\right|_{\mathcal A\times_S\mathcal A^{\vee(n)}}
  \right),
  \qquad
  \widehat{\mathcal P}
  :=
  \varprojlim_{n\geq0}\mathcal P^{(n)};
\]
see \cite[Definition~2.7]{KS_2025}.
Thus $\widehat{\mathcal P}$ is the completion of $\mathcal P$ along
$\operatorname{id}_{\mathcal A}\times e^\vee:
\mathcal A\hookrightarrow\mathcal A\times_S\mathcal A^\vee$,
viewed as a sheaf on $\mathcal A$ with its natural
$\mathcal O_{\mathcal A\times_S\widehat{\mathcal A}^{\vee}}$-module
structure.

The rigidification gives
$\mathcal P^{(0)}\cong\mathcal O_{\mathcal A}$.
Writing
$\omega_{\mathcal A^\vee/S}:=(e^\vee)^*\Omega^1_{\mathcal A^\vee/S}$
for the sheaf of invariant differentials, we have, for every $n\geq1$,
the canonical short exact sequence
\[
  0
  \longrightarrow
  \pi^*\operatorname{Sym}^n_{\mathcal O_S}
  \bigl(\omega_{\mathcal A^\vee/S}\bigr)
  \longrightarrow
  \mathcal P^{(n)}
  \longrightarrow
  \mathcal P^{(n-1)}
  \longrightarrow
  0.
\]
In particular, each $\mathcal P^{(n)}$ is a finite locally free
$\mathcal O_{\mathcal A}$-module.

\begin{prop}[Comultiplication
  {\cite[Proposition~2.12]{KS_2025}}]
  \label{prop: comultiplication of Poincare bundles}
For all $m,n\geq0$, the group law on $\mathcal A^\vee$, together
with the theorem of the square for $\mathcal P$, induces canonical
homomorphisms
\[
  \mathcal P^{(m+n)}
  \longrightarrow
  \mathcal P^{(m)}
  \otimes_{\mathcal O_{\mathcal A}}
  \mathcal P^{(n)}.
\]
These homomorphisms are compatible with the transition maps and
therefore induce a comultiplication
\[
  \widehat{\mathcal P}
  \longrightarrow
  \widehat{\mathcal P}
  \widehat{\otimes}_{\mathcal O_{\mathcal A}}
  \widehat{\mathcal P},
\]
where the completed tensor product denotes
$\varprojlim_{m,n}
(\mathcal P^{(m)}\otimes_{\mathcal O_{\mathcal A}}\mathcal P^{(n)})$.
This comultiplication is cocommutative.
Under the identification
$\operatorname{gr}\widehat{\mathcal P}
\cong
\pi^*\operatorname{Sym}^{\bullet}_{\mathcal O_S}
(\omega_{\mathcal A^\vee/S})$
for the infinitesimal filtration, its associated graded is induced
by the diagonal
$\omega_{\mathcal A^\vee/S}
\to
\omega_{\mathcal A^\vee/S}\oplus\omega_{\mathcal A^\vee/S}$.
\end{prop}

\begin{cor}[{\cite[Corollary~2.14]{KS_2025}}]
  \label{cor_5}
For every $n\geq0$, iterating the cocommutative comultiplication
gives a canonical homomorphism
\[
  \mathcal P^{(n)}
  \longrightarrow
  \operatorname{TSym}^n_{\mathcal O_{\mathcal A}}
  \bigl(\mathcal P^{(1)}\bigr).
\]
If $n!$ is invertible on $S$, this homomorphism is an isomorphism.
\end{cor}

Under the rigidification
$e^*\widehat{\mathcal P}
\cong
(\widehat\pi^\vee)_*
\mathcal O_{\widehat{\mathcal A}^{\vee}}$,
pulling the comultiplication back along $e$ identifies it with
the coproduct induced by the formal group law on
$\widehat{\mathcal A}^{\vee}$, which we use below to construct
the moment maps at the zero section.

The rigidification along
$e\times\operatorname{id}_{\mathcal A^\vee}$ induces canonical
identifications
\[
  e^*\mathcal P^{(n)}
  \cong
  (\pi^{\vee(n)})_*
  \mathcal O_{\mathcal A^{\vee(n)}},
  \qquad
  e^*\widehat{\mathcal P}
  \cong
  (\widehat\pi^\vee)_*
  \mathcal O_{\widehat{\mathcal A}^{\vee}}.
\]
We use these identifications to define moment maps at the zero section.

The group law on $\mathcal A^\vee$ induces compatible morphisms
$\mathcal A^{\vee(m)}\times_S\mathcal A^{\vee(n)}
\to\mathcal A^{\vee(m+n)}$, and hence cocommutative coproducts on the
corresponding $\mathcal O_S$-algebras.
The first-order algebra has the canonical augmentation splitting
$(\pi^{\vee(1)})_*\mathcal O_{\mathcal A^{\vee(1)}}
\cong\mathcal O_S\oplus\omega_{\mathcal A^\vee/S}$.
Iterating the coproduct and using cocommutativity together with this
splitting gives the finite-level moment homomorphisms
\[
  \operatorname{mom}_{e,n}:
  e^*\mathcal P^{(n)}
  \cong
  (\pi^{\vee(n)})_*
  \mathcal O_{\mathcal A^{\vee(n)}}
  \longrightarrow
  \bigoplus_{b=0}^{n}
  \operatorname{TSym}^b_{\mathcal O_S}
  \bigl(\omega_{\mathcal A^\vee/S}\bigr).
\]
Here
$\operatorname{TSym}^b_{\mathcal O_S}(M)
:=(M^{\otimes_{\mathcal O_S}b})^{\mathfrak S_b}$
denotes the symmetric tensors of a finite locally free
$\mathcal O_S$-module $M$. Set
$\widehat{\operatorname{TSym}}_{\mathcal O_S}
(\omega_{\mathcal A^\vee/S})
:=
\prod_{b\geq0}\operatorname{TSym}^b_{\mathcal O_S}
(\omega_{\mathcal A^\vee/S})$,
viewed as the inverse limit of the truncations
$\bigoplus_{b=0}^{n}\operatorname{TSym}^b_{\mathcal O_S}
(\omega_{\mathcal A^\vee/S})$.
The finite-level moment homomorphisms are compatible with the
transition maps and therefore induce the completed moment map
\[
  \operatorname{mom}_e:
  e^*\widehat{\mathcal P}
  \cong
  (\widehat\pi^\vee)_*
  \mathcal O_{\widehat{\mathcal A}^{\vee}}
  \longrightarrow
  \widehat{\operatorname{TSym}}_{\mathcal O_S}
  \bigl(\omega_{\mathcal A^\vee/S}\bigr).
\]
For every $b\geq0$, we denote its projection to the $b$-th component by
$\operatorname{mom}_e^b:
e^*\widehat{\mathcal P}\to
\operatorname{TSym}^b_{\mathcal O_S}
(\omega_{\mathcal A^\vee/S})$.

The functoriality of the infinitesimal and completed Poincaré bundles
under isogenies provides both their equivariant structures and the
splittings used to define moment maps at torsion sections.

\begin{thm}[Functoriality under isogenies
  {\cite[Theorem~2.8]{KS_2025}}]
  \label{thm: Functoriality of Poincaré bundles}
Let $\varphi:\mathcal A\to\mathcal B$ be an isogeny of abelian schemes
over $S$, with dual
$\varphi^\vee:\mathcal B^\vee\to\mathcal A^\vee$.
For every $n\geq0$, there is a canonical morphism
\[
  \varphi_\#^{(n)}:
  \mathcal P_{\mathcal A}^{(n)}
  \longrightarrow
  \varphi^*\mathcal P_{\mathcal B}^{(n)}.
\]
These morphisms are compatible with the transition maps in $n$ and
induce a canonical morphism
\[
  \varphi_\#:
  \widehat{\mathcal P}_{\mathcal A}
  \longrightarrow
  \varphi^*\widehat{\mathcal P}_{\mathcal B}.
\]
If $\varphi^\vee$ is étale, then $\varphi_\#^{(n)}$ and
$\varphi_\#$ are isomorphisms. This hypothesis holds, for example,
when $\deg(\varphi)$ is invertible on $S$.
The functoriality morphisms are compatible with composition of
isogenies.
\end{thm}

Applying this functoriality to automorphisms gives the equivariant
structures needed below.

\begin{cor}[Equivariant structure
  {\cite[Corollary~2.11]{KS_2025}}]
Let $\Gamma$ be a discrete group acting from the left on
$\mathcal A/S$ by automorphisms of abelian schemes.
Then $\mathcal P^{(n)}$ and $\widehat{\mathcal P}$ carry canonical
$\Gamma$-equivariant structures, compatible with the transition maps.
For every $\gamma\in\Gamma$, the structure isomorphism on
$\widehat{\mathcal P}$ is
$\gamma_\#^{-1}:
\gamma^*\widehat{\mathcal P}\xrightarrow{\sim}\widehat{\mathcal P}$,
and these isomorphisms satisfy the cocycle condition.
\end{cor}

Let $\varphi:\mathcal A\to\mathcal B$ be an isogeny of abelian
schemes over $S$. A section $x:S\to\mathcal A$ is called a
\textit{$\varphi$-torsion section} if it factors through
$\ker(\varphi)$, or equivalently, if
$\varphi\circ x=e_{\mathcal B}$, where
$e_{\mathcal B}:S\to\mathcal B$ is the zero section.
Restricting the isogeny functoriality to $\ker(\varphi)$ gives
the following splitting principle.

\begin{cor}[Splitting principle
  {\cite[Corollary~2.9]{KS_2025}}]
  \label{cor: splitting principle}
Let $\varphi:\mathcal A\to\mathcal B$ be an isogeny whose dual
$\varphi^\vee$ is étale, and let
$\pi_{\ker(\varphi)}:\ker\varphi\to S$ be the structure morphism.
There is a canonical isomorphism
\[
  \left.\widehat{\mathcal P}_{\mathcal A}\right|_{\ker\varphi}
  \cong
  \pi_{\ker(\varphi)}^*
  \bigl(\widehat\pi_{\mathcal B}^{\vee}\bigr)_*
  \mathcal O_{\widehat{\mathcal B}^{\vee}},
\]
where
$\widehat\pi_{\mathcal B}^{\vee}:
\widehat{\mathcal B}^{\vee}\to S$
is the structure morphism of the formal completion of
$\mathcal B^\vee$ along its zero section.
The analogous splitting holds at every finite level.

For every $\varphi$-torsion section $x:S\to\mathcal A$, there are canonical isomorphisms
\[
  \begin{aligned}
    \rho_x^{(n)}:
    x^*\mathcal P_{\mathcal A}^{(n)}
    &\xrightarrow{\sim}
    \bigl(\pi_{\mathcal A}^{\vee(n)}\bigr)_*
    \mathcal O_{\mathcal A^{\vee(n)}},
    \qquad
    \rho_x:
    x^*\widehat{\mathcal P}_{\mathcal A}
    &\xrightarrow{\sim}
    \bigl(\widehat\pi_{\mathcal A}^{\vee}\bigr)_*
    \mathcal O_{\widehat{\mathcal A}^{\vee}}.
  \end{aligned}
\]
The finite-level isomorphisms are compatible with the transition
maps and induce $\rho_x$ on passing to the inverse limit.
\end{cor}

Using the rigidification at the zero section, these splittings
transport the moment maps defined above to $x$.

\begin{defn}[Moment maps at torsion sections
  {\cite[Definition~2.10]{KS_2025}}]
  \label{def: moment map x}\rm
Let $\varphi$ and $x$ be as in
Corollary~\ref{cor: splitting principle}.
For every $n\geq0$, the finite-level moment map at $x$ is
\[
  \operatorname{mom}_{x,n}
  :=
  \operatorname{mom}_{e,n}\circ\rho_x^{(n)}:
  x^*\mathcal P^{(n)}
  \longrightarrow
  \bigoplus_{b=0}^{n}
  \operatorname{TSym}^b_{\mathcal O_S}
  \bigl(\omega_{\mathcal A^\vee/S}\bigr).
\]
The completed moment map at $x$ is
\[
  \operatorname{mom}_x
  :=
  \operatorname{mom}_e\circ\rho_x:
  x^*\widehat{\mathcal P}
  \longrightarrow
  \widehat{\operatorname{TSym}}_{\mathcal O_S}
  \bigl(\omega_{\mathcal A^\vee/S}\bigr).
\]
For $0\leq b\leq n$, we denote by
$\operatorname{mom}_{x,n}^b$ the projection of
$\operatorname{mom}_{x,n}$ to
$\operatorname{TSym}^b_{\mathcal O_S}
(\omega_{\mathcal A^\vee/S})$.
For every $b\geq0$, we similarly denote the $b$-th component of
$\operatorname{mom}_x$ by $\operatorname{mom}_x^b$.
The finite-level moment maps are compatible with the transition
maps and induce the completed moment map.
\end{defn}

If $\mathcal B$ also carries a $\Gamma$-action, $\varphi$ is
$\Gamma$-equivariant, and $x$ is $\Gamma$-invariant, then the
splittings $\rho_x^{(n)}$, $\rho_x$ and the corresponding moment
maps are $\Gamma$-equivariant.
Moreover, if $\psi:\mathcal B\to\mathcal C$ is an isogeny with
étale dual, then $x$ is also a $(\psi\circ\varphi)$-torsion section.
Replacing $\varphi$ by $\psi\circ\varphi$ leaves the splittings and
moment maps unchanged, by compatibility of the functoriality
morphisms with composition.

The construction of the equivariant coherent Eisenstein--Kronecker
class uses the following cohomological vanishing theorem.

\begin{thm}[Vanishing of cohomology
  {\cite[Theorem~2.15]{KS_2025}}]
  \label{thm: Vanishing of cohomology}
There are canonical isomorphisms
\[
  R^q\pi_*
  \left(
    \widehat{\mathcal P}
    \otimes_{\mathcal O_{\mathcal A}}
    \Omega_{\mathcal A/S}^d
  \right)
  \cong
  \begin{cases}
    \mathcal O_S, & q=d,\\
    0, & q\neq d.
  \end{cases}
\]
\end{thm}

\begin{cor}[Equivariant vanishing
  {\cite[Corollary~2.17]{KS_2025}}]
  \label{cor: vanishing of equiv coh}
For $0\leq q\leq d$, there are canonical isomorphisms
\[
  H^q\left(
    \mathcal A,\Gamma;
    \widehat{\mathcal P}\otimes\Omega_{\mathcal A/S}^d
  \right)
  \cong
  \begin{cases}
    H^0(S,\mathcal O_S), & q=d,\\
    0, & q<d.
  \end{cases}
\]
\end{cor}

For the localization construction, we use the following notation
for finite étale supports and trace-zero functions.

\begin{defn}[Trace-zero functions
  {\cite[Definition~2.18]{KS_2025}}]\rm
Let $\mathcal D\subset\mathcal A$ be a $\Gamma$-stable closed
subscheme which is finite étale over $S$. We write
\[
  \mathcal O_S[\mathcal D]
  :=
  H^0(\mathcal D,\mathcal O_{\mathcal D})
\]
for the ring of global functions on $\mathcal D$. The finite étale
morphism $\mathcal D\to S$ induces a trace map
\[
  \operatorname{Tr}_{\mathcal D/S}:
  \mathcal O_S[\mathcal D]
  \longrightarrow
  H^0(S,\mathcal O_S).
\]
The submodule of \emph{trace-zero functions} is defined by
\[
  \mathcal O_S[\mathcal D]^0
  :=
  \ker\bigl(\operatorname{Tr}_{\mathcal D/S}\bigr).
\]
Since $\mathcal D$ is $\Gamma$-stable, this submodule is preserved by
the induced $\Gamma$-action, and we denote its submodule of
$\Gamma$-invariants by
\[
  \mathcal O_S[\mathcal D]^{0,\Gamma}
  :=
  \bigl(\mathcal O_S[\mathcal D]^0\bigr)^\Gamma.
\]
\end{defn}

\begin{notation}\label{notation: cond of D}\rm
Let $\mathcal A'/S$ be an abelian scheme equipped with a
$\Gamma$-action, and let
$\delta:\mathcal A\to\mathcal A'$ be a $\Gamma$-equivariant étale
isogeny whose dual $\delta^\vee$ is étale.
Let $\mathcal D$ be either $\ker\delta$ or
$\ker\delta\setminus e(S)$.
We denote the closed immersion by
$i:\mathcal D\hookrightarrow\mathcal A$ and the finite étale
structure morphism by $\pi_{\mathcal D}:\mathcal D\to S$.
Set $\mathcal U_{\mathcal D}:=\mathcal A\setminus\mathcal D$,
with open immersion
$j:\mathcal U_{\mathcal D}\hookrightarrow\mathcal A$.
Both $\mathcal D$ and $\mathcal U_{\mathcal D}$ are
$\Gamma$-stable.
% https://q.uiver.app/#q=WzAsNSxbMCwwLCJcXG1hdGhjYWx7RH0iXSxbMSwwLCJcXG1hdGhjYWx7QX0iXSxbMiwwLCJcXG1hdGhjYWx7VX1fe1xcbWF0aGNhbHtEfX06PSBcXG1hdGhjYWx7QX1cXGJhY2tzbGFzaFxcbWF0aGNhbHtEfSJdLFszLDBdLFsxLDEsIlMiXSxbMCwxLCJpIl0sWzIsMSwiaiIsMl0sWzAsNCwiXFxwaV9cXG1hdGhjYWx7RH0iLDJdLFsxLDQsIlxccGkiLDJdLFsyLDQsIlxccGlfe1xcbWF0aGNhbHtVfV9cXG1hdGhjYWx7RH19Il1d
\[
\begin{tikzcd}
\mathcal{D}
\arrow[r,"i"]
\arrow[dr,swap,"\pi_{\mathcal{D}}"]
&
\mathcal{A}
\arrow[d,"\pi"]
&
\mathcal{U}_{\mathcal{D}}
:=
\mathcal{A}\setminus\mathcal{D}
\arrow[l,swap,"j"]
\arrow[dl,"\pi_{\mathcal{U}_{\mathcal{D}}}"]
\\
&
S
&
\end{tikzcd}
\]
\end{notation}

Since $\delta^\vee$ is étale, the splitting principle of
Corollary~\ref{cor: splitting principle} gives a canonical
$\Gamma$-equivariant isomorphism
\[
  i^*\widehat{\mathcal P}
  \cong
  \pi_{\mathcal D}^*
  \bigl(\widehat\pi_{\mathcal A'}^\vee\bigr)_*
  \mathcal O_{\widehat{\mathcal A'}^{\vee}}.
\]
The unit of this formal function algebra therefore gives a
canonical $\Gamma$-equivariant inclusion
$\mathcal O_{\mathcal D}\hookrightarrow i^*\widehat{\mathcal P}$.
On the other hand, equivariant localization and
Corollary~\ref{cor: vanishing of equiv coh} give the exact sequence
\begin{equation}\label{eq: localization for Eisenstein-Kronecker classes}
  \begin{aligned}
    0
    &\longrightarrow
    H^{d-1}\!\left(
      \mathcal U_{\mathcal D},\Gamma;
      j^{-1}\bigl(
        \widehat{\mathcal P}\otimes\Omega_{\mathcal A/S}^d
      \bigr)
    \right)\xlongrightarrow{\partial}
    H_{\mathcal D}^d\!\left(
      \mathcal A,\Gamma;
      \widehat{\mathcal P}\otimes\Omega_{\mathcal A/S}^d
    \right)
    \longrightarrow
    H^0(S,\mathcal O_S).
  \end{aligned}
\end{equation}
Here the last map is the natural map from cohomology with support
to global equivariant cohomology, followed by the degree-$d$
identification in Corollary~\ref{cor: vanishing of equiv coh}.
The following theorem relates this map to the trace on
$\mathcal O_S[\mathcal D]^\Gamma$.

\begin{thm}[Construction of the class
  {\cite[Theorem~2.20]{KS_2025}}]
  \label{thm: function to equiv coh class}
There is a canonical inclusion
\[
  \mathcal O_S[\mathcal D]^\Gamma
  \hookrightarrow
  H_{\mathcal D}^d\!\left(
    \mathcal A,\Gamma;
    \widehat{\mathcal P}\otimes\Omega_{\mathcal A/S}^d
  \right)
\]
whose composition with the last map in
\eqref{eq: localization for Eisenstein-Kronecker classes}
is the restriction of $\operatorname{Tr}_{\mathcal D/S}$.
Consequently, localization induces a canonical inclusion
\[
  EK_{\Gamma,\mathcal A}:
  \mathcal O_S[\mathcal D]^{0,\Gamma}
  \hookrightarrow
  H^{d-1}\!\left(
    \mathcal U_{\mathcal D},\Gamma;
    j^{-1}\bigl(
      \widehat{\mathcal P}\otimes\Omega_{\mathcal A/S}^d
    \bigr)
  \right),
\]
characterized by the requirement that its composition with
$\partial$ is the restriction of the preceding inclusion.
Moreover, the projections
$\widehat{\mathcal P}\to\mathcal P^{(n)}$ induce a canonical
isomorphism
\begin{equation}\label{eq: finite-level comparison for EK}
  \begin{aligned}
    &H^{d-1}\!\left(
      \mathcal U_{\mathcal D},\Gamma;
      j^{-1}\bigl(
        \widehat{\mathcal P}\otimes\Omega_{\mathcal A/S}^d
      \bigr)
    \right)\xrightarrow{\sim}
    \varprojlim_{n\geq0}
    H^{d-1}\!\left(
      \mathcal U_{\mathcal D},\Gamma;
      j^{-1}\bigl(
        \mathcal P^{(n)}\otimes\Omega_{\mathcal A/S}^d
      \bigr)
    \right).
  \end{aligned}
\end{equation}
\end{thm}

\begin{defn}[Equivariant coherent Eisenstein--Kronecker class]
  \label{def: Eisenstein-Kronecker class f}\rm
For $f\in\mathcal O_S[\mathcal D]^{0,\Gamma}$, the image
\[
  EK_{\Gamma,\mathcal A}(f)
  \in
  H^{d-1}\!\left(
    \mathcal U_{\mathcal D},\Gamma;
    j^{-1}\bigl(
      \widehat{\mathcal P}\otimes\Omega_{\mathcal A/S}^d
    \bigr)
  \right)
\]
of $f$ under the map of
Theorem~\ref{thm: function to equiv coh class}
is called the \textit{equivariant coherent Eisenstein--Kronecker
class associated to $f$}.
\end{defn}

We first record the compatibility of the coherent class with base
change, before specializing it at torsion sections.

\begin{prop}[Base change]\label{prop_8}
Let $T$ be a Noetherian scheme endowed with the trivial
$\Gamma$-action, and let $g:T\to S$ be a $\Gamma$-equivariant
morphism. Let $\mathcal A_T$ and $\mathcal D_T$ denote the base
changes of $\mathcal A$ and $\mathcal D$ along $g$, respectively,
and write $j_T:\mathcal U_{\mathcal D_T}\hookrightarrow\mathcal A_T$
for the base change of $j$.
For $f\in\mathcal O_S[\mathcal D]^{0,\Gamma}$, set $f_T:=g^*f$.
Then $f_T\in\mathcal O_T[\mathcal D_T]^{0,\Gamma}$, and the
canonical pullback morphism
\[
  \begin{aligned}
    g^*:\;
    &H^{d-1}\!\left(
      \mathcal U_{\mathcal D},\Gamma;
      j^{-1}\bigl(
        \widehat{\mathcal P}\otimes\Omega_{\mathcal A/S}^d
      \bigr)
    \right)
    \longrightarrow
    H^{d-1}\!\left(
      \mathcal U_{\mathcal D_T},\Gamma;
      j_T^{-1}\bigl(
        \widehat{\mathcal P}_T
        \otimes\Omega_{\mathcal A_T/T}^d
      \bigr)
    \right)
  \end{aligned}
\]
satisfies
\[
  g^*\bigl(EK_{\Gamma,\mathcal A}(f)\bigr)
  =
  EK_{\Gamma,\mathcal A_T}(f_T).
\]
\end{prop}

\begin{proof}
The base changes of $\delta$ and $\delta^\vee$ are again étale,
so $\mathcal D_T$ satisfies the hypotheses of
Notation~\ref{notation: cond of D}.
Since the finite étale trace commutes with base change and the
pullback is $\Gamma$-equivariant, $f_T$ is trace zero and
$\Gamma$-invariant.
For every $n\geq0$, base change for the Poincaré bundle and relative
differentials gives canonical identifications
$g^*\mathcal P^{(n)}\cong\mathcal P_T^{(n)}$ and
$g^*\Omega_{\mathcal A/S}^d\cong\Omega_{\mathcal A_T/T}^d$,
where we also write $g^*$ for pullback along the induced
morphisms.
Under these identifications, the construction in the proof of
Theorem~\ref{thm: function to equiv coh class}
(see \cite[Section~2.6]{KS_2025}) gives a commutative diagram
\[
\begin{tikzcd}[column sep=large, row sep=large]
  {H^0(\mathcal D,\mathcal O_{\mathcal D})}
    \arrow[r,"g^*"]
    \arrow[d,hook]
  &
  {H^0(\mathcal D_T,\mathcal O_{\mathcal D_T})}
    \arrow[d,hook]
  \\
  {H_{\mathcal D}^d\!\left(
    \mathcal A,
    \mathcal P^{(n)}\otimes\Omega_{\mathcal A/S}^d
  \right)}
    \arrow[r,"g^*"]
    \arrow[d]
  &
  {H_{\mathcal D_T}^d\!\left(
    \mathcal A_T,
    \mathcal P_T^{(n)}\otimes\Omega_{\mathcal A_T/T}^d
  \right)}
    \arrow[d]
  \\
  {H^d\!\left(
    \mathcal A,
    \mathcal P^{(n)}\otimes\Omega_{\mathcal A/S}^d
  \right)}
    \arrow[r,"g^*"]
    \arrow[d]
  &
  {H^d\!\left(
    \mathcal A_T,
    \mathcal P_T^{(n)}\otimes\Omega_{\mathcal A_T/T}^d
  \right)}
    \arrow[d]
  \\
  {H^0(S,\mathcal O_S)}
    \arrow[r,"g^*"]
  &
  {H^0(T,\mathcal O_T)}.
\end{tikzcd}
\]
The upper square follows from the base-change compatibility of
$\delta_{\#}^{(n)}$, the rigidifications, and the local-cohomology
inclusion used in the construction.
The middle square is the naturality of the map from cohomology
with support to global cohomology, while the lower square follows
from the compatibility of the augmentation and the relative trace
with base change.

All these maps are $\Gamma$-equivariant and compatible with the
transition morphisms in $n$.
Taking $\Gamma$-invariants and inverse limits, and using the
canonical degree-$d$ local-cohomology identifications, gives
\[
\begin{tikzcd}[column sep=large, row sep=large]
  {\mathcal O_S[\mathcal D]^\Gamma}
    \arrow[r,hook]
    \arrow[d,"g^*"']
  &
  {H_{\mathcal D}^d\!\left(
    \mathcal A,\Gamma;
    \widehat{\mathcal P}\otimes\Omega_{\mathcal A/S}^d
  \right)}
    \arrow[d,"g^*"]
  \\
  {\mathcal O_T[\mathcal D_T]^\Gamma}
    \arrow[r,hook]
  &
  {H_{\mathcal D_T}^d\!\left(
    \mathcal A_T,\Gamma;
    \widehat{\mathcal P}_T\otimes\Omega_{\mathcal A_T/T}^d
  \right)}.
\end{tikzcd}
\]
The compatible finite-level pullback maps likewise define the
pullback in the statement through
\eqref{eq: finite-level comparison for EK}.
By this diagram and the naturality of
\eqref{eq: localization for Eisenstein-Kronecker classes},
the classes $g^*\bigl(EK_{\Gamma,\mathcal A}(f)\bigr)$ and
$EK_{\Gamma,\mathcal A_T}(f_T)$ have the same image under the
injective localization map $\partial_T$.
\end{proof}

Fix $f\in\mathcal O_S[\mathcal D]^{0,\Gamma}$.
Let $\mathcal B/S$ be an abelian scheme with a $\Gamma$-action, and
let $\varphi:\mathcal A\to\mathcal B$ be a $\Gamma$-equivariant
isogeny whose dual $\varphi^\vee$ is étale.
Let $x:S\to\mathcal A$ be a $\Gamma$-invariant $\varphi$-torsion
section factoring through $\mathcal U_{\mathcal D}$.
Using the canonical identification
$x^*\Omega_{\mathcal A/S}^d\cong\omega_{\mathcal A/S}^d$, pullback
along $x$ gives
\[
  x^*EK_{\Gamma,\mathcal A}(f)
  \in
  H^{d-1}\!\left(
    S,\Gamma;
    x^*\widehat{\mathcal P}
    \otimes
    \omega_{\mathcal A/S}^d
  \right).
\]
Under these hypotheses, the moment maps of
Definition~\ref{def: moment map x} are $\Gamma$-equivariant.
Tensoring them with the identity on $\omega_{\mathcal A/S}^d$
therefore induces maps on equivariant cohomology.

\begin{defn}[Eisenstein--Kronecker class at $x$]
  \label{def: Eisenstein-Kronecker class at x}\rm
For every $b\geq0$, the \textit{equivariant coherent
Eisenstein--Kronecker class associated to $f$ and specialized at $x$
in degree $b$} is defined by
\[
  \begin{aligned}
    EK_{\Gamma,\mathcal A}^b(f,x)
    &:=
    \operatorname{mom}_x^b
    \bigl(x^*EK_{\Gamma,\mathcal A}(f)\bigr)
    \in
    H^{d-1}\!\left(
      S,\Gamma;
      \operatorname{TSym}^b
      (\omega_{\mathcal A^\vee/S})
      \otimes
      \omega_{\mathcal A/S}^d
    \right).
  \end{aligned}
\]
\end{defn}

If $S=\operatorname{Spec}R$ is affine and the coefficient sheaf is finite locally free, the vanishing of higher cohomology for quasi-coherent sheaves identifies the target with group cohomology. Thus
\[
  EK_{\Gamma,\mathcal A}^b(f,x)
  \in
  H^{d-1}\!\left(
    \Gamma,
    H^0\!\left(
      S,
      \operatorname{TSym}^b
      (\omega_{\mathcal A^\vee/S})
      \otimes
      \omega_{\mathcal A/S}^d
    \right)
  \right).
\]
In this affine setting, we also write
$\operatorname{TSym}^b(\omega_{\mathcal A^\vee/S})
\otimes\omega_{\mathcal A/S}^d$
for its $R$-module of global sections, equipped with the induced
$\Gamma$-action.

\begin{cor}[Base change after specialization]
  \label{cor: base change of specialized EK classes}
Let $g:T\to S$ be as in Proposition~\ref{prop_8}, and put
$x_T:=g^*x$.
Then $x_T$ is a $\Gamma$-invariant $\varphi_T$-torsion section
factoring through $\mathcal U_{\mathcal D_T}$, where $\varphi_T$
denotes the base change of $\varphi$.
For every $b\geq0$, the canonical pullback satisfies
\[
  g^*\bigl(EK_{\Gamma,\mathcal A}^b(f,x)\bigr)
  =
  EK_{\Gamma,\mathcal A_T}^b(f_T,x_T).
\]
\end{cor}

\begin{proof}
The equivariance and torsion conditions, as well as the étaleness
of the dual isogeny, are preserved by base change.
Proposition~\ref{prop_8} identifies the pullback of the basic class
with $EK_{\Gamma,\mathcal A_T}(f_T)$.
Since pullback along the section and the moment maps commute with
base change, applying $\operatorname{mom}_{x_T}^b$ gives the
asserted identity.
\end{proof}

\begin{prop}[Restriction and corestriction]
  \label{prop:restriction-corestriction}
Assume that $S=\operatorname{Spec}R$ is affine, and let
$\Gamma'\subseteq\Gamma$ be a subgroup of finite index.
For the $\Gamma$-invariant data $f$ and $x$ fixed above and every
$b\geq0$, one has
\[
  \begin{aligned}
    \operatorname{res}_{\Gamma'}^\Gamma
    \bigl(EK_{\Gamma,\mathcal A}^b(f,x)\bigr)
    &=
    EK_{\Gamma',\mathcal A}^b(f,x),
    \qquad
    \operatorname{cor}_{\Gamma'}^\Gamma
    \bigl(EK_{\Gamma',\mathcal A}^b(f,x)\bigr)
    &=
    [\Gamma:\Gamma']\,
    EK_{\Gamma,\mathcal A}^b(f,x).
  \end{aligned}
\]
If $[\Gamma:\Gamma']$ is invertible in $R$, the restriction map is
injective and
\[
  EK_{\Gamma,\mathcal A}^b(f,x)
  =
  \frac{1}{[\Gamma:\Gamma']}
  \operatorname{cor}_{\Gamma'}^\Gamma
  \bigl(EK_{\Gamma',\mathcal A}^b(f,x)\bigr).
\]
In particular, when $\Gamma'$ is torsion-free, the class may be
computed after restriction to $\Gamma'$ and recovered by
corestriction after inverting the index.
\end{prop}

\begin{proof}
The construction of the basic coherent class, its pullback along
$x$, and the moment maps are natural under restriction of the
acting group. This gives the restriction identity.
Since $S$ is affine, the relevant equivariant cohomology is group
cohomology of the coefficient module, where the standard relation
$\operatorname{cor}_{\Gamma'}^\Gamma
\circ\operatorname{res}_{\Gamma'}^\Gamma
=[\Gamma:\Gamma']\operatorname{id}$
holds.
Applying this relation gives the corestriction formula.
If the index is invertible in $R$, the same relation proves the
injectivity of restriction and yields the recovery formula.
\end{proof}

\subsection{The Serre tensor setting and the rank-two boundary decomposition}

We retain the CM and Serre tensor notation of Section~\ref{Section: Serre tensor constructions of CM abelian schemes}. Let
$R\subset\C$ be an $\mathcal O_{K^{\mathrm{Gal}}}[1/d_K]$-algebra,
set $S=\Spec R$, and let $\mathcal A_0/S$ be an abelian scheme of
relative dimension $g$ with CM by $\mathcal O_K$. The Serre tensor
abelian scheme
\[
  \mathcal A
  :=
  \mathcal O_L\otimes_{\mathcal O_K}\mathcal A_0
\]
has relative dimension $d=gn$. We retain the action of
$\GL_{\mathcal O_K}(\mathcal O_L)
=\operatorname{Aut}_{\mathcal O_K}(\mathcal O_L)$
on the first tensor factor and fix a subgroup
$\Gamma\subset\GL_{\mathcal O_K}(\mathcal O_L)$.

Throughout this subsection, we specialize to $g=1$ and $n=2$, so that
$d=2$. Thus $K$ is imaginary quadratic, $\mathcal A_0/S$ is an elliptic
scheme, and $\mathcal A/S$ is an abelian surface. Using the fixed
$\mathcal O_K$-basis of $\mathcal O_L$, we identify
$\GL_{\mathcal O_K}(\mathcal O_L)$ with $\GL_2(\mathcal O_K)$ and
assume that $\Gamma$ has finite index. This is the setting of the
Bianchi boundary calculation, where the unipotent cusp groups are
lattices of rank two. We study the boundary decomposition of the
group cohomology containing the Eisenstein--Kronecker classes recalled
in the preceding subsection; the arithmetic specialization at torsion
sections will be treated separately below.

The cusps are indexed by the finite set
$\Gamma\backslash\mathbb P^1(K)$. For a representative
$c\in\mathbb P^1(K)$, choose $\sigma_c\in\GL_2(K)$ such that
$\sigma_c(\infty)=c$, and set
\[
  P_c:=\sigma_cP_\infty\sigma_c^{-1},
  \qquad
  N_c:=\sigma_cN_\infty\sigma_c^{-1},
\]
where $P_\infty\subset\GL_2(K)$ is the upper triangular parabolic
subgroup and
\[
  N_\infty
  =
  \left\{
    n(a)=
    \begin{pmatrix}
      1 & a\\
      0 & 1
    \end{pmatrix}
    \,\middle|\,
    a\in K
  \right\}.
\]
Then $P_c$ is the parabolic subgroup stabilizing $c$, and $N_c$ is
its unipotent radical. Both depend only on $c$, although the choice
of $\sigma_c$ is not canonical. Write
$\Gamma_c:=\Gamma\cap P_c$ for the cusp stabilizer and
$\Gamma_{N,c}:=\Gamma\cap N_c$ for its unipotent subgroup.
For $b\geq 0$, the inclusion $\Gamma_c\hookrightarrow\Gamma$ induces
the restriction morphism
\[
  \operatorname{res}_c:
  H^1\left(
    \Gamma,
    \mathrm{TSym}^b(\omega_{\mathcal A^\vee/R})
    \otimes\omega_{\mathcal A/R}^2
  \right)
  \longrightarrow
  H^1\left(
    \Gamma_c,
    \mathrm{TSym}^b(\omega_{\mathcal A^\vee/R})
    \otimes\omega_{\mathcal A/R}^2
  \right),
\]
which we regard as the restriction of a group cohomology class to
the cusp $c$.

Since $N_c$ is normal in $P_c$, there is an exact sequence
\[
  1
  \longrightarrow
  \Gamma_{N,c}
  \longrightarrow
  \Gamma_c
  \longrightarrow
  \Gamma_c/\Gamma_{N,c}
  \longrightarrow
  1.
\]
The associated Hochschild--Serre spectral sequence is
\[
  E_2^{p,q}
  =
  H^p\left(
    \Gamma_c/\Gamma_{N,c},
    H^q\left(
      \Gamma_{N,c},
      \mathrm{TSym}^b(\omega_{\mathcal A^\vee/R})
      \otimes\omega_{\mathcal A/R}^2
    \right)
  \right)
  \Longrightarrow
  H^{p+q}\left(
    \Gamma_c,
    \mathrm{TSym}^b(\omega_{\mathcal A^\vee/R})
    \otimes\omega_{\mathcal A/R}^2
  \right).
\]
For $\gamma\in\Gamma_c$, the diagonal entries of
$\sigma_c^{-1}\gamma\sigma_c$ lie in $\mathcal O_K^\times$:
they belong to $K$, and both they and their inverses are integral
over $\mathcal O_K$, being eigenvalues of $\gamma$ and
$\gamma^{-1}$, respectively. The diagonal map therefore embeds
$\Gamma_c/\Gamma_{N,c}$ into
$\mathcal O_K^\times\times\mathcal O_K^\times$, which is finite
because $K$ is imaginary quadratic. Assume that the order of
$\Gamma_c/\Gamma_{N,c}$ is invertible in $R$. Its higher group
cohomology with coefficients in an $R$-module then vanishes, so the
spectral sequence degenerates and restriction induces a canonical
isomorphism
\[
  H^1\left(
    \Gamma_c,
    \mathrm{TSym}^b(\omega_{\mathcal A^\vee/R})
    \otimes\omega_{\mathcal A/R}^2
  \right)
  \cong
  H^1\left(
    \Gamma_{N,c},
    \mathrm{TSym}^b(\omega_{\mathcal A^\vee/R})
    \otimes\omega_{\mathcal A/R}^2
  \right)^{\Gamma_c/\Gamma_{N,c}}.
\]
Thus the restriction of an Eisenstein--Kronecker class to $c$ may
be studied in the cohomology of $\Gamma_{N,c}$, followed by taking
$\Gamma_c/\Gamma_{N,c}$-invariants.

At the standard cusp $c=\infty$, take $\sigma_\infty=1$.
Under the identification $n(a)\mapsto a$, the group
$\Gamma_{N,\infty}$ is a finite-index sublattice of
$\mathcal O_K$, hence a free abelian group of rank two. In
particular, at full level $\Gamma=\GL_2(\mathcal O_K)$, one has
\[
  \Gamma_{N,\infty}\cong(\mathcal O_K,+),
  \qquad
  \Gamma_\infty/\Gamma_{N,\infty}
  \cong
  \mathcal O_K^\times\times\mathcal O_K^\times.
\]
After localizing $R$ so that the lattice index
$[\mathcal O_K:\Gamma_{N,\infty}]$ and the order of
$\Gamma_\infty/\Gamma_{N,\infty}$ are invertible, we work over
$R[1/b!]$. The boundary calculation at $\infty$ is then reduced to
the group
\(
  H^1\left(
    \Gamma_{N,\infty},
    \mathrm{TSym}^b(\omega_{\mathcal A^\vee/R[1/b!]})
    \otimes\omega_{\mathcal A/R[1/b!]}^2
  \right).
\)
We now compute this unipotent cohomology explicitly.

\begin{prop}\label{prop_9}
Let $S=\Spec R$, where $R$ is an
$\mathcal O_{K^{\mathrm{Gal}}}[1/d_K]$-algebra, and let $b\geq 0$.
Assume that, under the identification $n(a)\mapsto a$,
$\Gamma_{N,\infty}$ is a finite-index sublattice of $\mathcal O_K$
whose index is invertible in $R$. After base change to $R[1/b!]$,
the CM decomposition
\[
  \mathfrak n_{R[1/b!]}
  =
  \mathfrak n_\sigma\oplus\mathfrak n_{\bar\sigma}
\]
of the unipotent Lie algebra
$\mathfrak n_{R[1/b!]}=\mathcal O_K\otimes_{\mathbb Z}R[1/b!]$
induces a natural isomorphism
\[
\begin{aligned}
  H^1\left(
    \Gamma_{N,\infty},
    \mathrm{TSym}^b\left(\omega_{\mathcal A^\vee/R[1/b!]}\right)
    \otimes\omega_{\mathcal A/R[1/b!]}^2
  \right)
  &\cong
  \left(
    \mathrm{TSym}^b\left(\omega_{\mathcal A^\vee/R[1/b!]}\right)
    \otimes\omega_{\mathcal A/R[1/b!]}^2
  \right)^{\Gamma_{N,\infty}}
  \otimes_{R[1/b!]}\mathfrak n_\sigma^\vee\\
  &\qquad\qquad\oplus
  \left(
    \mathrm{TSym}^b\left(\omega_{\mathcal A^\vee/R[1/b!]}\right)
    \otimes\omega_{\mathcal A/R[1/b!]}^2
  \right)_{\Gamma_{N,\infty}}
  \otimes_{R[1/b!]}\mathfrak n_{\bar\sigma}^\vee,
\end{aligned}
\]
where $\mathfrak n_{\bar\sigma}$ is the summand acting nontrivially
on $\omega_{\mathcal A^\vee/R[1/b!]}$.
\end{prop}

\begin{proof}
If $b=0$, then
$\mathrm{TSym}^0(\omega_{\mathcal A^\vee/R})=R$.
Moreover, $\Gamma_{N,\infty}$ acts trivially on
$\omega_{\mathcal A/R}^2$, since its elements have determinant one.
Thus
\[
  H^1\left(
    \Gamma_{N,\infty},\omega_{\mathcal A/R}^2
  \right)
  \cong
  \operatorname{Hom}_{\mathbb Z}\left(
    \Gamma_{N,\infty},\omega_{\mathcal A/R}^2
  \right).
\]
The invertibility of the lattice index and of $d_K$ gives
\[
  \Gamma_{N,\infty}\otimes_{\mathbb Z}R
  \cong
  \mathcal O_K\otimes_{\mathbb Z}R
  \cong
  \mathfrak n_\sigma\oplus\mathfrak n_{\bar\sigma}.
\]
Taking $R$-linear duals and tensoring with
$\omega_{\mathcal A/R}^2$ proves the assertion in this case, since
both the invariants and the coinvariants equal
$\omega_{\mathcal A/R}^2$. We may therefore assume that $b\geq 1$.

Since $d_K$ is invertible in $R$, the embeddings
$\sigma,\bar\sigma$ induce an isomorphism
$\mathcal O_K\otimes_{\mathbb Z}R[1/b!]
\cong R[1/b!]\oplus R[1/b!]$, and hence the decomposition
\[
  \mathfrak n_{R[1/b!]}
  =
  \mathfrak n_\sigma\oplus\mathfrak n_{\bar\sigma}.
\]
Under the CM-type convention fixed above, $\mathfrak n_\sigma$
acts trivially on $\omega_{\mathcal A^\vee/R[1/b!]}$.
Zariski-locally, a generator of $\mathfrak n_{\bar\sigma}$ acts
through a nilpotent endomorphism $N$, and one may choose a basis
$\omega_1,\omega_2$ such that
$N(\omega_1)=0$ and $N(\omega_2)=\omega_1$.
The induced endomorphism on
$\mathrm{TSym}^b(\omega_{\mathcal A^\vee/R[1/b!]})$
satisfies $N^{b+1}=0$. Since $\Gamma_{N,\infty}$ acts trivially on
$\omega_{\mathcal A/R[1/b!]}^2$, we also write $N$ for its induced
endomorphism on
$\mathrm{TSym}^b(\omega_{\mathcal A^\vee/R[1/b!]})
\otimes\omega_{\mathcal A/R[1/b!]}^2$.

Choose a $\mathbb Z$-basis $a_1,a_2$ of $\Gamma_{N,\infty}$.
Its group cohomology is computed by the Koszul complex
\[
\begin{aligned}
  0\longrightarrow
  \mathrm{TSym}^b\left(\omega_{\mathcal A^\vee/R[1/b!]}\right)
  \otimes\omega_{\mathcal A/R[1/b!]}^2
  \longrightarrow\;&
  \left(
    \mathrm{TSym}^b\left(\omega_{\mathcal A^\vee/R[1/b!]}\right)
    \otimes\omega_{\mathcal A/R[1/b!]}^2
  \right)^{\oplus 2}\\
  &\qquad\qquad\longrightarrow
  \mathrm{TSym}^b\left(\omega_{\mathcal A^\vee/R[1/b!]}\right)
  \otimes\omega_{\mathcal A/R[1/b!]}^2
  \longrightarrow 0,
\end{aligned}
\]
whose differentials are
\[
  h\longmapsto
  \left(
    \bigl(n(a_1)-1\bigr)h,
    \bigl(n(a_2)-1\bigr)h
  \right)
\]
and
\[
  (h_1,h_2)\longmapsto
  \bigl(n(a_1)-1\bigr)h_2
  -
  \bigl(n(a_2)-1\bigr)h_1.
\]
The operators $n(a_i)-1$ are nilpotent of order at most $b+1$,
so their logarithms are defined over $R[1/b!]$ by
\[
  \log n(a_i)
  =
  \sum_{j=1}^{b}
  \frac{(-1)^{j+1}}{j}
  \bigl(n(a_i)-1\bigr)^j,
  \qquad i=1,2.
\]
The inverse relation is the finite exponential polynomial
\[
  n(a_i)-1
  =
  \sum_{j=1}^{b}
  \frac{1}{j!}\bigl(\log n(a_i)\bigr)^j.
\]
The polynomials relating these operators have invertible linear
terms. Thus each $n(a_i)-1$ is the product of $\log n(a_i)$ with
an invertible polynomial in the latter. These polynomials commute
with both logarithms, so the Koszul complexes associated with
$n(a_1)-1,n(a_2)-1$ and with $\log n(a_1),\log n(a_2)$ are
isomorphic over $R[1/b!]$.

The change from the lattice basis $a_1,a_2$ to the CM decomposition
is given by the matrix
\[
  \begin{pmatrix}
    \sigma(a_1) & \sigma(a_2)\\
    \bar\sigma(a_1) & \bar\sigma(a_2)
  \end{pmatrix}.
\]
The square of its determinant is, up to sign, the discriminant of
the lattice $\Gamma_{N,\infty}$, which differs from $d_K$ by the
square of the index $[\mathcal O_K:\Gamma_{N,\infty}]$.
Both factors are invertible in $R$, so this matrix is invertible
over $R[1/b!]$. After this change of basis, the logarithmic
operators become the infinitesimal actions $0$ and $N$ of
$\mathfrak n_\sigma$ and $\mathfrak n_{\bar\sigma}$, respectively.
The preceding logarithmic identification therefore gives the
van Est comparison
\[
\begin{aligned}
  &H^1\left(
    \Gamma_{N,\infty},
    \mathrm{TSym}^b\left(\omega_{\mathcal A^\vee/R[1/b!]}\right)
    \otimes\omega_{\mathcal A/R[1/b!]}^2
  \right)\cong
  H^1\left(
    \mathfrak n_{R[1/b!]},
    \mathrm{TSym}^b\left(\omega_{\mathcal A^\vee/R[1/b!]}\right)
    \otimes\omega_{\mathcal A/R[1/b!]}^2
  \right).
\end{aligned}
\]

The Lie algebra cohomology on the right is computed by the
Koszul complex
\[
\begin{aligned}
  0\longrightarrow
  \mathrm{TSym}^b\left(\omega_{\mathcal A^\vee/R[1/b!]}\right)
  \otimes\omega_{\mathcal A/R[1/b!]}^2
  \xrightarrow{d_0}\;&
  \left(
    \mathrm{TSym}^b\left(\omega_{\mathcal A^\vee/R[1/b!]}\right)
    \otimes\omega_{\mathcal A/R[1/b!]}^2
  \right)
  \otimes
  \left(
    \mathfrak n_\sigma^\vee\oplus\mathfrak n_{\bar\sigma}^\vee
  \right)\\
  &\quad\xrightarrow{d_1}
  \left(
    \mathrm{TSym}^b\left(\omega_{\mathcal A^\vee/R[1/b!]}\right)
    \otimes\omega_{\mathcal A/R[1/b!]}^2
  \right)
  \otimes
  \left(
    \mathfrak n_\sigma^\vee\wedge\mathfrak n_{\bar\sigma}^\vee
  \right)
  \longrightarrow 0.
\end{aligned}
\]
Choose local generators
$n_\sigma\in\mathfrak n_\sigma^\vee$ and
$n_{\bar\sigma}\in\mathfrak n_{\bar\sigma}^\vee$, with
$n_{\bar\sigma}$ dual to the generator defining $N$.
The differentials are then
\[
  d_0(h)=N(h)\otimes n_{\bar\sigma}
\]
and
\[
  d_1\left(
    h_\sigma\otimes n_\sigma
    +
    h_{\bar\sigma}\otimes n_{\bar\sigma}
  \right)
  =
  -N(h_\sigma)\otimes n_\sigma\wedge n_{\bar\sigma}.
\]
Consequently, the first cohomology of this complex is
\[
  \ker(N)\otimes_{R[1/b!]}\mathfrak n_\sigma^\vee
  \oplus
  \operatorname{coker}(N)\otimes_{R[1/b!]}
  \mathfrak n_{\bar\sigma}^\vee.
\]
The logarithmic comparison and the invertible change of basis
identify the common kernel of $n(a_1)-1$ and $n(a_2)-1$ with
$\ker(N)$. Hence
\[
  \ker(N)
  =
  \left(
    \mathrm{TSym}^b\left(\omega_{\mathcal A^\vee/R[1/b!]}\right)
    \otimes\omega_{\mathcal A/R[1/b!]}^2
  \right)^{\Gamma_{N,\infty}}.
\]
Similarly, the submodule generated by the images of
$n(a_1)-1$ and $n(a_2)-1$ equals the image of $N$, so
\[
  \operatorname{coker}(N)
  =
  \left(
    \mathrm{TSym}^b\left(\omega_{\mathcal A^\vee/R[1/b!]}\right)
    \otimes\omega_{\mathcal A/R[1/b!]}^2
  \right)_{\Gamma_{N,\infty}}.
\]
These identifications are independent of the chosen local
generators and therefore glue to the asserted natural isomorphism.
\end{proof}

\begin{rem}\rm
The decomposition of Proposition~\ref{prop_9} extends to every cusp
$c$ after a suitable localization of $R$. Indeed, conjugation by
$\sigma_c$ identifies $N_\infty$ with $N_c$, and
\(
  \sigma_c^{-1}\Gamma_{N,c}\sigma_c
\)
is a rank-two lattice in $N_\infty(K)\cong(K,+)$. After rescaling
by a suitable element of $K^\times$, it becomes a finite-index
sublattice of $\mathcal O_K$. Localize $R$ by inverting this index
and the denominators needed for $\sigma_c$, its inverse, and the
rescaling, and retain the notation $R$ for the resulting ring.
Over $R[1/b!]$, conjugation transports the CM decomposition
$\mathfrak n_{R[1/b!]}=\mathfrak n_\sigma\oplus
\mathfrak n_{\bar\sigma}$, the nilpotent action on the coefficient
module, and the associated group and Lie algebra Koszul complexes.
Thus the argument of Proposition~\ref{prop_9} gives the corresponding
decomposition of the first cohomology of $\Gamma_{N,c}$.

If $\sigma_c'$ is another choice with $\sigma_c'(\infty)=c$, then
$\sigma_c^{-1}\sigma_c'\in P_\infty(K)$. Since $P_\infty$ normalizes
$N_\infty$ and its adjoint action preserves the two CM summands,
conjugation by $\sigma_c^{-1}\sigma_c'$ identifies the corresponding
lattices and Koszul complexes after a common localization. Hence
the transported decompositions are canonically identified.
\end{rem}

We carry out the following coordinate calculation at $\infty$,
working over a base on which the local bases and generators used
in the proof of Proposition~\ref{prop_9} are defined, and retaining
the notation $R[1/b!]$. At an arbitrary cusp $c$, the resulting
generators are understood via the preceding transport.
The endomorphism $N$ induces
\[
  N_b:
  \mathrm{TSym}^b\left(\omega_{\mathcal A^\vee/R[1/b!]}\right)
  \longrightarrow
  \mathrm{TSym}^b\left(\omega_{\mathcal A^\vee/R[1/b!]}\right),
\]
with $N_b^{b+1}=0$. With respect to the basis $\omega_1,\omega_2$
fixed in that proof, its action is
\[
  N_b\left(\omega_1^{[i]}\omega_2^{[b-i]}\right)
  =
  (i+1)\omega_1^{[i+1]}\omega_2^{[b-i-1]}
  \qquad (0\leq i<b),
\]
and $N_b(\omega_1^{[b]})=0$. Since $1,\ldots,b$ are invertible in
$R[1/b!]$, these formulas give
\[
  \ker(N_b)=R[1/b!]\omega_1^{[b]},
  \qquad
  \operatorname{coker}(N_b)
  =
  R[1/b!]\left[\omega_2^{[b]}\right].
\]

Choose the generator $\omega(\mathcal A)^{\otimes 2}$ of
$\omega_{\mathcal A/R[1/b!]}^2$, and retain
$n_\sigma,n_{\bar\sigma}$ from the proof of
Proposition~\ref{prop_9}. Since the unipotent group acts trivially
on $\omega_{\mathcal A/R[1/b!]}^2$, the proposition and the preceding
calculation show that the invariant contribution
\(
  \left(
    \mathrm{TSym}^b\left(\omega_{\mathcal A^\vee/R[1/b!]}\right)
    \otimes\omega_{\mathcal A/R[1/b!]}^2
  \right)^{\Gamma_{N,\infty}}
  \otimes_{R[1/b!]}\mathfrak n_\sigma^\vee
\)
is free of rank one, generated by
\(
  \left[
    \omega_1^{[b]}
    \otimes\omega(\mathcal A)^{\otimes 2}
    \otimes n_\sigma
  \right].
\)
Similarly, the coinvariant contribution
\(
  \left(
    \mathrm{TSym}^b\left(\omega_{\mathcal A^\vee/R[1/b!]}\right)
    \otimes\omega_{\mathcal A/R[1/b!]}^2
  \right)_{\Gamma_{N,\infty}}
  \otimes_{R[1/b!]}\mathfrak n_{\bar\sigma}^\vee
\)
is free of rank one, generated by
\(
  \left[
    \omega_2^{[b]}
    \otimes\omega(\mathcal A)^{\otimes 2}
    \otimes n_{\bar\sigma}
  \right].
\)

\begin{notation}\rm
For a cusp $c$, fix the two generators transported from $\infty$
as above. They define $R[1/b!]$-linear projections
\[
  \mathrm{pr}_A,\mathrm{pr}_B:
  H^1\left(
    \Gamma_{N,c},
    \mathrm{TSym}^b\left(\omega_{\mathcal A^\vee/R[1/b!]}\right)
    \otimes\omega_{\mathcal A/R[1/b!]}^2
  \right)
  \longrightarrow R[1/b!].
\]
More precisely, every class $\phi'$ in this group has a unique
expression
\[
  \phi'
  =
  A\left[
    \omega_1^{[b]}
    \otimes\omega(\mathcal A)^{\otimes 2}
    \otimes n_\sigma
  \right]
  +
  B\left[
    \omega_2^{[b]}
    \otimes\omega(\mathcal A)^{\otimes 2}
    \otimes n_{\bar\sigma}
  \right],
\]
with $A,B\in R[1/b!]$, where the displayed generators are understood
at $c$ via transport. The projections are characterized by
\[
  \mathrm{pr}_A(\phi')=A,
  \qquad
  \mathrm{pr}_B(\phi')=B.
\]
These projections depend on the chosen generators of the two
rank-one summands.
\end{notation}

\begin{rem}\rm
Assume, as above, that the order of $\Gamma_c/\Gamma_{N,c}$ is
invertible in $R[1/b!]$. If $\phi'$ is invariant under this quotient,
the Hochschild--Serre identification established above identifies
it with a unique class in
\[[\phi']_c\in
H^1\bigl(
\Gamma_c,
\mathrm{TSym}^b(\omega_{\mathcal A^\vee/R[1/b!]})
\otimes\omega_{\mathcal A/R[1/b!]}^2
\bigr).
\]
The invariance condition concerns the entire expression for
$\phi'$, including the action on its two generators, rather than
the scalar coordinates $A$ and $B$ considered independently of
those generators. In particular, the projections are not
automatically $\Gamma_c/\Gamma_{N,c}$-invariant scalar-valued maps;
they have this property when the quotient acts trivially on the
chosen generators.
\end{rem}

\subsection{Specialization at \texorpdfstring{$\mathfrak F$}{F}-torsion sections}

We retain the CM and Serre tensor notation of the preceding subsection.
Let $R$ be an $\mathcal O_{K^{\mathrm{Gal}}}[1/d_K]$-algebra and set
$S=\Spec R$. Fix a nonzero integral ideal
$\mathfrak F\subset\mathcal O_L$; in the arithmetic applications below,
we usually take $\mathfrak F=\mathfrak f\mathcal O_L$ for an integral
ideal $\mathfrak f\subset\mathcal O_K$. Let
$\mathfrak C\subsetneq\mathcal O_L$ be a nonzero integral ideal coprime
to $\mathfrak F$ such that $[\mathfrak C]$ is $\Gamma$-equivariant and
$N_{L/\mathbb Q}(\mathfrak C)\in R^\times$. Put
$\mathcal D:=\mathcal A[\mathfrak C]$ and
$\mathcal U_{\mathcal D}:=\mathcal A\setminus\mathcal D$, so that
$\mathcal D$ is a finite étale $\Gamma$-stable subscheme satisfying the
hypotheses of Section~\ref{Section: The equivariant coherent Eisenstein--Kronecker classes}. Let
$x\colon S\to\mathcal U_{\mathcal D}$ be a $\Gamma$-invariant
$\mathfrak F$-torsion section. We denote by
$[\mathfrak F]\colon\mathcal A\to
\mathcal B:=\mathcal A(\mathfrak F^{-1})$
the isogeny induced by
$\mathcal O_L\hookrightarrow\mathfrak F^{-1}$, and assume that it is
$\Gamma$-equivariant. Unlike $\mathfrak C$, we do not require
$N_{L/\mathbb Q}(\mathfrak F)$ to be invertible in $R$; thus neither
$[\mathfrak F]$ nor its dual is assumed to be étale. Instead of using a
torsion splitting on $\mathcal A$, we therefore use isogeny
functoriality and the identity $[\mathfrak F]\circ x=e_{\mathcal B}$
to specialize at the unit section of $\mathcal B$.

\begin{prop}\label{prop_10}
Let $\mathfrak F$, $\mathfrak C$, and $x$ be as above, and put
$\mathcal B=\mathcal A(\mathfrak F^{-1})$. For every $b\geq 0$ and
$f\in R[\mathcal D]^{0,\Gamma}$, there is a canonical class
\[
\prescript{}{[\mathfrak F]}{EK}_{\Gamma,\mathcal B}^b(f,x)
\in H^{d-1}\left(\Gamma,\mathrm{TSym}^b(\omega_{\mathcal B^\vee/R})\otimes_R\omega_{\mathcal B/R}^d\right).
\]
If $g=1$ and $n=2$, then, with the projections of the preceding
subsection applied to $\mathcal B$, one has
\[
\mathrm{pr}_i\left(\prescript{}{[\mathfrak F]}{EK}_{\Gamma,\mathcal B}^b(f,x)\right)\in R[1/b!],
\qquad i\in\{A,B\}.
\]
\end{prop}

\begin{proof}
Let $j\colon\mathcal U_{\mathcal D}\hookrightarrow\mathcal A$ denote
the open immersion. By Theorem~\ref{thm: function to equiv coh class}, we may regard $EK_{\Gamma,\mathcal A}(f)$ as an
element of
\[
EK_{\Gamma,\mathcal A}(f)\in\varprojlim_m H^{d-1}\left(\mathcal U_{\mathcal D},\Gamma;
j^{-1}\left(\mathcal P_{\mathcal A}^{(m)}\otimes\Omega_{\mathcal A/R}^d\right)\right).
\]
For each $m\geq 0$, the functoriality of the Poincaré bundle gives a
$\Gamma$-equivariant morphism
$[\mathfrak F]_\#^{(m)}\colon
\mathcal P_{\mathcal A}^{(m)}
\to[\mathfrak F]^*\mathcal P_{\mathcal B}^{(m)}$,
compatible with the transition maps in $m$.
Pullback by the dual isogeny also gives
$[\mathfrak F]^{\vee *}\colon
\omega_{\mathcal A^\vee/R}\to\omega_{\mathcal B^\vee/R}$.
The canonical determinant identifications
$\omega_{\mathcal A/R}^d\cong\omega_{\mathcal A^\vee/R}^d$
and
$\omega_{\mathcal B/R}^d\cong\omega_{\mathcal B^\vee/R}^d$,
obtained from cup product and relative Serre duality, therefore yield
a $\Gamma$-equivariant morphism
\[
\Omega_{\mathcal A/R}^d
\cong\pi_{\mathcal A}^*\omega_{\mathcal A^\vee/R}^d
\xrightarrow{\;\pi_{\mathcal A}^*(\bigwedge^d[\mathfrak F]^{\vee *})\;}
\pi_{\mathcal A}^*\omega_{\mathcal B^\vee/R}^d
\cong[\mathfrak F]^*\Omega_{\mathcal B/R}^d.
\]
We also denote this top-degree morphism by
$[\mathfrak F]^{\vee *}$. Tensoring it with
$[\mathfrak F]_\#^{(m)}$ and using compatibility with the transition
maps gives
\[
[\mathfrak F]^{\vee *}[\mathfrak F]_\#EK_{\Gamma,\mathcal A}(f)
\in\varprojlim_m H^{d-1}\left(\mathcal U_{\mathcal D},\Gamma;
j^{-1}[\mathfrak F]^*\left(\mathcal P_{\mathcal B}^{(m)}\otimes\Omega_{\mathcal B/R}^d\right)\right).
\]
Since $[\mathfrak F]\circ x=e_{\mathcal B}$, rigidification at the
unit section identifies
$x^*[\mathfrak F]^*\widehat{\mathcal P}_{\mathcal B}$
with
$e_{\mathcal B}^*\widehat{\mathcal P}_{\mathcal B}
\cong(\widehat\pi_{\mathcal B}^\vee)_*
\mathcal O_{\widehat{\mathcal B}^\vee}$.
The unit-section moment map recalled in Definition~\ref{def: moment map x} thus gives
\[
\operatorname{mom}_{e_{\mathcal B}}^b\colon
x^*[\mathfrak F]^*\widehat{\mathcal P}_{\mathcal B}
\cong e_{\mathcal B}^*\widehat{\mathcal P}_{\mathcal B}
\longrightarrow\mathrm{TSym}^b(\omega_{\mathcal B^\vee/R}).
\]
Likewise,
$x^*[\mathfrak F]^*\Omega_{\mathcal B/R}^d
\cong e_{\mathcal B}^*\Omega_{\mathcal B/R}^d
\cong\omega_{\mathcal B/R}^d$.
The degree-$b$ moment map factors through the finite-level
rigidification for every $m\geq b$, so applying pullback and the
moment map to the preceding inverse system is independent of the
chosen level. Since $S=\Spec R$ is affine, equivariant cohomology
over $S$ with these coefficients identifies with group cohomology,
and we define
\[
\prescript{}{[\mathfrak F]}{EK}_{\Gamma,\mathcal B}^b(f,x)
:=
\left(\operatorname{mom}_{e_{\mathcal B}}^b\otimes
\mathrm{id}_{\omega_{\mathcal B/R}^d}\right)
x^*\left([\mathfrak F]^{\vee *}[\mathfrak F]_\#
EK_{\Gamma,\mathcal A}(f)\right)\in H^{d-1}\left(\Gamma,\mathrm{TSym}^b(\omega_{\mathcal B^\vee/R})\otimes_R\omega_{\mathcal B/R}^d\right).
\]
\end{proof}

\begin{prop}\label{prop: F-torsion restriction}
Assume that $g=1$ and $n=2$, so that $d=2$, and let
$\Gamma_c\subseteq\Gamma$ be the cusp stabilizer fixed in the
preceding subsection. Under the above equivariance hypotheses,
the coefficient morphism of
Proposition~\ref{prop_10} is compatible with
restriction from $\Gamma$ to $\Gamma_c$. More precisely, for every
$b\geq 0$, the following diagram commutes:
\[
\begin{tikzcd}[column sep=large, row sep=large]
{H^1\bigl(
\Gamma,\mathrm{TSym}^b(\omega_{\mathcal A^\vee/R})
\otimes_R\omega_{\mathcal A/R}^d
\bigr)}
\arrow[r, "{\mathrm{res}_c}"]
\arrow[d, "{[\mathfrak F]_\#}"']
&
{H^1\bigl(
\Gamma_c,\mathrm{TSym}^b(\omega_{\mathcal A^\vee/R})
\otimes_R\omega_{\mathcal A/R}^d
\bigr)}
\arrow[d, "{[\mathfrak F]_\#}"]
\\
{H^1\bigl(
\Gamma,\mathrm{TSym}^b(\omega_{\mathcal B^\vee/R})
\otimes_R\omega_{\mathcal B/R}^d
\bigr)}
\arrow[r, "{\mathrm{res}_c}"']
&
{H^1\bigl(
\Gamma_c,\mathrm{TSym}^b(\omega_{\mathcal B^\vee/R})
\otimes_R\omega_{\mathcal B/R}^d
\bigr)}.
\end{tikzcd}
\]
Here the vertical arrows denote the maps induced by the coefficient
morphism, including its top-differential factor. In particular, for
every $f\in R[\mathcal D]^{0,\Gamma}$, one has
\[
\mathrm{res}_c\left(
\prescript{}{[\mathfrak F]}{EK}_{\Gamma,\mathcal B}^b(f,x)
\right)
=
\prescript{}{[\mathfrak F]}{EK}_{\Gamma_c,\mathcal B}^b(f,x).
\]
\end{prop}

\begin{proof}
By the construction in Proposition~\ref{prop_10},
the morphisms $[\mathfrak F]_\#$ and $[\mathfrak F]^{\vee *}$ induce
a homomorphism of coefficient modules
\[
\mathrm{TSym}^b(\omega_{\mathcal A^\vee/R})\otimes_R\omega_{\mathcal A/R}^d
\longrightarrow
\mathrm{TSym}^b(\omega_{\mathcal B^\vee/R})\otimes_R\omega_{\mathcal B/R}^d.
\]
Since $[\mathfrak F]$ is $\Gamma$-equivariant, this homomorphism is
$\Gamma$-equivariant and hence also $\Gamma_c$-equivariant.
Naturality of restriction in group cohomology with respect to
equivariant homomorphisms of coefficient modules gives
\[
\mathrm{res}_c\circ[\mathfrak F]_\#
=
[\mathfrak F]_\#\circ\mathrm{res}_c,
\]
which proves the commutativity of the diagram. The finite-level
Poincaré morphisms, pullback along $x$, and the moment map at
$e_{\mathcal B}$ are likewise compatible with restriction.
Applying these compatibilities to the construction of
$\prescript{}{[\mathfrak F]}{EK}_{\Gamma,\mathcal B}^b(f,x)$
in Proposition~\ref{prop_10} proves the
asserted identity for the specialized classes.
\end{proof}

\section{Analytic realization of the Eisenstein--Kronecker classes}

In this section, we study the complex analytic realization of the equivariant
Eisenstein--Kronecker classes constructed in the previous section. We first
develop an equivariant Dolbeault and current model and construct explicit
Fourier representatives. We then specialize to the case $g=1$ and $n=2$,
analyze the restriction to the Borel--Serre boundary and its Fourier constant
terms, compare the resulting expressions with the algebraic Hodge bases, and
finally evaluate the class along the quadratic torus to recover the
Kings--Sprang formula.

\subsection{Analytic realization and the equivariant current model}

We now construct an analytic model for the equivariant
Eisenstein--Kronecker classes of the preceding section, following the
strategy of \cite[Section~3.2]{KS_2025}. Since the group $\Gamma$ need
not be commutative, we use the symmetric space attached to
$\operatorname{Res}_{K/\Q}\GL_n$ in the Borel construction. We first
fix the analytic uniformization, coordinates, and equivariant quotient
used throughout this section.

\begin{notation}\rm
We retain the CM and Serre tensor notation of the preceding sections
and work over $S=\Spec\C$. Thus $[K:\Q]=2g$, $[L:K]=n$, and
$\mathcal A=\mathcal O_L\otimes_{\mathcal O_K}\mathcal A_0$ has
dimension $d=gn$. Write
$\Sigma_K=\{\sigma_1,\ldots,\sigma_g\}$ for the CM type of
$\mathcal A_0$. For each $i$, put
$\C_{\sigma_i}:=\mathcal O_L\otimes_{\mathcal O_K,\sigma_i}\C$.
Using the CM-type decomposition of
$\operatorname{Lie}(\mathcal A_0/\C)$, we fix identifications
\[
\operatorname{Lie}(\mathcal A/\C)
\cong
\mathcal O_L\otimes_{\mathcal O_K}
\operatorname{Lie}(\mathcal A_0/\C)
\cong
\bigoplus_{i=1}^g\C_{\sigma_i}.
\]
A choice of a $K$-basis of $L$ identifies each $\C_{\sigma_i}$
with $\C^n$. We write
$Z_i=(z_{i,1},\ldots,z_{i,n})^t\in\C^n$ for the corresponding
coordinate vector, and hence
\[
z=(Z_1,\ldots,Z_g)
\in
\operatorname{Lie}(\mathcal A/\C)
\cong
\bigoplus_{i=1}^g\C^n
\cong
\C^d.
\]
The exponential map identifies
$H_1(\mathcal A(\C),\Z)$ with a lattice
$\Lambda\subset\C^d$, namely its kernel. We therefore obtain the
analytic uniformization
\[
\mathcal A(\C)\cong\C^d/\Lambda.
\]

Let $\Gamma\subset\GL_{\mathcal O_K}(\mathcal O_L)$ be the group
acting on the Serre tensor construction. The chosen $K$-basis
realizes $\GL_{\mathcal O_K}(\mathcal O_L)$ as the subgroup of
$\GL_n(K)$ preserving the lattice $\mathcal O_L$. Here
$\mathcal O_L$ is only assumed to be finite projective over
$\mathcal O_K$: an identification with $\GL_n(\mathcal O_K)$
would additionally require an $\mathcal O_K$-basis of
$\mathcal O_L$. For $\gamma\in\Gamma$, write
$\gamma_\infty=(\gamma_1,\ldots,\gamma_g)$, where
$\gamma_i=\gamma\otimes_{\mathcal O_K,\sigma_i}1$ acts on
$\C_{\sigma_i}$. In the chosen coordinates,
$\gamma_i=\sigma_i(\gamma)\in\GL_n(\C)$, and the induced action
on the Lie algebra is
\[
\gamma\cdot(Z_1,\ldots,Z_g)
=
(\gamma_1Z_1,\ldots,\gamma_gZ_g).
\]
Indeed, functoriality of the Serre tensor construction gives an
automorphism of $\mathcal A$ whose differential acts through
$\gamma$ on the first tensor factor. The exponential map intertwines
this differential with the action on $\mathcal A(\C)$, so
$\gamma_\infty\Lambda=\Lambda$. Consequently, the linear action
descends to $\mathcal A(\C)=\C^d/\Lambda$ by
$[z]\mapsto[\gamma_\infty z]$. Changing the $K$-basis changes
only the coordinate description of this action.

Set $G=\operatorname{Res}_{K/\Q}\GL_n$ and
$G_\infty=G(\R)$. The CM type determines an identification
$G_\infty\cong\prod_{i=1}^g\GL_n(\C)$, under which we regard
$\Gamma$ as a subgroup through $\gamma\mapsto\gamma_\infty$.
We fix the maximal compact subgroup
$K_\infty=\prod_{i=1}^gU(n)$ and put
$S_\infty=G_\infty/K_\infty$. The action of $\Gamma$ on
$S_\infty$ is induced by left multiplication:
\[
\gamma\cdot gK_\infty=\gamma_\infty gK_\infty.
\]
\end{notation}

Let $H_0$ denote the standard positive-definite Hermitian form on
each copy of $\C^n$, and let $^*$ denote conjugate transpose.
The symmetric space has the real-analytic realization
\[
S_\infty\xrightarrow{\sim}(\operatorname{Herm}_n^+)^g,
\qquad
gK_\infty\longmapsto H_g:=(g^{-1})^*H_0g^{-1}.
\]
This map is well defined because $K_\infty$ preserves $H_0$.
Every positive-definite Hermitian form is obtained from $H_0$
by such a change of basis, and the stabilizer of $H_0$ is
precisely $K_\infty$, which gives the stated identification.
Under this realization, the $\Gamma$-action becomes
\[
\gamma\cdot H_g
=
H_{\gamma_\infty g}
=
(\gamma_\infty^{-1})^*H_g\gamma_\infty^{-1}.
\]
In particular, the simultaneous action on vectors and Hermitian
forms is isometric:
\[
H_{\gamma_\infty g}(\gamma_\infty v,\gamma_\infty w)
=
H_g(v,w)
\qquad
(v,w\in\C^d).
\]
This compatibility will be used when defining the norms and
contraction operators in the explicit analytic construction.
Moreover, $\operatorname{Herm}_n^+$ is an open convex cone in
the real vector space of Hermitian matrices. Each factor is
therefore contractible, and hence so is $S_\infty$.

The archimedean image of
$\operatorname{End}_{\mathcal O_K}(\mathcal O_L)$ is a lattice
in $\prod_{i=1}^gM_n(\C)$, so $\Gamma$ is discrete in
$G_\infty$. Since $K_\infty$ is compact, its action on
$S_\infty=G_\infty/K_\infty$ is properly discontinuous.
For $gK_\infty\in S_\infty$, the stabilizer is
\[
\operatorname{Stab}_\Gamma(gK_\infty)
=
\Gamma\cap gK_\infty g^{-1}.
\]
This intersection is finite because $\Gamma$ is discrete and
$gK_\infty g^{-1}$ is compact. For the construction below,
we assume that $\Gamma$ is torsion-free, so all these
stabilizers are trivial. The diagonal action on
$\mathcal A(\C)\times S_\infty$ is then free and properly
discontinuous, and is given by
\[
\gamma\cdot([z],gK_\infty)
=
([\gamma_\infty z],\gamma_\infty gK_\infty).
\]
We define
\[
X_\Gamma
:=
\mathcal A(\C)\times_\Gamma S_\infty
=
\Gamma\backslash\bigl(\mathcal A(\C)\times S_\infty\bigr).
\]
This is a smooth real-analytic manifold, and its natural projection
to $\Gamma\backslash S_\infty$ is a real-analytic fiber bundle
with fiber $\mathcal A(\C)$. Since $S_\infty$ is contractible,
the covering $S_\infty\to\Gamma\backslash S_\infty$ is a
model for the universal $\Gamma$-bundle, and $X_\Gamma$ is
the corresponding Borel construction.

\begin{rem}\rm
If $\Gamma$ has torsion, the stabilizers remain finite, and the
quotients above are naturally interpreted as orbifold quotients.
Alternatively, in the arithmetic setting considered here, one may
pass to a torsion-free subgroup $\Gamma'\subset\Gamma$ of
finite index and use the restriction and corestriction maps
established in the preceding section. Since our analytic
coefficients are over $\C$, a $\Gamma$-equivariant class is
recovered from its restriction to $\Gamma'$ by normalized
corestriction.
\end{rem}

We regard the analytic coefficient sheaves as $\Gamma$-equivariant
sheaves of complex vector spaces, equipped with isomorphisms
$\rho_\gamma:\gamma^{-1}\mathcal F\xrightarrow{\sim}\mathcal F$
satisfying the cocycle condition. This category is abelian and has
enough injective objects by
\cite[Propositions~5.1.1 and~5.1.2]{tohuku_Grothendieck}.
As before, $H^i(X,\Gamma;\mathcal F)$ denotes the right
derived functor of
$\mathcal F\mapsto H^0(X,\mathcal F)^\Gamma$.
All sheaves on a product with $S_\infty$ will be considered
with respect to the diagonal $\Gamma$-action.

\begin{prop}\label{prop: equiv coh via Borel}
Let $X$ be a smooth manifold with a $\Gamma$-action, and let
$\mathcal F$ be a $\Gamma$-equivariant sheaf of complex
vector spaces on $X$. Consider the projection and quotient maps
\[
X\xleftarrow{\mathrm{pr}_X}X\times S_\infty
\xrightarrow{q}X\times_\Gamma S_\infty.
\]
For a $\Gamma$-equivariant sheaf $\mathcal G$ on
$X\times S_\infty$, let $q_*^\Gamma\mathcal G$ denote
the sheaf on the quotient defined by
\[
H^0(V,q_*^\Gamma\mathcal G)
=
H^0(q^{-1}(V),\mathcal G)^\Gamma
\]
for every open subset $V\subset X\times_\Gamma S_\infty$.
Then, for every $i\geq0$, there is a canonical isomorphism
\[
H^i(X,\Gamma;\mathcal F)
\cong
H^i\bigl(
X\times_\Gamma S_\infty,
q_*^\Gamma\mathrm{pr}_X^{-1}\mathcal F
\bigr).
\]
These isomorphisms are natural in $\mathcal F$ and compatible
with restriction to $\Gamma$-stable open subsets of $X$.
\end{prop}

\begin{proof}
The diagonal action on $X\times S_\infty$ is free and
properly discontinuous. Over an evenly covered open subset of the
quotient, the inverse image is a disjoint union of
$\Gamma$-translates of a single sheet. A $\Gamma$-equivariant
sheaf on this inverse image is determined by its restriction to
that sheet, and the equivariant structure identifies the
restrictions on the remaining sheets. It follows that $q^{-1}$
and $q_*^\Gamma$ are inverse equivalences between sheaves on
$X\times_\Gamma S_\infty$ and $\Gamma$-equivariant sheaves
on $X\times S_\infty$. In particular, $q_*^\Gamma$ is
exact and preserves injective objects. Taking derived functors
of the equality of global-section functors therefore gives
\[
H^i(X\times S_\infty,\Gamma;\mathcal G)
\cong
H^i(X\times_\Gamma S_\infty,q_*^\Gamma\mathcal G).
\]

We next compare the product with $X$. Under the Hermitian
realization, the contraction
$H\mapsto(1-t)H+tH_0$ of $S_\infty$ induces a homotopy
over $X$ from the identity of $X\times S_\infty$ to the
section $x\mapsto(x,H_0)$ composed with $\mathrm{pr}_X$.
The coefficient sheaf $\mathrm{pr}_X^{-1}\mathcal F$
is pulled back from $X$ throughout this homotopy. The product
homotopy lemma for inverse-image coefficients thus gives
\[
\mathcal F
\xrightarrow{\sim}
R(\mathrm{pr}_X)_*
\mathrm{pr}_X^{-1}\mathcal F,
\qquad
R^j(\mathrm{pr}_X)_*
\mathrm{pr}_X^{-1}\mathcal F=0
\quad(j>0).
\]
Although the chosen contraction need not be $\Gamma$-equivariant,
the adjunction morphism in this formula is canonical and
$\Gamma$-equivariant. The ordinary Leray spectral sequence
therefore gives $\Gamma$-equivariant isomorphisms
$H^b(X,\mathcal F)\cong
H^b(X\times S_\infty,\mathrm{pr}_X^{-1}\mathcal F)$.
We now compare the fundamental spectral sequence
\[
E_2^{a,b}
=
H^a\bigl(\Gamma,H^b(X,\mathcal F)\bigr)
\Longrightarrow
H^{a+b}(X,\Gamma;\mathcal F)
\]
from \cite[Theorem~5.2.1]{tohuku_Grothendieck}
with the corresponding spectral sequence for
$X\times S_\infty$. Pullback identifies their $E_2$-terms,
and hence their abutments. Applying the descent comparison above
to $\mathcal G=\mathrm{pr}_X^{-1}\mathcal F$ completes
the argument.
\end{proof}

In particular, this comparison applies to
$X=\mathcal A(\C)$ and to the $\Gamma$-stable open subset
$\mathcal U_{\mathcal D}(\C)$. For the latter, the coefficients
used below are
$j^{-1}(\mathcal P^{(m),\mathrm{an}}
\otimes\Omega_{\mathcal A(\C)}^d)$ for $m\geq0$, where
$j:\mathcal U_{\mathcal D}(\C)\hookrightarrow\mathcal A(\C)$
is the open immersion. Naturality makes the resulting Borel
models compatible with the transition morphisms in $m$.
With this equivariant analytic model fixed, we turn to the
translation-invariant $C^\infty$-trivialization and the
tautological form used in the realization of the completed
Poincaré bundle.

Write
\[
\mathrm{pr}_{\mathcal A}:
\mathcal A(\C)\times S_\infty\longrightarrow\mathcal A(\C),
\qquad
\mathrm{pr}_{S_\infty}:
\mathcal A(\C)\times S_\infty\longrightarrow S_\infty
\]
for the two projections. The product structure gives a canonical
decomposition of real tangent bundles
\[
T_{\mathcal A(\C)\times S_\infty}
\cong
\mathrm{pr}_{\mathcal A}^*T_{\mathcal A(\C)}
\oplus
\mathrm{pr}_{S_\infty}^*T_{S_\infty}.
\]
Set $\mathcal H_{\Z}:=H_1(\mathcal A(\C),\Z)$ and
$\mathcal H_{\R}:=\mathcal H_{\Z}\otimes_{\Z}\R$.
Under the uniformization $\mathcal A(\C)\cong\C^d/\Lambda$,
the universal covering identifies $\mathcal H_{\Z}$ with
$\Lambda$. Since $\Lambda$ is a full lattice in the underlying
real vector space of $\C^d$, its real span identifies
$\mathcal H_{\R}$ with $T_0\mathcal A(\C)$. Translation therefore
gives the canonical $C^\infty$-trivialization
\[
T_{\mathcal A(\C)}
\xrightarrow{\sim}
\mathcal A(\C)\times\mathcal H_{\R},
\qquad
\xi_a\longmapsto
\bigl(a,d(t_{-a})_a(\xi_a)\bigr),
\]
where $t_{-a}$ denotes translation by $-a$. We regard
$\mathcal H_{\R}$ as a constant vector bundle on
$\mathcal A(\C)$ and use its pullback on the product.

\begin{defn}[Tautological differential form]
\label{def: Tautological differential form}\rm
The \emph{tautological differential form} on $\mathcal A(\C)$
is the $\mathcal H_{\R}$-valued form
$\nu_{\mathcal A}\in
\mathcal E^1(\mathcal A(\C),\mathcal H_{\R})$ defined by
\[
(\nu_{\mathcal A})_{[z]}(\xi)
:=
d(t_{-[z]})_{[z]}(\xi).
\]
On $\mathcal A(\C)\times S_\infty$, we put
$\nu:=\mathrm{pr}_{\mathcal A}^*\nu_{\mathcal A}$.
Thus $\nu$ is the identity in the torus direction under the
translation-invariant trivialization, depends only on the
$\mathcal A(\C)$-coordinate, and vanishes on
$\mathrm{pr}_{S_\infty}^*T_{S_\infty}$.
\end{defn}

\begin{notation}\label{notation: decomp of H_C}\rm
Set $\mathcal H_{\C}:=\mathcal H_{\R}\otimes_{\R}\C$.
The Betti--de Rham comparison identifies $\mathcal H_{\C}$ with
$\operatorname{Hom}_{\C}
(H^1_{\mathrm{dR}}(\mathcal A/\C),\C)$.
We use the homological Hodge decomposition
$\mathcal H_{\C}
=\mathcal H^{-1,0}\oplus\mathcal H^{0,-1}$, with
\[
\mathcal H^{-1,0}\cong T_0^{1,0}\mathcal A(\C),
\qquad
\mathcal H^{0,-1}\cong T_0^{0,1}\mathcal A(\C).
\]
By Proposition~\ref{prop_5} and Remark~\ref{rem_2}, this
decomposition agrees with the CM-type decomposition:
$\mathcal H^{-1,0}=\mathcal H_{\C}(\Sigma_L)$ and
$\mathcal H^{0,-1}=\mathcal H_{\C}(\overline{\Sigma}_L)$.
In particular, the latter component identifies with
$\omega_{\mathcal A^\vee/\C}$.
Write $z_1,\ldots,z_d$ for the coordinates
$z_{1,1},\ldots,z_{g,n}$ fixed above, and let
$e_1,\ldots,e_d$ and $\bar e_1,\ldots,\bar e_d$ denote the
bases of $\mathcal H^{-1,0}$ and $\mathcal H^{0,-1}$
corresponding to $\partial/\partial z_i$ and
$\partial/\partial\bar z_i$, respectively. Our duality
conventions are
$dz_j(e_i)=\delta_{ij}$ and
$d\bar z_j(\bar e_i)=\delta_{ij}$.
\end{notation}

\begin{defn}\label{def: cplx decomp tautological diff form}\rm
The complexification of the tautological form decomposes as
$\nu_{\mathcal A,\C}
=\nu_{\mathcal A}^{1,0}+\nu_{\mathcal A}^{0,1}$, where
\[
\nu_{\mathcal A}^{1,0}
=
\sum_{i=1}^d e_i\,dz_i,
\qquad
\nu_{\mathcal A}^{0,1}
=
\sum_{i=1}^d\bar e_i\,d\bar z_i.
\]
We denote their pullbacks to the product by
$\nu^{1,0}:=\mathrm{pr}_{\mathcal A}^*
\nu_{\mathcal A}^{1,0}$ and
$\nu^{0,1}:=\mathrm{pr}_{\mathcal A}^*
\nu_{\mathcal A}^{0,1}$.
The bidegrees here refer only to the complex structure in the
$\mathcal A(\C)$-direction; both forms vanish in the
$S_\infty$-direction.
\end{defn}

The translation-invariant trivialization and the tautological
forms are compatible with the $\Gamma$-action. Indeed, since
$\gamma\in\Gamma$ acts by an automorphism of the complex Lie
group $\mathcal A(\C)$ fixing the origin, one has
\[
t_{-\gamma[z]}\circ\gamma
=
\gamma\circ t_{-[z]}.
\]
Differentiating at $[z]$ gives
\[
d(t_{-\gamma[z]})_{\gamma[z]}\circ d\gamma_{[z]}
=
d\gamma_0\circ d(t_{-[z]})_{[z]}.
\]
Under $\mathcal H_{\R}\cong T_0\mathcal A(\C)$, the
differential $d\gamma_0$ is the real-linear action induced by
$\gamma_\infty$. Consequently,
$\gamma^*\nu_{\mathcal A}=\gamma_\infty\nu_{\mathcal A}$,
and equivariance of $\mathrm{pr}_{\mathcal A}$ gives
$\gamma^*\nu=\gamma_\infty\nu$.
Here pullback on the left acts on the differential-form
component, whereas $\gamma_\infty$ on the right acts on the
coefficient space. Since $\gamma$ is holomorphic, its action
on $\mathcal H_{\C}$ preserves the homological Hodge
decomposition. The two Hodge components are therefore
separately equivariant, both on $\mathcal A(\C)$ and on the
product; in particular,
\[
\gamma^*\nu^{1,0}=\gamma_\infty\nu^{1,0},
\qquad
\gamma^*\nu^{0,1}=\gamma_\infty\nu^{0,1}.
\]
We use the same notation $\gamma_\infty$ for the complexified
action on $\mathcal H_{\C}$; on $\mathcal H^{0,-1}$ its
matrix in the basis $\bar e_1,\ldots,\bar e_d$ is the complex
conjugate of its matrix on $\mathcal H^{-1,0}$.

We now describe the holomorphic structure of the finite-level
Poincaré bundles in terms of $\nu_{\mathcal A}^{0,1}$.
For every $m\geq0$, put
\[
\mathcal P^{(m),\mathrm{sm}}
:=
\mathcal P^{(m),\mathrm{an}}
\otimes_{\mathcal O_{\mathcal A(\C)}^{\mathrm{an}}}
\mathcal C^\infty_{\mathcal A(\C)},
\qquad
\widehat{\mathcal P}^{\mathrm{sm}}
:=
\varprojlim_m\mathcal P^{(m),\mathrm{sm}}.
\]
Similarly, set
$\mathcal H^{\mathrm{sm}}
:=\mathcal H_{\C}\otimes_{\C}
\mathcal C^\infty_{\mathcal A(\C)}$ and, for
$J=\Sigma_L$ or $\overline{\Sigma}_L$, write
$\mathcal H(J)^{\mathrm{sm}}
:=\mathcal H_{\C}(J)\otimes_{\C}
\mathcal C^\infty_{\mathcal A(\C)}$.
The Hodge decomposition then gives
$\mathcal H^{\mathrm{sm}}
=\mathcal H(\Sigma_L)^{\mathrm{sm}}
\oplus\mathcal H(\overline{\Sigma}_L)^{\mathrm{sm}}$.
The holomorphic structure on $\mathcal P^{(m),\mathrm{an}}$
determines an integrable $(0,1)$-connection
\[
\nabla'':
\mathcal P^{(m),\mathrm{sm}}
\longrightarrow
\mathcal P^{(m),\mathrm{sm}}
\otimes_{\mathcal C^\infty_{\mathcal A(\C)}}
\mathcal E_{\mathcal A(\C)}^{0,1}.
\]
It annihilates holomorphic sections and satisfies the Leibniz
rule
$\nabla''(fs)=f\nabla''s+s\otimes\bar\partial f$
for a smooth function $f$ and a smooth local section $s$.
Extending by the graded Leibniz rule gives a differential on
$\mathcal P^{(m),\mathrm{sm}}
\otimes_{\mathcal C^\infty_{\mathcal A(\C)}}
\mathcal E_{\mathcal A(\C)}^{0,\bullet}$, with
$(\nabla'')^2=0$.
The transition morphisms
$\mathcal P^{(m),\mathrm{an}}\to
\mathcal P^{(m-1),\mathrm{an}}$
and the $\Gamma$-linearizations inherited from the algebraic
construction are holomorphic. Their smooth realizations
therefore commute with $\nabla''$.

\begin{thm}\label{thm: compare P_n and TSym with connection}
For every $m\geq0$, there is a $\Gamma$-equivariant
isomorphism of smooth modules with integrable
$(0,1)$-connections
\[
\alpha_m:
\left(\mathcal P^{(m),\mathrm{sm}},\nabla''\right)
\xlongrightarrow{\sim}
\left(
\bigoplus_{k=0}^m
\mathrm{TSym}^k
\bigl(\mathcal H(\overline{\Sigma}_L)^{\mathrm{sm}}\bigr),
\bar\partial+\nu_{\mathcal A}^{0,1}
\right).
\]
Under the identification
$\mathrm{TSym}^k
(\mathcal H(\overline{\Sigma}_L)^{\mathrm{sm}})
\cong
\mathrm{TSym}^k(\mathcal H_{\C}(\overline{\Sigma}_L))
\otimes_{\C}\mathcal C^\infty_{\mathcal A(\C)}$,
the operator $\bar\partial$ acts on the smooth coefficient
functions:
$\bar\partial(v\otimes f)=v\otimes\bar\partial f$.
The form $\nu_{\mathcal A}^{0,1}$ acts by tensor-symmetric
multiplication, raising the symmetric degree by one, with
products of degree greater than $m$ set equal to zero.
Thus
\[
(\bar\partial+\nu_{\mathcal A}^{0,1})
(s_0,\ldots,s_m)
=
\bigl(
\bar\partial s_k+\nu_{\mathcal A}^{0,1}s_{k-1}
\bigr)_{k=0}^m,
\qquad s_{-1}:=0.
\]
The isomorphisms $\alpha_m$ are compatible with the
transition maps: for $m\geq1$, the morphism
$\mathcal P^{(m),\mathrm{sm}}\to
\mathcal P^{(m-1),\mathrm{sm}}$
corresponds to projection onto the summands of symmetric
degree at most $m-1$.
Moreover, under
$e^*\mathcal H(\overline{\Sigma}_L)^{\mathrm{sm}}
\cong\omega_{\mathcal A^\vee/\C}$, their pullbacks along
the unit section agree with the finite-level moment maps:
\[
e^*\alpha_m
=
\mathrm{mom}_{e,m}^{\mathrm{sm}}:
e^*\mathcal P^{(m),\mathrm{sm}}
\xlongrightarrow{\sim}
\bigoplus_{k=0}^m
\mathrm{TSym}^k(\omega_{\mathcal A^\vee/\C}).
\]
\end{thm}

\begin{proof}
At level zero, the identification
$\mathcal P^{(0),\mathrm{an}}
\cong\mathcal O_{\mathcal A(\C)}^{\mathrm{an}}$
gives $\alpha_0$.
At the first infinitesimal level, the defining exact
sequence of the Poincaré bundle becomes
\[
0\longrightarrow
\mathcal H(\overline{\Sigma}_L)^{\mathrm{sm}}
\xrightarrow{\iota}
\mathcal P^{(1),\mathrm{sm}}
\xrightarrow{\epsilon}
\mathcal C^\infty_{\mathcal A(\C)}
\longrightarrow0.
\]
This sequence of smooth vector bundles admits a smooth
splitting. Choose a splitting $s$ whose pullback along $e$
agrees with the splitting induced by the rigidification;
such a normalization can be obtained by adjusting any
smooth splitting by a constant section of
$\mathcal H(\overline{\Sigma}_L)^{\mathrm{sm}}$.
Under the resulting identification
$\mathcal P^{(1),\mathrm{sm}}
\cong
\mathcal C^\infty_{\mathcal A(\C)}
\oplus\mathcal H(\overline{\Sigma}_L)^{\mathrm{sm}}$,
the Dolbeault operator has the form
\[
\nabla''(f,h)
=
\bigl(\bar\partial f,\bar\partial h+f\beta\bigr),
\]
where $\beta$ is an
$\mathcal H_{\C}(\overline{\Sigma}_L)$-valued
$(0,1)$-form. Integrability implies
$\bar\partial\beta=0$, and its Dolbeault class is the
extension class of $\mathcal P^{(1),\mathrm{an}}$.
By the universal property of the Poincaré bundle, this
extension class corresponds to the identity under the
canonical identifications
\[
\begin{aligned}
H^1\bigl(
\mathcal A(\C),\mathcal O_{\mathcal A(\C)}^{\mathrm{an}}
\bigr)
\otimes_{\C}\omega_{\mathcal A^\vee/\C}
&\cong
\operatorname{Lie}(\mathcal A^\vee/\C)
\otimes_{\C}\omega_{\mathcal A^\vee/\C}\cong
\operatorname{End}_{\C}(\omega_{\mathcal A^\vee/\C}).
\end{aligned}
\]
Indeed, the first infinitesimal neighborhood of the origin
in $\mathcal A^\vee$ parametrizes the universal first-order
deformation of the trivial line bundle.
Under the translation-invariant description of Dolbeault
cohomology, the classes of $d\bar z_i$ are dual to the
basis $\bar e_i$ of
$\mathcal H_{\C}(\overline{\Sigma}_L)
\cong\omega_{\mathcal A^\vee/\C}$.
Consequently,
$\nu_{\mathcal A}^{0,1}
=\sum_i\bar e_i\,d\bar z_i$
is the translation-invariant representative of this
identity class.
It follows that
$\beta-\nu_{\mathcal A}^{0,1}=\bar\partial h$
for a smooth section $h$ of
$\mathcal H(\overline{\Sigma}_L)^{\mathrm{sm}}$.
Viewing $h(e)$ as a constant section, replace $s$ by the
smooth splitting $s'$ determined by
\[
s'(1):=s(1)-\iota\bigl(h-h(e)\bigr).
\]
This modification preserves the normalization at $e$ and
changes the off-diagonal term to
$\beta-\bar\partial(h-h(e))
=\nu_{\mathcal A}^{0,1}$.
It therefore gives the required isomorphism $\alpha_1$.

The normalized splitting is unique. Two splittings with
this off-diagonal term differ by a holomorphic section of
the constant bundle
$\mathcal H(\overline{\Sigma}_L)^{\mathrm{sm}}$.
Such a section is constant because $\mathcal A(\C)$ is
compact, and the normalization forces its value at $e$
to be zero.
The level-one extension and its rigidification are
$\Gamma$-equivariant, as is
$\nu_{\mathcal A}^{0,1}$.
Hence the $\gamma$-translate of the normalized splitting
satisfies the same conditions for every $\gamma\in\Gamma$.
Uniqueness shows that the splitting, and therefore
$\alpha_1$, is $\Gamma$-equivariant.

For arbitrary $m$, Corollary~\ref{cor_5} gives an
isomorphism
\[
\mathcal P^{(m),\mathrm{sm}}
\xlongrightarrow{\sim}
\mathrm{TSym}^m_{\mathcal C^\infty_{\mathcal A(\C)}}
\bigl(\mathcal P^{(1),\mathrm{sm}}\bigr),
\]
since $m!$ is invertible over $\C$.
This map is induced by iterated comultiplication.
The comultiplication is holomorphic and functorial for the
$\Gamma$-action, so the isomorphism is
$\Gamma$-equivariant and compatible with the induced
$(0,1)$-connections.
Applying $\mathrm{TSym}^m$ to $\alpha_1$ and using the
canonical decomposition
\[
\mathrm{TSym}^m
\left(
\mathcal C^\infty_{\mathcal A(\C)}
\oplus\mathcal H(\overline{\Sigma}_L)^{\mathrm{sm}}
\right)
\cong
\bigoplus_{k=0}^m
\mathrm{TSym}^k
\bigl(\mathcal H(\overline{\Sigma}_L)^{\mathrm{sm}}\bigr)
\]
defines $\alpha_m$.
To describe the induced connection, write $u=(1,0)$ for
the generator of the first summand at level one.
The degree-$k$ summand in the displayed decomposition is
embedded by tensor-symmetric multiplication with
$u^{[m-k]}$.
Since $\nabla''u=\nu_{\mathcal A}^{0,1}$, the tensor
product connection gives, for $k<m$,
\[
\nabla''\bigl(u^{[m-k]}s_k\bigr)
=
u^{[m-k]}\bar\partial s_k
+
u^{[m-k-1]}\nu_{\mathcal A}^{0,1}s_k,
\]
with the second term absent when $k=m$.
Thus $\bar\partial$ preserves the symmetric degree,
whereas the remaining term is multiplication by
$\nu_{\mathcal A}^{0,1}$.
Compatibility of comultiplication with truncation
identifies the transition from level $m$ to level $m-1$
with applying the augmentation $\epsilon$ to one tensor
factor. Since $\epsilon(u)=1$ and $\epsilon$ vanishes on
$\mathcal H(\overline{\Sigma}_L)^{\mathrm{sm}}$, this
map is the identity on symmetric degrees at most $m-1$
and annihilates degree $m$.

Finally, the normalization of the level-one splitting
identifies its pullback with the canonical decomposition
$e^*\mathcal P^{(1),\mathrm{sm}}
\cong\C\oplus\omega_{\mathcal A^\vee/\C}$.
Under the rigidification, the pullback of the Poincaré
comultiplication is the formal-group coproduct used to
define the moment maps.
Both $e^*\alpha_m$ and
$\mathrm{mom}_{e,m}^{\mathrm{sm}}$ are therefore obtained
from this level-one identification by the same iterated
comultiplication and tensor-symmetric decomposition.
\end{proof}

\begin{cor}\label{cor_7}
For every $m\geq0$, inclusion of holomorphic sections
into smooth sections induces a natural
$\Gamma$-equivariant quasi-isomorphism
\[
\bigl(\mathcal P^{(m),\mathrm{an}}\bigr)[0]
\xlongrightarrow{\sim}
\left(
\mathcal P^{(m),\mathrm{sm}}
\otimes_{\mathcal C^\infty_{\mathcal A(\C)}}
\mathcal E_{\mathcal A(\C)}^{0,\bullet},
\nabla''
\right).
\]
For $0\leq p\leq d$, tensoring with the sheaf of
holomorphic $p$-forms gives a natural
$\Gamma$-equivariant quasi-isomorphism
\[
\left(
\mathcal P^{(m),\mathrm{an}}
\otimes_{\mathcal O_{\mathcal A(\C)}^{\mathrm{an}}}
\Omega_{\mathcal A(\C)}^p
\right)[0]
\xlongrightarrow{\sim}
\left(
\mathcal P^{(m),\mathrm{sm}}
\otimes_{\mathcal C^\infty_{\mathcal A(\C)}}
\mathcal E_{\mathcal A(\C)}^{p,\bullet},
\nabla''
\right).
\]
These quasi-isomorphisms are compatible with the
transition morphisms in $m$.
\end{cor}

\begin{proof}
The sheaf of $\nabla''$-horizontal sections of
$\mathcal P^{(m),\mathrm{sm}}$ is precisely
$\mathcal P^{(m),\mathrm{an}}$.
In a local holomorphic frame, the corresponding
Dolbeault complex is the ordinary Dolbeault complex
applied componentwise, so the Dolbeault lemma gives
the first quasi-isomorphism.
Since $\Omega_{\mathcal A(\C)}^p$ is locally free over
$\mathcal O_{\mathcal A(\C)}^{\mathrm{an}}$,
tensoring with it preserves quasi-isomorphisms.
The canonical identification
\[
\Omega_{\mathcal A(\C)}^p
\otimes_{\mathcal O_{\mathcal A(\C)}^{\mathrm{an}}}
\mathcal E_{\mathcal A(\C)}^{0,q}
\cong
\mathcal E_{\mathcal A(\C)}^{p,q}
\]
identifies the resulting complex with the one in the
second assertion.
All these maps are natural for holomorphic bundle
morphisms and for pullback by the holomorphic
automorphisms in $\Gamma$.
They therefore commute with the transition morphisms
and with the $\Gamma$-linearizations.
\end{proof}

We now combine the ordinary Dolbeault resolution of
Corollary~\ref{cor_7} with the de Rham complex along $S_\infty$
to compute equivariant cohomology on the Borel model.
For the projection
$\mathrm{pr}_{\mathcal A}:
\mathcal A(\C)\times S_\infty\to\mathcal A(\C)$,
the relative de Rham complex
$\mathcal E^\bullet_{\mathcal A(\C)\times S_\infty/\mathcal A(\C)}$
consists of smooth complex-valued forms in the $S_\infty$-direction,
with coefficients depending smoothly on both factors; its differential
is $d_{S_\infty}$.
Its degree-zero kernel consists of functions locally independent
of the $S_\infty$-coordinate and is therefore
$\mathrm{pr}_{\mathcal A}^{-1}\mathcal C^\infty_{\mathcal A(\C)}$.
Under the Hermitian realization, $S_\infty$ is an open convex cone.
The de Rham homotopy associated with its straight-line contraction
acts on smooth families parametrized by $\mathcal A(\C)$ and
contracts the positive relative degrees.
Locally on product neighborhoods, this gives the relative
Poincaré lemma and hence a quasi-isomorphism
\begin{equation}\label{eq_15}
\bigl(
\mathrm{pr}_{\mathcal A}^{-1}
\mathcal C^\infty_{\mathcal A(\C)}
\bigr)[0]
\xlongrightarrow{\sim}
\left(
\mathcal E^\bullet_{\mathcal A(\C)\times S_\infty/\mathcal A(\C)},
d_{S_\infty}
\right).
\end{equation}
Although a chosen contracting homotopy need not be
$\Gamma$-equivariant, the augmentation and the relative differential
are natural for the diagonal action, so this quasi-isomorphism
is $\Gamma$-equivariant.

\begin{defn}[Mixed bidegree]\label{def_1}\rm
On $\mathcal A(\C)\times S_\infty$, define
\[
\begin{aligned}
&\mathcal E^{1,0}_{\mathcal A(\C)\times S_\infty}
:=
\mathrm{pr}_{\mathcal A}^{*}
\mathcal E^{1,0}_{\mathcal A(\C)},\qquad
&\mathcal E^{0,1}_{\mathcal A(\C)\times S_\infty}
:=
\mathrm{pr}_{\mathcal A}^{*}
\mathcal E^{0,1}_{\mathcal A(\C)}
\oplus
\mathrm{pr}_{S_\infty}^{*}\mathcal E^1_{S_\infty},
\end{aligned}
\]
where forms on $S_\infty$ are taken with complex coefficients.
For $p,q\geq0$, put
\[
\mathcal E^{p,q}_{\mathcal A(\C)\times S_\infty}
:=
\bigwedge^p
\mathcal E^{1,0}_{\mathcal A(\C)\times S_\infty}
\otimes
\bigwedge^q
\mathcal E^{0,1}_{\mathcal A(\C)\times S_\infty}.
\]
Here the pullbacks, exterior powers, and tensor products are
taken in the category of smooth bundles and their sheaves of
sections.
Thus forms pulled back from $S_\infty$ are placed in
bidegree $(0,1)$: the second degree combines the
anti-holomorphic degree on $\mathcal A(\C)$ with the real
de Rham degree on $S_\infty$.
\end{defn}

The diagonal $\Gamma$-action preserves this decomposition,
since it acts holomorphically on $\mathcal A(\C)$ and smoothly
on $S_\infty$, without mixing the two factors.
The mixed form sheaves are therefore $\Gamma$-equivariant.
Moreover, pullback commutes with the Dolbeault differential on
$\mathcal A(\C)$ and the de Rham differential on $S_\infty$.
Together with the equivariant smooth realization of
Theorem~\ref{thm: compare P_n and TSym with connection}, these
operators give the equivariant mixed complex used below.
We write $\nabla_{\mathcal A}''$ for the Dolbeault operator on
$\mathcal A(\C)$ and $\nabla''$ for the total operator on the
product, with the totalization signs specified in the proof.

\begin{thm}[Equivariant Dolbeault resolution]
\label{thm: Dolbeault realization}
For every $m\geq0$, there is a natural $\Gamma$-equivariant
quasi-isomorphism of complexes of sheaves on
$\mathcal A(\C)\times S_\infty$:
\[
\left(
\mathrm{pr}_{\mathcal A}^{-1}
\bigl(
\mathcal P^{(m),\mathrm{an}}
\otimes
\Omega_{\mathcal A(\C)}^d
\bigr)
\right)[0]
\xlongrightarrow{\sim}
\left(
\mathrm{pr}_{\mathcal A}^{*}
\mathcal P^{(m),\mathrm{sm}}
\otimes
\mathcal E_{\mathcal A(\C)\times S_\infty}^{d,\bullet},
\nabla''
\right).
\]
For every $i\geq0$, the quasi-isomorphism induces a canonical
identification
\begin{equation*}
  \scalebox{0.99}{$\displaystyle
H^i\left(
\mathcal U_{\mathcal D}(\C),\Gamma;
j^{-1}\bigl(
\mathcal P^{(m),\mathrm{an}}
\otimes\Omega_{\mathcal A(\C)}^d
\bigr)
\right)\cong
H^i\left(
H^0\left(
\mathcal U_{\mathcal D}(\C)\times S_\infty,
(j\times\mathrm{id}_{S_\infty})^{-1}
\left(\mathrm{pr}_{\mathcal A}^{*}\mathcal P^{(m),\mathrm{sm}}
\otimes
\mathcal E_{\mathcal A(\C)\times S_\infty}^{d,\bullet}\right)
\right)^\Gamma,
\nabla''
\right).
  $}
\end{equation*}
These quasi-isomorphisms and cohomological identifications are
compatible with the transition morphisms in $m$.

Under the inverse-limit comparison for the completed coefficient
system,
\begin{equation}\label{eq: completed analytic inverse limit}
\begin{aligned}
&H^i\left(
\mathcal U_{\mathcal D}(\C),\Gamma;
j^{-1}\bigl(
\widehat{\mathcal P}^{\mathrm{an}}
\otimes\Omega_{\mathcal A(\C)}^d
\bigr)
\right)\xrightarrow{\sim}
\varprojlim_m
H^i\left(
\mathcal U_{\mathcal D}(\C),\Gamma;
j^{-1}\bigl(
\mathcal P^{(m),\mathrm{an}}
\otimes\Omega_{\mathcal A(\C)}^d
\bigr)
\right),
\end{aligned}
\end{equation}
every class in the completed group is represented by a compatible
system of smooth $\Gamma$-equivariant
$\nabla''$-closed $(d,i)$-forms with values in
$\mathrm{pr}_{\mathcal A}^{*}\mathcal P^{(m),\mathrm{sm}}$,
modulo compatible systems that are exact in the corresponding
complexes of $\Gamma$-invariant sections.
\end{thm}

\begin{proof}
Corollary~\ref{cor_7}, with $p=d$, gives the
$\Gamma$-equivariant quasi-isomorphism
\[
\left(
\mathcal P^{(m),\mathrm{an}}
\otimes\Omega_{\mathcal A(\C)}^d
\right)[0]
\xlongrightarrow{\sim}
\left(
\mathcal P^{(m),\mathrm{sm}}
\otimes
\mathcal E_{\mathcal A(\C)}^{d,\bullet},
\nabla_{\mathcal A}''
\right).
\]
Applying the exact inverse-image functor
$\mathrm{pr}_{\mathcal A}^{-1}$ gives a resolution of
$\mathrm{pr}_{\mathcal A}^{-1}
(\mathcal P^{(m),\mathrm{an}}\otimes
\Omega_{\mathcal A(\C)}^d)$
by the pulled-back Dolbeault complex.
We resolve each of its terms in the $S_\infty$-direction
using \eqref{eq_15}.
The resulting double complex has terms
\[
C^{q,r}
:=
\mathrm{pr}_{\mathcal A}^{*}
\mathcal P^{(m),\mathrm{sm}}
\otimes
\mathrm{pr}_{\mathcal A}^{*}
\mathcal E_{\mathcal A(\C)}^{d,q}
\otimes
\mathrm{pr}_{S_\infty}^{*}
\mathcal E_{S_\infty}^{r},
\]
where $q$ is the anti-holomorphic degree on $\mathcal A(\C)$
and $r$ is the real de Rham degree on $S_\infty$.
All tensor products in this double complex are over
$\mathcal C^\infty_{\mathcal A(\C)\times S_\infty}$.
The horizontal differential is
$d_h:=\mathrm{pr}_{\mathcal A}^{*}\nabla_{\mathcal A}''$,
acting in the $\mathcal A(\C)$-direction, while the vertical
differential on $C^{q,r}$ is
$d_v:=(-1)^q d_{S_\infty}$.
Thus the two degrees increase as follows:
% https://q.uiver.app/#q=WzAsMTIsWzEsMiwiXFxtYXRocm17cHJ9X3tcXG1hdGhjYWx7QX19XipcXG1hdGhjYWx7UH1eeyhuKSxcXG1hdGhybXtzbX19XFxvdGltZXMgXFxtYXRocm17cHJ9X3tcXG1hdGhjYWx7QX19XipcXG1hdGhjYWx7RX1ee2QscX1fe1xcbWF0aGNhbHtBfShcXEMpfVxcb3RpbWVzXFxtYXRocm17cHJ9X3tTX1xcaW5mdHl9XipcXG1hdGhjYWx7RX1ee3IrMX1fe1NfXFxpbmZ0eX0iXSxbMSw0LCJcXG1hdGhybXtwcn1fe1xcbWF0aGNhbHtBfX1eKlxcbWF0aGNhbHtQfV57KG4pLFxcbWF0aHJte3NtfX1cXG90aW1lcyBcXG1hdGhybXtwcn1fe1xcbWF0aGNhbHtBfX1eKlxcbWF0aGNhbHtFfV57ZCxxfV97XFxtYXRoY2Fse0F9KFxcQyl9XFxvdGltZXNcXG1hdGhybXtwcn1fe1NfXFxpbmZ0eX1eKlxcbWF0aGNhbHtFfV57cn1fe1NfXFxpbmZ0eX0iXSxbMiwyLCJcXG1hdGhybXtwcn1fe1xcbWF0aGNhbHtBfX1eKlxcbWF0aGNhbHtQfV57KG4pLFxcbWF0aHJte3NtfX1cXG90aW1lcyBcXG1hdGhybXtwcn1fe1xcbWF0aGNhbHtBfX1eKlxcbWF0aGNhbHtFfV57ZCxxKzF9X3tcXG1hdGhjYWx7QX0oXFxDKX1cXG90aW1lc1xcbWF0aHJte3ByfV97U19cXGluZnR5fV4qXFxtYXRoY2Fse0V9XntyKzF9X3tTX1xcaW5mdHl9Il0sWzIsNCwiXFxtYXRocm17cHJ9X3tcXG1hdGhjYWx7QX19XipcXG1hdGhjYWx7UH1eeyhuKSxcXG1hdGhybXtzbX19XFxvdGltZXMgXFxtYXRocm17cHJ9X3tcXG1hdGhjYWx7QX19XipcXG1hdGhjYWx7RX1ee2QscSsxfV97XFxtYXRoY2Fse0F9KFxcQyl9XFxvdGltZXNcXG1hdGhybXtwcn1fe1NfXFxpbmZ0eX1eKlxcbWF0aGNhbHtFfV57cn1fe1NfXFxpbmZ0eX0iXSxbMSwwLCJcXHZkb3RzIl0sWzIsMCwiXFx2ZG90cyJdLFsxLDYsIlxcdmRvdHMiXSxbMiw2LCJcXHZkb3RzIl0sWzMsMiwiXFxjZG90cyJdLFszLDQsIlxcY2RvdHMiXSxbMCwyLCJcXGNkb3RzIl0sWzAsNCwiXFxjZG90cyJdLFsxLDAsImRfdiJdLFsxLDMsImRfaCJdLFszLDIsImRfdiJdLFswLDIsImRfaCJdLFswLDQsImRfdiJdLFsyLDUsImRfdiJdLFszLDksImRfaCJdLFsyLDgsImRfaCJdLFs3LDMsImRfdiJdLFs2LDEsImRfdiJdLFsxMCwwLCJkX2giXSxbMTEsMSwiZF9oIl1d
\[\begin{tikzcd}
	& \vdots & \vdots & \\
	\\
	\cdots & {\mathrm{pr}_{\mathcal{A}}^*\mathcal{P}^{(n),\mathrm{sm}}\otimes \mathrm{pr}_{\mathcal{A}}^*\mathcal{E}^{d,q}_{\mathcal{A}(\C)}\otimes\mathrm{pr}_{S_\infty}^*\mathcal{E}^{r+1}_{S_\infty}} & {\mathrm{pr}_{\mathcal{A}}^*\mathcal{P}^{(n),\mathrm{sm}}\otimes \mathrm{pr}_{\mathcal{A}}^*\mathcal{E}^{d,q+1}_{\mathcal{A}(\C)}\otimes\mathrm{pr}_{S_\infty}^*\mathcal{E}^{r+1}_{S_\infty}} & \cdots \\
	\\
	\cdots & {\mathrm{pr}_{\mathcal{A}}^*\mathcal{P}^{(n),\mathrm{sm}}\otimes \mathrm{pr}_{\mathcal{A}}^*\mathcal{E}^{d,q}_{\mathcal{A}(\C)}\otimes\mathrm{pr}_{S_\infty}^*\mathcal{E}^{r}_{S_\infty}} & {\mathrm{pr}_{\mathcal{A}}^*\mathcal{P}^{(n),\mathrm{sm}}\otimes \mathrm{pr}_{\mathcal{A}}^*\mathcal{E}^{d,q+1}_{\mathcal{A}(\C)}\otimes\mathrm{pr}_{S_\infty}^*\mathcal{E}^{r}_{S_\infty}} & \cdots \\
	\\
	& \vdots & \vdots
	\arrow["{d_h}", from=3-1, to=3-2]
	\arrow["{d_v}", from=3-2, to=1-2]
	\arrow["{d_h}", from=3-2, to=3-3]
	\arrow["{d_v}", from=3-3, to=1-3]
	\arrow["{d_h}", from=3-3, to=3-4]
	\arrow["{d_h}", from=5-1, to=5-2]
	\arrow["{d_v}", from=5-2, to=3-2]
	\arrow["{d_h}", from=5-2, to=5-3]
	\arrow["{d_v}", from=5-3, to=3-3]
	\arrow["{d_h}", from=5-3, to=5-4]
	\arrow["{d_v}", from=7-2, to=5-2]
	\arrow["{d_v}", from=7-3, to=5-3]
\end{tikzcd}\]
The unsigned operators in the two directions commute:
$\nabla_{\mathcal A}''$ is pulled back from $\mathcal A(\C)$,
so its coefficients are independent of $S_\infty$.
Consequently, on $C^{q,r}$,
\[
d_hd_v+d_vd_h
=
(-1)^q d_h d_{S_\infty}
+
(-1)^{q+1}d_{S_\infty}d_h
=
0.
\]
The displayed squares therefore anticommute.
Together with $d_h^2=d_v^2=0$, this defines the total differential
$\nabla'':=d_h+d_v$.
The coefficient linearization commutes with $d_h$ by
Theorem~\ref{thm: compare P_n and TSym with connection},
and $d_v$ commutes with $\Gamma$ by naturality of the
de Rham differential.
Hence the double complex and its total differential are
$\Gamma$-equivariant.

For fixed $q$, the relative Poincaré lemma applies to the
smooth coefficient bundle
$\mathcal P^{(m),\mathrm{sm}}\otimes
\mathcal E_{\mathcal A(\C)}^{d,q}$.
Indeed, in a local smooth frame on $\mathcal A(\C)$,
the contraction in the $S_\infty$-direction acts componentwise
and leaves the coefficient frame fixed.
It follows that
\[
\mathcal H^r(C^{q,\bullet},d_v)
\cong
\begin{cases}
\mathrm{pr}_{\mathcal A}^{-1}
\left(
\mathcal P^{(m),\mathrm{sm}}
\otimes
\mathcal E_{\mathcal A(\C)}^{d,q}
\right),
& r=0,\\[4pt]
0,
& r>0.
\end{cases}
\]
The vertical-cohomology spectral sequence of the bounded
double complex therefore has only its row $r=0$ remaining.
The differential on that row is the inverse image of
$\nabla_{\mathcal A}''$.
Thus the natural map from the pulled-back Dolbeault complex
to $\operatorname{Tot}C^{\bullet,\bullet}$ is a
quasi-isomorphism.
By Definition~\ref{def_1}, the mixed forms decompose as
\[
\mathcal E_{\mathcal A(\C)\times S_\infty}^{d,k}
\cong
\bigoplus_{q+r=k}
\mathrm{pr}_{\mathcal A}^{*}
\mathcal E_{\mathcal A(\C)}^{d,q}
\otimes
\mathrm{pr}_{S_\infty}^{*}
\mathcal E_{S_\infty}^{r}.
\]
Tensoring this decomposition with the coefficient bundle gives
\[
\operatorname{Tot}^k C^{\bullet,\bullet}
\cong
\mathrm{pr}_{\mathcal A}^{*}
\mathcal P^{(m),\mathrm{sm}}
\otimes
\mathcal E_{\mathcal A(\C)\times S_\infty}^{d,k}
=
\mathrm{pr}_{\mathcal A}^{*}
\mathcal P^{(n),\mathrm{sm}}
\otimes
\mathcal E_{\mathcal A(\C)\times S_\infty}^{d,k}.
\]
Under this identification, the total differential is precisely
the mixed operator $\nabla''$.
Composing with the inverse image of the ordinary Dolbeault
resolution gives the asserted quasi-isomorphism on the product.

We next restrict to
$\mathcal U_{\mathcal D}(\C)\times S_\infty$.
The projection and restriction maps form the Cartesian square
\[
\begin{tikzcd}[column sep=4.5em, row sep=3em]
\mathcal U_{\mathcal D}(\C)\times S_\infty
  \arrow[r,"\mathrm{pr}_{\mathcal U_{\mathcal D}}"]
  \arrow[d,"j\times\mathrm{id}_{S_\infty}"']
&
\mathcal U_{\mathcal D}(\C)
  \arrow[d,"j"]
\\
\mathcal A(\C)\times S_\infty
  \arrow[r,"\mathrm{pr}_{\mathcal A}"']
&
\mathcal A(\C).
\end{tikzcd}
\]
In particular,
$(j\times\mathrm{id}_{S_\infty})^{-1}
\mathrm{pr}_{\mathcal A}^{-1}
=
\mathrm{pr}_{\mathcal U_{\mathcal D}}^{-1}j^{-1}$.
Since inverse image is exact, restriction gives the
$\Gamma$-equivariant quasi-isomorphism
\[
\begin{aligned}
&
\left(
\mathrm{pr}_{\mathcal U_{\mathcal D}}^{-1}
j^{-1}
\bigl(
\mathcal P^{(m),\mathrm{an}}
\otimes\Omega_{\mathcal A(\C)}^d
\bigr)
\right)[0]
\xlongrightarrow{\sim}
\left(
(j\times\mathrm{id}_{S_\infty})^{-1}
\left(
\mathrm{pr}_{\mathcal A}^{*}
\mathcal P^{(m),\mathrm{sm}}
\otimes
\mathcal E_{\mathcal A(\C)\times S_\infty}^{d,\bullet}
\right),
\nabla''
\right).
\end{aligned}
\]
Let
$q:\mathcal U_{\mathcal D}(\C)\times S_\infty
\to\mathcal U_{\mathcal D}(\C)\times_\Gamma S_\infty$
be the quotient map.
Proposition~\ref{prop: equiv coh via Borel} identifies
equivariant cohomology with the cohomology of the descended
coefficient sheaf:
\[
\begin{aligned}
&H^i\left(
\mathcal U_{\mathcal D}(\C),\Gamma;
j^{-1}\bigl(
\mathcal P^{(m),\mathrm{an}}
\otimes\Omega_{\mathcal A(\C)}^d
\bigr)
\right)
\cong
H^i\left(
\mathcal U_{\mathcal D}(\C)\times_\Gamma S_\infty,
q_*^\Gamma
\mathrm{pr}_{\mathcal U_{\mathcal D}}^{-1}
j^{-1}\bigl(
\mathcal P^{(m),\mathrm{an}}
\otimes\Omega_{\mathcal A(\C)}^d
\bigr)
\right).
\end{aligned}
\]
The restricted resolution is $\Gamma$-equivariant.
Since the action is free and properly discontinuous,
$q_*^\Gamma$ is exact, as in the proof of that proposition,
and the resolution descends to the quotient.
Each term of the descended complex is the sheaf of smooth
sections of a vector bundle.
Partitions of unity on the quotient make these sheaves fine,
so the hypercohomology of the descended resolution is computed
by its complex of global sections.
By the definition of $q_*^\Gamma$, those global sections are
exactly the $\Gamma$-invariant sections upstairs.
This gives
\begin{equation*}
  \scalebox{0.97}{$\displaystyle
H^i\left(
  \mathcal U_{\mathcal D}(\C),
  \Gamma;
  j^{-1}
  \bigl(
    \mathcal P^{(m),\mathrm{an}}
    \otimes
    \Omega_{\mathcal A(\C)}^d
  \bigr)
\right)
\cong
H^i\left(
  \left(
    H^0\left(
      \mathcal U_{\mathcal D}(\C)\times S_\infty,
      (j\times\mathrm{id}_{S_\infty})^{-1}
      \left(
        \mathrm{pr}_{\mathcal A}^{*}
        \mathcal P^{(m),\mathrm{sm}}
        \otimes
        \mathcal E_{\mathcal A(\C)\times S_\infty}^{d,\bullet}
      \right)
    \right)^\Gamma,
    \nabla''
  \right)
\right).
  $}
\end{equation*}

All constructions commute with the transition morphisms
$\mathcal P^{(m+1),\mathrm{sm}}\to
\mathcal P^{(m),\mathrm{sm}}$.
For the Dolbeault differential this follows from
Corollary~\ref{cor_7}; the relative de Rham differential acts
only in the $S_\infty$-direction, and restriction and descent
are functorial.
The finite-level cohomological identifications therefore form
an isomorphism of inverse systems.
Moreover, Theorem~\ref{thm: compare P_n and TSym with connection}
identifies the transition maps with projection onto the
summands of symmetric degree at most $m$.
They consequently induce degreewise surjective maps on the
complexes of $\Gamma$-invariant global forms.

This surjectivity allows compatible cohomology classes to be
represented by compatible closed forms.
Indeed, let $t_m$ denote the transition from level $m+1$
to level $m$, and suppose that a closed representative
$\omega_m$ has already been chosen.
Choose a closed representative $\beta_{m+1}$ of the next
class.
Compatibility of the classes means that
$t_m\beta_{m+1}-\omega_m=\nabla''\eta_m$
for an invariant form $\eta_m$ of degree $(d,i-1)$.
Lift $\eta_m$ to an invariant form
$\widetilde\eta_{m+1}$ at level $m+1$ and set
\[
\omega_{m+1}
:=
\beta_{m+1}-\nabla''\widetilde\eta_{m+1}.
\]
Then $\omega_{m+1}$ is closed, represents the chosen class,
and satisfies $t_m\omega_{m+1}=\omega_m$.
For $i=0$, compatibility already holds for the closed
representatives, since there are no boundaries in degree zero.
Two compatible systems give the same element of the inverse
limit of cohomology groups exactly when their difference is
exact in the invariant-section complex at every finite level.
Using \eqref{eq: completed analytic inverse limit} gives
the completed description.
\end{proof}

\begin{defn}[Mixed currents]\rm
Since $\dim_{\R}S_\infty=gn^2$, the real dimension of
$\mathcal A(\C)\times S_\infty$ is $2d+gn^2$.
For $0\leq p\leq d$ and $0\leq q\leq d+gn^2$, the sheaf
$\mathcal D^{p,q}_{\mathcal A(\C)\times S_\infty}$
of $(p,q)$-currents is defined on an open subset
$W\subset\mathcal A(\C)\times S_\infty$ by
\[
H^0\left(
W,\mathcal D^{p,q}_{\mathcal A(\C)\times S_\infty}
\right)
:=
\left(
\mathcal E_c^{d-p,d+gn^2-q}(W)
\right)^*,
\]
where $^*$ denotes the continuous complex-linear dual of
the space of compactly supported smooth forms, equipped
with its usual test-function topology.
Restriction of currents is defined by extension by zero
of compactly supported test forms.
The differentials on currents are induced by duality from
the differentials on test forms, with the signs prescribed
by Stokes' formula and the totalization convention above.
\end{defn}

\begin{rem}\label{rem_3}\rm
Integration defines a morphism from smooth forms to currents:
\[
\iota:
\mathcal E^{p,q}_{\mathcal A(\C)\times S_\infty}
\longrightarrow
\mathcal D^{p,q}_{\mathcal A(\C)\times S_\infty},
\qquad
\omega\longmapsto
\left(
\eta\longmapsto
\int_{\mathcal A(\C)\times S_\infty}\eta\wedge\omega
\right),
\]
where the integral is taken over the support of the test form.
Let $d_0''$ denote the mixed differential for the trivial
coefficient bundle, namely the totalization of the ordinary
Dolbeault and de Rham differentials.
The duality convention makes $\iota$ a morphism of complexes,
and it induces a quasi-isomorphism
\[
\left(
\mathcal E^{d,\bullet}_{\mathcal A(\C)\times S_\infty},
d_0''
\right)
\xlongrightarrow{\sim}
\left(
\mathcal D^{d,\bullet}_{\mathcal A(\C)\times S_\infty},
d_0''
\right).
\]

The integration morphism is also $\Gamma$-equivariant.
The action on $\mathcal A(\C)$ preserves its complex
orientation, while the action on $S_\infty$ is induced
by the connected group
$G_\infty=\prod_{i=1}^g\GL_n(\C)$ and preserves its
orientation.
Equivariance of $\iota$ therefore follows from change
of variables under the diagonal action.
\end{rem}

We extend $\nabla''$ to currents with values in
$\mathrm{pr}_{\mathcal A}^{*}\mathcal P^{(m),\mathrm{sm}}$
by distributional differentiation.
In a local frame, its smooth coefficient matrices act
on currents by multiplication, and its differential
operators act by duality.
This construction is independent of the frame and
commutes with integration.

\begin{cor}\label{cor_8}
For every $n\geq0$, there is a natural $\Gamma$-equivariant
quasi-isomorphism
\[
\left(
  \mathrm{pr}_{\mathcal A}^{-1}
  \bigl(
    \mathcal P^{(n),\mathrm{an}}
    \otimes\Omega_{\mathcal A(\C)}^d
  \bigr)
\right)[0]
\xlongrightarrow{\sim}
\left(
  \mathrm{pr}_{\mathcal A}^{*}
  \mathcal P^{(n),\mathrm{sm}}
  \otimes
  \mathcal D_{\mathcal A(\C)\times S_\infty}^{d,\bullet},
  \nabla''
\right),
\]
compatible with the transition morphisms in $n$. Under the inverse-limit identification of
Theorem~\ref{thm: Dolbeault realization}, for every $i\geq0$, every class in
\(
H^i\left(
  \mathcal U_{\mathcal D}(\C),\Gamma;
  j^{-1}\bigl(
    \widehat{\mathcal P}^{\mathrm{an}}
    \otimes\Omega_{\mathcal A(\C)}^d
  \bigr)
\right)
\)
can be represented by a compatible system of $\Gamma$-equivariant
$\nabla''$-closed $(d,i)$-currents on
$\mathcal U_{\mathcal D}(\C)\times S_\infty$ with values in
$\mathrm{pr}_{\mathcal A}^{*}\mathcal P^{(n),\mathrm{sm}}$,
modulo compatible systems of $\nabla''$-exact currents in the
$\Gamma$-invariant complex at each level.
\end{cor}

\begin{proof}
By Remark~\ref{rem_3}, integration gives a $\Gamma$-equivariant
quasi-isomorphism from smooth forms to currents.
The same holds with coefficients in
$\mathrm{pr}_{\mathcal A}^{*}\mathcal P^{(n),\mathrm{sm}}$:
in local holomorphic frames pulled back from $\mathcal A(\C)$,
the comparison reduces to a finite direct sum of scalar comparisons.
Composing with the resolution in
Theorem~\ref{thm: Dolbeault realization} gives the first assertion.
The current sheaves descend to fine sheaves on the Borel quotient,
so the global-sections argument of that theorem applies.
Compatibility with the transition morphisms in $n$ then gives
the completed description.
\end{proof}

The equivariant Eisenstein--Kronecker classes may therefore
be studied through this current model.

\subsection{Fourier construction of the Eisenstein--Kronecker current}

To construct an explicit representative in the equivariant current model
established above, we first introduce the Fourier data and
Eisenstein--Kronecker series associated with the varying Hermitian
structures. For $g\in S_\infty$, we use the same notation $H_g$ for the
positive definite Hermitian form and its matrix in the chosen
coordinates, so that $H_g(v,w)=v^*H_gw$. Write
$\|v\|_g^2:=H_g(v,v)$ and
$\la v,w\ra_g:=\operatorname{Im}H_g(v,w)$. The latter is a
non-degenerate alternating $\R$-bilinear form, with respect to which
the dual lattice is
\[
  \Lambda_g^*
  :=
  \bigl\{
    l\in\C^d:\la l,\Lambda\ra_g\subset\Z
  \bigr\}.
\]
The $\Gamma$-equivariance of the Hermitian forms gives
\[
  \la\gamma_\infty v,\gamma_\infty w\ra_{\gamma\cdot g}
  =
  \la v,w\ra_g.
\]
Since $\gamma_\infty\Lambda=\Lambda$, it follows that
$\gamma_\infty\Lambda_g^*=\Lambda_{\gamma\cdot g}^*$.
For the standard Hermitian form $H_0$, we write
$\la\cdot,\cdot\ra:=\la\cdot,\cdot\ra_{H_0}$ and
$\Lambda^*:=\Lambda_{H_0}^*$.
For $l_g\in\Lambda_g^*$, the functional
$w\mapsto\la l_g,w\ra_g$ is $\Z$-valued on $\Lambda$, and hence defines
the smooth character
\[
  \chi_{l_g}:\C^d/\Lambda\longrightarrow S^1\subset\C^\times,
  \qquad
  [w]\longmapsto
  \exp\bigl(2\pi i\la l_g,w\ra_g\bigr).
\]
Finally, $\operatorname{covol}_g(\Lambda)$ denotes the covolume of
$\Lambda$ for the Riemannian metric induced by $H_g$.

Under the identification
$\Lambda\cong H_1(\mathcal A(\C),\Z)$, the alternating form associated
with $H_0$ agrees with the polarization pairing after removing the
Tate twist. More precisely, with the convention
$\Z(1)=2\pi i\,\Z$, the Tate-twisted polarization pairing satisfies
\[
  \la v,w\ra_{\mathrm{pol}}
  =
  2\pi i\,\la v,w\ra,
  \qquad v,w\in\Lambda.
\]
Thus identifying $\Z(1)$ with $\Z$ by $2\pi i\,m\mapsto m$ recovers
the restriction of $\la\cdot,\cdot\ra$ to $\Lambda$; compare
\cite[Section~2.1.1]{Levin_2000}.
This normalization accounts for the factors $2\pi i$ in the Fourier
characters and in the differential forms and residue currents used
below. For general $g$, the corresponding dual lattice and characters
are defined using the transported alternating form
$\la\cdot,\cdot\ra_g$.

The analytic functions entering the Fourier construction are the
following Eisenstein--Kronecker series, considered in the framework
of \cite[Section~3.3]{KS_2025}.

\begin{defn}[Eisenstein--Kronecker series]\rm
Let $\mu\in\mathbb N^d$, let $\Lambda\subset\C^d$ be a full lattice,
and let $z,w\in\C^d$. Let $H=H_g$ be a positive definite Hermitian
form as above, with associated alternating form
$\la\cdot,\cdot\ra$. We define
\[
  K^\mu(H,z,w,s,\Lambda)
  :=
  \sum_{\lambda\in\Lambda\setminus\{-z\}}
  \frac{
    (\overline z+\overline\lambda)^\mu
  }{
    |\!|z+\lambda|\!|_H^{2s}
  }
  e^{2\pi i\la\lambda,w\ra},
\]
where $|\!|z|\!|_H^2:=H(z,z)$ and the term $\lambda=-z$ is omitted
whenever $z\in-\Lambda$.
\end{defn}

These series are special cases of the generalized Epstein zeta
functions considered in \cite[Chapter~I, Section~5]{Siegel_1980}.
They converge absolutely for
$\operatorname{Re}(s)>d+\frac{|\mu|}{2}$, locally uniformly on compact
subsets of this domain with $z\notin-\Lambda$.
Indeed, the exponential factor has absolute value one. For $z$ in a
fixed compact set and all but finitely many $\lambda\in\Lambda$,
the norms $|\!|z+\lambda|\!|_H$ and $|\!|\lambda|\!|_H$ are uniformly
comparable, while
$|(\overline z+\overline\lambda)^\mu|
\ll(1+|\!|\lambda|\!|_H)^{|\mu|}$.
The tail is therefore bounded by a constant multiple of
\(
  \sum_{\lambda\in\Lambda\setminus\{0\}}
  |\!|\lambda|\!|_H^{|\mu|-2\operatorname{Re}(s)}.
\)
Since $\Lambda$ has real rank $2d$, this converges when
$2\operatorname{Re}(s)-|\mu|>2d$, and the same estimate gives local
uniform convergence. At $z\in-\Lambda$, absolute convergence holds
with the omitted-term convention in the definition.

\begin{prop}[Meromorphic continuation and functional equation]
\label{prop:functional equation of Eisenstein-Kronecker series}
The completed Eisenstein--Kronecker series
$\Gamma(s)K^\mu(H,z,w,s,\Lambda)/\pi^s$ admits a meromorphic
continuation to $\C$, whose only possible poles are simple poles
at $s=0$ and $s=d$. Define
\[
  \delta_{\mu,z}
  :=
  \begin{cases}
    1,& z\in\Lambda\text{ and }|\mu|=0,\\
    0,& \text{otherwise},
  \end{cases}
  \qquad
  \delta_{\mu,w}
  :=
  \begin{cases}
    1,& w\in\Lambda_H^*\text{ and }|\mu|=0,\\
    0,& \text{otherwise}.
  \end{cases}
\]
The residues are
\[
  \operatorname*{Res}_{s=0}
  \frac{\Gamma(s)K^\mu(H,z,w,s,\Lambda)}{\pi^s}
  =
  -\delta_{\mu,z}e^{2\pi i\la-z,w\ra}
\]
and
\[
  \operatorname*{Res}_{s=d}
  \frac{\Gamma(s)K^\mu(H,z,w,s,\Lambda)}{\pi^s}
  =
  \frac{\delta_{\mu,w}}{\mathrm{covol}_H(\Lambda)}.
\]
Moreover, the functional equation is
\[
  \frac{\Gamma(s)K^\mu(H,z,w,s,\Lambda)}{\pi^s}
  =
  \frac{e^{2\pi i\la z,w\ra}}{\mathrm{covol}_H(\Lambda)}
  \frac{
    \Gamma(d+|\mu|-s)
    K^\mu(H,w,z,d+|\mu|-s,\Lambda_H^*)
  }{
    \pi^{d+|\mu|-s}
  }.
\]
\end{prop}

\begin{proof}
We follow the method of \cite[Proposition~3.13]{KS_2025}, based on
Poisson summation and the Mellin transform of the associated theta
function. For $t>0$, set
\[
  \theta_t^\mu(H,z,w,\Lambda)
  :=
  \sum_{\lambda\in\Lambda}
  (\overline z+\overline\lambda)^\mu
  e^{-\pi t|\!|z+\lambda|\!|_H^2}
  e^{2\pi i\la\lambda,w\ra}.
\]
The Gaussian factor gives absolute and local uniform convergence
on $\R_{>0}\times\C^d\times\C^d$.
With respect to the Riemannian measure induced by $H$, Poisson
summation uses the dual lattice $\Lambda_H^*$ determined by
$\la\cdot,\cdot\ra$ and contributes the factor
$\mathrm{covol}_H(\Lambda)^{-1}$. Applied to the
polynomial-weighted Gaussian, it gives the theta transformation
\[
  \theta_t^\mu(H,z,w,\Lambda)
  =
  t^{-d-|\mu|}
  \frac{e^{2\pi i\la w,z\ra}}{\mathrm{covol}_H(\Lambda)}
  \theta_{t^{-1}}^\mu(H,w,z,\Lambda_H^*).
\]

In the region of absolute convergence, Euler's integral gives
\[
  \int_0^\infty
  e^{-\pi t|\!|z+\lambda|\!|_H^2}t^s\frac{dt}{t}
  =
  \frac{\Gamma(s)}{\pi^s|\!|z+\lambda|\!|_H^{2s}}
  \qquad(\lambda\neq-z).
\]
Termwise integration therefore yields
\[
  \frac{\Gamma(s)K^\mu(H,z,w,s,\Lambda)}{\pi^s}
  =
  \int_0^\infty
  \left(
    \theta_t^\mu(H,z,w,\Lambda)
    -
    \delta_{\mu,z}e^{2\pi i\la-z,w\ra}
  \right)t^s\frac{dt}{t}.
\]
Split this integral at $t=1$. After removal of the constant term,
the theta series decays exponentially as $t\to\infty$, so the
integral over $[1,\infty)$ is entire in $s$. On $(0,1]$, apply the
theta transformation and substitute $t\mapsto t^{-1}$.
Removing the constant term of the dual theta series again leaves
an exponentially decaying integrand on $[1,\infty)$.
The only non-decaying contributions are
\(
  -\frac{\delta_{\mu,z}e^{2\pi i\la-z,w\ra}}{s}
  +
  \frac{\delta_{\mu,w}}{\mathrm{covol}_H(\Lambda)}
  \frac{1}{s-d}.
\)
Here the phase of the dual constant term cancels the prefactor
$e^{2\pi i\la w,z\ra}$ in the theta transformation.
These terms give the possible simple poles and their residues.
The same transformation interchanges the two Mellin integrals
under $s\mapsto d+|\mu|-s$ and
$(z,w,\Lambda)\mapsto(w,z,\Lambda_H^*)$.
\end{proof}

\begin{cor}\label{cor:special value EK series at zero}
The Eisenstein--Kronecker series itself is holomorphic at $s=0$ and
satisfies
\[
  K^\mu(H,z,w,0,\Lambda)
  =
  -\delta_{\mu,z}e^{2\pi i\la-z,w\ra}.
\]
Its only possible pole is at $s=d$, where
\[
  \operatorname*{Res}_{s=d}K^\mu(H,z,w,s,\Lambda)
  =
  \frac{\pi^d}{\Gamma(d)}
  \frac{\delta_{\mu,w}}{\mathrm{covol}_H(\Lambda)}.
\]
In particular, if $|\mu|>0$, both residues in
Proposition~\ref{prop:functional equation of Eisenstein-Kronecker series}
vanish.
\end{cor}

\begin{proof}
Since $\Gamma(s)=s^{-1}+O(1)$ at $s=0$, the residue of the completed
series gives the regularity and value of $K^\mu$ there.
At $s=d$, the factor $\Gamma(s)/\pi^s$ is holomorphic and non-zero,
so its value converts the residue of the completed series into
the displayed residue of $K^\mu$.
\end{proof}

For $H=H_g$, we have $\Lambda_H^*=\Lambda_g^*$ and
$\mathrm{covol}_H(\Lambda)=\operatorname{covol}_g(\Lambda)$.
These series, with their varying Hermitian structure, will occur
in the Fourier expansion of the Eisenstein--Kronecker current
constructed next.

% Normalized volume and Fourier completeness.

We now prepare the Fourier-mode homotopy used to construct the current.
For $g\in S_\infty$, let
$\omega_g:=\pi i\la\nu,\nu\ra_g
=2\pi i\la\nu^{1,0},\nu^{0,1}\ra_g$
be the Kähler form associated with $\la\cdot,\cdot\ra_g$.
In the chosen coordinates, it is given by
\[
  \omega_g
  =
  -\pi\sum_{i,j=1}^d
  (H_g)_{ji}\,dz_i\wedge d\overline{z}_j.
\]
The covolume satisfies
$\mathrm{covol}_g(\Lambda)
=(2\pi i)^{-d}\int_{\C^d/\Lambda}\omega_g^d/d!$.
We therefore define the translation-invariant fiberwise
$(d,d)$-form
\[
  \vol_g
  :=
  \frac{\omega_g^d}{d!\,\mathrm{covol}_g(\Lambda)},
  \qquad
  \int_{\mathcal A(\C)}\vol_g=(2\pi i)^d.
\]
Its coordinate expression is
\[
  \begin{aligned}
    \vol_g
    &=
    \frac{(-\pi)^d\det(H_g)}{\mathrm{covol}_g(\Lambda)}
    \bigwedge_{i=1}^d
    \bigl(dz_i\wedge d\overline{z}_i\bigr)
    =
    \frac{
      \pi^d(-1)^{d(d+1)/2}\det(H_g)
    }{
      \mathrm{covol}_g(\Lambda)
    }
    dz_1\wedge\cdots\wedge dz_d
    \wedge
    d\overline{z}_1\wedge\cdots\wedge d\overline{z}_d.
  \end{aligned}
\]
The second equality uses the sign $(-1)^{d(d-1)/2}$ obtained by
moving all the holomorphic differentials to the front.
For general $g$, the form $\omega_g$ need not be $2\pi i$-integral
with respect to $\Lambda$, since this would additionally require
$\la\Lambda,\Lambda\ra_g\subseteq\Z$.
The covolume normalization does not require this integrality condition
and gives the normalization needed for Fourier inversion; compare
\cite[Sections~3.1.3 and~3.2.1]{Levin_2000}.

\begin{prop}\label{prop_3}
For every $g\in S_\infty$, the Fourier characters satisfy
\[
  \sum_{l_g\in\Lambda_g^*}\chi_{l_g}\vol_g
  =
  (2\pi i)^d\delta_0
\]
in the sense of fiberwise $(d,d)$-currents, where $\delta_0$ denotes
the fiberwise Dirac current supported on the zero section
$e(S_\infty)$.
\end{prop}

\begin{proof}
Fix $g\in S_\infty$. By the definition of $\Lambda_g^*$, the
characters $\chi_{l_g}$ are precisely the unitary characters of
$\mathcal A(\C)=\C^d/\Lambda$, and
$(2\pi i)^{-d}\vol_g$ is its normalized Haar measure.
Thus Fourier inversion at the origin gives, for every smooth test
function $f$,
\[
  \sum_{l_g\in\Lambda_g^*}
  \int_{\mathcal A(\C)}
  f(u)\chi_{l_g}(u)\vol_g
  =
  (2\pi i)^d f(0).
\]
\end{proof}

% Fourier vector fields, equivariance, and Lie derivative nilpotence.

For $l_g\in\Lambda_g^*$, define
\[
  \tilde{l}_g
  :=
  [\overline{e}_1,\ldots,\overline{e}_d]\,\overline{l}_g
  =
  \sum_{i=1}^d(\overline{l}_g)_i\,\overline{e}_i.
\]
This is a section of
$\mathrm{pr}_{\mathcal A}^*\mathcal H^{0,-1}$
on $\mathcal A(\C)\times S_\infty$.
Using the identification of $\mathcal H^{0,-1}$ with the
anti-holomorphic vertical tangent bundle, we regard $\tilde{l}_g$
as a vertical anti-holomorphic vector field and denote contraction
along it by $\iota_{\tilde{l}_g}$.
For the mixed differential
$d''=\overline{\partial}_{\mathcal A(\C)}+d_{S_\infty}$,
the corresponding Lie derivative is defined by Cartan's formula
\[
  L_{\tilde{l}_g}
  :=
  \iota_{\tilde{l}_g}d''
  +
  d''\iota_{\tilde{l}_g}.
\]

\begin{lem}\label{lem:equivariance-lie-derivative}
The Fourier vector fields are compatible with the $\Gamma$-action:
under the identification
$l_g\mapsto\gamma_\infty l_g$
of $\Lambda_g^*$ with $\Lambda_{\gamma\cdot g}^*$, one has
$\gamma_*\tilde{l}_g=\tilde{l}_{\gamma\cdot g}$.
Consequently,
\[
  \gamma^*\iota_{\tilde{l}_{\gamma\cdot g}}
  =
  \iota_{\tilde{l}_g}\gamma^*,
  \qquad
  \gamma^*L_{\tilde{l}_{\gamma\cdot g}}
  =
  L_{\tilde{l}_g}\gamma^*.
\]
\end{lem}

\begin{proof}
The basis transforms by
$\gamma_*\overline{e}_i
=\sum_{j=1}^d\overline{(\gamma_\infty)}_{ji}\overline{e}_j$,
so
\[
  \gamma_*\tilde{l}_g
  =
  \sum_{j=1}^d
  \bigl(\overline{\gamma_\infty l_g}\bigr)_j\overline{e}_j
  =
  \tilde{l}_{\gamma\cdot g}.
\]
Naturality of contraction gives the first operator identity.
The second follows from Cartan's formula, since $d''$ commutes
with the $\Gamma$-action.
\end{proof}

\begin{prop}\label{prop_1}
For $l_g\in\Lambda_g^*$, one has
\[
  L_{\tilde{l}_g}^k\vol_g=0
  \quad(k>d),
  \qquad
  \iota_{\tilde{l}_g}L_{\tilde{l}_g}^k\vol_g=0
  \quad(k\geq d).
\]
For $k=d-1$, the contraction is given explicitly by
\[
  \begin{aligned}
    \iota_{\tilde{l}_g}L_{\tilde{l}_g}^{d-1}\vol_g
    =
    \frac{
      \pi^d(-1)^{d(d-1)/2}(d-1)!\det(H_g)
    }{
      \mathrm{covol}_g(\Lambda)
    }
    \bigwedge_{j=1}^d dz_j
    \wedge
    \sum_{i=1}^d
    (-1)^{i-1}(\overline{l}_g)_i\,
    d_{S_\infty}(\overline{l}_g)_1
    \wedge\cdots\wedge
    \widehat{d_{S_\infty}(\overline{l}_g)_i}
    \wedge\cdots\wedge
    d_{S_\infty}(\overline{l}_g)_d.
  \end{aligned}
\]
\end{prop}

\begin{proof}
Since $\overline{e}_i$ corresponds to
$\partial/\partial\overline{z}_i$, contraction by $\tilde{l}_g$
annihilates $dz_i$ and sends $d\overline{z}_i$ to
$(\overline{l}_g)_i$.
The graded Leibniz rule therefore gives
$\iota_{\tilde{l}_g}(dz_i\wedge d\overline{z}_i)
=-(\overline{l}_g)_i\,dz_i$.
As $d''(dz_i\wedge d\overline{z}_i)=0$ and
$(\overline{l}_g)_i$ depends only on $S_\infty$, Cartan's formula
yields
\[
  L_{\tilde{l}_g}
  \bigl(dz_i\wedge d\overline{z}_i\bigr)
  =
  -d''(\overline{l}_g)_i\wedge dz_i
  =
  -d_{S_\infty}(\overline{l}_g)_i\wedge dz_i.
\]
A second application of $L_{\tilde{l}_g}$ gives zero:
contraction vanishes on $dz_i$ and on forms pulled back from
$S_\infty$, while $d_{S_\infty}^2=0$.
Moreover, the coefficient
$(-\pi)^d\det(H_g)/\mathrm{covol}_g(\Lambda)$ of $\vol_g$
depends only on $S_\infty$, so its Lie derivative along
$\tilde{l}_g$ vanishes.
Since $L_{\tilde{l}_g}$ is a derivation of degree zero, each of the
$d$ factors $dz_i\wedge d\overline{z}_i$ can be acted on
nontrivially at most once.
This gives $L_{\tilde{l}_g}^k\vol_g=0$ for $k>d$.
For $k=d$, every anti-holomorphic vertical differential has been
replaced by a differential from $S_\infty$, so contraction also
vanishes.

For $k=d-1$, exactly one factor remains undifferentiated.
For each choice of this factor, the other $d-1$ factors can be
differentiated in $(d-1)!$ orders.
In the $i$-th summand, contraction replaces the remaining factor
by $-(\overline{l}_g)_i\,dz_i$, while each differentiated factor
is $-d_{S_\infty}(\overline{l}_g)_j\wedge dz_j$.
These $d$ minus signs cancel the sign in $(-\pi)^d$.
Moving all the $dz_j$ to the front contributes
$(-1)^{d(d-1)/2+i-1}$, giving the coefficient and sign in the
displayed formula.
\end{proof}

% Fourier Cartan operator and its finite-filtration inverse.

At a fixed truncation level $n\geq0$, set
\[
  A_{l_g}
  :=
  2\pi i\,d''\la l_g,\cdot\ra_g+\nu^{0,1}
  \in
  \mathcal E^1_{X_\infty}
  \otimes
  \mathrm{TSym}^{\leq n}
  \bigl(\mathrm{pr}_{\mathcal A}^*\mathcal H^{0,-1}\bigr).
\]
The corresponding Fourier-mode Cartan operator on
$\mathcal E^\bullet_{X_\infty}
\otimes
\mathrm{TSym}^{\leq n}
(\mathrm{pr}_{\mathcal A}^*\mathcal H^{0,-1})$
is the degree-one operator
\[
  C_{l_g}(\omega)
  :=
  d''\omega+A_{l_g}\wedge\omega.
\]
The same operator acts on the corresponding complex of currents.

\begin{lem}\label{lem:cartan-operator-invertible}
Let $l_g\in\Lambda_g^*$ with $l_g\neq0$.
Then
\(
  C_{l_g}\circ\iota_{\tilde{l}_g}
  +
  \iota_{\tilde{l}_g}\circ C_{l_g}
\)
is invertible on the $l_g$-th Fourier coefficient subspace of
$(d,d)$-currents with coefficients in
$\mathrm{TSym}^{\leq n}
(\mathrm{pr}_{\mathcal A}^*\mathcal H^{0,-1})$.
\end{lem}

\begin{proof}
Cartan's formula and the graded Leibniz rule give
\[
  C_{l_g}\circ\iota_{\tilde{l}_g}
  +
  \iota_{\tilde{l}_g}\circ C_{l_g}
  =
  L_{\tilde{l}_g}+\iota_{\tilde{l}_g}A_{l_g},
\]
where the last term acts by multiplication.
To compute it, write
\[
  \la l_g,z\ra_g
  =
  \frac{l_g^*H_gz-z^*H_gl_g}{2i},
  \qquad
  \frac{\partial}{\partial\overline{z}_j}
  \la l_g,z\ra_g
  =
  -\frac{1}{2i}(H_gl_g)_j.
\]
Since $\tilde{l}_g$ is vertical, its contraction with the
$S_\infty$-component of $d''\la l_g,\cdot\ra_g$ vanishes.
Consequently,
\[
  \iota_{\tilde{l}_g}d''\la l_g,\cdot\ra_g
  =
  -\frac{1}{2i}l_g^*H_gl_g
  =
  -\frac{1}{2i}|\!|l_g|\!|_g^2.
\]
The tautological form satisfies
$\iota_{\tilde{l}_g}\nu^{0,1}=\overline{l}_g$,
where the right-hand side is regarded as an element of
$\mathrm{TSym}^1(\mathrm{pr}_{\mathcal A}^*\mathcal H^{0,-1})$.
Thus
$\iota_{\tilde{l}_g}A_{l_g}
=
-\pi|\!|l_g|\!|_g^2+\overline{l}_g$.

For $l_g\neq0$, the function $\pi|\!|l_g|\!|_g^2$ is nowhere
vanishing on $S_\infty$.
The positive symmetric-degree term is nilpotent in the truncated
coefficient algebra, so $\iota_{\tilde{l}_g}A_{l_g}$ is invertible,
with
\[
  \bigl(\iota_{\tilde{l}_g}A_{l_g}\bigr)^{-1}
  =
  -\sum_{b=0}^n
  \frac{
    b!\,\tilde{l}_g^{[b]}
  }{
    \bigl(\pi|\!|l_g|\!|_g^2\bigr)^{b+1}
  }.
\]

On the Fourier coefficient subspace, the coefficients depend
only on $S_\infty$.
The operator $L_{\tilde{l}_g}$ replaces an anti-holomorphic
vertical differential $d\overline{z}_i$ by
$d_{S_\infty}(\overline{l}_g)_i$ and vanishes on coefficients
and differential forms pulled back from $S_\infty$.
It therefore increases the $S_\infty$-differential degree by one
while decreasing the vertical anti-holomorphic degree by one.
Since the latter is at most $d$, this operator is nilpotent.
Multiplication by
$\bigl(\iota_{\tilde{l}_g}A_{l_g}\bigr)^{-1}$
preserves these exterior degrees, so
$\bigl(\iota_{\tilde{l}_g}A_{l_g}\bigr)^{-1}
L_{\tilde{l}_g}$
is nilpotent as well.
In the factorization
\[
  L_{\tilde{l}_g}+\iota_{\tilde{l}_g}A_{l_g}
  =
  \bigl(\iota_{\tilde{l}_g}A_{l_g}\bigr)
  \left(
    1+
    \bigl(\iota_{\tilde{l}_g}A_{l_g}\bigr)^{-1}
    L_{\tilde{l}_g}
  \right),
\]
the first factor has the inverse computed above, and the second
is inverted by a finite geometric series.
\end{proof}

This finite-filtration argument is a version of the
Lie-derivative method of
\cite[Sections~3.2--3.4]{Levin_2000}; compare also the Fourier-mode
Cartan argument in \cite[Theorem~3.20]{KS_2025}.
We will use this inverse, together with contraction by
$\tilde{l}_g$, to solve the current equation on each nonzero
Fourier mode.

The inverse in Lemma~\ref{lem:cartan-operator-invertible} gives the
following solution of the Fourier-mode equation.

\begin{prop}\label{prop_2}\label{prop:explicit-fourier-mode-current}
Let $l_g\in\Lambda_g^*\setminus\{0\}$ and $t\in\mathcal A(\C)$.
At the fixed truncation level $n$, the $(d,d-1)$-current
\[
  \phi_{l_g,t}
  :=
  \iota_{\tilde{l}_g}
  \left(
    C_{l_g}\circ\iota_{\tilde{l}_g}
    +
    \iota_{\tilde{l}_g}\circ C_{l_g}
  \right)^{-1}
  \bigl(\chi_{l_g}(-t)\vol_g\bigr)
\]
with coefficients in
$\mathrm{TSym}^{\leq n}
(\mathrm{pr}_{\mathcal A}^*\mathcal H^{0,-1})$
satisfies
\[
  d''\phi_{l_g,t}
  +
  A_{l_g}\wedge\phi_{l_g,t}
  =
  \chi_{l_g}(-t)\vol_g.
\]
Moreover, it admits the finite expansion
\[
  \phi_{l_g,t}
  =
  -
  \sum_{b=0}^n
  \sum_{j=0}^{d-1}
  \frac{\Gamma(b+j+1)}{j!}
  \frac{
    \tilde{l}_g^{[b]}
  }{
    \bigl(\pi|\!|l_g|\!|_g^2\bigr)^{b+j+1}
  }
  \iota_{\tilde{l}_g}
  L_{\tilde{l}_g}^j
  \bigl(\chi_{l_g}(-t)\vol_g\bigr).
\]
\end{prop}

\begin{proof}
One has $\iota_{\tilde{l}_g}^2=0$. The graded Leibniz rule and the
flatness of $\nabla''$ give
\[
  d''A_{l_g}+A_{l_g}\wedge A_{l_g}
  =
  d''\nu^{0,1}+\nu^{0,1}\wedge\nu^{0,1}
  =
  0,
\]
so $C_{l_g}^2=0$. Consequently, both $C_{l_g}$ and
$\iota_{\tilde{l}_g}$ commute with
$C_{l_g}\circ\iota_{\tilde{l}_g}
+\iota_{\tilde{l}_g}\circ C_{l_g}$
and hence with its inverse, which exists by
Lemma~\ref{lem:cartan-operator-invertible}.
For a fixed Fourier character, $\chi_{l_g}(-t)$ and $\vol_g$ are
independent of $g$ under the chosen trivialization. Since $\vol_g$
has top vertical anti-holomorphic degree, one has
\[
  d''\bigl(\chi_{l_g}(-t)\vol_g\bigr)=0,
  \qquad
  A_{l_g}\wedge\chi_{l_g}(-t)\vol_g=0.
\]
Thus the source term is $C_{l_g}$-closed, and the commutation with
the inverse yields
\[
  \begin{aligned}
    C_{l_g}\phi_{l_g,t}
    &=
    \left(
      C_{l_g}\circ\iota_{\tilde{l}_g}
      +
      \iota_{\tilde{l}_g}\circ C_{l_g}
    \right)
    \left(
      C_{l_g}\circ\iota_{\tilde{l}_g}
      +
      \iota_{\tilde{l}_g}\circ C_{l_g}
    \right)^{-1}
    \bigl(\chi_{l_g}(-t)\vol_g\bigr)
    =
    \chi_{l_g}(-t)\vol_g.
  \end{aligned}
\]

To obtain the explicit formula, use the finite geometric expansion
of the inverse from the first part of the proof:
\[
  \phi_{l_g,t}
  =
  \sum_{j=0}^{d-1}
  (-1)^j
  \bigl(\iota_{\tilde{l}_g}A_{l_g}\bigr)^{-(j+1)}
  \iota_{\tilde{l}_g}
  L_{\tilde{l}_g}^j
  \bigl(\chi_{l_g}(-t)\vol_g\bigr).
\]
Here $L_{\tilde{l}_g}$ annihilates coefficients depending only on
$S_\infty$, and the upper bound $d-1$ follows from
Proposition~\ref{prop_1}, since
$\iota_{\tilde{l}_g}L_{\tilde{l}_g}^j\vol_g=0$ for $j\geq d$.
In the coefficient algebra, $\tilde{l}_g$ is the degree-one
element with coordinate vector $\overline l_g$, so
\[
  \iota_{\tilde{l}_g}A_{l_g}
  =
  -\pi|\!|l_g|\!|_g^2+\tilde{l}_g.
\]
The finite binomial expansion in
$\mathrm{TSym}^{\leq n}
(\mathrm{pr}_{\mathcal A}^*\mathcal H^{0,-1})$
therefore gives
\[
  \bigl(\iota_{\tilde{l}_g}A_{l_g}\bigr)^{-(j+1)}
  =
  (-1)^{j+1}
  \sum_{b=0}^n
  \frac{(j+b)!}{j!}
  \frac{
    \tilde{l}_g^{[b]}
  }{
    \bigl(\pi|\!|l_g|\!|_g^2\bigr)^{b+j+1}
  }.
\]
Substitution combines the signs to $-1$ and gives the coefficient
$\Gamma(b+j+1)/j!$, since $(b+j)!=\Gamma(b+j+1)$.
\end{proof}

We now assemble these solutions over all nonzero Fourier modes.

% Single point current and distributional 
\begin{defn}\label{def:single-point-ek-current}\rm
For $n\geq0$ and $t\in\mathcal A(\C)$, define
\[
  \phi_{g,t}^{(n)}
  :=
  \sum_{l_g\in\Lambda_g^*\setminus\{0\}}
  \chi_{l_g}(z)\phi_{l_g,t}.
\]
Via the $\mathcal C^\infty$-trivialization of
$\mathcal P^{(n),\mathrm{sm}}$, we regard this as a
$(d,d-1)$-current
\[
  \phi_{g,t}^{(n)}
  \in
  H^0\left(
    \mathcal A(\C)\times S_\infty,
    \mathcal D^{d,d-1}_{X_\infty}
    \otimes
    \mathrm{pr}_{\mathcal A}^*
    \mathcal P^{(n),\mathrm{sm}}
  \right).
\]
\end{defn}

The Fourier series is interpreted in the sense of currents.
Locally on $S_\infty$, its coefficients have at most polynomial
growth in $l_g$, and hence define a distributional current;
compare \cite[Definition~3.15]{KS_2025}.
The character identity
$\chi_{l_g}(z)\chi_{l_g}(-t)=\chi_{l_g}(z-t)$
expresses the translation in each Fourier mode.

\[
  \phi_{g,t}^{(n)}
  =
  -
  \sum_{b=0}^n
  \sum_{j=0}^{d-1}
  \frac{\Gamma(b+j+1)}{j!}
  \sum_{l_g\in\Lambda_g^*\backslash\{0\}}
  \frac{
    \tilde{l}_g^{[b]}
  }{
    \bigl(\pi|\!|l_g|\!|_g^2\bigr)^{b+j+1}
  }
  \iota_{\tilde{l}_g}
  L_{\tilde{l}_g}^j
  \bigl(\chi_{l_g}(z-t)\vol_g\bigr).
\]

The omission of the zero Fourier mode accounts for the volume
term in the following differential equation.

\begin{prop}\label{prop:single-point-current-equation}
For every $n\geq0$ and $t\in\mathcal A(\C)$, the current
$\phi_{g,t}^{(n)}$ satisfies
\[
  \nabla''\phi_{g,t}^{(n)}
  =
  (2\pi i)^d\delta_t-\vol_g,
\]
where $\delta_t$ denotes the fiberwise $(d,d)$-Dirac current
concentrated at $t$.
\end{prop}

\begin{proof}
Under the fixed $\mathcal C^\infty$-trivialization,
$\nabla''=d''+\nu^{0,1}\wedge-$.
Since
$d''\chi_{l_g}(z)
=2\pi i\,d''\la l_g,\cdot\ra_g\,\chi_{l_g}(z)$,
the Leibniz rule identifies the action on each Fourier mode as
\[
  \nabla''
  \bigl(\chi_{l_g}(z)\phi_{l_g,t}\bigr)
  =
  \chi_{l_g}(z)C_{l_g}\phi_{l_g,t}.
\]
Applying the differential termwise in the sense of currents and
using Proposition~\ref{prop_2}, we obtain
\[
  \begin{aligned}
    \nabla''\phi_{g,t}^{(n)}
    =
    \sum_{l_g\in\Lambda_g^*\backslash\{0\}}
    \chi_{l_g}(z)C_{l_g}\phi_{l_g,t}
    =
    \sum_{l_g\in\Lambda_g^*\backslash\{0\}}
    \chi_{l_g}(z-t)\vol_g
    =
    \sum_{l_g\in\Lambda_g^*}
    \chi_{l_g}(z-t)\vol_g-\vol_g
    =
    (2\pi i)^d\delta_t-\vol_g.
  \end{aligned}
\]
The penultimate equality restores the zero Fourier mode, for which
$\chi_0=1$; the last is Proposition~\ref{prop_3} after translation
by $t$.
\end{proof}

% Weighted current construction:

We now combine the single-point currents using the trace-zero function
defining the Eisenstein--Kronecker class. Let $\mathcal D$ be the finite
$\Gamma$-stable torsion subscheme fixed in
Section~\ref{Section: The equivariant coherent Eisenstein--Kronecker classes}.
Over $\C$, we regard $f\in\C[\mathcal D]^{0,\Gamma}$ as a
$\Gamma$-invariant function on $\mathcal D(\C)$ satisfying
$\sum_{t\in\mathcal D(\C)}f(t)=0$.
For this analytic construction, only the finiteness and
$\Gamma$-stability of $\mathcal D$ are needed; in particular, it applies
to the arithmetic choice $\mathcal D=\mathcal A[\mathfrak C]$.
With $\delta_t$ normalized by $\la\delta_t,h\ra=h(t)$ for smooth
test functions $h$, put
$\delta_f:=\sum_{t\in\mathcal D(\C)}f(t)\delta_t$.
Let $\tau_t$ denote the translation $z\mapsto z-t$ on
$\mathcal A(\C)$, acting trivially on $S_\infty$.
In the fixed $\mathcal C^\infty$-trivialization, define
\[
  \phi_{g,f}^{(n)}
  :=
  \sum_{t\in\mathcal D(\C)}f(t)\tau_t^*\phi_{g,0}^{(n)}
  =
  \sum_{t\in\mathcal D(\C)}f(t)\phi_{g,t}^{(n)}.
\]
The second equality follows from the translation identity for the
Fourier characters established above.

% Differential equation, equivariance, and transition compatibility:

\begin{thm}\label{thm:differential-equation-ek-current}
For every $n\geq0$, the current $\phi_{g,f}^{(n)}$ is
$\Gamma$-equivariant and satisfies
\[
  \nabla''\phi_{g,f}^{(n)}=(2\pi i)^d\delta_f.
\]
These currents are compatible with the transition morphisms
$\mathcal P^{(n),\mathrm{sm}}\to
\mathcal P^{(n-1),\mathrm{sm}}$ for $n\geq1$.
Hence they define a $\Gamma$-equivariant pro-system
$\phi_{g,f}=(\phi_{g,f}^{(n)})_{n\geq0}$ satisfying
$\nabla''\phi_{g,f}=(2\pi i)^d\delta_f$.
\end{thm}

\begin{proof}
Summing the equation of
Proposition~\ref{prop:single-point-current-equation}
with coefficients $f(t)$ gives
\[
  \nabla''\phi_{g,f}^{(n)}
  =
  (2\pi i)^d\delta_f
  -
  \left(\sum_{t\in\mathcal D(\C)}f(t)\right)\vol_g
  =
  (2\pi i)^d\delta_f.
\]

For $\gamma\in\Gamma$, the map $l_g\mapsto\gamma_\infty l_g$
identifies $\Lambda_g^*$ with $\Lambda_{\gamma\cdot g}^*$.
Under this identification,
$\|\gamma_\infty l_g\|_{\gamma\cdot g}=\|l_g\|_g$ and
$\gamma^*\vol_{\gamma\cdot g}=\vol_g$.
The Fourier characters are compatible with the simultaneous action
on the Fourier variable and the argument:
\[
  \chi_{\gamma_\infty l_g}
  (\gamma\cdot z-\gamma\cdot t)
  =
  \chi_{l_g}(z-t).
\]
Together with the compatibility of contraction and Lie differentiation
in Lemma~\ref{lem:equivariance-lie-derivative}, these identities applied
to the single-mode formula in
Proposition~\ref{prop:explicit-fourier-mode-current} and the defining
Fourier sum give
$\gamma^*\phi_{\gamma\cdot g,\gamma\cdot t}^{(n)}
=\phi_{g,t}^{(n)}$.
Since $\mathcal D$ is $\Gamma$-stable and $f(\gamma\cdot t)=f(t)$,
reindexing the finite sum by $t\mapsto\gamma\cdot t$ yields
$\gamma^*\phi_{\gamma\cdot g,f}^{(n)}=\phi_{g,f}^{(n)}$.

Under the $\mathcal C^\infty$-trivialization, the transition morphism
$\mathcal P^{(n),\mathrm{sm}}\to
\mathcal P^{(n-1),\mathrm{sm}}$
is the truncation of the coefficient algebra from
$\mathrm{TSym}^{\leq n}$ to $\mathrm{TSym}^{\leq n-1}$.
The finite expansion in
Proposition~\ref{prop:explicit-fourier-mode-current}
is compatible with this truncation, so
$\phi_{g,f}^{(n)}$ maps to $\phi_{g,f}^{(n-1)}$.
\end{proof}

To compare the singular term in this equation with the local
cohomology class used in the algebraic construction, we record
the residue normalization.

% Residue current normalization and compatibility with the trace:

\begin{lem}[Residue current and trace normalization]
\label{lem:trace-residue-normalization}
For $t\in\mathcal D(\C)$, the current $(2\pi i)^d\delta_t$,
with coefficient given by the unit under the canonical splitting
of the completed Poincaré bundle along $\mathcal D$, represents
the local cohomology class corresponding to $1$ in the
$t$-component of $\C[\mathcal D]$.
Consequently, the assignment
$[t]\mapsto[(2\pi i)^d\delta_t]$ induces the canonical inclusion
\[
  \operatorname{cl}_{\mathcal D}:
  \C[\mathcal D]^\Gamma
  \hookrightarrow
  H^d_{\mathcal D(\C)}\left(
    \mathcal A(\C),\Gamma;
    \widehat{\mathcal P}^{\mathrm{an}}
    \otimes\Omega^d_{\mathcal A(\C)}
  \right).
\]
Its composite with the natural map to global cohomology and the
canonical identification
\[
  H^d\left(
    \mathcal A(\C),\Gamma;
    \widehat{\mathcal P}^{\mathrm{an}}
    \otimes\Omega^d_{\mathcal A(\C)}
  \right)
  \cong\C
\]
is the trace map
$\sum_{t\in\mathcal D(\C)}a_t[t]\mapsto
\sum_{t\in\mathcal D(\C)}a_t$.
In particular, $\operatorname{cl}_{\mathcal D}(f)$ lies in the
kernel of the map to global cohomology.
\end{lem}

\begin{proof}
Fix $t\in\mathcal D(\C)$, choose a lift to $\C^d$, and take
holomorphic coordinates $u_i=z_i-t_i$ on a sufficiently small
polydisc centered at $t$.
The standard Grothendieck residue class
$\left[du_1\wedge\cdots\wedge du_d/(u_1\cdots u_d)\right]$
has residue $1$ at $t$.
In the Dolbeault-current realization, it is represented by the
ordered product of the one-variable residue currents.
With our normalization,
\[
  \overline{\partial}\left(\frac{1}{u_i}\right)\wedge du_i
  =
  2\pi i\,\delta_{u_i=0}.
\]
Taking the ordered product over all coordinates therefore gives
\[
  \bigwedge_{i=1}^d
  \left(
    \overline{\partial}\left(\frac{1}{u_i}\right)\wedge du_i
  \right)
  =
  (2\pi i)^d\delta_t.
\]
The canonical splitting of the completed Poincaré bundle along
$\mathcal D$ identifies this residue-one class with the image of
$[t]\in\C[\mathcal D]$ under the inclusion into local cohomology.
This construction is compatible with the $\Gamma$-action and hence
induces $\operatorname{cl}_{\mathcal D}$ on $\Gamma$-invariants;
in particular,
$\operatorname{cl}_{\mathcal D}(f)
=[(2\pi i)^d\delta_f]$.
By Theorem~\ref{thm: function to equiv coh class}, the composite with the
map to global cohomology is the trace map; see
\cite[Section~2.6]{KS_2025}. It sends each residue-one local class to
$1$. Thus $\operatorname{cl}_{\mathcal D}(f)$ maps to
$\sum_{t\in\mathcal D(\C)}f(t)=0$.
\end{proof}

The differential equation and this residue normalization now
identify the Fourier current with the analytic realization of the
equivariant coherent Eisenstein--Kronecker class.

% Identification with the Eisenstein--Kronecker class:

\begin{thm}\label{prop:functional equation for currents}
The restriction
\[
  \left.\phi_{g,f}\right|_{\mathcal U_{\mathcal D}(\C)\times S_\infty}
  =
  \left(
    \left.\phi_{g,f}^{(n)}\right|_{\mathcal U_{\mathcal D}(\C)\times S_\infty}
  \right)_{n\geq0}
\]
is a compatible pro-system of smooth, $\Gamma$-equivariant,
$\nabla''$-closed $(d,d-1)$-forms.
Under the Dolbeault realization established above, it defines a class
\[
  [\phi_{g,f}]_{\mathrm{coh}}
  \in
  H^{d-1}\left(
    \mathcal U_{\mathcal D}(\C),\Gamma;
    j^{-1}\bigl(
      \widehat{\mathcal P}^{\mathrm{an}}
      \otimes\Omega^d_{\mathcal A(\C)}
    \bigr)
  \right).
\]
This class coincides with the equivariant coherent
Eisenstein--Kronecker class:
\[
  [\phi_{g,f}]_{\mathrm{coh}}=EK_{\Gamma,\mathcal A}(f).
\]
\end{thm}

\begin{proof}
By Theorem~\ref{thm:differential-equation-ek-current},
$\nabla''\phi_{g,f}^{(n)}=(2\pi i)^d\delta_f$.
Since $\delta_f$ is supported on $\mathcal D(\C)$, the restriction
of each $\phi_{g,f}^{(n)}$ to
$\mathcal U_{\mathcal D}(\C)\times S_\infty$ is $\nabla''$-closed.
The defining Fourier expansion is smooth away from
$\mathcal D(\C)$, and equivariance and compatibility with the
transition morphisms were proved above.
Corollary~\ref{cor_8}, together with the inverse-limit identification
of Theorem~\ref{thm: function to equiv coh class}, therefore gives
the class $[\phi_{g,f}]_{\mathrm{coh}}$ in the stated cohomology group.

In the current realization, the localization connecting morphism
$\partial$ of
\eqref{eq: localization for Eisenstein-Kronecker classes}
is computed by extending a closed form on
$\mathcal U_{\mathcal D}(\C)\times S_\infty$ to a current on
$\mathcal A(\C)\times S_\infty$ and applying $\nabla''$.
The pro-current $\phi_{g,f}$ supplies such an extension, so
Lemma~\ref{lem:trace-residue-normalization} gives
\[
  \partial\bigl([\phi_{g,f}]_{\mathrm{coh}}\bigr)
  =
  [\nabla''\phi_{g,f}]
  =
  [(2\pi i)^d\delta_f]
  =
  \operatorname{cl}_{\mathcal D}(f).
\]
The injectivity of the localization morphism and the characterization
of the Eisenstein--Kronecker class in
Theorem~\ref{thm: function to equiv coh class} imply
\(
  [\phi_{g,f}]_{\mathrm{coh}}=EK_{\Gamma,\mathcal A}(f).
\)
\end{proof}

\subsection{The Bianchi realization in dimension two}

We now specialize the preceding construction to the case $g=1$ and
$n=2$. Thus $K/\Q$ is imaginary quadratic, $L/K$ is quadratic, and
$d=2$. Under the Hermitian realization
$S_\infty=\GL_2(\C)/U(2)\cong\mathrm{Herm}_2^+$, the determinant-one
part is identified with
\[
  \mathrm{PGL}_2(\C)/\mathrm{PU}(2)
  \cong
  \mathrm{SL}_2(\C)/\mathrm{SU}(2)
  \cong
  \mathbb H^3,
\]
where $\mathbb H^3=\{(u,\xi)\in\C\times\R_{>0}\}$ is the Bianchi
upper half-space. Hence
$S_\infty\cong\mathbb H^3\times\R_{>0}$, with the second factor
recorded by $\det(H_g)$. More precisely, the Hermitian form associated
with $g\in S_\infty$ has the unique expression
\[
  H_g
  =
  \sqrt{\det(H_g)}
  \begin{bmatrix}
    \dfrac{1}{\xi}
    &
    -\dfrac{u}{\xi}
    \\[6pt]
    -\dfrac{\overline u}{\xi}
    &
    \dfrac{|u|^2+\xi^2}{\xi}
  \end{bmatrix}.
\]

\begin{lem}
With respect to the fixed $\mathcal O_K$-basis of $\mathcal O_L$, let
\[
  \gamma
  =
  \begin{bmatrix}
    \gamma_{11} & \gamma_{12} \\
    \gamma_{21} & \gamma_{22}
  \end{bmatrix}
  \in
  \Gamma
  \subset
  \GL_{\mathcal O_K}(\mathcal O_L)
  \cong
  \GL_2(\mathcal O_K).
\]
The action on Hermitian forms is given by
\[
  H_{\gamma\cdot g}
  =
  (\gamma^{-1})^*H_g\gamma^{-1}.
\]
Since $\det(\gamma)\in\mathcal O_K^\times$ and $K$ is imaginary
quadratic, the determinant coordinate satisfies
\[
  \det(H_{\gamma\cdot g})
  =
  |\det(\gamma)|^{-2}\det(H_g)
  =
  \det(H_g).
\]
Writing $(u',\xi')$ for the Bianchi coordinates of
$H_{\gamma\cdot g}$, one has
\begin{align*}
  \xi'
  &=
  \frac{\xi}
  {|\gamma_{21}u+\gamma_{22}|^2+|\gamma_{21}|^2\xi^2},
  \\
  u'
  &=
  \frac{
    (\gamma_{11}u+\gamma_{12})
    \overline{(\gamma_{21}u+\gamma_{22})}
    +
    \gamma_{11}\overline{\gamma_{21}}\xi^2
  }{
    |\gamma_{21}u+\gamma_{22}|^2+|\gamma_{21}|^2\xi^2
  }.
\end{align*}
The corresponding complexified coframe transforms by
\[
  \gamma^*
  \begin{bmatrix}
    \dfrac{du}{\xi} \\[4pt]
    \dfrac{d\overline u}{\xi} \\[4pt]
    \dfrac{d\xi}{\xi}
  \end{bmatrix}
  =
  \begin{bmatrix}
    \dfrac{du'}{\xi'} \\[4pt]
    \dfrac{d\overline{u'}}{\xi'} \\[4pt]
    \dfrac{d\xi'}{\xi'}
  \end{bmatrix}
  =
  M
  \begin{bmatrix}
    \dfrac{du}{\xi} \\[4pt]
    \dfrac{d\overline u}{\xi} \\[4pt]
    \dfrac{d\xi}{\xi}
  \end{bmatrix},
\]
where
\[
  \begin{aligned}
    M
    &=
    \frac{1}{
      |\gamma_{21}u+\gamma_{22}|^2+|\gamma_{21}|^2\xi^2
    }
    \begin{bmatrix}
      \det(\gamma)
      \overline{(\gamma_{21}u+\gamma_{22})}^{\,2}
      &
      -\det(\gamma)\overline{\gamma_{21}}^{\,2}\xi^2
      &
      2\det(\gamma)\overline{\gamma_{21}}\,
      \overline{(\gamma_{21}u+\gamma_{22})}\xi
      \\[6pt]
      -\overline{\det(\gamma)}\gamma_{21}^2\xi^2
      &
      \overline{\det(\gamma)}(\gamma_{21}u+\gamma_{22})^2
      &
      2\overline{\det(\gamma)}\gamma_{21}
      (\gamma_{21}u+\gamma_{22})\xi
      \\[6pt]
      -\gamma_{21}\overline{(\gamma_{21}u+\gamma_{22})}\xi
      &
      -(\gamma_{21}u+\gamma_{22})\overline{\gamma_{21}}\xi
      &
      |\gamma_{21}u+\gamma_{22}|^2-|\gamma_{21}|^2\xi^2
    \end{bmatrix}.
  \end{aligned}
\]
This matrix belongs to $SO(B)(\C)$, where
\[
  B
  =
  \begin{pmatrix}
    0 & \frac12 & 0 \\
    \frac12 & 0 & 0 \\
    0 & 0 & 1
  \end{pmatrix},
  \qquad
  M^tBM=B,
  \qquad
  \det(M)=1.
\]
The determinant-direction form is invariant:
\[
  \gamma^*
  \left(
    \frac{d\det(H_g)}{\det(H_g)}
  \right)
  =
  \frac{d\det(H_g)}{\det(H_g)}.
\]
\end{lem}

\begin{proof}
The Hermitian realization of $S_\infty$ gives
\[
  H_{\gamma\cdot g}
  =
  ((\gamma g)^{-1})^*H_0(\gamma g)^{-1}
  =
  (\gamma^{-1})^*H_g\gamma^{-1}.
\]
Taking determinants yields the factor $|\det(\gamma)|^{-2}$,
which equals $1$ because $\det(\gamma)$ is a unit of the imaginary
quadratic field $K$. Substituting the Bianchi expression for $H_g$
into $(\gamma^{-1})^*H_g\gamma^{-1}$ gives the displayed formulas
for $u'$ and $\xi'$. These are the usual Möbius transformation
formulas on $\mathbb H^3$, and differentiating them gives the
matrix $M$. The hyperbolic metric is
\[
  ds_{\mathbb H^3}^2
  =
  \frac{du\,d\overline u+d\xi^2}{\xi^2},
\]
whose Gram matrix in the complexified coframe
$du/\xi,d\overline u/\xi,d\xi/\xi$ is $B$. Since the Möbius
action of $\mathrm{PGL}_2(\C)$ is isometric and
orientation-preserving, its coframe transformation satisfies
$M^tBM=B$ and $\det(M)=1$. Finally, the equality
$\det(H_{\gamma\cdot g})=\det(H_g)$ implies
\[
  \gamma^*
  \left(
    \frac{d\det(H_g)}{\det(H_g)}
  \right)
  =
  \frac{d\det(H_{\gamma\cdot g})}{\det(H_{\gamma\cdot g})}
  =
  \frac{d\det(H_g)}{\det(H_g)}.
\]
\end{proof}

For $u=x+iy$, the real coframe $dx/\xi,dy/\xi,d\xi/\xi$ is
orthonormal for the hyperbolic metric, and its transformation
matrix belongs to $SO(3,\R)$. The terminology
``$\Gamma$-equivariant coframe'' does not mean that
$du/\xi,d\overline u/\xi,d\xi/\xi$ are individually
$\Gamma$-invariant: their span is preserved, with action described
by the automorphy matrix $M$, whereas
$d\det(H_g)/\det(H_g)$ is invariant. In this sense,
\[
  \left\{
    \frac{du}{\xi},
    \frac{d\overline u}{\xi},
    \frac{d\xi}{\xi},
    \frac{d\det(H_g)}{\det(H_g)}
  \right\}
\]
is a $\Gamma$-equivariant complexified coframe on $S_\infty$.
We next use these coordinates to describe the variation of the
dual lattice vectors and compute the specialized moment form.

Recall that $x$ is a $\Gamma$-invariant torsion section with image
in $\mathcal U_{\mathcal D}(\C)$.

\begin{prop}\label{prop: explicit moment class in dimension two}
For every $b\geq0$, the differential form
\[
  \psi_g^b(f,x)
  \in
  H^0\left(
    S_\infty,
    \mathrm{TSym}^b(\omega_{\mathcal A^\vee/\C})
    \otimes
    \omega_{\mathcal A/\C}^2
    \otimes
    \mathcal E^1_{S_\infty}
  \right)^\Gamma
\]
given by
\begin{align*}
  \psi_g^b(f,x)
  &=
  \det(H_g)
  \sum_{t\in\mathcal D(\C)}f(t)
  \sum_{|\beta|=b}
  \Bigg[
    \Bigg(
      \frac{1}{\xi}
      K^{\beta+(2,0)}(H_g,x-t,0,2,\Lambda)
      -
      2\frac{\overline u}{\xi}
      K^{\beta+(1,1)}(H_g,x-t,0,2,\Lambda)
  \\
  &\qquad\qquad\qquad\qquad
      +
      \frac{\overline u^{\,2}}{\xi}
      K^{\beta+(0,2)}(H_g,x-t,0,2,\Lambda)
    \Bigg)
    \frac{du}{\xi}
    -
    \xi K^{\beta+(0,2)}(H_g,x-t,0,2,\Lambda)
    \frac{d\overline u}{\xi}
  \\
  &\qquad\qquad\qquad
    +
    \Big(
      -2K^{\beta+(1,1)}(H_g,x-t,0,2,\Lambda)
      +
      2\overline u
      K^{\beta+(0,2)}(H_g,x-t,0,2,\Lambda)
    \Big)
    \frac{d\xi}{\xi}
  \Bigg]
 \otimes
  \overline e^{[\beta]}
  \otimes
  \bigwedge_{i=1}^2dz_i
\end{align*}
is closed, and its cohomology class satisfies
\[
  \operatorname{mom}_x^b
  \bigl(x^*EK_{\Gamma,\mathcal A}(f)\bigr)
  =
  [\psi_g^b(f,x)]
  \in
  H^1\left(
    \Gamma,
    \mathrm{TSym}^b(\omega_{\mathcal A^\vee/\C})
    \otimes
    \omega_{\mathcal A/\C}^2
  \right).
\]
\end{prop}

\begin{proof}
By Theorem~\ref{prop:functional equation for currents}, the
pro-system $\phi_{g,f}$ represents $EK_{\Gamma,\mathcal A}(f)$.
Since the degree-$b$ moment factors through every finite
truncation level $n\geq b$, the specialized moment is represented
by $\operatorname{mom}_{x,n}^b(x^*\phi_{g,f}^{(n)})$.
In dimension $d=2$, the degree in the current resolution is one.
Specialization along $x$ kills the antiholomorphic fibre
differentials, so only the component containing one
$S_\infty$-differential survives:
\[
  x^*\left(
    \iota_{\widetilde l_g}
    L_{\widetilde l_g}^{k}\vol_g
  \right)
  =
  0
  \qquad\text{for }k\neq1.
\]
Evaluating this component and applying the moment map gives
\[
  \begin{aligned}
    &\frac{\Gamma(b+2)\pi^2\det(H_g)}
    {\operatorname{covol}_g(\Lambda)}
    \sum_{t\in\mathcal D(\C)}f(t)
    \sum_{l_g\in\Lambda_g^*\setminus\{0\}}
    \frac{\widetilde l_g^{[b]}}
    {(\pi\|l_g\|_g^2)^{b+2}}
    \chi_{l_g}(x-t)
   \otimes
    \bigwedge_{i=1}^2dz_i
    \otimes
    \left(
      (\overline l_g)_1
      d_{S_\infty}(\overline l_g)_2
      -
      (\overline l_g)_2
      d_{S_\infty}(\overline l_g)_1
    \right).
  \end{aligned}
\]

To compute the differential factor, use the identification
$\Lambda^*\xrightarrow{\sim}\Lambda_g^*$,
$l\mapsto l_g:=H_g^{-1}l$.
For $l=(l_1,l_2)^t$, inversion of the Bianchi matrix gives
\[
  (l_g)_1
  =
  \frac{(|u|^2+\xi^2)l_1+ul_2}
  {\sqrt{\det(H_g)}\,\xi},
  \qquad
  (l_g)_2
  =
  \frac{\overline u l_1+l_2}
  {\sqrt{\det(H_g)}\,\xi}.
\]
In particular,
$(l_g)_1-u(l_g)_2=\xi l_1/\sqrt{\det(H_g)}$.
Keeping $l\in\Lambda^*$ fixed and differentiating the complex
conjugates with respect to $S_\infty$ yields
\begin{align*}
  d_{S_\infty}(\overline l_g)_1
  &=
  \frac{
    \overline u
    \bigl(
      (\overline l_g)_1-\overline u(\overline l_g)_2
    \bigr)
  }{\xi}
  \frac{du}{\xi}
  +
  \xi(\overline l_g)_2
  \frac{d\overline u}{\xi}
  +
  \bigl(
    (\overline l_g)_1-2\overline u(\overline l_g)_2
  \bigr)
  \frac{d\xi}{\xi}
  -
  \frac12(\overline l_g)_1
  \frac{d\det(H_g)}{\det(H_g)},
  \\
  d_{S_\infty}(\overline l_g)_2
  &=
  \frac{
    (\overline l_g)_1-\overline u(\overline l_g)_2
  }{\xi}
  \frac{du}{\xi}
  -
  (\overline l_g)_2
  \frac{d\xi}{\xi}
  -
  \frac12(\overline l_g)_2
  \frac{d\det(H_g)}{\det(H_g)}.
\end{align*}
In the alternating combination occurring above, the two
contributions involving $d\det(H_g)/\det(H_g)$ cancel, leaving
\[
  \begin{aligned}
    (\overline l_g)_1
    d_{S_\infty}(\overline l_g)_2
    -
    (\overline l_g)_2
    d_{S_\infty}(\overline l_g)_1
    =
    \frac{
      \bigl(
        (\overline l_g)_1-\overline u(\overline l_g)_2
      \bigr)^2
    }{\xi}
    \frac{du}{\xi}
    -
    \xi(\overline l_g)_2^2
    \frac{d\overline u}{\xi}
   -
    2(\overline l_g)_2
    \bigl(
      (\overline l_g)_1-\overline u(\overline l_g)_2
    \bigr)
    \frac{d\xi}{\xi}.
  \end{aligned}
\]

The tensor-symmetric coefficient expands as
\[
  \widetilde l_g
  =
  [\overline e_1,\overline e_2]\,\overline l_g,
  \qquad
  \widetilde l_g^{[b]}
  =
  \sum_{|\beta|=b}
  (\overline l_g)^\beta\,
  \overline e^{[\beta]}.
\]
Substituting this expansion and the preceding differential
identity into the moment expression produces the three
monomials $(\overline l_g)^{\beta+\varepsilon}$, where
$\varepsilon\in\{(2,0),(1,1),(0,2)\}$.
By the definition of the Eisenstein--Kronecker series, the
corresponding lattice sums are
\[
  \sum_{l_g\in\Lambda_g^*\setminus\{0\}}
  \frac{(\overline l_g)^{\beta+\varepsilon}}
  {(\pi\|l_g\|_g^2)^{b+2}}
  \chi_{l_g}(x-t)
  =
  \frac{
    K^{\beta+\varepsilon}
    (H_g,0,x-t,b+2,\Lambda_g^*)
  }{\pi^{b+2}}.
\]
Since $d=2$ and $|\varepsilon|=2$,
Proposition~\ref{prop:functional equation of Eisenstein-Kronecker series}
gives
\[
  \frac{
    K^{\beta+\varepsilon}(H_g,x-t,0,2,\Lambda)
  }{\pi^2}
  =
  \frac{
    \Gamma(b+2)
    K^{\beta+\varepsilon}
    (H_g,0,x-t,b+2,\Lambda_g^*)
  }{
    \operatorname{covol}_g(\Lambda)\pi^{b+2}
  }.
\]
Applying this identity to the three terms cancels the common
factor $\Gamma(b+2)\pi^2/\operatorname{covol}_g(\Lambda)$ and
gives the stated expression for $\psi_g^b(f,x)$.

Finally, pullback along $x$ and the moment map are
$\Gamma$-equivariant, so $\psi_g^b(f,x)$ is
$\Gamma$-invariant. Since the image of $x$ is disjoint from
$\mathcal D$, the residue current vanishes after specialization.
Thus the resulting form is closed and represents the
specialized moment identified at the beginning of the proof.
\end{proof}

\subsection{Boundary restriction and Fourier constant terms}

We now study the Borel--Serre boundary restriction of the class represented
by $\psi_g^b(f,x)$. We retain the setting $g=1$, $n=2$, and $d=2$, with
$K$ imaginary quadratic and $\Gamma\subseteq\GL_2(\mathcal O_K)$ the
finite-index subgroup fixed above. Since
$\det(\gamma)\in\mathcal O_K^\times$ for every $\gamma\in\Gamma$, one has
$|\det(\gamma)|=1$. Thus the $\Gamma$-action on
$S_\infty\cong\mathbb H^3\times\R_{>0}$ is trivial on the second factor,
and
\[
  \Gamma\backslash S_\infty
  \cong
  (\Gamma\backslash\mathbb H^3)\times\R_{>0}.
\]
The contractibility of $\R_{>0}$ shows that restriction to any fixed
determinant slice is a homotopy equivalence. By Proposition
\ref{prop: explicit moment class in dimension two}, restriction to a
determinant slice identifies the Eisenstein--Kronecker moment class with
\[
  \mathrm{mom}_x^b\bigl(x^*EK_{\Gamma,\mathcal A}(f)\bigr)
  =
  [\psi_g^b(f,x)]
  \in
  H^1\left(\Gamma\backslash\mathbb H^3,\mathrm{TSym}^b(\omega_{\mathcal A^\vee/\C})
  \otimes
  \omega_{\mathcal A/\C}^2\right).
\]

\begin{rem}\rm
For the geometric discussion, we may replace $\Gamma$ by a torsion-free
subgroup $\Gamma'\subseteq\Gamma$ of finite index. By Proposition
\ref{prop:restriction-corestriction}, restriction carries the
Eisenstein--Kronecker class to the corresponding class for $\Gamma'$,
and on cohomology with complex coefficients one has
\[
  \operatorname{cor}_{\Gamma'}^\Gamma
  \circ
  \operatorname{res}_{\Gamma'}^\Gamma
  =
  [\Gamma:\Gamma']\operatorname{id}.
\]
Since the index is invertible over $\C$, normalized corestriction
recovers the original class. We therefore assume $\Gamma$ torsion-free
when discussing the quotient spaces as manifolds. For groups with
torsion, the corresponding cohomology is understood in the equivariant
sense introduced above.
\end{rem}

The visual boundary of $\mathbb H^3$ is $\mathbb P^1(\C)$, with the
standard cusp $\infty$ corresponding to $\xi\to+\infty$ in the
upper half-space model. The proper $K$-rational parabolic subgroups of
$\GL_2$ are precisely the stabilizers of $K$-lines in $K^2$, and hence
are parametrized by $\mathbb P^1(K)$. The cusps of
$\Gamma\backslash\mathbb H^3$ are therefore indexed by the finite set
$\Gamma\backslash\mathbb P^1(K)$. For
$\Gamma=\GL_2(\mathcal O_K)$, this set is naturally identified with
$\operatorname{Cl}(K)$; see
\cite[Chapter~7, Theorem~2.4]{EGM98}.

Fix a representative $c\in\mathbb P^1(K)$ and choose
$\sigma_c\in\GL_2(K)$ with $\sigma_c(\infty)=c$. Let $P_\infty$ be the
upper-triangular parabolic subgroup and $N_\infty$ its unipotent radical,
so that
\[
  P_\infty(K)
  =
  \left\{
    \begin{bmatrix}
      a&z\\
      0&d
    \end{bmatrix}
    \,\middle|\,
    a,d\in K^\times,\ z\in K
  \right\},
  \qquad
  N_\infty(K)
  =
  \left\{
    n(a)=
    \begin{bmatrix}
      1&a\\
      0&1
    \end{bmatrix}
    \,\middle|\,
    a\in K
  \right\}.
\]
Put $P_c:=\sigma_cP_\infty\sigma_c^{-1}$ and
$N_c:=\sigma_cN_\infty\sigma_c^{-1}$, and write
$\Gamma_c:=\Gamma\cap P_c(K)$ and
$\Gamma_{N,c}:=\Gamma\cap N_c(K)$. The groups $P_c$ and $N_c$ depend
only on $c$, although the chosen identification with the standard
parabolic data depends on $\sigma_c$. Conjugation by $\sigma_c$
identifies $N_\infty(\C)\cong(\C,+)$ with $N_c(\C)$, under which
$\Gamma_{N,c}$ is a rank-two lattice. Consequently,
\[
  T_c
  :=
  \Gamma_{N,c}\backslash N_c(\C)
  \cong
  \Gamma_{N,c}\backslash\C
\]
is a compact real two-dimensional torus, where the last expression uses
this identification of the unipotent group with $\C$.

% Boundary geometry:
\begin{prop}\label{prop: Borel--Serre cusp components}
Let $e_c$ be the Borel--Serre boundary component corresponding to $c$.
Then
\[
  e_c
  \cong
  (\Gamma_c/\Gamma_{N,c})\backslash T_c,
\]
where $\Gamma_c/\Gamma_{N,c}$ is finite.
\end{prop}

\begin{proof}
After passing to the hyperbolic factor, the symmetric space attached
to the Levi quotient of $P_c$ is a point. The Borel--Serre face
associated with $P_c$ is therefore modeled by $N_c(\C)$; see
\cite[Sections~2.1--2.3]{Harder87}. Its arithmetic quotient is obtained
by dividing by $\Gamma_c$. The normal subgroup $\Gamma_{N,c}$ acts
by translations, giving $T_c$, and the remaining action is that of
$\Gamma_c/\Gamma_{N,c}$. After conjugation by $\sigma_c$, diagonal
projection identifies this quotient with an arithmetic subgroup of the
diagonal torus of $\GL_2(K)$. Such a subgroup is finite because
$\mathcal O_K^\times$ is finite.
\end{proof}

For example, at the standard cusp and for
$\Gamma=\GL_2(\mathcal O_K)$, one has
$\Gamma_{N,\infty}=\{n(a)\mid a\in\mathcal O_K\}$ and
$T_\infty=\mathcal O_K\backslash\C$. Moreover,
$\Gamma_\infty/\Gamma_{N,\infty}
\cong\mathcal O_K^\times\times\mathcal O_K^\times$, whose action on
$T_\infty$ is given by
$(\varepsilon_1,\varepsilon_2)\cdot[u]
=[\varepsilon_1\varepsilon_2^{-1}u]$.

Let $j_{\Gamma\backslash\mathbb H^3}$ denote the inclusion of the
Borel--Serre boundary into
$\overline{\Gamma\backslash\mathbb H^3}^{\,\mathrm{BS}}$, and let
$j_c:e_c\hookrightarrow
\partial\overline{\Gamma\backslash\mathbb H^3}^{\,\mathrm{BS}}$
be the inclusion of the component at $c$. Since the inclusion of the interior into the Borel--Serre
compactification is a homotopy equivalence, the composite
\begin{equation*}
  \scalebox{0.92}{$\displaystyle
H^1\left(\Gamma\backslash\mathbb H^3,\mathrm{TSym}^b(\omega_{\mathcal A^\vee/\C})
  \otimes
  \omega_{\mathcal A/\C}^2\right)
  \cong
  H^1\left(
    \overline{\Gamma\backslash\mathbb H^3}^{\,\mathrm{BS}},\mathrm{TSym}^b(\omega_{\mathcal A^\vee/\C})
  \otimes
  \omega_{\mathcal A/\C}^2
  \right)
  \xrightarrow{\,j_{\Gamma\backslash\mathbb H^3}^*\,}
  H^1\left(
    \partial\overline{\Gamma\backslash\mathbb H^3}^{\,\mathrm{BS}},\mathrm{TSym}^b(\omega_{\mathcal A^\vee/\C})
  \otimes
  \omega_{\mathcal A/\C}^2
  \right)
  $}
\end{equation*}
defines the boundary restriction map $\mathrm{res}_{\partial}$.
For each cusp $c$, we put
\[
  \mathrm{res}_c
  :=
  j_c^*\circ\mathrm{res}_{\partial}
  :
  H^1\left(\Gamma\backslash\mathbb H^3,\mathrm{TSym}^b(\omega_{\mathcal A^\vee/\C})
  \otimes
  \omega_{\mathcal A/\C}^2\right)
  \longrightarrow
  H^1\left(e_c,\mathrm{TSym}^b(\omega_{\mathcal A^\vee/\C})
  \otimes
  \omega_{\mathcal A/\C}^2\right).
\]
The cohomology of each boundary component can be computed using its
unipotent subgroup.

% Cohomological reduction:
\begin{prop}\label{prop: cusp cohomology via unipotent subgroup}
There are canonical identifications
\begin{equation*}
  \scalebox{0.99}{$\displaystyle
H^1\left(e_c,\mathrm{TSym}^b(\omega_{\mathcal A^\vee/\C})
  \otimes
  \omega_{\mathcal A/\C}^2\right)
  \cong
  H^1\left(\Gamma_c,\mathrm{TSym}^b(\omega_{\mathcal A^\vee/\C})
  \otimes
  \omega_{\mathcal A/\C}^2\right)
  \xrightarrow{\;\sim\;}
  H^1\left(\Gamma_{N,c},\mathrm{TSym}^b(\omega_{\mathcal A^\vee/\C})
  \otimes
  \omega_{\mathcal A/\C}^2\right)^{\Gamma_c/\Gamma_{N,c}},
  $}
\end{equation*}
where the second map is induced by restriction to $\Gamma_{N,c}$.
\end{prop}

\begin{proof}
The Borel--Serre face $N_c(\C)\cong\C$ is contractible, so the Borel
construction identifies its quotient cohomology with
$H^1\left(\Gamma_c,\mathrm{TSym}^b(\omega_{\mathcal A^\vee/\C})
  \otimes
  \omega_{\mathcal A/\C}^2\right)$. The exact sequence
\[
  1
  \longrightarrow
  \Gamma_{N,c}
  \longrightarrow
  \Gamma_c
  \longrightarrow
  \Gamma_c/\Gamma_{N,c}
  \longrightarrow
  1
\]
gives the Hochschild--Serre spectral sequence
\[
  E_2^{p,q}
  =
  H^p\left(
    \Gamma_c/\Gamma_{N,c},
    H^q\left(\Gamma_{N,c},\mathrm{TSym}^b(\omega_{\mathcal A^\vee/\C})
  \otimes
  \omega_{\mathcal A/\C}^2\right)
  \right)
  \Longrightarrow
  H^{p+q}\left(\Gamma_c,\mathrm{TSym}^b(\omega_{\mathcal A^\vee/\C})
  \otimes
  \omega_{\mathcal A/\C}^2\right).
\]
By Proposition~\ref{prop: Borel--Serre cusp components}, the quotient
$\Gamma_c/\Gamma_{N,c}$ is finite. Since the coefficient groups are
complex vector spaces, averaging over this finite group makes the
invariants functor exact. Hence
$H^p(\Gamma_c/\Gamma_{N,c},-)=0$ for $p>0$, and the spectral sequence
degenerates at $E_2$. Its degree-one edge morphism, which is restriction
to $\Gamma_{N,c}$, is therefore an isomorphism onto the invariant
subspace.
\end{proof}

Since $\C$ is contractible, one also has
$H^1\left(T_c,\mathrm{TSym}^b(\omega_{\mathcal A^\vee/\C})
  \otimes
  \omega_{\mathcal A/\C}^2\right)\cong H^1\left(\Gamma_{N,c},\mathrm{TSym}^b(\omega_{\mathcal A^\vee/\C})
  \otimes
  \omega_{\mathcal A/\C}^2\right)$. Under this identification,
pullback along the natural projection $T_c\to e_c$ identifies
$H^1\left(e_c,\mathrm{TSym}^b(\omega_{\mathcal A^\vee/\C})
  \otimes
  \omega_{\mathcal A/\C}^2\right)$ with the
$\Gamma_c/\Gamma_{N,c}$-invariant subspace of $H^1\left(T_c,\mathrm{TSym}^b(\omega_{\mathcal A^\vee/\C})
  \otimes
  \omega_{\mathcal A/\C}^2\right)$.
Consequently, to determine
$\mathrm{res}_c[\psi_g^b(f,x)]$, it is enough to compute its pullback
to the covering torus $T_c$ and then pass to the residual invariant
part. We shall carry this out first at the standard cusp $c=\infty$.
The restriction at an arbitrary cusp will then be obtained by
transporting the complete coefficient-valued form through a matrix
$\sigma_c$ with $\sigma_c(\infty)=c$.

At the standard cusp, a cusp neighborhood provides, for every sufficiently
large $\xi_0>0$, a horospherical map
$j_{\infty,\xi_0}:e_\infty\to\Gamma\backslash\mathbb H^3$
induced by $u\mapsto(u,\xi_0)$. Composing with the projection
$T_\infty\to e_\infty$ gives
\[
  j_{T_\infty,\xi_0}:
  T_\infty\longrightarrow\Gamma\backslash\mathbb H^3,
  \qquad
  [u]\longmapsto[(u,\xi_0)].
\]
The boundary restriction can be computed by pullback along this map;
no pointwise limit of the Eisenstein--Kronecker series as
$\xi_0\to+\infty$ is required.

% Standard cusp restriction: 
\begin{prop}[Restriction at the standard cusp]
\label{prop:restriction-standard-cusp}
For every sufficiently large $\xi_0>0$, one has
\begin{align*}
  &j_{T_\infty,\xi_0}^*\psi_g^b(f,x)=
  \det(H_g)
  \sum_{t\in \mathcal{D}(\C)}f(t)
  \sum_{|\beta|=b}
  \Biggl[
    \Biggl(
      \frac{1}{\xi_0}
      K^{\beta+(2,0)}(H_g,x-t,0,2,\Lambda)
  \\
  &\qquad\qquad\hspace{5.5em}
      -
      \frac{2\overline{u}}{\xi_0}
      K^{\beta+(1,1)}(H_g,x-t,0,2,\Lambda)
      +
      \frac{\overline{u}^{\,2}}{\xi_0}
      K^{\beta+(0,2)}(H_g,x-t,0,2,\Lambda)
    \Biggr)
    \frac{du}{\xi_0}
  \\
  &\qquad\qquad\qquad\qquad\hspace{5.5em}
    -
    K^{\beta+(0,2)}(H_g,x-t,0,2,\Lambda)
    d\overline{u}
  \Biggr]
  \otimes\overline{e}^{[\beta]}
  \otimes\bigwedge_{i=1}^d dz_i,
\end{align*}
where $H_g$ and the Eisenstein--Kronecker series are evaluated at
$(u,\xi_0)$. This form is invariant under the residual action of
$\Gamma_\infty/\Gamma_{N,\infty}$ and therefore descends to $e_\infty$.
Denoting the descended form by
$\mathrm{res}_{\infty,\xi_0}(\psi_g^b(f,x))$, one has
\[
  \left[
    \mathrm{res}_{\infty,\xi_0}\bigl(\psi_g^b(f,x)\bigr)
  \right]
  =
  \mathrm{res}_\infty\bigl([\psi_g^b(f,x)]\bigr).
\]
In particular, its cohomology class is independent of the sufficiently
large choice of $\xi_0$.
\end{prop}

\begin{proof}
Since $\xi$ is constant along the horospherical map,
\[
  j_{T_\infty,\xi_0}^*\left(\frac{d\xi}{\xi}\right)=0,
  \qquad
  j_{T_\infty,\xi_0}^*\left(\frac{du}{\xi}\right)
  =\frac{du}{\xi_0},
  \qquad
  j_{T_\infty,\xi_0}^*\left(\frac{d\overline{u}}{\xi}\right)
  =\frac{d\overline{u}}{\xi_0}.
\]
Thus the normal component vanishes and only the tangential components
remain. Substitution into the formula of Proposition
\ref{prop: explicit moment class in dimension two} gives the displayed
expression. The $\Gamma$-equivariance of $\psi_g^b(f,x)$ implies
invariance under the residual cusp action, and hence descent to
$e_\infty$.

Inside the Borel--Serre cusp neighborhood, the horospherical maps for
sufficiently large heights are homotopic to one another and, after
composition with the inclusion of the interior, to the boundary
inclusion. Since $\psi_g^b(f,x)$ is closed, their pullbacks determine
the same cohomology class. Under the homotopy equivalence between the
interior and its Borel--Serre compactification, this class is precisely
$\mathrm{res}_\infty([\psi_g^b(f,x)])$.
\end{proof}

Here and below, Eisenstein--Kronecker series evaluated at $s=2$ are
understood through their meromorphic continuations.

For an arbitrary cusp $c$, retain the choice
$\sigma_c\in\GL_2(K)$ with $\sigma_c(\infty)=c$. Under the identification
\[
  T_c
  \cong
  \bigl(\sigma_c^{-1}\Gamma_{N,c}\sigma_c\bigr)\backslash\C,
\]
the corresponding horospherical map is
\[
  j_{T_c,\xi_0}:
  T_c\longrightarrow\Gamma\backslash\mathbb H^3,
  \qquad
  [u]\longmapsto[\sigma_c\cdot(u,\xi_0)].
\]
In particular, the unipotent lattice in these coordinates is
$\sigma_c^{-1}\Gamma_{N,c}\sigma_c$, not necessarily the lattice
$\Gamma_{N,\infty}$ of the standard cusp of the original group.

% Transport and its independence: 
\begin{prop}[Restriction at an arbitrary cusp]
\label{prop:restriction-arbitrary-cusp}
\label{thm: explicit restriction at an arbitrary cusp}
With the action convention fixed above, one has
\[
  j_{T_c,\xi_0}^*\psi_g^b(f,x)
  =
  j_{T_\infty,\xi_0}^*
  \bigl(\sigma_c^{-1}\cdot\psi_g^b(f,x)\bigr).
\]
Here the map on the right-hand side is the standard horospherical map
for the conjugated group $\sigma_c^{-1}\Gamma\sigma_c$, with unipotent
lattice $\sigma_c^{-1}\Gamma_{N,c}\sigma_c$, and the coefficient local
systems are identified through $\sigma_c$.

The explicit expression is obtained by applying Proposition
\ref{prop:restriction-standard-cusp} to the transported form and data.
It is invariant under $\Gamma_c/\Gamma_{N,c}$ and descends to a form
$\mathrm{res}_{c,\xi_0}(\psi_g^b(f,x))$ on $e_c$ satisfying
\[
  \left[
    \mathrm{res}_{c,\xi_0}\bigl(\psi_g^b(f,x)\bigr)
  \right]
  =
  \mathrm{res}_c\bigl([\psi_g^b(f,x)]\bigr).
\]
The resulting class is independent of the sufficiently large height
$\xi_0$ and of the choice of $\sigma_c$.
\end{prop}

\begin{proof}
The horospherical map at $c$ is obtained by composing the standard
horospherical map for the conjugated group with the action of
$\sigma_c$. Pullback, together with the induced identification of
coefficients, gives the asserted identity. Proposition
\ref{prop:restriction-standard-cusp} then supplies the explicit
expression. Equivariance gives descent under the residual finite
group, while the homotopy through the cusp neighborhood identifies
the descended cohomology class with the boundary restriction and
shows its independence of the height.

If $\sigma_c'$ is another choice with $\sigma_c'(\infty)=c$, then
$\sigma_c^{-1}\sigma_c'\in P_\infty(K)$. Since $P_\infty$ normalizes
$N_\infty$, this change of coordinates identifies the corresponding
horospherical tori and transports one restricted form to the other.
Together with independence of the height, this gives the same class
on $e_c$.
\end{proof}

The action of $\sigma_c^{-1}$ is applied to the complete
coefficient-valued differential form: it acts simultaneously on the
Hermitian form, the arguments and lattice of the Eisenstein--Kronecker
series, the coefficient vectors $\overline{e}^{[\beta]}$, and the
holomorphic top form. Since $\sigma_c$ need not preserve the chosen
basis, this action may mix terms indexed by different multi-indices
$\beta$ and must therefore be applied after forming the complete sum
over $|\beta|=b$. Before transport, the torsion sum remains indexed by
the original finite set $\mathcal{D}(\C)$.

For the full group $\GL_2(\mathcal O_K)$, the transported data admit
the following ideal-class description. The quotients and their
cohomology are understood in the equivariant sense explained above.

\begin{exmp}[Ideal-class description of the cusp data]
\label{exmp: ideal-class description of the cusp data}\rm
Suppose that $\Gamma=\GL_2(\mathcal O_K)$. For a cusp representative
$c=[\alpha:\beta]\in\mathbb P^1(K)$, put
$\mathfrak a=\alpha\mathcal O_K+\beta\mathcal O_K$.
Scaling $(\alpha,\beta)$ by an element of $K^\times$ multiplies
$\mathfrak a$ by a principal fractional ideal, whereas applying an
element of $\GL_2(\mathcal O_K)$ leaves the generated ideal unchanged.
Thus the assignment $c\mapsto[\mathfrak a]$ is well-defined on cusp
orbits and gives the classical bijection
\[
  \GL_2(\mathcal O_K)\backslash\mathbb P^1(K)
  \xrightarrow{\;\sim\;}
  \operatorname{Cl}(K);
\]
see \cite[Chapter~7, Section~2, Theorem~2.4]{EGM98}.

The intersection $K(\alpha,\beta)\cap\mathcal O_K^2$ is identified
with $\mathfrak a^{-1}$ through $v\mapsto(\alpha v,\beta v)$.
Its complementary projective rank-one summand has ideal class
$[\mathfrak a]$, so we may choose the $K$-basis defining $\sigma_c$
compatibly with this decomposition, with $\sigma_c(\infty)=c$ and
\[
  \sigma_c^{-1}\mathcal O_K^2=\mathfrak a^{-1}\oplus\mathfrak a.
\]
For $u\in K$, the action of the standard unipotent element is
\(
  n(u)(v_1,v_2)=(v_1+uv_2,v_2).
\)
Hence $n(u)$ preserves $\mathfrak a^{-1}\oplus\mathfrak a$ if and
only if $u\mathfrak a\subseteq\mathfrak a^{-1}$, equivalently
$u\in\mathfrak a^{-2}$. Consequently,
\[
  \sigma_c^{-1}\Gamma_{N,c}\sigma_c
  =
  \{n(u)\mid u\in\mathfrak a^{-2}\},
  \qquad
  T_c\cong\mathfrak a^{-2}\backslash\C;
\]
compare \cite[Chapter~8, Lemma~2.2]{EGM98}.

For the fixed CM period $\Omega_K$ with
$\mathcal A_0(\C)\cong\C/\mathcal O_K\Omega_K$, the Serre construction
gives the uniformization
\[
  \bigl((\mathfrak a^{-1}\oplus\mathfrak a)
  \otimes_{\mathcal O_K}\mathcal A_0\bigr)(\C)
  \cong
  \C^2/\bigl(\mathfrak a^{-1}\Omega_K\oplus\mathfrak a\Omega_K\bigr).
\]
Continuing to denote the transported abelian lattice by $\Lambda$,
we therefore have
\[
  \Lambda=\mathfrak a^{-1}\Omega_K\oplus\mathfrak a\Omega_K.
\]
Here $\Lambda$ is the lattice defining the abelian surface, whereas
$\mathfrak a^{-2}$ is the unipotent lattice defining the
horospherical torus $T_c$.

Let $\mathfrak c,\mathfrak f\subset\mathcal O_K$ be the nonzero
integral ideals determining the torsion data, and put
$\mathfrak C=\mathfrak c\mathcal O_L$. After transport to the cusp
chart, retain the notation $\mathcal{D}(\C)$, $x$, and $f$ for the
torsion set, torsion point, and function. Since the scalar
$\mathcal O_K$-action respects the two summands of $\Lambda$,
the torsion set $\mathcal{D}(\C)=\mathcal A[\mathfrak C](\C)$ has
the coordinate realization
\[
  \mathcal{D}(\C)
  =
  \frac{\mathfrak c^{-1}\Lambda}{\Lambda}
  \cong
  \frac{\mathfrak c^{-1}\mathfrak a^{-1}\Omega_K}
       {\mathfrak a^{-1}\Omega_K}
  \oplus
  \frac{\mathfrak c^{-1}\mathfrak a\Omega_K}
       {\mathfrak a\Omega_K}.
\]
Similarly, the transported $\mathfrak f$-torsion point satisfies
\[
  x\in
  \frac{\mathfrak f^{-1}\Lambda}{\Lambda}
  \cong
  \frac{\mathfrak f^{-1}\mathfrak a^{-1}\Omega_K}
       {\mathfrak a^{-1}\Omega_K}
  \oplus
  \frac{\mathfrak f^{-1}\mathfrak a\Omega_K}
       {\mathfrak a\Omega_K}.
\]
The function $f$ is transported to this finite quotient, and its
trace-zero condition is preserved. Writing each torsion point in
these coordinates gives, for every multi-index $\mu$,
\begin{align*}
  &\sum_{t\in\mathcal{D}(\C)}
  f(t)K^\mu(H_g,x-t,0,2,\Lambda)
  =
  \sum_{\substack{
    t_1\in
    \mathfrak c^{-1}\mathfrak a^{-1}\Omega_K/
    \mathfrak a^{-1}\Omega_K\\
    t_2\in
    \mathfrak c^{-1}\mathfrak a\Omega_K/
    \mathfrak a\Omega_K
  }}
  f(t_1,t_2)
  K^\mu\bigl(H_g,x-(t_1,t_2),0,2,\Lambda\bigr).
\end{align*}
The three cases needed in the restriction formula are
$\mu=\beta+(2,0)$, $\beta+(1,1)$, and $\beta+(0,2)$, with
$|\beta|=b$.

Finally, if $\mathfrak a'=\lambda\mathfrak a$ with
$\lambda\in K^\times$, the diagonal change of coordinates identifies
the corresponding cusp-adapted lattices by
\[
  \begin{pmatrix}
    \lambda^{-1}&0\\
    0&\lambda
  \end{pmatrix}
  \bigl(\mathfrak a^{-1}\Omega_K\oplus\mathfrak a\Omega_K\bigr)
  =
  (\mathfrak a')^{-1}\Omega_K\oplus\mathfrak a'\Omega_K.
\]
It transports the torsion data and Eisenstein--Kronecker series,
together with their coefficient vectors and holomorphic top form,
compatibly. By Proposition~\ref{prop:restriction-arbitrary-cusp},
these identifications preserve the resulting cuspwise boundary
class. Thus, for the fixed global data, this class depends only on
$[\mathfrak a]$, not on the chosen ideal representative. For the
principal class, taking $\mathfrak a=\mathcal O_K$ recovers the
standard cusp description
\[
  T_\infty=\mathcal O_K\backslash\C,
  \qquad
  \Lambda=\mathcal O_K\Omega_K\oplus\mathcal O_K\Omega_K.
\]
\end{exmp}

For a finite-index subgroup $\Gamma\subseteq\GL_2(\mathcal O_K)$,
the natural map
$\Gamma\backslash\mathbb P^1(K)\to\operatorname{Cl}(K)$ is
surjective but need not be injective, so a cusp must still be
specified by its $\Gamma$-orbit. We now return to the standard cusp
$c=\infty$ to compute the Fourier expansion of the horospherical
restriction. The calculation at an arbitrary cusp is obtained by
transporting the complete form and using its conjugated unipotent
lattice, with the explicit ideal-class coordinates of
Example~\ref{exmp: ideal-class description of the cusp data} in the full-group case.

For the Fourier calculation, we work on the determinant-one component
at height $\xi_0=1$, using the height independence of the horospherical
restriction established in
Proposition~\ref{prop:restriction-standard-cusp}. The Hermitian form
and its associated norm are then
\[
  H_u=
  \begin{bmatrix}
    1&-u\\
    -\overline{u}&|u|^2+1
  \end{bmatrix},
  \qquad
  \|v\|_u^2=|v_1-uv_2|^2+|v_2|^2.
\]
Introduce the moving basis
\[
  \overline{e}_1^u:=\overline{e}_1,
  \qquad
  \overline{e}_2^u:=\overline{u}\,\overline{e}_1+\overline{e}_2,
\]
and, for $v=(v_1,v_2)$, put
$v_1^u:=v_1-uv_2$ and $v_2^u:=v_2$. Thus
$\|v\|_u^2=|v_1^u|^2+|v_2^u|^2$, and the change of basis satisfies
\[
  \overline{v}_1\overline{e}_1+
  \overline{v}_2\overline{e}_2
  =
  \overline{v_1^u}\,\overline{e}_1^u+
  \overline{v_2^u}\,\overline{e}_2^u.
\]

% Moving frame scalar coefficients: 
We reintroduce the complex parameter $s$ in the Eisenstein--Kronecker
series and initially work in the region
$\operatorname{Re}(s)\gg0$, where the sums are absolutely convergent.
For every multi-index $\beta$ with $|\beta|=b$, define
\begin{align*}
  A_f^{\beta,x}(u,s)
  :=
  \sum_{t\in\mathcal{D}(\C)}f(t)
  \sum_{v\in(x-t+\Lambda)\setminus\{0\}}
  \frac{
    (\overline{v^u})^\beta(\overline{v_1^u})^2
  }{
    \|v\|_u^{2s}
  },
  \qquad
  B_f^{\beta,x}(u,s)
  :=
  \sum_{t\in\mathcal{D}(\C)}f(t)
  \sum_{v\in(x-t+\Lambda)\setminus\{0\}}
  \frac{
    (\overline{v^u})^\beta(\overline{v_2^u})^2
  }{
    \|v\|_u^{2s}
  }.
\end{align*}
The corresponding coefficient-valued functions are
\[
  A_f^{b,x}(u,s)
  =
  \sum_{|\beta|=b}
  A_f^{\beta,x}(u,s)(\overline{e}^u)^{[\beta]},
  \qquad
  B_f^{b,x}(u,s)
  =
  \sum_{|\beta|=b}
  B_f^{\beta,x}(u,s)(\overline{e}^u)^{[\beta]}.
\]
Writing $\psi_u^b(f,x)$ for the form in these normalized coordinates,
the restriction formula of
Proposition~\ref{prop:restriction-standard-cusp} becomes
\[
  \mathrm{res}_{\infty,1}\bigl(\psi_u^b(f,x)\bigr)(s)
  =
  \left(
    A_f^{b,x}(u,s)\,du
    -
    B_f^{b,x}(u,s)\,d\overline{u}
  \right)
  \otimes\bigwedge_{i=1}^2dz_i.
\]
The moving basis converts the unipotent equivariance of this form into
ordinary periodicity of its scalar coefficients.

\begin{lem}[Unipotent equivariance and periodicity]
\label{lem: unipotent periodicity moving frame}
For $n(a)\in\Gamma_{N,\infty}$, one has
\[
  n(a)\cdot
  \mathrm{res}_{\infty,1}\bigl(\psi_u^b(f,x)\bigr)(s)
  =
  \mathrm{res}_{\infty,1}\bigl(\psi_{u+a}^b(f,x)\bigr)(s).
\]
Moreover, for every $|\beta|=b$,
\[
  A_f^{\beta,x}(u+a,s)=A_f^{\beta,x}(u,s),
  \qquad
  B_f^{\beta,x}(u+a,s)=B_f^{\beta,x}(u,s).
\]
\end{lem}

\begin{proof}
The action of $n(a)$ on the horospherical coordinate is
$u\mapsto u+a$, and the moving basis satisfies
\[
  n(a)\cdot\overline{e}_i^u
  =
  \overline{e}_i^{u+a},
  \qquad i=1,2.
\]
Since $n(a)w=(w_1+aw_2,w_2)$, the transformed coordinates satisfy
\[
  (n(a)w)_1^{u+a}=w_1^u,
  \qquad
  (n(a)w)_2^{u+a}=w_2^u,
  \qquad
  \|n(a)w\|_{u+a}^2=\|w\|_u^2.
\]
The point $x$ is $\Gamma$-invariant, the torsion set
$\mathcal{D}(\C)$ is $\Gamma$-stable, and $f$ is $\Gamma$-invariant.
Consequently, the substitutions $t=n(a)t'$ and $v=n(a)w$ reindex
the torsion and lattice sums without changing their summands in the
moving coordinates. This gives the two periodicity identities.
Combining them with the equivariance of the moving basis and the
invariance of $du$, $d\overline{u}$, and
$\bigwedge_{i=1}^2dz_i$ under $n(a)$ gives the equivariance of the
complete form.
\end{proof}

The scalar functions therefore admit Fourier expansions on
$T_\infty=\Gamma_{N,\infty}\backslash\C$, where
$\Gamma_{N,\infty}$ is identified with its lattice of translation
parameters. Choose a fundamental domain $\mathcal F$ and use the
oriented volume
\[
  \mathrm{vol}(\mathcal F)
  :=
  \int_{\mathcal F}du\wedge d\overline{u}.
\]
The zeroth Fourier coefficients are
\begin{align*}
  A_{f,0}^{\beta,x}(s)
  :=
  \frac{1}{\mathrm{vol}(\mathcal F)}
  \int_{\mathcal F}
  A_f^{\beta,x}(u,s)\,du\wedge d\overline{u},
  \qquad
  B_{f,0}^{\beta,x}(s)
  :=
  \frac{1}{\mathrm{vol}(\mathcal F)}
  \int_{\mathcal F}
  B_f^{\beta,x}(u,s)\,du\wedge d\overline{u}.
\end{align*}
For $t\in\mathcal{D}(\C)$, let
\[
  \operatorname{pr}_i:
  x-t+\Lambda\longrightarrow\C,
  \qquad
  (v_1,v_2)\longmapsto v_i,
  \qquad i=1,2,
\]
with fibers $\operatorname{pr}_i^{-1}(v_i)$ taken inside the indicated
coset. We also put
\[
  \Gamma_{N,\infty,t}
  :=
  \{n(a)\in\Gamma_{N,\infty}\mid n(a)t=t\}.
\]
In the fiber quotients below, $v_2\Gamma_{N,\infty,t}$ acts by
translation of the first coordinate.

% Explicit zeroth Fourier coefficients: 
\begin{prop}[Zeroth Fourier coefficients]
\label{prop: zeroth Fourier coefficients}
For $\operatorname{Re}(s)\gg0$, one has
$A_{f,0}^{\beta,x}(s)=0$ unless $\beta=(b,0)$, in which case
\[
  A_{f,0}^{(b,0),x}(s)
  =
  \sum_{t\in\mathcal{D}(\C)}f(t)
  \sum_{\substack{
    v\in(x-t+\Lambda)\setminus\{0\}\\
    v_2=0
  }}
  \frac{\overline{v}_1^{\,b+2}}{|v_1|^{2s}}.
\]
Similarly, $B_{f,0}^{\beta,x}(s)=0$ unless $\beta=(0,b)$, in which
case
\begin{align*}
  B_{f,0}^{(0,b),x}(s)
  &=
  -\frac{2\pi i}{(s-1)\mathrm{vol}(\mathcal F)}
  \sum_{t\in\mathcal{D}(\C)}
  \frac{f(t)}
       {[\Gamma_{N,\infty}:\Gamma_{N,\infty,t}]}
  \sum_{\substack{
    v_2\in\operatorname{pr}_2(x-t+\Lambda)\\
    v_2\neq0
  }}
  \#
  \left(
    \frac{\operatorname{pr}_2^{-1}(v_2)}
         {v_2\Gamma_{N,\infty,t}}
  \right)
  \frac{\overline{v}_2^{\,b+2}}{|v_2|^{2s}}.
\end{align*}
\end{prop}

\begin{proof}
Absolute convergence for $\operatorname{Re}(s)\gg0$ permits
interchanging the sums with the integral over $\mathcal F$.
For the $A$-coefficient, split the lattice sum according to whether
$v_2$ vanishes. When $v_2=0$, one has $v_1^u=v_1$ and $v_2^u=0$,
so the numerator vanishes unless $\beta_2=0$, equivalently
$\beta=(b,0)$. In this case the summand is independent of $u$,
and its average gives the displayed contribution to
$A_{f,0}^{(b,0),x}(s)$.

For $v_2\neq0$, set $w:=v_1-uv_2$. Then
$dw=-v_2\,du$ and
$d\overline{w}=-\overline{v}_2\,d\overline{u}$, so
\[
  du\wedge d\overline{u}
  =
  |v_2|^{-2}\,dw\wedge d\overline{w}.
\]
Decomposing $\mathcal{D}(\C)$ into $\Gamma_{N,\infty}$-orbits and
unfolding with respect to $\Gamma_{N,\infty,t}$ reduces the
contribution for fixed $t$ and $v_2\neq0$ to finite multiples of
\[
  \frac{\overline{v}_2^{\,\beta_2}}{|v_2|^2}
  \int_{\C}
  \frac{\overline{w}^{\,\beta_1+2}}
       {(|w|^2+|v_2|^2)^s}
  \,dw\wedge d\overline{w}.
\]
The angular integral vanishes because $\beta_1+2>0$. Thus the
terms with $v_2\neq0$ do not contribute to the zeroth Fourier
coefficient of $A_f^{\beta,x}$.

For the $B$-coefficient, the terms with $v_2=0$ vanish because of
the factor $(\overline{v_2^u})^2$. We describe the unfolding for
$v_2\neq0$, keeping track of its multiplicities. For a representative
$t$ of a $\Gamma_{N,\infty}$-orbit in $\mathcal{D}(\C)$, summing
the integrals associated with its orbit replaces $\mathcal F$ by
a fundamental domain $\mathcal F_t$ for
$\Gamma_{N,\infty,t}$. This domain consists of
$[\Gamma_{N,\infty}:\Gamma_{N,\infty,t}]$ translates of
$\mathcal F$. For fixed $v_2\neq0$, the possible first coordinates
form finitely many orbits under
\[
  v_1\longmapsto v_1+av_2,
  \qquad n(a)\in\Gamma_{N,\infty,t},
\]
and their number is
\(
  \#
  \left(
    \operatorname{pr}_2^{-1}(v_2)/v_2\Gamma_{N,\infty,t}
  \right).
\)
For each such fiber orbit, summing over the translations unfolds
the integral over $\mathcal F_t$ to an integral over $\C$.
Rewriting the sum over torsion-orbit representatives as a sum over
all $t\in\mathcal{D}(\C)$ introduces the factor
$[\Gamma_{N,\infty}:\Gamma_{N,\infty,t}]^{-1}$.

With the same substitution $w=v_1-uv_2$, the integral contributed
by each fiber orbit is
\[
  \frac{\overline{v}_2^{\,\beta_2+2}}{|v_2|^2}
  \int_{\C}
  \frac{\overline{w}^{\,\beta_1}}
       {(|w|^2+|v_2|^2)^s}
  \,dw\wedge d\overline{w}.
\]
Its angular integral vanishes unless $\beta_1=0$, equivalently
$\beta=(0,b)$. In this remaining case, polar coordinates give
\[
  \int_{\C}
  \frac{dw\wedge d\overline{w}}
       {(|w|^2+|v_2|^2)^s}
  =
  -\frac{2\pi i}{s-1}|v_2|^{2-2s}.
\]
Multiplication by the Jacobian factor $|v_2|^{-2}$ leaves the
power $|v_2|^{-2s}$. Including the fiber-orbit multiplicity and
the stabilizer-index factor, summing with weight $f(t)$ over
$t$ and $v_2\neq0$, and dividing by
$\mathrm{vol}(\mathcal F)$ gives the stated $B$-coefficient.
\end{proof}

\begin{cor}[Constant term at the standard cusp]
\label{cor: constant term standard cusp}
For $\operatorname{Re}(s)\gg0$, the zeroth Fourier coefficient of
the horospherical restriction, taken in the moving basis, is
\begin{align*}
  &\left.
    \mathrm{res}_{\infty,1}\bigl(\psi_u^b(f,x)\bigr)(s)
  \right|_0
  =
  \Biggl[
    \sum_{t\in\mathcal{D}(\C)}f(t)
    \sum_{\substack{
      v\in(x-t+\Lambda)\setminus\{0\}\\
      v_2=0
    }}
    \frac{\overline{v}_1^{\,b+2}}{|v_1|^{2s}}
    (\overline{e}_1^u)^{[b]}\,du
  \\
  &\qquad\qquad\qquad\qquad+
    \frac{2\pi i}{(s-1)\mathrm{vol}(\mathcal F)}
    \sum_{t\in\mathcal{D}(\C)}
    \frac{f(t)}
         {[\Gamma_{N,\infty}:\Gamma_{N,\infty,t}]}
    \sum_{\substack{
      v_2\in\operatorname{pr}_2(x-t+\Lambda)\\
      v_2\neq0
    }}
    \#
    \left(
      \frac{\operatorname{pr}_2^{-1}(v_2)}
           {v_2\Gamma_{N,\infty,t}}
    \right)
    \frac{\overline{v}_2^{\,b+2}}{|v_2|^{2s}}
    (\overline{e}_2^u)^{[b]}\,d\overline{u}
  \Biggr]
  \otimes\bigwedge_{i=1}^2dz_i.
\end{align*}
Thus only the two extreme multi-indices $\beta=(b,0)$ and
$\beta=(0,b)$ contribute to the zeroth Fourier coefficient.
\end{cor}

\begin{proof}
Apply Proposition~\ref{prop: zeroth Fourier coefficients} to the
scalar coefficients in
$A_f^{b,x}(u,s)\,du-B_f^{b,x}(u,s)\,d\overline{u}$.
The minus sign in $B_{f,0}^{(0,b),x}(s)$ cancels the minus sign
preceding the $B$-term in the restricted form.
\end{proof}

\begin{rem}\rm
The Fourier expansion is taken coefficientwise with respect to the
$\Gamma_{N,\infty}$-equivariant moving basis. Thus the zeroth
Fourier coefficients are those of the scalar functions
$A_f^{\beta,x}(u,s)$ and $B_f^{\beta,x}(u,s)$, whereas the vectors
$(\overline{e}_1^u)^{[b]}$ and $(\overline{e}_2^u)^{[b]}$ still
vary with $u$ according to the local system. Periodicity under
$\Gamma_{N,\infty}$ does not imply that the two scalar
contributions are separately invariant under the full cusp
stabilizer $\Gamma_\infty$. The residual finite group acts on
the complete coefficient-valued expression; at the cohomological
level, passage from $T_\infty$ to $e_\infty$ is given by taking
$\Gamma_\infty/\Gamma_{N,\infty}$-invariants, as in
Proposition~\ref{prop: cusp cohomology via unipotent subgroup}.
\end{rem}

\begin{rem}[Constant terms at an arbitrary cusp]\rm
For a cusp $c$, Proposition~\ref{prop:restriction-arbitrary-cusp}
reduces the calculation to the transported form
$\sigma_c^{-1}\cdot\psi_g^b(f,x)$, with Fourier expansion taken
with respect to the lattice
\[
  \sigma_c^{-1}\Gamma_{N,c}\sigma_c
  \subset N_\infty(K)\cong(K,+).
\]
Corollary~\ref{cor: constant term standard cusp} applies with the
transported lattice, coordinate projections, stabilizers, and
moving basis. The same angular-integral argument leaves only
$\beta=(b,0)$ and $\beta=(0,b)$ in this cusp-adapted basis.
For $\Gamma=\GL_2(\mathcal O_K)$,
Example~\ref{exmp: ideal-class description of the cusp data} supplies the unipotent
lattice $\mathfrak a^{-2}$ and the abelian lattice
$\Lambda=\mathfrak a^{-1}\Omega_K\oplus\mathfrak a\Omega_K$.
The extreme-index statement is relative to the transported basis:
on returning to the original fixed basis, the action of
$\sigma_c^{-1}$ may mix the coefficient vectors. The complete
coefficient-valued constant term must therefore be transported
before it is expressed in the fixed basis.
\end{rem}

\begin{rem}\rm
The preceding identities were established in the region of absolute
convergence. Since both sides arise from the same
Eisenstein--Kronecker series, they extend by meromorphic
continuation in $s$. For a related use of constant terms in
studying boundary restrictions of polylogarithmic Eisenstein
classes, see
\cite[Section~4.5, Proposition~4.5.1]{graf2016polylogarithmsgl2totallyreal}.
\end{rem}

It remains to show that the nonzero Fourier modes do not contribute
to the boundary cohomology class. We do this by constructing a
modewise contracting homotopy using the nilpotent monodromy of the
moving coefficient frame.

\begin{prop}[The zeroth Fourier coefficient determines the boundary class]
\label{prop:zero-fourier-coefficient-determines-boundary-class}
\label{prop: zero Fourier coefficient determines boundary class}
The forms
\[
  \mathrm{res}_{\infty,1}\bigl(\psi_u^b(f,x)\bigr)(2)
  \qquad\text{and}\qquad
  \left.
    \mathrm{res}_{\infty,1}\bigl(\psi_u^b(f,x)\bigr)(2)
  \right|_0
\]
represent the same class in $H^1\left(e_\infty,\mathrm{TSym}^b(\omega_{\mathcal A^\vee/\C})
  \otimes
  \omega_{\mathcal A/\C}^2\right)$, under the
identification with the residual-invariant cohomology of $T_\infty$
in Proposition~\ref{prop: cusp cohomology via unipotent subgroup}.
\end{prop}

%Fourier homotopy argument: 
\begin{proof}
By Proposition~\ref{prop: cusp cohomology via unipotent subgroup},
it is enough to compare the two classes on
$T_\infty=\Gamma_{N,\infty}\backslash\C$ and then pass to
$\Gamma_\infty/\Gamma_{N,\infty}$-invariants.
Let $N$ be the nilpotent endomorphism of
$\omega_{\mathcal A^\vee/\C}$ defined by
\[
  N(\overline{e}_1)=0,
  \qquad
  N(\overline{e}_2)=\overline{e}_1.
\]
Denote by $N_b$ its induced action on $\mathrm{TSym}^b(\omega_{\mathcal A^\vee/\C})
  \otimes
  \omega_{\mathcal A/\C}^2$, acting on symmetric
tensors by the Leibniz rule and trivially on the top-form factor.
Since $N^2=0$, one has $N_b^{b+1}=0$.
The moving frame
$\overline{e}_1^u=\overline{e}_1$ and
$\overline{e}_2^u=\overline{u}\,\overline{e}_1+\overline{e}_2$
satisfies
\[
  d\overline{e}_1^u=0,
  \qquad
  d\overline{e}_2^u=d\overline{u}\,\overline{e}_1^u.
\]
Consequently, in the moving-frame trivialization with periodic
scalar coefficients, the local-system differential is
\(
  \nabla_b=d+d\overline{u}\wedge N_b.
\)
Choose an oriented $\Z$-basis $\lambda_1,\lambda_2$ of the
translation lattice $\Gamma_{N,\infty}$ and write
\[
  u=x_1\lambda_1+x_2\lambda_2,
  \qquad
  (x_1,x_2)\in\R^2/\Z^2.
\]
Since
$d\overline{u}=\overline{\lambda}_1dx_1+
\overline{\lambda}_2dx_2$, this becomes
\[
  \nabla_b=d+dx_1\wedge N_1+dx_2\wedge N_2,
  \qquad
  N_i:=\overline{\lambda}_iN_b.
\]

Set
\[
  \omega
  :=
  \mathrm{res}_{\infty,1}\bigl(\psi_u^b(f,x)\bigr)(2),
  \qquad
  \omega_0:=\left.\omega\right|_0.
\]
By Lemma~\ref{lem: unipotent periodicity moving frame}, the scalar
coefficients of $\omega$ in the moving frame are periodic.
They are smooth, so their Fourier expansion gives
\[
  \omega=\sum_{m\in\Z^2}\chi_m(x)\omega_m,
  \qquad
  \chi_m(x)
  :=
  \exp\!\bigl(2\pi i(m_1x_1+m_2x_2)\bigr),
\]
where the $\omega_m$ have constant coefficients in this
trivialization, and the term with $m=0$ is $\omega_0$.
The constant-coefficient operator $\nabla_b$ preserves this
decomposition. On the $m$-th mode, its differential is
\[
  \nabla_{b,m}
  =
  \sum_{k=1}^2dx_k\wedge L_{m,k},
  \qquad
  L_{m,k}:=2\pi i m_k\operatorname{id}+N_k.
\]
In particular, projection $P_0$ onto the zeroth mode commutes with
$\nabla_b$. Since $\nabla_b\omega=0$, each $\omega_m$ is closed
under $\nabla_{b,m}$, and $\nabla_b\omega_0=0$.

For $m\neq0$, choose $j\in\{1,2\}$ such that
$|m_j|=\max(|m_1|,|m_2|)$. Then $m_j\neq0$, and the nilpotence
$N_j^{b+1}=0$ gives the finite inverse
\[
  L_{m,j}^{-1}
  =
  \frac{1}{2\pi i m_j}
  \sum_{r=0}^{b}
  \left(-\frac{N_j}{2\pi i m_j}\right)^r.
\]
Define the degree-minus-one operator
\(
  h_m:=L_{m,j}^{-1}\iota_{\partial/\partial x_j}.
\)
The operators $L_{m,1}$ and $L_{m,2}$ commute because $N_1$ and
$N_2$ are scalar multiples of $N_b$. Together with the
wedge-contraction identity
\[
  dx_k\wedge\iota_{\partial/\partial x_j}
  +
  \iota_{\partial/\partial x_j}(dx_k\wedge-)
  =
  \delta_{kj}\operatorname{id},
\]
this gives
\[
  \nabla_{b,m}h_m+h_m\nabla_{b,m}
  =
  \sum_{k=1}^2
  \delta_{kj}L_{m,k}L_{m,j}^{-1}
  =
  \operatorname{id}.
\]
Thus every nonzero Fourier subcomplex is contractible.

The finite inverse also gives the uniform estimate
\(
  \|h_m\|
  =
  O\!\left(\max(|m_1|,|m_2|)^{-1}\right).
\)
Since the Fourier coefficients of a smooth form decay faster than
any power, the series
\[
  K\omega
  :=
  \sum_{m\neq0}\chi_m(x)h_m\omega_m
\]
converges in the $\mathcal C^\infty$-topology. The same construction
defines $K$ on smooth coefficient-valued forms in every degree,
and convergence with all derivatives permits summing the modewise
homotopy identities:
\[
  \nabla_bK+K\nabla_b=\operatorname{id}-P_0.
\]
Applying this identity to the closed form $\omega$ yields
\[
  \omega-\omega_0=\nabla_bK\omega.
\]
Hence $[\omega]=[\omega_0]$ on $T_\infty$.
Their common class is invariant under
$\Gamma_\infty/\Gamma_{N,\infty}$ because $\omega$ comes from
the cusp restriction, so it determines a class in
$H^1\left(e_\infty,\mathrm{TSym}^b(\omega_{\mathcal A^\vee/\C})
  \otimes
  \omega_{\mathcal A/\C}^2\right)$ through the identification used at the
beginning of the proof.
\end{proof}

We may now specialize the constant-term formula to the value $s=2$
occurring in the Eisenstein--Kronecker representative.

% Explicit specialization and regularity: 
\begin{cor}[Explicit representative of the boundary class]
\label{cor:zero-fourier-representative-boundary-class}
\label{cor: zero Fourier representative of boundary class}
The boundary restriction of $[\psi_g^b(f,x)]$ at the standard cusp is
represented on $T_\infty$ by the zeroth Fourier coefficient
\begin{align*}
  &\left.
    \mathrm{res}_{\infty,1}\bigl(\psi_u^b(f,x)\bigr)(2)
  \right|_0
  =
  \Biggl[
    \sum_{t\in\mathcal{D}(\C)}f(t)
    \sum_{\substack{
      v\in(x-t+\Lambda)\setminus\{0\}\\
      v_2=0
    }}
    \frac{\overline{v}_1^{\,b+2}}{|v_1|^4}
    (\overline{e}_1^u)^{[b]}\,du
  \\
  &\qquad\qquad\qquad\qquad\qquad+
    \frac{2\pi i}{\mathrm{vol}(\mathcal F)}
    \sum_{t\in\mathcal{D}(\C)}
    \frac{f(t)}
         {[\Gamma_{N,\infty}:\Gamma_{N,\infty,t}]}
    \sum_{\substack{
      v_2\in\operatorname{pr}_2(x-t+\Lambda)\\
      v_2\neq0
    }}
    \#
    \left(
      \frac{\operatorname{pr}_2^{-1}(v_2)}
           {v_2\Gamma_{N,\infty,t}}
    \right)
    \frac{\overline{v}_2^{\,b+2}}{|v_2|^4}
    (\overline{e}_2^u)^{[b]}\,d\overline{u}
  \Biggr]
  \otimes\bigwedge_{i=1}^2dz_i.
\end{align*}
The sums at $s=2$ are understood as the values of their
meromorphic continuations.
\end{cor}

\begin{proof}
The identity of Corollary~\ref{cor: constant term standard cusp}
extends meromorphically in $s$ and is regular at $s=2$:
the relevant Eisenstein--Kronecker series have multi-indices
of total degree $b+2>0$, so the possible pole there does not
occur by
Proposition~\ref{prop:functional equation of Eisenstein-Kronecker series}.
Evaluate this identity at $s=2$ and apply
Proposition~\ref{prop:zero-fourier-coefficient-determines-boundary-class}.
\end{proof}

\begin{cor}[Boundary class at an arbitrary cusp]
\label{cor:boundary-class-arbitrary-cusp}
Let $c\in\mathbb P^1(K)$ represent a cusp and choose
$\sigma_c\in\GL_2(K)$ with $\sigma_c(\infty)=c$.
The class
\[
  \mathrm{res}_c\bigl([\psi_g^b(f,x)]\bigr)
  \in H^1\left(e_c,\mathrm{TSym}^b(\omega_{\mathcal A^\vee/\C})
  \otimes
  \omega_{\mathcal A/\C}^2\right)
\]
is represented by the zeroth Fourier coefficient at $s=2$ of
the standard horospherical restriction of
$\sigma_c^{-1}\cdot\psi_g^b(f,x)$, computed at normalized height
one with respect to the lattice
$\sigma_c^{-1}\Gamma_{N,c}\sigma_c$ and the transported moving
frame. Its explicit expression is obtained from
Corollary~\ref{cor:zero-fourier-representative-boundary-class}
by replacing all data with their cusp-adapted counterparts.

For $\Gamma=\GL_2(\mathcal O_K)$, the component at
$[\mathfrak a]\in\operatorname{Cl}(K)$ is computed using the
unipotent lattice $\mathfrak a^{-2}$ and the abelian lattice
$\Lambda=\mathfrak a^{-1}\Omega_K\oplus\mathfrak a\Omega_K$
of Example~\ref{exmp: ideal-class description of the cusp data}.
\end{cor}

\begin{proof}
Proposition~\ref{prop:restriction-arbitrary-cusp} identifies the
restriction at $c$ with the standard horospherical restriction
of the transported form, and
Corollary~\ref{cor: constant term standard cusp} computes its
zeroth Fourier coefficient using the transported data.
In the cusp-adapted moving frame, the local-system differential
again has the constant-coefficient form
$d+d\overline{u}\wedge N_b$, where $N_b$ now denotes the
transported nilpotent action and still satisfies
$N_b^{b+1}=0$. Choosing a basis of the conjugated unipotent
lattice therefore gives mode operators of the form
$2\pi i m_k\operatorname{id}+\overline{\lambda}_kN_b$.
The finite inverses and convergent Fourier homotopy of
Proposition~\ref{prop:zero-fourier-coefficient-determines-boundary-class}
apply to these operators, so the restricted form and its zeroth
mode determine the same class on $T_c$. Passage to
$\Gamma_c/\Gamma_{N,c}$-invariants identifies this class with
the boundary restriction on $e_c$.
\end{proof}

Here the surviving indices $(b,0)$ and $(0,b)$ are relative to
the cusp-adapted coefficient basis; expressing the transported
result in the original fixed basis may mix their coefficient
vectors through the action of $\sigma_c^{-1}$.

The cuspwise representatives now assemble into the full
Borel--Serre boundary class.

\begin{prop}[The boundary restriction class]
\label{prop:boundary-restriction-class}
\label{prop: boundary restriction class}
The decomposition
\(
  \partial\overline{\Gamma\backslash\mathbb H^3}^{\,\mathrm{BS}}
  =
  \coprod_{c\in\Gamma\backslash\mathbb P^1(K)}e_c
\)
induces canonical identifications
\begin{align*}
  H^1\left(
    \partial\overline{\Gamma\backslash\mathbb H^3}^{\,\mathrm{BS}},
    \mathrm{TSym}^b(\omega_{\mathcal A^\vee/\C})
    \otimes
    \omega_{\mathcal A/\C}^2
  \right)
  &\cong
  \bigoplus_{c\in\Gamma\backslash\mathbb P^1(K)}
  H^1\left(e_c,\mathrm{TSym}^b(\omega_{\mathcal A^\vee/\C})
  \otimes
  \omega_{\mathcal A/\C}^2\right)
  \\
  &\cong
  \bigoplus_{c\in\Gamma\backslash\mathbb P^1(K)}
  H^1\left(\Gamma_{N,c},\mathrm{TSym}^b(\omega_{\mathcal A^\vee/\C})
  \otimes
  \omega_{\mathcal A/\C}^2\right)^{\Gamma_c/\Gamma_{N,c}}.
\end{align*}
Under these identifications,
\begin{align*}
  \mathrm{res}_{\partial}\bigl([\psi_g^b(f,x)]\bigr)
  &=
  \left(
    \mathrm{res}_c\bigl([\psi_g^b(f,x)]\bigr)
  \right)_{c\in\Gamma\backslash\mathbb P^1(K)}
  =
  \left(
    \left[
      \left.j_{T_c,1}^*\psi_g^b(f,x)\right|_0
    \right]
  \right)_{c\in\Gamma\backslash\mathbb P^1(K)}.
\end{align*}
Here each zeroth Fourier coefficient is computed at normalized
height one in the corresponding cusp-adapted moving frame, and
its class is viewed in the residual-invariant subspace.
Moreover,
\[
  \mathrm{res}_{\partial}\bigl([\psi_g^b(f,x)]\bigr)
  =
  \mathrm{res}_{\partial}
  \left(
    \mathrm{mom}_x^b
    \bigl(x^*EK_{\Gamma,\mathcal A}(f)\bigr)
  \right).
\]
Thus these cuspwise constant terms describe the boundary
restriction of the algebraically constructed equivariant
Eisenstein--Kronecker moment class.
\end{prop}

\begin{proof}
By Proposition~\ref{prop: Borel--Serre cusp components}, the
boundary is the finite disjoint union of the components $e_c$,
which gives the first cohomological decomposition.
The second follows from
Proposition~\ref{prop: cusp cohomology via unipotent subgroup}.
Restriction to the disjoint components gives the tuple of
cuspwise classes, and
Corollary~\ref{cor:boundary-class-arbitrary-cusp} supplies their
zeroth Fourier representatives.
Finally, Proposition~\ref{prop: explicit moment class in dimension two}
identifies
$[\psi_g^b(f,x)]
=\mathrm{mom}_x^b(x^*EK_{\Gamma,\mathcal A}(f))$;
applying $\mathrm{res}_{\partial}$ gives the last equality.
\end{proof}

The boundary restriction is therefore reduced to explicit
cuspwise constant terms in the analytic moving frame.

\subsection{Algebraic normalization of the boundary class}

We now compare the analytic moving frame occurring in the boundary
calculation with the algebraic Hodge bases fixed in
Section~\ref{subsection: Paring and decompositions}. This rewrites the
zeroth Fourier coefficient obtained in the preceding subsection in
algebraic coordinates and makes explicit the normalization arising from
the de Rham pairing.

% Basis normalization and comparison:
We retain $g=1$, $n=2$, and $S=\Spec\C$, together with the Hodge
identification
$\mathcal H^{0,-1}\cong\omega_{\mathcal A^\vee/\C}
\cong\overline{\operatorname{Lie}(\mathcal A/\C)}$
fixed above.
Assume that
$\mathcal O_L=\mathcal O_K[\alpha]
=\mathcal O_K\oplus\mathcal O_K\alpha$,
and let $\{1^\vee,\alpha^\vee\}$ be the basis dual to
$\{1,\alpha\}$. For the fixed CM type $\Sigma_K=\{\sigma\}$, normalize
the analytic generator
$E_\sigma\in\operatorname{Lie}(\mathcal A_0/\C)(\sigma)$ by
$\omega_{\mathcal A_0}(\sigma^{-1})(E_\sigma)=1$.
The corresponding Serre tensor basis is
\[
  e_1=E_\sigma\otimes_{\mathcal O_K,\sigma}1,
  \qquad
  e_2=E_\sigma\otimes_{\mathcal O_K,\sigma}\alpha.
\]
Choose the coordinates $(z_1,z_2)$ compatibly with this basis, so that
\[
  dz_1=\omega_{\mathcal A_0}(\sigma^{-1})
       \otimes_{\mathcal O_K}1^\vee,
  \qquad
  dz_2=\omega_{\mathcal A_0}(\sigma^{-1})
       \otimes_{\mathcal O_K}\alpha^\vee.
\]
In particular, $dz_1\wedge dz_2$ is the wedge of these algebraic
basis vectors.

\begin{lem}\label{lem:comparison analytic algebraic Hodge basis}
Under the Hodge identification
$\mathcal H^{0,-1}\cong\omega_{\mathcal A^\vee/\C}$, one has
\[
  \begin{aligned}
    \overline e_1
    =
    \frac{
      \omega_{\mathcal A_0^\vee}(\overline\sigma)
      \otimes_{\mathcal O_K}1
    }{
      \left\langle
        \overline{\omega(\mathcal A)},\omega(\mathcal A^\vee)
      \right\rangle_{dR,\mathcal A}
    },\qquad
    \overline e_2
    =
    \frac{
      \omega_{\mathcal A_0^\vee}(\overline\sigma)
      \otimes_{\mathcal O_K}\alpha
    }{
      \left\langle
        \overline{\omega(\mathcal A)},\omega(\mathcal A^\vee)
      \right\rangle_{dR,\mathcal A}
    }.
  \end{aligned}
\]
\end{lem}

\begin{proof}
By the Serre construction and the decomposition recalled in
Section~\ref{subsection: Paring and decompositions}, the conjugate Hodge
component is generated by
$\omega_{\mathcal A_0^\vee}(\overline\sigma)\otimes_{\mathcal O_K}1$
and
$\omega_{\mathcal A_0^\vee}(\overline\sigma)\otimes_{\mathcal O_K}\alpha$.
Dualizing $\{e_1,e_2\}$ gives the algebraic basis of
$\omega_{\mathcal A/\C}$ identified above with $\{dz_1,dz_2\}$.
The de Rham pairing for the Serre construction is induced by the pairing
for $\mathcal A_0$ together with the evaluation pairing between
$\mathcal O_L^\vee$ and $\mathcal O_L$. Since
$\{1^\vee,\alpha^\vee\}$ and $\{1,\alpha\}$ are dual, the two components
are orthogonal and have the same pairing scalar, inherited from
$\mathcal A_0$. Under the Hodge decomposition, $\overline e_i$
corresponds to $\partial/\partial\overline z_i$; normalizing the
algebraic generators by their pairing with the conjugate differential
basis therefore gives the displayed expressions.
\end{proof}

% Moving frame and tensor symmetric coefficients: 
Recall that the moving frame in the Bianchi boundary calculation satisfies
$\overline e_1^u=\overline e_1$ and
$\overline e_2^u=\overline u\,\overline e_1+\overline e_2$.
Lemma~\ref{lem:comparison analytic algebraic Hodge basis} consequently gives
\[
  \begin{aligned}
    \overline e_1^u
    =
    \frac{
      \omega_{\mathcal A_0^\vee}(\overline\sigma)
      \otimes_{\mathcal O_K}1
    }{
      \left\langle
        \overline{\omega(\mathcal A)},\omega(\mathcal A^\vee)
      \right\rangle_{dR,\mathcal A}
    },\qquad
    \overline e_2^u
    =
    \frac{
      \omega_{\mathcal A_0^\vee}(\overline\sigma)
      \otimes_{\mathcal O_K}\alpha
      +
      \overline u\,
      \omega_{\mathcal A_0^\vee}(\overline\sigma)
      \otimes_{\mathcal O_K}1
    }{
      \left\langle
        \overline{\omega(\mathcal A)},\omega(\mathcal A^\vee)
      \right\rangle_{dR,\mathcal A}
    }.
  \end{aligned}
\]
The change-of-frame matrix
$n(\overline u)=
\bigl(\begin{smallmatrix}1&\overline u\\0&1\end{smallmatrix}\bigr)$
is complex unipotent, not in general unitary. The second numerator is a
complex linear combination in $\omega_{\mathcal A^\vee/\C}$, rather
than a tensor involving $\alpha+\overline u$ over $\mathcal O_K$,
since $\overline u$ need not belong to $\mathcal O_L$.
For $\beta=(\beta_1,\beta_2)\in\mathbb N^2$ with $|\beta|=b$, the
corresponding tensor-symmetric coefficient is therefore
\[
  \begin{aligned}
    (\overline e_1^u)^{[\beta_1]}
    (\overline e_2^u)^{[\beta_2]}
    =
    \frac{1}{
      \left\langle
        \overline{\omega(\mathcal A)},\omega(\mathcal A^\vee)
      \right\rangle_{dR,\mathcal A}^{\,b}
    }
    \left(
      \omega_{\mathcal A_0^\vee}(\overline\sigma)
      \otimes_{\mathcal O_K}1
    \right)^{[\beta_1]}
    \left(
      \omega_{\mathcal A_0^\vee}(\overline\sigma)
      \otimes_{\mathcal O_K}\alpha
      +
      \overline u\,
      \omega_{\mathcal A_0^\vee}(\overline\sigma)
      \otimes_{\mathcal O_K}1
    \right)^{[\beta_2]}.
  \end{aligned}
\]

We now combine the zeroth Fourier coefficient computed above with the
comparison between the moving frame and the algebraic Hodge basis.

\begin{thm}[Algebraic formula for the constant term at the cusp $\infty$]
\label{thm:algebraic constant term at infinity}

Keep the assumptions and algebraic Hodge bases fixed above, and let
$b\in\mathbb N$. For $\operatorname{Re}(s)\gg0$, the zeroth Fourier
coefficient of the horospherical restriction
$\mathrm{res}_{\infty,1}(\psi_u^b(f,x))(s)$ is
\begin{align*}
  &\left.
    \mathrm{res}_{\infty,1}\bigl(\psi_u^b(f,x)\bigr)(s)
  \right|_0
  =
  \frac{1}{
    \left\langle
      \overline{\omega(\mathcal A)},\omega(\mathcal A^\vee)
    \right\rangle_{dR,\mathcal A}^{\,b}
  }
  \sum_{t\in\mathcal{D}(\C)}
  f(t)
  \Biggl(
    \sum_{\substack{
      v\in x-t+\Lambda\setminus\{0\}\\
      v_2=0
    }}
    \frac{\overline v_1^{\,b+2}}{|v_1|^{2s}}
    \left(
      \omega(\mathcal A^\vee)
      \begin{bmatrix}
        1\\
        0
      \end{bmatrix}
    \right)^{[b]}
    du
  \\
  &\qquad\qquad+
    \frac{
      2\pi i
    }{
      (s-1)\mathrm{vol}(\mathcal F)
      [\Gamma_{N,\infty}:\Gamma_{N,\infty,t}]
    }
    \sum_{\substack{
      v_2\in \mathrm{pr}_2(x-t+\Lambda\setminus\{0\})\\
      v_2\neq0
    }}
    \#
    \left(
      \frac{
        \mathrm{pr}_2^{-1}(v_2)
      }{
        v_2\Gamma_{N,\infty,t}
      }
    \right)
    \frac{\overline v_2^{\,b+2}}{|v_2|^{2s}}
    \left(
      \omega(\mathcal A^\vee)
      \begin{bmatrix}
        \overline u\\
        1
      \end{bmatrix}
    \right)^{[b]}
    d\overline u
  \Biggr)
  \otimes\omega(\mathcal A)^{\otimes2}.
\end{align*}
Both sides admit meromorphic continuation in $s$, holomorphic at $s=2$.
The specialization at $s=2$ is obtained by evaluating this continued
identity, formally replacing $|v_j|^{2s}$ by $|v_j|^4$ for $j=1,2$ and
$(s-1)^{-1}$ by $1$. Moreover,
\[
  \left[
    \left.
      \mathrm{res}_{\infty,1}\bigl(\psi_u^b(f,x)\bigr)(2)
    \right|_0
  \right]
  =
  \left[
    \mathrm{res}_{\infty,1}\bigl(\psi_u^b(f,x)\bigr)(2)
  \right]
\]
in
$H^1\bigl(
  e_\infty,
  \mathrm{TSym}^b(\omega_{\mathcal A^\vee/\C})
  \otimes\omega_{\mathcal A/\C}^2
\bigr)$.
Under the Hochschild--Serre identification fixed above, their common
class is the complex realization of the restriction to the cusp
$\infty$ of
\[
  \operatorname{mom}_x^b
  \bigl(x^*EK_{\Gamma,\mathcal A}(f)\bigr)
  \in
  H^1\left(
    \Gamma,
    \mathrm{TSym}^b(\omega_{\mathcal A^\vee/R})
    \otimes\omega_{\mathcal A/R}^2
  \right).
\]
\end{thm}

\begin{proof}
By Corollary~\ref{cor: zero Fourier representative of boundary class},
the restriction of $[\psi_g^b(f,x)]$ to the cusp $\infty$ is represented
by
$\left.
  \mathrm{res}_{\infty,1}(\psi_u^b(f,x))(2)
\right|_0$.
In the constant-term formula computed above,
Lemma~\ref{lem:comparison analytic algebraic Hodge basis} and the
moving-frame identities replace
$(\overline e_1^u)^{[b]}$ and $(\overline e_2^u)^{[b]}$
by the displayed algebraic coefficient vectors, contributing the factor
$\left\langle
  \overline{\omega(\mathcal A)},\omega(\mathcal A^\vee)
\right\rangle_{dR,\mathcal A}^{-b}$.
The chosen normalization identifies $dz_1\wedge dz_2$ with
$\omega(\mathcal A)^{\otimes2}$, which gives the asserted formula.
Its meromorphic continuation and holomorphy at $s=2$ follow from
Proposition~\ref{prop:functional equation of Eisenstein-Kronecker series}.
Finally, Theorem~\ref{prop:functional equation for currents}
identifies $[\psi_g^b(f,x)]$ with the complex realization of
$\operatorname{mom}_x^b(x^*EK_{\Gamma,\mathcal A}(f))$.
\end{proof}

\begin{rem}\rm
The displayed series are not to be interpreted as an absolutely
convergent definition at $s=2$. The identity is initially obtained
for $\operatorname{Re}(s)\gg0$, and the value representing the boundary
class is the value of its meromorphic continuation at $s=2$.
\end{rem}

\begin{cor}[The algebraic boundary restriction class]
\label{cor: algebraic boundary restriction class}
For each cusp $c\in\Gamma\backslash\mathbb P^1(K)$, choose
$\sigma_c\in\GL_2(K)$ with $\sigma_c(\infty)=c$ and a sufficiently
large cusp parameter $\xi_0$. Transporting the formula of
Theorem~\ref{thm:algebraic constant term at infinity} by $\sigma_c$
expresses
$\left.
  \mathrm{res}_{c,\xi_0}(\psi_g^b(f,x))(2)
\right|_0$
in the algebraic Hodge bases. The tuple of its cuspwise cohomology
classes is independent of these choices and belongs to
\begin{align*}
  \bigoplus_{c\in\Gamma\backslash\mathbb P^1(K)}
  H^1\left(
    e_c,
    \mathrm{TSym}^b(\omega_{\mathcal A^\vee/\C})
    \otimes\omega_{\mathcal A/\C}^2
  \right)
  \cong
  H^1\left(
    \partial\overline{\Gamma\backslash\mathbb H^3}^{\,\mathrm{BS}},
    \mathrm{TSym}^b(\omega_{\mathcal A^\vee/\C})
    \otimes\omega_{\mathcal A/\C}^2
  \right).
\end{align*}
Under this identification, and after base change to $\C$ on the
right-hand side, one has
\[
  \left(
    \left[
      \left.
        \mathrm{res}_{c,\xi_0}\bigl(\psi_g^b(f,x)\bigr)(2)
      \right|_0
    \right]
  \right)_{c\in\Gamma\backslash\mathbb P^1(K)}
  =
  \mathrm{res}_{\partial}
  \left(
    \operatorname{mom}_x^b
    \bigl(x^*EK_{\Gamma,\mathcal A}(f)\bigr)
  \right).
\]
\end{cor}

\begin{proof}
Apply Theorem~\ref{thm:algebraic constant term at infinity} to
$\sigma_c^{-1}\cdot\psi_g^b(f,x)$ at each cusp.
Proposition~\ref{prop: boundary restriction class} identifies the
resulting tuple with the indicated boundary restriction and establishes
its independence of the auxiliary choices.
\end{proof}

\begin{rem}[Comparison with Harder's Eisenstein construction]\rm
Harder constructs global Eisenstein cohomology classes from suitable
boundary data by an analytic Eisenstein operator defined using
Eisenstein series and meromorphic continuation, and subsequently
establishes rationality properties of the resulting classes; see
\cite[Sections~2.3--2.9 and 4.1--4.3]{Harder87}.
Here $EK_{\Gamma,\mathcal A}(f)$ is first constructed algebraically
from the completed Poincaré bundle and the equivariant localization
sequence, as recalled from
\cite[Theorem~2.20 and Definition~2.21]{KS_2025}, and its boundary
forms are computed only after base change to $\C$.
Thus the forms in
Corollary~\ref{cor: algebraic boundary restriction class}
give the explicit complex realization of the boundary restriction
of an already constructed algebraic class, rather than being used
to define that class.
\end{rem}

\subsection{Quadratic torus periods and the Kings--Sprang formula}

Returning to the ideal-theoretic torsion data, we restrict the Bianchi
realization of the Eisenstein--Kronecker class to the quadratic torus
associated with $L/K$. We compute its period and, after passing to the
algebraic Hodge bases, compare the resulting lattice sum with the
Kings--Sprang formula.

Let $\mathcal D=\mathcal A[\mathfrak C]$ and
$\mathcal U_{\mathcal D}=\mathcal A\setminus\mathcal D$, and let
$x\in\mathcal U_{\mathcal D}(\C)$ be the fixed $\Gamma$-invariant
$\mathfrak f$-torsion point. Under the uniformization
$\mathcal A(\C)=\C^2/\Lambda$, we have
$\mathcal D(\C)=\mathfrak c^{-1}\Lambda/\Lambda$, and we regard
$f\in\mathcal O_S[\mathcal D]^{0,\Gamma}$ as the corresponding
$\Gamma$-invariant trace-zero function on
$\mathfrak c^{-1}\Lambda/\Lambda$. For the comparison in this subsection,
we restrict the action to $\Gamma\cap\mathcal O_L^\times$, where units
act by multiplication on $\mathcal O_L$.

Recall that $G=\operatorname{Res}_{K/\Q}\GL_2$ and
$H=\operatorname{Res}_{L/\Q}\mathbb G_m$. Multiplication on the
two-dimensional $K$-vector space $L$ defines the regular representation
$i_H:H\longrightarrow G$, whose map on $\Q$-points is
$L^\times\longrightarrow\GL_K(L)$. Let $\{1,\alpha\}$ be the fixed
$\mathcal O_K$-basis of $\mathcal O_L$, and write $X^2-TX+N$ for the
minimal polynomial of $\alpha$ over $K$, where
$T=\Tr_{L/K}(\alpha)$ and $N=\Norm_{L/K}(\alpha)$. Fix the chosen
$\sigma\in\Sigma_K$, and let $\sigma_1,\sigma_2:L\hookrightarrow\C$
be the two embeddings above $\sigma$. Put $\alpha_i=\sigma_i(\alpha)$
and set
\[
  E_\alpha=
  \begin{pmatrix}
    1&\alpha_1\\
    1&\alpha_2
  \end{pmatrix}.
\]

\begin{lem}
\label{lem:quadratic torus algebraic embedding}
\label{lem:quadratic torus archimedean embedding}
With respect to the basis $\{1,\alpha\}$, the regular representation is
given by
\[
  i_H(a+b\alpha)=
  \begin{pmatrix}
    a&-bN\\
    b&a+bT
  \end{pmatrix},
  \qquad a+b\alpha\in L^\times.
\]
In particular, it restricts to an embedding
$\mathcal O_L^\times\longrightarrow\GL_2(\mathcal O_K)$.
Under $H(\R)\cong\C^\times\times\C^\times$ and
$G_\infty\cong\GL_2(\C)$, its archimedean realization is
\[
  i_{H_\infty}(h_1,h_2)
  =
  E_\alpha^{-1}
  \begin{pmatrix}
    h_1&0\\
    0&h_2
  \end{pmatrix}
  E_\alpha.
\]
For $h\in L^\times$, the corresponding coordinates are
$h_i=\sigma_i(h)$.
\end{lem}

\begin{proof}
The relation $\alpha^2=T\alpha-N$ gives the matrix of multiplication
by $a+b\alpha$. For an integral unit, this multiplication is an
automorphism of the $\mathcal O_K$-module $\mathcal O_L$.
The matrix $E_\alpha$ represents the isomorphism
$L\otimes_{K,\sigma}\C\cong\C\times\C$ induced by
$\sigma_1,\sigma_2$. In these coordinates, multiplication by $h$ is
diagonal with entries $h_1,h_2$, giving the archimedean formula.
\end{proof}

Put $K_\infty^H=(S^1)^2\subset H(\R)$ and
$H_\infty=H(\R)/K_\infty^H$. The point
$g_\alpha=E_\alpha^{-1}$ satisfies
\[
  g_\alpha^{-1}i_{H_\infty}(k_1,k_2)g_\alpha
  =
  \begin{pmatrix}
    k_1&0\\
    0&k_2
  \end{pmatrix}
  \in U(2),
  \qquad (k_1,k_2)\in K_\infty^H.
\]
It therefore defines a well-defined map
\[
  j_\infty(g_\alpha):
  H_\infty\longrightarrow S_\infty,
  \qquad
  hK_\infty^H\longmapsto
  i_{H_\infty}(h)g_\alpha U(2)
  =
  E_\alpha^{-1}
  \begin{pmatrix}
    h_1&0\\
    0&h_2
  \end{pmatrix}
  U(2).
\]
For $\gamma\in\Gamma\cap\mathcal O_L^\times$, multiplicativity of
$i_H$ gives
$j_\infty(g_\alpha)(\gamma h)=
\gamma\cdot j_\infty(g_\alpha)(h)$, so this map is
$\Gamma\cap\mathcal O_L^\times$-equivariant.

The center $Z_G=\operatorname{Res}_{K/\Q}\mathbb G_m$ is identified
with the diagonal subtorus of $H$, with
$i_{H_\infty}(c,c)=cI_2$. Thus the relative torus is
\[
  T_{L/K}
  =
  H/(Z_G\cap H)
  =
  \operatorname{Res}_{L/\Q}\mathbb G_m/
  \operatorname{Res}_{K/\Q}\mathbb G_m.
\]
On real points, the ratio $(h_1,h_2)\mapsto h_1/h_2$ identifies
\[
  T_{L/K}(\R)
  \cong
  (\C^\times\times\C^\times)/\C^\times
  \cong
  \C^\times.
\]
Let $K_\infty^T$ denote the image of $K_\infty^H$ in
$T_{L/K}(\R)$. It is the maximal compact subgroup, with
$K_\infty^T\cong
K_\infty^H/(K_\infty^H\cap Z_G(\R))\cong S^1$.
Consequently,
\[
  T_{L/K}(\R)/K_\infty^T
  \cong
  H(\R)\big/\bigl((Z_G\cap H)(\R)K_\infty^H\bigr)
  \cong
  \R_{>0},
\]
where the positive coordinate is $\eta=|h_1/h_2|$.
Writing $h_i=r_ie^{i\theta_i}$ with $r_i>0$, the unitary factors
disappear in $S_\infty$, so $j_\infty(g_\alpha)(h)$ is represented
by $E_\alpha^{-1}\operatorname{diag}(r_1,r_2)$.
The common scalar $(r_1r_2)^{1/2}$ accounts for the central
determinant direction, while $\eta=r_1/r_2$ parametrizes the
relative torus.

\begin{prop}\label{prop:quadratic torus Bianchi embedding}
Under the Bianchi identification
$S_\infty\cong\mathbb H^3\times\R_{>0}$, the fixed-determinant
component of the relative torus is the map
\[
  j_{T,\infty}:\R_{>0}\longrightarrow\mathbb H^3,
  \qquad
  \eta\longmapsto\bigl(u(\eta),\xi(\eta)\bigr),
\]
where
\[
  u(\eta)
  =
  -\frac{\alpha_1+\alpha_2\eta^2}{1+\eta^2},
  \qquad
  \xi(\eta)
  =
  \frac{|\alpha_2-\alpha_1|\eta}{1+\eta^2}.
\]
\end{prop}

\begin{proof}
For $g=E_\alpha^{-1}\operatorname{diag}(r_1,r_2)$, the associated
Hermitian matrix and its determinant are
\[
  H_g
  =
  E_\alpha^*
  \begin{pmatrix}
    r_1^{-2}&0\\
    0&r_2^{-2}
  \end{pmatrix}
  E_\alpha,
  \qquad
  \det(H_g)
  =
  \frac{|\alpha_2-\alpha_1|^2}{r_1^2r_2^2}.
\]
Its first-row entries are $r_1^{-2}+r_2^{-2}$ and
$r_1^{-2}\alpha_1+r_2^{-2}\alpha_2$. In the Bianchi-coordinate
expression for $H_g$, these are
$\sqrt{\det(H_g)}/\xi$ and
$-\sqrt{\det(H_g)}u/\xi$, respectively. Comparing them gives
\[
  u
  =
  -\frac{r_1^{-2}\alpha_1+r_2^{-2}\alpha_2}
        {r_1^{-2}+r_2^{-2}},
  \qquad
  \xi
  =
  \frac{|\alpha_2-\alpha_1|r_1r_2}{r_1^2+r_2^2},
\]
which depend only on $\eta=r_1/r_2$.
\end{proof}

The image of $j_{T,\infty}$ is the hyperbolic geodesic with endpoints
$-\alpha_1$ and $-\alpha_2$, traversed from $-\alpha_1$ to
$-\alpha_2$ as $\eta$ increases from $0$ to $\infty$.
Its $\Gamma\cap\mathcal O_L^\times$-equivariance follows from that of
$j_\infty(g_\alpha)$.

By the previously established restriction--corestriction formalism,
we may assume that $\Gamma$ is torsion-free. Since the $\Gamma$-action
preserves the determinant coordinate, the determinant-slice comparison
allows us to regard the moment class as
\[
  \mathrm{mom}_x^b
  \bigl(x^*EK_{\Gamma,\mathcal A}(f)\bigr)
  =
  [\psi_g^b(f,x)]
  \in
  H^1\left(
    \Gamma\backslash\mathbb H^3,
    \mathrm{TSym}^b(\omega_{\mathcal A^\vee/\C})
    \otimes\omega_{\mathcal A/\C}^2
  \right).
\]
Equivariance of $j_{T,\infty}$ gives the descended map
\[
  j_\Gamma:
  (\Gamma\cap\mathcal O_L^\times)
  \backslash T_{L/K}(\R)/K_\infty^T
  \longrightarrow
  \Gamma\backslash\mathbb H^3,
\]
and hence the induced restriction map
\[
  \begin{aligned}
    j_\Gamma^*:\quad
    &H^1\left(
      \Gamma\backslash\mathbb H^3,
      \mathrm{TSym}^b(\omega_{\mathcal A^\vee/\C})
      \otimes\omega_{\mathcal A/\C}^2
    \right)
    \longrightarrow
    H^1\left(
      (\Gamma\cap\mathcal O_L^\times)
      \backslash T_{L/K}(\R)/K_\infty^T,
      \mathrm{TSym}^b(\omega_{\mathcal A^\vee/\C})
      \otimes\omega_{\mathcal A/\C}^2
    \right).
  \end{aligned}
\]

We compute the restriction using the explicit Bianchi representative.

\begin{prop}\label{prop:restriction to quadratic torus}
The restriction of the Eisenstein--Kronecker class to the quadratic
torus is represented by
\[
  j_\Gamma^*EK^b_{\Gamma,\mathcal A}(f,x)
  =
  \bigl[j_\Gamma^*\psi_g^b(f,x)\bigr],
\]
where
\[
\begin{aligned}
  j_\Gamma^*\psi_g^b(f,x)
  ={}&
  -\frac{2\det(H_\eta)}{\overline{\alpha_2-\alpha_1}}
  \sum_{t\in\mathfrak c^{-1}\Lambda/\Lambda}
  f(t)\sum_{|\beta|=b}
  \Bigl(
    K^{\beta+(2,0)}(H_\eta,x-t,0,2,\Lambda)
    +
    (\overline{\alpha_1}+\overline{\alpha_2})
    K^{\beta+(1,1)}(H_\eta,x-t,0,2,\Lambda)
  \\
  &\qquad\qquad\qquad\qquad\qquad\qquad\qquad\qquad+
    \overline{\alpha_1}\,\overline{\alpha_2}
    K^{\beta+(0,2)}(H_\eta,x-t,0,2,\Lambda)
  \Bigr)
  \frac{d\eta}{\eta}
  \otimes
  \overline e^{[\beta]}
  \otimes
  \bigwedge_{i=1}^2dz_i.
\end{aligned}
\]
Here $H_\eta$ is the Hermitian form along the chosen fixed determinant
slice, given by
\[
  H_\eta
  =
  \frac{\sqrt{\det(H_\eta)}}{|\alpha_2-\alpha_1|\eta}
  E_\alpha^*
  \begin{pmatrix}
    1&0\\
    0&\eta^2
  \end{pmatrix}
  E_\alpha.
\]
\end{prop}

\begin{proof}
Differentiating the coordinate formulas in
Proposition~\ref{prop:quadratic torus Bianchi embedding} gives
\[
\begin{aligned}
  j_\Gamma^*\left(\frac{du}{\xi}\right)
  =
  -\frac{2(\alpha_2-\alpha_1)}
        {|\alpha_2-\alpha_1|(1+\eta^2)}\,d\eta,
  \qquad
  j_\Gamma^*\left(\frac{d\overline u}{\xi}\right)
  =
  -\frac{2\overline{\alpha_2-\alpha_1}}
        {|\alpha_2-\alpha_1|(1+\eta^2)}\,d\eta,
  \qquad
  j_\Gamma^*\left(\frac{d\xi}{\xi}\right)
  =
  \frac{1-\eta^2}{1+\eta^2}\frac{d\eta}{\eta}.
\end{aligned}
\]
In particular,
\[
  j_\Gamma^*\left(\frac{1}{\xi}\frac{du}{\xi}\right)
  =
  -\frac{2}{\overline{\alpha_2-\alpha_1}}
  \frac{d\eta}{\eta}.
\]
Substituting these identities into the explicit formula for
$\psi_g^b(f,x)$ obtained above gives the common prefactor
$-2\det(H_\eta)/\overline{\alpha_2-\alpha_1}$; the remaining
coefficients of $K^{\beta+(1,1)}$ and $K^{\beta+(0,2)}$ simplify to
$\overline{\alpha_1}+\overline{\alpha_2}$ and
$\overline{\alpha_1}\,\overline{\alpha_2}$, respectively.
Since $\psi_g^b(f,x)$ is $\Gamma$-equivariant and $j_{T,\infty}$ is
$\Gamma\cap\mathcal O_L^\times$-equivariant, the resulting form
descends to
$(\Gamma\cap\mathcal O_L^\times)
\backslash T_{L/K}(\R)/K_\infty^T$.
\end{proof}

We now evaluate the restricted form on the quadratic torus cycle.

\begin{prop}\label{prop_11}
With the orientation induced by increasing $\eta$, the quadratic
torus period is
 \begin{equation*}
    \scalebox{0.92}{$\displaystyle
\begin{aligned}
  \int_{
    (\Gamma\cap\mathcal O_L^\times)
    \backslash T_{L/K}(\R)/K_\infty^T
  }
  j_\Gamma^*\bigl(\psi_g^b(f,x)\bigr)
 =
  -(\alpha_2-\alpha_1)
  \sum_{
    t\in
    (\Gamma\cap\mathcal O_L^\times)
    \backslash(\mathfrak c^{-1}\Lambda/\Lambda)
  }
  f(t)\sum_{|\beta|=b}
  \sum_{
    \substack{
      v\in
      (\Gamma\cap\mathcal O_L^\times)
      \backslash
      \bigl(\Lambda+(x-t)(\Gamma\cap\mathcal O_L^\times)\bigr)
      \\
      v\neq0
    }
  }
  \frac{\overline v^\beta}{v^{(1)}v^{(2)}}
  \otimes
  \overline e^{[\beta]}
  \otimes
  \bigwedge_{i=1}^2dz_i,
\end{aligned}
    $}
 \end{equation*}
where, for $v=(v_1,v_2)$,
\[
  v^{(1)}=v_1+\alpha_1v_2,
  \qquad
  v^{(2)}=v_1+\alpha_2v_2.
\]
\end{prop}

\begin{proof}
Choose lifts of $x$ and $t$ to $\C^2$ and write
$v=x-t+\lambda$ with $\lambda\in\Lambda$.
The factorization
$(\overline v_1+\overline{\alpha_1}\,\overline v_2)
 (\overline v_1+\overline{\alpha_2}\,\overline v_2)
 =\overline{v^{(1)}v^{(2)}}$
combines the three Eisenstein--Kronecker terms into
 \begin{equation*}
    \scalebox{0.90}{$\displaystyle
\begin{aligned}
  K^{\beta+(2,0)}(H_\eta,x-t,0,2,\Lambda)
  +
  (\overline{\alpha_1}+\overline{\alpha_2})
  K^{\beta+(1,1)}(H_\eta,x-t,0,2,\Lambda)
  +
  \overline{\alpha_1}\,\overline{\alpha_2}
  K^{\beta+(0,2)}(H_\eta,x-t,0,2,\Lambda)
  =
  \sum_{v\in(x-t+\Lambda)\setminus\{0\}}
  \frac{
    \overline v^\beta\,\overline{v^{(1)}v^{(2)}}
  }{
    \|v\|_\eta^4
  }.
\end{aligned}
    $}
 \end{equation*}
The expression for $H_\eta$ gives
\[
  \|v\|_\eta^2
  =
  \frac{\sqrt{\det(H_\eta)}}{|\alpha_2-\alpha_1|\eta}
  \left(
    |v^{(1)}|^2+\eta^2|v^{(2)}|^2
  \right).
\]
Since $\det(H_\eta)$ is constant on the chosen slice, the integral
of a lattice summand is reduced to
\[
  \int_0^\infty
  \frac{\eta\,d\eta}
       {\left(|v^{(1)}|^2+\eta^2|v^{(2)}|^2\right)^2}
  =
  \frac{1}{2|v^{(1)}|^2|v^{(2)}|^2}.
\]
Together with the preceding norm formula, this yields
\[
  \int_0^\infty
  \frac{
    \overline v^\beta\,\overline{v^{(1)}v^{(2)}}
  }{
    \|v\|_\eta^4
  }
  \frac{d\eta}{\eta}
  =
  \frac{|\alpha_2-\alpha_1|^2}{2\det(H_\eta)}
  \frac{\overline v^\beta}{v^{(1)}v^{(2)}}.
\]
Multiplying by the prefactor in
Proposition~\ref{prop:restriction to quadratic torus} gives the
single-vector contribution
$-(\alpha_2-\alpha_1)\overline v^\beta/(v^{(1)}v^{(2)})$.

Finally, the $\Gamma\cap\mathcal O_L^\times$-invariance of $f$
allows us to group the torsion points into unit orbits; since $x$
is fixed, the shifted lattices associated with the orbit of $t$
have union
$\Lambda+(x-t)(\Gamma\cap\mathcal O_L^\times)$.
For each unit orbit of nonzero lattice vectors, equivariance
identifies the sum of the corresponding integrals over a
fundamental domain with the integral of a representative summand
over $\R_{>0}$, obtained by combining the unit translates of that
domain.
\end{proof}

In the eigen-coordinates determined by $\alpha_1$ and $\alpha_2$,
the product $v^{(1)}v^{(2)}$ is the relative norm factor. We now rewrite the period of Proposition~\ref{prop_11} in the algebraic
Hodge bases fixed in the preceding subsection.

\begin{cor}
With respect to these bases, the quadratic torus period is
\[
\begin{aligned}
  &\int_{
    (\Gamma\cap\mathcal O_L^\times)
    \backslash T_{L/K}(\R)/K_\infty^T
  }
  j_\Gamma^*\bigl(\psi_g^b(f,x)\bigr)
 =
  -\frac{
    \alpha_2-\alpha_1
  }{
    \left\langle
      \omega(\mathcal A_0),
      \omega(\mathcal A_0^\vee)
    \right\rangle_{dR,\mathcal A_0}^{\,b}
  }
  \sum_{
    t\in
    (\Gamma\cap\mathcal O_L^\times)
    \backslash(\mathfrak c^{-1}\Lambda/\Lambda)
  }
  f(t)\sum_{|\beta|=b}
  \sum_{
    \substack{
      0\neq v\in
      (\Gamma\cap\mathcal O_L^\times)\backslash
      \\
      \bigl(\Lambda+(x-t)(\Gamma\cap\mathcal O_L^\times)\bigr)
    }
  }
  \frac{\overline v^\beta}{v^{(1)}v^{(2)}}
  \\
  &\qquad\qquad\qquad\qquad\qquad\qquad\otimes
  \left(
    \omega_{\mathcal A_0^\vee,\overline\sigma}
    \otimes_{\mathcal O_K}1
  \right)^{[\beta_1]}
  \otimes
  \left(
    \omega_{\mathcal A_0^\vee,\overline\sigma}
    \otimes_{\mathcal O_K}\alpha
  \right)^{[\beta_2]}
 \otimes
  \left(
    \left(
      \omega_{\mathcal A_0,-\sigma}
      \otimes_{\mathcal O_K}1^\vee
    \right)
    \wedge
    \left(
      \omega_{\mathcal A_0,-\sigma}
      \otimes_{\mathcal O_K}\alpha^\vee
    \right)
  \right).
\end{aligned}
\]
\end{cor}

\begin{proof}
By Lemma~\ref{lem:comparison analytic algebraic Hodge basis},
$\overline e^{[\beta]}$ is the tensor of the two algebraic
divided-power factors displayed above multiplied by
$\left\langle
  \omega(\mathcal A_0),\omega(\mathcal A_0^\vee)
\right\rangle_{dR,\mathcal A_0}^{-b}$,
since $|\beta|=b$. The chosen normalization identifies
$dz_1\wedge dz_2$ with the displayed algebraic top form.
Substituting these two identifications into
Proposition~\ref{prop_11} gives the formula.
\end{proof}

\begin{rem}\rm
After restriction to the commutative subgroup
$\Gamma\cap\mathcal O_L^\times$ and decomposition of the coefficient
system into its CM eigenspaces, this formula recovers the
$d=2$ specialization of \cite[Corollary~3.28]{KS_2025} with zero
holomorphic multi-index. Indeed, the factor
$\left\langle
  \omega(\mathcal A_0),\omega(\mathcal A_0^\vee)
\right\rangle_{dR,\mathcal A_0}^{-b}$
is precisely the de Rham pairing factor in their algebraic
normalization, while
\[
  (\alpha_2-\alpha_1)\,dz_1\wedge dz_2
  =
  (dz_1+\alpha_1dz_2)\wedge(dz_1+\alpha_2dz_2).
\]
Thus $\alpha_2-\alpha_1$ is the change-of-basis factor from the
chosen $\mathcal O_K$-basis to the ordered CM eigenbasis.
After the change of variable $t\mapsto-t$ in the torsion sum, the
orbit lattice sum is the corresponding Eisenstein series in
their formula.
\end{rem}

\section{Application to special values of Hecke \texorpdfstring{$L$}{L}-functions}

Building on the quadratic-torus period formula obtained in the preceding
section, we now relate the resulting Eisenstein--Kronecker sums to critical
values of Hecke $L$-functions. After recalling the relevant partial
$L$-functions and fixing compatible CM periods, we use the algebraic
boundary classes constructed above to deduce the corresponding
algebraicity statements.

\subsection{Hecke characters and partial \texorpdfstring{$L$}{L}-functions}

We recall the Hecke-theoretic notation needed for the special-value
calculation and give a ray-class description of the relevant partial
$L$-functions. We retain the notation for types and CM types from
Section~\ref{section:CM field and CM type}; in particular, $K$ is a CM
field with CM type $\Sigma_K\subset J_K$, and $I_K=\mathbb Z[J_K]$.

Let $\mathfrak f\subset\mathcal O_K$ be a nonzero integral ideal.
We denote by $\mathcal I(\mathfrak f)$ the group of fractional ideals
of $K$ prime to $\mathfrak f$, and by $\mathcal P_{\mathfrak f}$ the
subgroup of principal ideals generated by elements $\lambda\in K^\times$
satisfying the ray congruence $\lambda\equiv1\pmod{\mathfrak f}$.
Since $K$ is totally imaginary, no positivity condition at the
archimedean places is required, and
$\mathcal I(\mathfrak f)/\mathcal P_{\mathfrak f}$ is the ray class
group modulo $\mathfrak f$. We also put
$\mathcal O_{\mathfrak f}^{\times}
:=\{u\in\mathcal O_K^\times\mid u\equiv1\pmod{\mathfrak f}\}$.

A type $\mu\in I_K$ is called a \emph{Hecke character type} of
\emph{weight} $w\in\mathbb Z$ if
$\mu(\sigma)+\mu(\overline{\sigma})=w$ for every
$\sigma\in\Sigma_K$.

\begin{defn}[Algebraic Hecke character]\rm
Let $E\subset\C$ be a number field. An \emph{algebraic Hecke character}
of $K$ with values in $E$, conductor dividing $\mathfrak f$, and
infinity type $\mu\in I_K$ is a homomorphism
$\chi:\mathcal I(\mathfrak f)\longrightarrow E^\times$ such that
\[
  \chi((\lambda))
  =\lambda^\mu
  =\prod_{\sigma\in J_K}\sigma(\lambda)^{\mu(\sigma)}
\]
for every $\lambda\in K^\times$ satisfying
$\lambda\equiv1\pmod{\mathfrak f}$.
An algebraic Hecke character of finite order is called a
\emph{Dirichlet character}.
\end{defn}

Characters are identified under restriction along
$\mathcal I(\mathfrak f')\subset\mathcal I(\mathfrak f)$ whenever
$\mathfrak f\mid\mathfrak f'$, and the least modulus for which such
a description exists is the \emph{conductor} of $\chi$; see
\cite[Definition~4.1]{KS_2025}.

For $\operatorname{Re}(s)\gg0$, the Hecke $L$-function of $\chi$ is
defined by
\[
  L_{\mathfrak f}(\chi,s)
  :=\sum_{\substack{
    \mathfrak a\in\mathcal I(\mathfrak f)\\
    \mathfrak a\subset\mathcal O_K
  }}
  \frac{\chi(\mathfrak a)}{(N_{K/\Q}\mathfrak a)^s}.
\]
For an integral ideal $\mathfrak b$ prime to $\mathfrak f$, with ray
class
$[\mathfrak b]\in\mathcal I(\mathfrak f)/\mathcal P_{\mathfrak f}$,
the associated \emph{partial $L$-function} is
\[
  L_{\mathfrak f}(\chi,s,[\mathfrak b])
  :=\sum_{\substack{
    \mathfrak a\in[\mathfrak b]\\
    \mathfrak a\subset\mathcal O_K
  }}
  \frac{\chi(\mathfrak a)}{(N_{K/\Q}\mathfrak a)^s}.
\]
The ray classes partition the integral ideals prime to $\mathfrak f$,
so that
\[
  L_{\mathfrak f}(\chi,s)
  =\sum_{[\mathfrak b]\in\mathcal I(\mathfrak f)/\mathcal P_{\mathfrak f}}
  L_{\mathfrak f}(\chi,s,[\mathfrak b]).
\]
These functions admit meromorphic continuation to $\C$.
For $k\in\mathbb Z$, twisting by the norm character
$(\chi N_{K/\Q}^{-k})(\mathfrak a)
=\chi(\mathfrak a)(N_{K/\Q}\mathfrak a)^{-k}$ gives
\[
  L_{\mathfrak f}(\chi,k,[\mathfrak b])
  =L_{\mathfrak f}(\chi N_{K/\Q}^{-k},0,[\mathfrak b]),
\]
and the analogous identity holds for the full $L$-function.
Thus integral values may be formulated at $s=0$.
In the CM case, criticality is determined by the infinity type,
and the embeddings on which a critical infinity type is negative
form its attached CM type.

\begin{lem}[Ray-class parametrization]
\label{lem:Ray-class parametrization}\rm
Let $\mathfrak b\subset\mathcal O_K$ be an integral ideal prime to
$\mathfrak f$. Multiplication by $\mathfrak b$ induces a bijection
\[
  \mathcal O_{\mathfrak f}^{\times}
  \backslash(1+\mathfrak f\mathfrak b^{-1})
  \xrightarrow{\;\sim\;}
  \{\mathfrak a\in[\mathfrak b]\mid\mathfrak a\subset\mathcal O_K\},
  \qquad
  \lambda\longmapsto\lambda\mathfrak b,
\]
where $\lambda=0$ is omitted if it occurs.
Consequently, for $\operatorname{Re}(s)\gg0$,
\[
  L_{\mathfrak f}(\chi,s,[\mathfrak b])
  =\chi(\mathfrak b)(N_{K/\Q}\mathfrak b)^{-s}
  \sum_{\substack{
    \lambda\in\mathcal O_{\mathfrak f}^{\times}
    \backslash(1+\mathfrak f\mathfrak b^{-1})\\
    \lambda\neq0
  }}
  \frac{\chi((\lambda))}{|N_{K/\Q}(\lambda)|^s}.
\]
\end{lem}

\begin{proof}
For nonzero $\lambda\in1+\mathfrak f\mathfrak b^{-1}$, one has
$\lambda\mathfrak b\subset\mathfrak b+\mathfrak f\subset\mathcal O_K$.
Since $\mathfrak b$ is prime to $\mathfrak f$, the element $\lambda$
satisfies $\lambda\equiv1\pmod{\mathfrak f}$ locally at the primes
dividing $\mathfrak f$, so $\lambda\mathfrak b$ belongs to
$[\mathfrak b]$. Two such elements $\lambda,\lambda'$ determine the
same ideal precisely when $\lambda'=u\lambda$ for some
$u\in\mathcal O_K^\times$, and their ray congruences force
$u\in\mathcal O_{\mathfrak f}^{\times}$.
Conversely, an integral ideal $\mathfrak a\in[\mathfrak b]$ can be
written as $\mathfrak a=(\lambda)\mathfrak b$ with
$\lambda\equiv1\pmod{\mathfrak f}$ in the ray sense.
Integrality gives $\lambda\in\mathfrak b^{-1}$ and hence
$\lambda-1\in\mathfrak b^{-1}$; combining this with the ray
congruence and the coprimality of $\mathfrak b$ and $\mathfrak f$
yields $\lambda-1\in\mathfrak f\mathfrak b^{-1}$.
Reindexing the defining series by this bijection and substituting
$\chi(\lambda\mathfrak b)=\chi((\lambda))\chi(\mathfrak b)$ and
$N_{K/\Q}(\lambda\mathfrak b)
=|N_{K/\Q}(\lambda)|\,N_{K/\Q}\mathfrak b$
gives the formula.
\end{proof}

In the case $g=1$ arising in the dimension-two computation, $K$ is
imaginary quadratic. Thus $\mathcal O_K^\times$, and hence
$\mathcal O_{\mathfrak f}^{\times}$, is finite, and under the fixed
embedding $K\subset\C$ one has $N_{K/\Q}(\lambda)=|\lambda|^2$.

\subsection{CM periods and compatible Serre twists}
\label{subsection: CM periods and compatible Serre twists}

We fix the period normalizations needed for the arithmetic applications.
Following \cite[Section~4.2]{KS_2025} for the period formalism over a
fixed CM field, we recall the Betti structures and pairings defining
the CM periods of $\mathcal A_0$ and its dual. We then compare these
with the periods of $\mathcal A$ and $\mathcal A^\vee$ through the
Serre tensor construction.

Throughout, let $S=\Spec R$, where $R\subset\C$ is an
$\mathcal O_{K^{\mathrm{Gal}}}[1/d_K]$-algebra. We retain the CM-type
decompositions, pairings, and algebraic Hodge bases of
Section~\ref{subsection: Paring and decompositions}.

\begin{defn}\label{def: ideal structures}\rm
Let $\mathcal A_0/S$ be an abelian scheme with CM by $\mathcal O_K$,
and let $\mathfrak a$ be a fractional ideal of $K$.
An $\mathfrak a$-structure and a $K$-structure on $\mathcal A_0$
are, respectively, an isomorphism of $\mathcal O_K$-modules and an
isomorphism of $K$-vector spaces
\[
  \xi_{\mathfrak a}:\mathfrak a\xrightarrow{\sim}
  H_1(\mathcal A_0(\C),\Z),
  \qquad
  \xi_K:K\xrightarrow{\sim}H_1(\mathcal A_0(\C),\Q).
\]
\end{defn}

Tensoring $\xi_{\mathfrak a}$ with $\Q$ gives a $K$-structure,
again denoted $\xi_{\mathfrak a}$; the expression
$\xi_{\mathfrak a}(1)$ always refers to this rational extension.
For a fractional ideal $\mathfrak b$ of $K$, the Serre tensor
construction~\ref{example_A} gives an induced
$\mathfrak a\mathfrak b$-structure on
$\mathfrak b\otimes_{\mathcal O_K}\mathcal A_0$.
The same construction induces an $\mathfrak A$-structure on
$\mathcal A=\mathcal O_L\otimes_{\mathcal O_K}\mathcal A_0$,
where $\mathfrak A=\mathcal O_L\otimes_{\mathcal O_K}\mathfrak a$.

The compatibility between the Betti and de Rham pairings is expressed
by the commutative diagram
% https://q.uiver.app/#q=WzAsNSxbMCwwLCJIXzEoXFxtYXRoY2FsIEFfMChcXEMpLFxcWikgICAgIFxcdGltZXMgICAgIEhfMShcXG1hdGhjYWwgQV8wXlxcdmVlKFxcQyksXFxaKSJdLFsyLDBdLFswLDEsIlxcbWF0aGNhbCBIX3tcXG1hdGhjYWwgQV8wL1xcQ30gICAgIFxcdGltZXMgICAgIFxcbWF0aGNhbCBIX3tcXG1hdGhjYWwgQV8wXlxcdmVlL1xcQ30iXSxbMywwLCIoMlxccGkgaSlcXFoiXSxbMywxLCJcXEMiXSxbMCwyXSxbMCwzLCJcXGxhLSwtXFxyYV97XFxtYXRocm17dG9wfSxcXG1hdGhjYWwgQV8wfSJdLFsyLDQsIlxcbGEtLC1cXHJhX3tcXG1hdGhybXtkUn0sXFxtYXRoY2FsIEFfMH0iXSxbMyw0XV0=
\[\begin{tikzcd}
	{H_1(\mathcal A_0(\C),\Z)     \times     H_1(\mathcal A_0^\vee(\C),\Z)} && {} & {(2\pi i)\Z} \\
	{\mathcal H_{\mathcal A_0/\C}     \times     \mathcal H_{\mathcal A_0^\vee/\C}} &&& \C
	\arrow["{\la-,-\ra_{\mathrm{top},\mathcal A_0}}", from=1-1, to=1-4]
	\arrow[from=1-1, to=2-1]
	\arrow[from=1-4, to=2-4]
	\arrow["{\la-,-\ra_{\mathrm{dR},\mathcal A_0}}", from=2-1, to=2-4]
\end{tikzcd}\]
Here
$\mathcal H_{\mathcal A_0/\C}
=H^1_{\mathrm{dR}}(\mathcal A_0/\C)^\vee$
as in Definition~\ref{def_2}.
The pairing induces a $K$-structure $\xi_K^\vee$ on
$\mathcal A_0^\vee$ from $\xi_K$; for an $\mathfrak a$-structure,
trace duality identifies the integral Betti lattice of the dual
with $\mathfrak a^{-1}\mathfrak d_K^{-1}$.
By Proposition~\ref{prop: refined period pairings}, the refined period
pairings $\langle-,-\rangle_{\mathcal A_0}$ and
$\langle-,-\rangle_{\mathcal A_0^\vee}$ take values in
$\C^{\Sigma_K}$ and $\C^{\overline{\Sigma}_K}$, respectively,
while the refined de Rham pairing pairs
$\overline{\omega_{\mathcal A_0/\C}}$ with
$\omega_{\mathcal A_0^\vee/\C}$ in
$\C^{\overline{\Sigma}_K}$.
Composing with the trace recovers the ordinary pairings.
The analogous statements hold for $\mathcal A$ and the lifted
CM type $\Sigma_L$.

\begin{defn}[CM periods]\label{def: CM periods}\rm
Suppose that $\mathcal A_0$ is equipped with an
$\mathcal O_K$-structure $\xi_{\mathcal O_K}$, and denote by
$\xi_K^\vee$ the induced $K$-structure on $\mathcal A_0^\vee$.
For the algebraic bases $\omega(\mathcal A_0)$ and
$\omega(\mathcal A_0^\vee)$ of
Definition~\ref{def: algebraic bases}, define
\[
  \Omega_K
  :=\langle
      \omega(\mathcal A_0),\xi_{\mathcal O_K}(1)
    \rangle_{\mathcal A_0},
  \qquad
  \Omega_K^\vee
  :=\langle
      \omega(\mathcal A_0^\vee),\xi_K^\vee(1)
    \rangle_{\mathcal A_0^\vee},
\]
in $\C^{\Sigma_K}$ and $\C^{\overline{\Sigma}_K}$, respectively.
\end{defn}

\begin{prop}\label{prop: dual CM periods}
With the above normalizations, the dual CM period satisfies
\[
  \Omega_K^\vee
  =
  \frac{
    2\pi i\,
    \left\langle
      \overline{\omega(\mathcal A_0)},
      \omega(\mathcal A_0^\vee)
    \right\rangle_{\mathrm{dR},\mathcal A_0}
  }{
    \overline{\Omega}_K
  }
\]
componentwise in $\C^{\overline{\Sigma}_K}$.
\end{prop}

\begin{proof}
By Remark~\ref{rem_2} and Proposition~\ref{prop_5}, the
Betti--de Rham comparison and the Hodge decomposition give
\[
  H_1(\mathcal A_0(\C),\C)
  \cong
  \mathcal H_{\mathcal A_0/\C}
  \cong
  \mathrm{Lie}(\mathcal A_0/\C)
  \oplus
  \overline{\mathrm{Lie}(\mathcal A_0/\C)}.
\]
Under this identification, $\xi_{\mathcal O_K}(1)$ maps to
$\bigl(
  \Omega_K\omega(\mathcal A_0)^\vee,
  \overline{\Omega}_K\,\overline{\omega(\mathcal A_0)}^\vee
\bigr)$.
Dualizing the comparison and using the $2\pi i$-normalization
of the topological pairing gives
\[
  \left(
    \frac{\omega(\mathcal A_0)}{\Omega_K},
    \frac{\overline{\omega(\mathcal A_0)}}{\overline{\Omega}_K}
  \right)
  \longmapsto
  \frac{1}{2\pi i}\xi_K^\vee(1).
\]
By the CM-type decomposition, the holomorphic component
$\omega(\mathcal A_0)$ pairs trivially with
$\omega(\mathcal A_0^\vee)$.
Evaluating the dual comparison against
$\omega(\mathcal A_0^\vee)$ therefore yields
\[
  \frac{\Omega_K^\vee}{2\pi i}
  =
  \left\langle
    \frac{\overline{\omega(\mathcal A_0)}}{\overline{\Omega}_K},
    \omega(\mathcal A_0^\vee)
  \right\rangle_{\mathrm{dR},\mathcal A_0}.
\]
\end{proof}

The Serre tensor identification
\[
  H_1(\mathcal A(\C),\Z)
  \cong
  \mathcal O_L\otimes_{\mathcal O_K}
  H_1(\mathcal A_0(\C),\Z)
\]
induces an $\mathcal O_L$-structure $\xi_{\mathcal O_L}$ on
$\mathcal A$ from $1\otimes\xi_{\mathcal O_K}$.
Denote by $\xi_L^\vee$ the corresponding $L$-structure on
$\mathcal A^\vee$.
Using the induced algebraic Hodge bases of
Definition~\ref{def: algebraic bases}, we similarly define
\[
  \Omega_L
  :=\langle
      \omega(\mathcal A),\xi_{\mathcal O_L}(1)
    \rangle_{\mathcal A},
  \qquad
  \Omega_L^\vee
  :=\langle
      \omega(\mathcal A^\vee),\xi_L^\vee(1)
    \rangle_{\mathcal A^\vee},
\]
in $\C^{\Sigma_L}$ and $\C^{\overline{\Sigma}_L}$, respectively.

\begin{prop}
For the Betti structures and algebraic Hodge bases induced by the
Serre tensor construction, one has
\[
  \Omega_L=1\otimes_{\mathcal O_K}\Omega_K,
  \qquad
  \Omega_L^\vee=1\otimes_{\mathcal O_K}\Omega_K^\vee,
\]
under the identifications associated with the lifted CM type.
\end{prop}

\begin{proof}
Under the displayed Betti identification,
$\xi_{\mathcal O_L}(1)$ corresponds to
$1\otimes_{\mathcal O_K}\xi_{\mathcal O_K}(1)$,
and the algebraic Hodge basis $\omega(\mathcal A)$ is induced
from $\omega(\mathcal A_0)$.
The compatibility of the refined period pairings with the Serre
tensor construction in
Proposition~\ref{prop: refined period pairings} therefore gives
\[
  \left\langle
    \omega(\mathcal A),\xi_{\mathcal O_L}(1)
  \right\rangle_{\mathcal A}
  =
  1\otimes_{\mathcal O_K}
  \left\langle
    \omega(\mathcal A_0),\xi_{\mathcal O_K}(1)
  \right\rangle_{\mathcal A_0}.
\]
For the dual, use the rational structures $\xi_K^\vee$ and
$\xi_L^\vee$ induced by the Betti--de Rham pairings, together
with the dual Hodge basis $\omega(\mathcal A^\vee)$ induced
from $\omega(\mathcal A_0^\vee)$.
The corresponding compatibility of the dual refined period
pairings gives
\[
  \left\langle
    \omega(\mathcal A^\vee),\xi_L^\vee(1)
  \right\rangle_{\mathcal A^\vee}
  =
  1\otimes_{\mathcal O_K}
  \left\langle
    \omega(\mathcal A_0^\vee),\xi_K^\vee(1)
  \right\rangle_{\mathcal A_0^\vee}.
\]
\end{proof}

Fix the reference data
$(\mathcal A_0/S,\omega(\mathcal A_0),
\omega(\mathcal A_0^\vee),\xi_{\mathcal O_K})$
as above. For a fractional ideal $\mathfrak a$ of $K$, set
$\mathcal A_0(\mathfrak a)
:=\mathfrak a\otimes_{\mathcal O_K}\mathcal A_0$
and equip it with the induced $\mathfrak a$-structure
$\xi_{\mathfrak a}$. Under the Serre tensor identification
$\mathrm{Lie}(\mathcal A_0(\mathfrak a)/R)
\cong\mathfrak a\otimes_{\mathcal O_K}
\mathrm{Lie}(\mathcal A_0/R)$,
the basis $\omega(\mathcal A_0)$, viewed as an isomorphism
via \eqref{eq_1}, induces
\[
  \mathfrak a\otimes\omega(\mathcal A_0):
  \mathrm{Lie}(\mathcal A_0(\mathfrak a)/R)
  \xrightarrow{\sim}
  (\mathfrak a\otimes_{\Z}R)(\Sigma_K).
\]
After base change to $\C$, this identifies
$\mathrm{Lie}(\mathcal A_0(\mathfrak a)/\C)$ with
$\C^{\Sigma_K}$. 

\begin{prop}[Compatible Serre twists]\label{prop_4}
With the above notation, one has
\[
  \left\langle
    \mathfrak a\otimes\omega(\mathcal A_0),\xi_{\mathfrak a}(1)
  \right\rangle_{\mathcal A_0(\mathfrak a)}
  =
  \left\langle
    \omega(\mathcal A_0),\xi_{\mathcal O_K}(1)
  \right\rangle_{\mathcal A_0}
  =\Omega_K.
\]
The image of $H_1(\mathcal A_0(\mathfrak a)(\C),\Z)$ in
$\C^{\Sigma_K}$ under the map induced by
$\mathfrak a\otimes\omega(\mathcal A_0)$ is
$\mathfrak a\Omega_K$. Consequently,
\[
  \mathcal A_0(\mathfrak a)(\C)
  \cong\C^{\Sigma_K}/\mathfrak a\Omega_K.
\]

Suppose further that $\mathfrak a=\mathfrak f\mathfrak b^{-1}$
for integral ideals $\mathfrak f,\mathfrak b$, and that the isogenies
\[
  [\mathfrak f]:
  \mathcal A_0(\mathfrak a)
  \longrightarrow\mathcal A_0(\mathfrak b^{-1}),
  \qquad
  [\mathfrak b]:
  \mathcal A_0
  \longrightarrow\mathcal A_0(\mathfrak b^{-1})
\]
are étale. Then
\[
  \omega(\mathcal A_0)\otimes\mathfrak a
  =
  [\mathfrak f]^*([\mathfrak b]^*)^{-1}\omega(\mathcal A_0)
\]
is a basis of $\omega_{\mathcal A_0(\mathfrak a)/R}$.

The same assertions hold for $\mathcal A$, with
$\Sigma_K,\Omega_K$ replaced by $\Sigma_L,\Omega_L$ and
$\mathfrak a,\mathfrak b,\mathfrak f$ replaced, respectively, by
\[
  \mathfrak A:=\mathcal O_L\otimes_{\mathcal O_K}\mathfrak a,
  \qquad
  \mathfrak B:=\mathcal O_L\otimes_{\mathcal O_K}\mathfrak b,
  \qquad
  \mathfrak F:=\mathcal O_L\otimes_{\mathcal O_K}\mathfrak f.
\]
In particular, the corresponding period lattice is
$\mathfrak A\Omega_L$, and under the étaleness hypotheses
the compatible algebraic basis is
\[
  \omega(\mathcal A)\otimes\mathfrak A
  =
  [\mathfrak F]^*([\mathfrak B]^*)^{-1}\omega(\mathcal A).
\]
\end{prop}

\begin{proof}
Functoriality of the Serre construction gives
\[
  H_1(\mathcal A_0(\mathfrak a)(\C),\Q)
  \cong
  \mathfrak a\otimes_{\mathcal O_K}
  H_1(\mathcal A_0(\C),\Q),
\]
compatibly with the maps to the Lie algebras. In the coordinates
$\mathfrak a\otimes\omega(\mathcal A_0)$, the period map is
therefore induced from the corresponding map for $\mathcal A_0$.
By the definition of the induced rational structure,
$\xi_{\mathfrak a}(1)$ has the same image as
$\xi_{\mathcal O_K}(1)$ in the reference coordinates, namely
$\Omega_K$. With integral coefficients, the same identification
tensors the reference period lattice $\mathcal O_K\Omega_K$
with $\mathfrak a$, giving precisely $\mathfrak a\Omega_K$.
The asserted complex uniformization follows.

Suppose now that $\mathfrak a=\mathfrak f\mathfrak b^{-1}$.
The uniformization also shows that the class of $\Omega_K$
is an $[\mathfrak f]$-torsion point, since
$\Omega_K\in\mathfrak b^{-1}\Omega_K$.
The étaleness assumptions give pullback isomorphisms
\[
  \omega_{\mathcal A_0(\mathfrak a)/R}
  \xleftarrow[\sim]{[\mathfrak f]^*}
  \omega_{\mathcal A_0(\mathfrak b^{-1})/R}
  \xrightarrow[\sim]{[\mathfrak b]^*}
  \omega_{\mathcal A_0/R}.
\]
Thus $[\mathfrak f]^*([\mathfrak b]^*)^{-1}\omega(\mathcal A_0)$
is an algebraic basis. On the complex uniformizations, the two
isogenies are induced by the lattice inclusions
$\mathfrak a\Omega_K\subset\mathfrak b^{-1}\Omega_K$ and
$\mathcal O_K\Omega_K\subset\mathfrak b^{-1}\Omega_K$.
Their tangent maps are therefore the identity in the preceding
coordinates, so the pullback basis agrees with
$\omega(\mathcal A_0)\otimes\mathfrak a$ after base change to
$\C$, and hence over $R\subset\C$.

For $\mathcal A$, the Betti and Lie algebra identifications of
Construction~\ref{example_A}, together with
$\Omega_L=1\otimes_{\mathcal O_K}\Omega_K$, give the corresponding
period and lattice assertions by the same argument.
Locally on $S$, the morphisms $[\mathfrak F]$ and
$[\mathfrak B]$ are identified with finite products of
$[\mathfrak f]$ and $[\mathfrak b]$, respectively.
They are therefore étale, and their pullbacks on invariant
differentials give the stated basis comparison.
\end{proof}

\subsection{Algebraicity of critical \texorpdfstring{$L$}{L}-values}

We now combine the algebraic boundary restriction with the period
normalizations fixed above. Specializing the two rank-one contributions
at suitable torsion sections relates their constant terms to regularized
critical partial Hecke $L$-values. Summing over ray classes then gives the
corresponding algebraicity statements.

We retain the arithmetic setting of the preceding subsections:
$K$ is imaginary quadratic, $[L:K]=2$, and
$\mathcal A=\mathcal O_L\otimes_{\mathcal O_K}\mathcal A_0$ over
$S=\operatorname{Spec}R$, with the fixed $\mathcal O_K$-basis
$\{1,\alpha\}$ of $\mathcal O_L$. Let
$\mathfrak f,\mathfrak b,\mathfrak c$ be pairwise coprime nonzero integral
ideals of $\mathcal O_K$, with
$\mathfrak f,\mathfrak c\neq\mathcal O_K$, and put
$\mathfrak a=\mathfrak f\mathfrak b^{-1}$. Their extensions to
$\mathcal O_L$ are denoted by
\[
  \mathfrak F=\mathfrak f\mathcal O_L,\qquad
  \mathfrak B=\mathfrak b\mathcal O_L,\qquad
  \mathfrak C=\mathfrak c\mathcal O_L,\qquad
  \mathfrak A=\mathfrak a\mathcal O_L
  =\mathfrak F\mathfrak B^{-1}.
\]
In particular,
$N_{L/\Q}(\mathfrak C)=N_{K/\Q}(\mathfrak c)^2$.
We also write
\[
  \mathcal O_{\mathfrak f}^{\times}
  :=
  \{k\in\mathcal O_K^\times\mid k\equiv1\pmod{\mathfrak f}\}
\]
for the group of units satisfying the ray congruence.

For the coefficient rings occurring in the arithmetic specialization,
we use the notation
\[
  R\left[\frac{1}{\mathfrak f}\right]
  :=
  \bigcap_{\substack{
    \mathfrak m\subset R\ \mathrm{maximal}\\
    \mathfrak m\cap\mathcal O_K\nmid\mathfrak f
  }}
  R_{\mathfrak m}.
\]
After passing to
$\operatorname{Spec}R[1/\mathfrak f]$, the relevant
$\mathfrak f$- and $\mathfrak F$-isogenies become étale, since the
primes lying above the prime divisors of $\mathfrak f$ have been
removed. We keep this localization explicit rather than absorbing it
into the notation $R$.

Fix $b\in\mathbb N$. Let $\chi_A$ and $\chi_B$ be algebraic Hecke
characters of $K$, with values in $E$ and conductor dividing
$\mathfrak f$, satisfying
\[
  \chi_A((\lambda))
  =
  \frac{\overline{\lambda}^{\,b}}{\lambda^2},
  \qquad
  \chi_B((\lambda))
  =
  \frac{\overline{\lambda}^{\,b+1}}{\lambda}
\]
for every $\lambda\equiv1\pmod{\mathfrak f}$. In the conventions fixed
above, their infinity types are
$b\overline{\Sigma}_K-2\Sigma_K$ and
$(b+1)\overline{\Sigma}_K-\Sigma_K$, respectively. Thus $s=0$ is
critical for both characters, with attached CM type $\Sigma_K$.

Under the compatible uniformization
$\mathcal A_0(\mathfrak f)(\C)
\cong\C^{\Sigma_K}/\mathfrak f\Omega_K$
of the preceding subsection, let $x_{\Omega_K}$ denote the
$\mathfrak f$-torsion point represented by $\Omega_K$. After a finite
extension of $R$, if necessary, we assume that it is defined over $R$.
Since $\mathfrak b$ is integral, the inclusion
$\mathfrak f\subset\mathfrak a$ induces an isogeny
$[\mathfrak b]:\mathcal A_0(\mathfrak f)\to
\mathcal A_0(\mathfrak a)$ under which this point is still represented
by $\Omega_K$; we use the same notation for its image. The corresponding
isogeny
$[\mathfrak B]:\mathcal A(\mathfrak F)\to\mathcal A(\mathfrak A)$
likewise preserves the representatives of the two
$\mathfrak F$-torsion sections
\[
  x_{\Omega_L}^{A}
  :=
  1\otimes_{\mathcal O_K}x_{\Omega_K},
  \qquad
  x_{\Omega_L}^{B}
  :=
  \alpha\otimes_{\mathcal O_K}x_{\Omega_K}.
\]
We retain these symbols for their images on
$\mathcal A(\mathfrak A)$. With respect to the fixed basis
$\{1,\alpha\}$, their representatives are
\[
  x_{\Omega_L}^{A}=(\Omega_K,0),
  \qquad
  x_{\Omega_L}^{B}=(0,\Omega_K).
\]

With the convention
$\Omega_L=1\otimes_{\mathcal O_K}\Omega_K$, the period lattice at the
standard cusp is
\[
  \Lambda
  =
  \mathfrak A\Omega_L
  =
  \mathfrak a\Omega_K
  \oplus
  \mathfrak a\Omega_K\alpha.
\]
The auxiliary torsion subscheme
$\mathcal D=\mathcal A(\mathfrak A)[\mathfrak C]$
therefore has the coordinate realization
\[
  \mathcal D(\C)
  =
  \mathfrak C^{-1}\Lambda/\Lambda
  \cong
  \frac{\mathfrak c^{-1}\mathfrak a\Omega_K}
       {\mathfrak a\Omega_K}
  \oplus
  \frac{\mathfrak c^{-1}\mathfrak a\Omega_K}
       {\mathfrak a\Omega_K}\alpha.
\]
Accordingly, we write its points as
$t=(t_1\Omega_K,t_2\Omega_K)$, with
$t_1,t_2\in\mathfrak c^{-1}\mathfrak a/\mathfrak a$.
Since $\mathfrak f\neq\mathcal O_K$ and
$(\mathfrak f,\mathfrak b)=1$, both chosen torsion sections are
nonzero. Coprimality of $\mathfrak f$ and $\mathfrak c$ then shows that
they lie outside $\mathcal D$, and hence factor through
$\mathcal U_{\mathcal D}$.

With respect to the same basis, take the congruence group
\[
  \Gamma(\mathfrak f)
  =
  \ker\!\left(
    \GL_2(\mathcal O_K)
    \longrightarrow
    \GL_2(\mathcal O_K/\mathfrak f)
  \right).
\]
Its action on $\mathcal O_L/\mathfrak F$ is trivial, so every
$\mathfrak F$-torsion section, in particular
$x_{\Omega_L}^{A}$ and $x_{\Omega_L}^{B}$, is fixed by
$\Gamma(\mathfrak f)$. Moreover,
$\mathfrak C=\mathfrak c\mathcal O_L$ is preserved by every
$\mathcal O_K$-linear automorphism of $\mathcal O_L$; thus
$[\mathfrak C]$ and $\mathcal D$ are compatible with this action.
At the standard cusp, the determinant-one part of the stabilizer
satisfies
\[
  \frac{
    \Gamma_\infty(\mathfrak f)\cap\mathrm{SL}_2(\mathcal O_K)
  }{
    \Gamma_{N,\infty}(\mathfrak f)
  }
  \cong
  \mathcal O_{\mathfrak f}^{\times}.
\]
The class corresponding to $k\in\mathcal O_{\mathfrak f}^{\times}$
is represented by
$\left(\begin{smallmatrix}k&0\\0&k^{-1}\end{smallmatrix}\right)$
and acts on the torsion coordinates by
$k\cdot(t_1,t_2)=(kt_1,k^{-1}t_2)$.
The group $\mathcal O_{\mathfrak f}^{\times}$ is finite, since $K$ is
imaginary quadratic.

On the full torsion subscheme $\mathcal D$, we use the distinguished
function
\[
  f_{L,[\mathfrak C]}
  =
  N_{L/\Q}(\mathfrak C)\,1_{e(R)}-1_{\mathcal D}
  \in
  R\left[\frac{1}{N_{L/\Q}(\mathfrak C)}\right]
  [\mathcal D]^{0,\Gamma(\mathfrak f)}.
\]
Its trace vanishes because
$\#\mathcal D(\C)=N_{L/\Q}(\mathfrak C)$, and its invariance follows
from the stability of $\mathcal D$ and the zero section.
For $t_1,t_2\in\mathfrak c^{-1}\mathfrak a/\mathfrak a$, the
corresponding one-variable functions are
\begin{align*}
  f_{L,[\mathfrak C]}^{A}(t_1)
  &:=
  f_{L,[\mathfrak C]}\bigl(-(t_1\Omega_K,0)\bigr)
  =
  N_{K/\Q}(\mathfrak c)^2\,1_0(t_1)-1,
  \\
  f_{L,[\mathfrak C]}^{B}(t_2)
  &:=
  \sum_{t_1\in\mathfrak c^{-1}\mathfrak a/\mathfrak a}
  f_{L,[\mathfrak C]}
  \bigl(-(t_1\Omega_K,t_2\Omega_K)\bigr)
  =
  N_{K/\Q}(\mathfrak c)^2\,1_0(t_2)-N_{K/\Q}(\mathfrak c).
\end{align*}
Thus $f_{L,[\mathfrak C]}^{A}$ is obtained by restriction to the first
coordinate, whereas $f_{L,[\mathfrak C]}^{B}$ is obtained by summing
over that coordinate. These are auxiliary functions for the two
projections; the trace-zero function defining the
Eisenstein--Kronecker class remains the single function
$f_{L,[\mathfrak C]}$ on $\mathcal D$.

The same distribution relation applies to the two resulting
one-variable sums. It is obtained by grouping the torsion parameters
and lattice variables into $\mathcal O_{\mathfrak f}^{\times}$-orbits.

\begin{lem}\label{lem: distribution relation}
Let $\mathfrak f,\mathfrak b,\mathfrak c$ be pairwise coprime nonzero
integral ideals of $\mathcal O_K$, with
$\mathfrak c\neq\mathcal O_K$, and put
$\mathfrak a=\mathfrak f\mathfrak b^{-1}$.
Multiplication by $\mathcal O_{\mathfrak f}^{\times}$ defines an action
on $\mathfrak c^{-1}\mathfrak a/\mathfrak a$ and fixes the class of $1$
in $\mathfrak b^{-1}/\mathfrak a$.
Let $G:K^\times\to\C$ satisfy $G(kv)=G(v)$ for every
$k\in\mathcal O_{\mathfrak f}^{\times}$ and $v\in K^\times$.
Whenever the following sums are absolutely convergent, one has
\begin{align*}
  &\sum_{
    t\in\mathcal O_{\mathfrak f}^{\times}
    \backslash(\mathfrak c^{-1}\mathfrak a/\mathfrak a)
  }
  \;
  \sum_{
    v\in\mathcal O_{\mathfrak f}^{\times}
    \backslash
    \bigl(
      (1+\mathfrak a+\mathcal O_{\mathfrak f}^{\times}t)
      \setminus\{0\}
    \bigr)
  }
  G(v)
=
  \sum_{
    v\in\mathcal O_{\mathfrak f}^{\times}
    \backslash
    \bigl(
      (1+\mathfrak c^{-1}\mathfrak a)\setminus\{0\}
    \bigr)
  }
  G(v).
\end{align*}
\end{lem}

\begin{proof}
Since $\mathfrak b$ is integral,
$\mathfrak f\subset\mathfrak a$, and therefore
$k-1\in\mathfrak a$ for every
$k\in\mathcal O_{\mathfrak f}^{\times}$.
Multiplication by units preserves the fractional ideals involved,
so the indicated actions are well defined and the class of $1$ is
fixed. Grouping the cosets of $\mathfrak a$ according to their unit
orbits gives the disjoint decomposition
\[
  1+\mathfrak c^{-1}\mathfrak a
  =
  \coprod_{
    \mathcal O_{\mathfrak f}^{\times}t
    \in
    \mathcal O_{\mathfrak f}^{\times}
    \backslash(\mathfrak c^{-1}\mathfrak a/\mathfrak a)
  }
  \left(
    1+\mathfrak a+\mathcal O_{\mathfrak f}^{\times}t
  \right).
\]
Indeed, two pieces can intersect only if their torsion parameters
belong to the same unit orbit modulo $\mathfrak a$.
Each piece is stable under $\mathcal O_{\mathfrak f}^{\times}$ because
$k\equiv1\pmod{\mathfrak a}$.
After removing zero, we may therefore pass to unit orbits;
the invariance of $G$ and absolute convergence allow its sum to be
split over the resulting disjoint union.
\end{proof}

For the two functions
\[
  G_A(v)
  :=
  \frac{\overline v^{\,b+2}}{|v|^{2s}},
  \qquad
  G_B(v)
  :=
  N_{K/\Q}((v)\mathfrak b)
  \frac{\overline v^{\,b+2}}{|v|^{2s}},
\]
the invariance hypothesis of
Lemma~\ref{lem: distribution relation} follows from the defining
condition for $\chi_A$. Namely, for
$k\in\mathcal O_{\mathfrak f}^{\times}$, the identities
$\chi_A((k))=1$ and $k\overline k=1$ give
$\overline k^{\,b+2}=1$, hence $G_A(kv)=G_A(v)$.
Since $(kv)=(v)$ for units $k$, the same holds for $G_B$.
For these functions, the distribution identities are first established
in the region of absolute convergence, and their values at $s=2$ are
understood by meromorphic continuation.

We first consider the contribution of $\mathrm{pr}_A$ at the standard
cusp $\infty$, taking $x=x_{\Omega_L}^{A}$. The first rank-one summand
of the constant-term formula is supported on the locus $v_2=0$,
which reduces the calculation to the first torsion coordinate.

\begin{thm}[The $\mathrm{pr}_A$-contribution]
\label{thm: Special values of partial raw functions}
For
$f\in R[1/N_{K/\Q}(\mathfrak c)]
[\mathcal D]^{0,\Gamma(\mathfrak f)}$,
the $\mathrm{pr}_A$-component of the zeroth Fourier coefficient of
$\mathrm{res}_{\infty,1}
(\psi_u^b(f,x_{\Omega_L}^{A}))(s)$ is, for
$\operatorname{Re}(s)\gg0$, given by
\begin{equation}
\label{eq: prA general constant term}
\begin{aligned}
&\mathrm{pr}_A\left(
  \left[
    \left.
      \mathrm{res}_{\infty,1}
      \bigl(\psi_u^b(f,x_{\Omega_L}^{A})\bigr)(s)
    \right|_0
  \right]
\right)=\frac{
  (2\pi i)^b\overline{\Omega_K}^{\,2}
  |\mathcal O_{\mathfrak f}^{\times}|
}{
  (\Omega_K^\vee)^b|\Omega_K|^{2s}
}
\sum_{
  t_1\in\mathcal O_{\mathfrak f}^{\times}
  \backslash(\mathfrak c^{-1}\mathfrak a/\mathfrak a)
}
f\bigl(-(t_1\Omega_K,0)\bigr)
\sum_{
  v_1\in\mathcal O_{\mathfrak f}^{\times}
  \backslash
  \bigl(
    (1+\mathfrak a+\mathcal O_{\mathfrak f}^{\times}t_1)
    \setminus\{0\}
  \bigr)
}
\frac{\overline{v_1}^{\,b+2}}{|v_1|^{2s}}.
\end{aligned}
\end{equation}
This expression admits meromorphic continuation in $s$ and is
holomorphic at $s=2$. At this value, its prefactor simplifies to
$(2\pi i)^b|\mathcal O_{\mathfrak f}^{\times}|
/((\Omega_K^\vee)^b\Omega_K^2)$, and its value is the projected
algebraic boundary class
\begin{equation}
\label{eq: prA algebraic specialization}
\begin{aligned}
&\mathrm{pr}_A\left(
  \left[
    \left.
      \mathrm{res}_{\infty,1}
      \bigl(\psi_u^b(f,x_{\Omega_L}^{A})\bigr)(2)
    \right|_0
  \right]
\right)
\in
R\left[\frac{1}{\mathfrak f}\right]
\left[
  \frac{1}{
    |\mathcal O_{\mathfrak f}^{\times}|\,b!\,
    N_{K/\Q}(\mathfrak b\mathfrak c)
  }
\right].
\end{aligned}
\end{equation}
\end{thm}

\begin{proof}
By Theorem~\ref{thm:algebraic constant term at infinity}, after
replacing $t$ by $-t$, the first scalar coefficient is
\[
\frac{1}{
  \left\langle
    \overline{\omega(\mathcal A)},
    \omega(\mathcal A^\vee)
  \right\rangle_{dR,\mathcal A}^{\,b}
}
\sum_{t\in\mathfrak C^{-1}\Lambda/\Lambda}
f(-t)
\sum_{\substack{
  v\in(x_{\Omega_L}^{A}+t+\Lambda)\setminus\{0\}\\
  v_2=0
}}
\frac{\overline{v_1}^{\,b+2}}{|v_1|^{2s}}.
\]
Recall that
$\Lambda=\mathfrak a\Omega_K\oplus
\mathfrak a\Omega_K\alpha$ and that
$x_{\Omega_L}^{A}$ is represented by $(\Omega_K,0)$.
Writing $t=(t_1\Omega_K,t_2\Omega_K)$, with
$t_1,t_2\in\mathfrak c^{-1}\mathfrak a/\mathfrak a$, the condition
$v_2=0$ forces $t_2=0$ in this quotient.
For these torsion parameters, the surviving first coordinate runs
through
$((1+t_1)\Omega_K+\mathfrak a\Omega_K)\setminus\{0\}$.
The compatible Serre-tensor Hodge bases identify the pairing factor
with
$\langle\overline{\omega(\mathcal A_0)},
\omega(\mathcal A_0^\vee)\rangle_{dR,\mathcal A_0}^{-b}$.

The diagonal action of
$k\in\mathcal O_{\mathfrak f}^{\times}$ preserves the locus $t_2=0$
and sends $(t_1,v_1)$ to $(kt_1,kv_1)$.
Since $k-1\in\mathfrak f\subset\mathfrak a$, this action is compatible
with the corresponding first-coordinate cosets. It is free on pairs
with $v_1\neq0$, and both $f$ and the summand are invariant, by the
$\Gamma(\mathfrak f)$-invariance of $f$ and the observation following
Lemma~\ref{lem: distribution relation}.
Passing to unit orbits therefore introduces the factor
$|\mathcal O_{\mathfrak f}^{\times}|$.
Proposition~\ref{prop: dual CM periods} gives
\[
\Omega_K^\vee
=
\frac{
  2\pi i
  \left\langle
    \overline{\omega(\mathcal A_0)},
    \omega(\mathcal A_0^\vee)
  \right\rangle_{dR,\mathcal A_0}
}{
  \overline{\Omega_K}
}.
\]
Factoring $\Omega_K$ out of the surviving lattice coordinate now
gives \eqref{eq: prA general constant term}.
Its meromorphic continuation and holomorphy at $s=2$ follow from
Theorem~\ref{thm:algebraic constant term at infinity};
the stated specialization of the prefactor uses
$\overline{\Omega_K}^{\,2}/|\Omega_K|^4=1/\Omega_K^2$.

The same constant-term theorem identifies this value with the
complex realization of the algebraic boundary restriction.
Proposition~\ref{prop_9} defines its first rank-one projection after
inverting $b!$, while passage through the finite unit quotient is
defined after inverting $|\mathcal O_{\mathfrak f}^{\times}|$.
Propositions~\ref{prop_10} and
\ref{prop: F-torsion restriction} give compatibility with the
isogenies and restriction to the cusp stabilizer.
The remaining localization
$R[1/\mathfrak f][1/N_{K/\Q}(\mathfrak b\mathfrak c)]$
is precisely the one required for the isogenies, compatible Serre
twists, and torsion data entering this specialization.
\end{proof}

\begin{cor}
\label{cor: prA partial L values}
Retain the arithmetic data and the character $\chi_A$ fixed above,
allowing also $\mathfrak f=\mathcal O_K$ in the second assertion.
The following partial-value formulas hold.
\begin{enumerate}[label=(\arabic*),font=\normalfont]
\item
If $\mathfrak f\neq\mathcal O_K$, then
$x_{\Omega_L}^{A}\neq e$, and specialization to
$f_{L,[\mathfrak C]}$ on
$\mathcal D=\mathcal A(\mathfrak A)[\mathfrak C]$ gives
\begin{equation}
\label{eq: prA partial L value}
\begin{aligned}
&\frac{(2\pi i)^b}{(\Omega_K^\vee)^b\Omega_K^2}
\Bigl(
  N_{K/\Q}(\mathfrak c)^2\chi_A(\mathfrak c)
  L_{\mathfrak f}(\chi_A,0,[\mathfrak b])
  -
  L_{\mathfrak f}(\chi_A,0,[\mathfrak b\mathfrak c])
\Bigr)
\in
\mathcal O_E R\left[\frac{1}{\mathfrak f}\right]
\left[
  \frac{1}{
    |\mathcal O_{\mathfrak f}^{\times}|\,b!\,
    N_{K/\Q}(\mathfrak b\mathfrak c)
  }
\right].
\end{aligned}
\end{equation}

\item
If $\mathfrak f=\mathcal O_K$, then $x_{\Omega_L}^{A}=e$.
Let $\mathfrak c'\subset\mathcal O_K$ be a nontrivial integral ideal
such that $\mathfrak b,\mathfrak c,\mathfrak c'$ are pairwise
coprime, and put $\mathfrak C'=\mathfrak c'\mathcal O_L$.
Then
\begin{equation}
\label{eq: prA trivial conductor partial L value}
\begin{aligned}
&\frac{(2\pi i)^b}{(\Omega_K^\vee)^b\Omega_K^2}
\Bigl(
  N_{K/\Q}(\mathfrak c)^2\chi_A(\mathfrak c)
  L_{\mathcal O_K}(\chi_A,0,[\mathfrak b\mathfrak c'])
  -
  N_{K/\Q}(\mathfrak c)^2
  \chi_A(\mathfrak c)\chi_A(\mathfrak c')
  L_{\mathcal O_K}(\chi_A,0,[\mathfrak b])
\\
&\qquad\qquad\qquad\qquad\qquad
  +
  \chi_A(\mathfrak c')
  L_{\mathcal O_K}(\chi_A,0,[\mathfrak b\mathfrak c])
  -
  L_{\mathcal O_K}
  (\chi_A,0,[\mathfrak b\mathfrak c\mathfrak c'])
\Bigr)
\in
\mathcal O_E R
\left[
  \frac{1}{
    |\mathcal O_K^\times|\,b!\,
    N_{K/\Q}(\mathfrak b\mathfrak c\mathfrak c')
  }
\right].
\end{aligned}
\end{equation}
\end{enumerate}
\end{cor}

\begin{proof}
Suppose first that $\mathfrak f\neq\mathcal O_K$ and take
$f=f_{L,[\mathfrak C]}$.
Since
$f_{L,[\mathfrak C]}^{A}(t_1)
=N_{K/\Q}(\mathfrak c)^2\,1_0(t_1)-1$,
Lemma~\ref{lem: distribution relation} replaces the double sum in
\eqref{eq: prA general constant term} by
$N_{K/\Q}(\mathfrak c)^2$ times the sum over
$\mathcal O_{\mathfrak f}^{\times}
\backslash((1+\mathfrak a)\setminus\{0\})$
minus the sum over
$\mathcal O_{\mathfrak f}^{\times}
\backslash((1+\mathfrak c^{-1}\mathfrak a)\setminus\{0\})$.
Using the infinity type of $\chi_A$ and
$N_{K/\Q}(v)=|v|^2$,
Lemma~\ref{lem:Ray-class parametrization} identifies these sums as
\begin{align*}
\sum_{
  v\in\mathcal O_{\mathfrak f}^{\times}
  \backslash((1+\mathfrak a)\setminus\{0\})
}
\frac{\overline v^{\,b+2}}{|v|^{2s}}
&=
N_{K/\Q}(\mathfrak b)^{s-2}
\chi_A(\mathfrak b)^{-1}
L_{\mathfrak f}(\chi_A,s-2,[\mathfrak b]),
\\
\sum_{
  v\in\mathcal O_{\mathfrak f}^{\times}
  \backslash
  ((1+\mathfrak c^{-1}\mathfrak a)\setminus\{0\})
}
\frac{\overline v^{\,b+2}}{|v|^{2s}}
&=
N_{K/\Q}(\mathfrak b\mathfrak c)^{s-2}
\chi_A(\mathfrak b\mathfrak c)^{-1}
L_{\mathfrak f}
(\chi_A,s-2,[\mathfrak b\mathfrak c]).
\end{align*}
These identities are initially valid in the region of absolute
convergence and then hold by meromorphic continuation.
Evaluate at $s=2$ and apply
Theorem~\ref{thm: Special values of partial raw functions}.
Multiplying the resulting value by
$\chi_A(\mathfrak b\mathfrak c)/
|\mathcal O_{\mathfrak f}^{\times}|$
gives the expression in \eqref{eq: prA partial L value}.
A ray-principal power of $\mathfrak b\mathfrak c$, together with
the defining infinity type, shows that
$\chi_A(\mathfrak b\mathfrak c)
\in\mathcal O_E[1/N_{K/\Q}(\mathfrak b\mathfrak c)]$;
the unit-group order is already inverted in the stated localization.

Now suppose that $\mathfrak f=\mathcal O_K$.
The chosen section is $e$, so the distinguished function on the
full $\mathfrak C$-torsion cannot be used directly for specialization
at this section. With $\mathfrak c'$ as in the statement, coprimality
gives the Chinese remainder identification
\[
\mathcal A(\mathfrak A)[\mathfrak C\mathfrak C']
\cong
\mathcal A(\mathfrak A)[\mathfrak C']
\times
\mathcal A(\mathfrak A)[\mathfrak C].
\]
On
$\mathcal D=
\mathcal A(\mathfrak A)[\mathfrak C\mathfrak C']
\setminus e(S)$,
define $\widetilde f_{L,[\mathfrak C]}$ by pulling back
$f_{L,[\mathfrak C]}$ along the second projection wherever the first
factor is nonzero, and by setting it equal to zero wherever the
first factor is zero.
This construction is $\Gamma(\mathfrak f)$-invariant.
Its trace vanishes because its restriction over each nonzero
$\mathfrak C'$-torsion point has the same trace as
$f_{L,[\mathfrak C]}$, while it vanishes over the zero
$\mathfrak C'$-torsion point. Thus
\[
\widetilde f_{L,[\mathfrak C]}
\in
R\left[
  \frac{1}{N_{K/\Q}(\mathfrak c\mathfrak c')}
\right]
[\mathcal D]^{0,\Gamma(\mathfrak f)},
\]
and $e$ now factors through $\mathcal U_{\mathcal D}$.

Here $\mathfrak a=\mathfrak b^{-1}$ contains $1$ and
$\mathcal O_{\mathfrak f}^{\times}=\mathcal O_K^\times$.
On the first-coordinate locus, excluding the zero
$\mathfrak C'$-torsion component subtracts the original lattice
sum from the $\mathfrak c'$-enlarged one.
Applying Lemma~\ref{lem: distribution relation} successively to the
$\mathfrak c'$- and $\mathfrak c$-torsion coordinates therefore
gives four lattice sums, corresponding respectively to
$[\mathfrak b\mathfrak c']$, $[\mathfrak b]$,
$[\mathfrak b\mathfrak c]$, and
$[\mathfrak b\mathfrak c\mathfrak c']$, with coefficients
$N_{K/\Q}(\mathfrak c)^2$,
$-N_{K/\Q}(\mathfrak c)^2$, $1$, and $-1$.
By Lemma~\ref{lem:Ray-class parametrization}, their values at
$s=2$ are the corresponding partial $L$-values multiplied by the
inverse character values of these ideals.
The calculation and algebraic-specialization argument of
Theorem~\ref{thm: Special values of partial raw functions} apply to
$\widetilde f_{L,[\mathfrak C]}$ on this torsion subscheme, with
auxiliary torsion ideal $\mathfrak c\mathfrak c'$ and $x=e$.
Multiplying the resulting value by
$\chi_A(\mathfrak b\mathfrak c\mathfrak c')/
|\mathcal O_K^\times|$
gives \eqref{eq: prA trivial conductor partial L value}.
The same ray-principal-power argument gives
$\chi_A(\mathfrak b\mathfrak c\mathfrak c')
\in
\mathcal O_E[1/N_{K/\Q}(\mathfrak b\mathfrak c\mathfrak c')]$,
so this multiplication requires no further localization.
\end{proof}

The second rank-one summand is obtained by taking
$x=x_{\Omega_L}^{B}$, represented by $(0,\Omega_K)$.
Unlike the first contribution, the fibers of the second-coordinate
projection may contain several unipotent orbits. Their contribution
therefore involves the following normalized multiplicity.

\begin{lem}[Normalized fiber multiplicity]
\label{lem: normalized fiber multiplicity}
Let $t=(t_1\Omega_K,t_2\Omega_K)$, with
$t_1,t_2\in\mathfrak c^{-1}\mathfrak a/\mathfrak a$, and let
$v_2=y\Omega_K$, where
$y\in(1+t_2+\mathfrak a)\setminus\{0\}$.
Under the identification
$\Gamma_{N,\infty}(\mathfrak f)\cong\mathfrak f$, put
\[
\mathfrak d(t_2)
:=
\operatorname{Ann}_{\mathcal O_K}(t_2)
=
\{r\in\mathcal O_K\mid rt_2\in\mathfrak a\}.
\]
Then
$\Gamma_{N,\infty,t}(\mathfrak f)=\mathfrak f\mathfrak d(t_2)$,
and
\[
[\Gamma_{N,\infty}(\mathfrak f):
 \Gamma_{N,\infty,t}(\mathfrak f)]
=
N_{K/\Q}(\mathfrak d(t_2)).
\]
Moreover, the first-coordinate fiber satisfies
\[
\#\left(
  \mathrm{pr}_2^{-1}(v_2)/
  v_2\Gamma_{N,\infty,t}(\mathfrak f)
\right)
=
N_{K/\Q}((y)\mathfrak b\mathfrak d(t_2)).
\]
Consequently,
\begin{equation}
\label{eq: normalized prB fiber multiplicity}
\frac{
  \#\left(
    \mathrm{pr}_2^{-1}(v_2)/
    v_2\Gamma_{N,\infty,t}(\mathfrak f)
  \right)
}{
  [\Gamma_{N,\infty}(\mathfrak f):
   \Gamma_{N,\infty,t}(\mathfrak f)]
}
=
N_{K/\Q}((y)\mathfrak b).
\end{equation}
\end{lem}

\begin{proof}
For $u\in\mathfrak f$, the unipotent element $n(u)$ acts on the
torsion coordinates by
$(t_1,t_2)\mapsto(t_1+ut_2,t_2)$.
It fixes $t$ precisely when $ut_2\in\mathfrak a$, so
$\Gamma_{N,\infty,t}(\mathfrak f)
=\mathfrak f\cap\mathfrak d(t_2)$.
Since
$\mathfrak c\subseteq\mathfrak d(t_2)\subseteq\mathcal O_K$
and $(\mathfrak f,\mathfrak c)=1$, the ideals
$\mathfrak f$ and $\mathfrak d(t_2)$ are coprime.
Their intersection is therefore
$\mathfrak f\mathfrak d(t_2)$, whose index in $\mathfrak f$ is
$N_{K/\Q}(\mathfrak d(t_2))$.

For $v_2=y\Omega_K$, the first-coordinate fiber is
$\mathrm{pr}_2^{-1}(v_2)=(t_1+\mathfrak a)\Omega_K$,
and its stabilizer acts by translations in
$y\mathfrak f\mathfrak d(t_2)\Omega_K$.
If $u\in\mathfrak f\mathfrak d(t_2)$, then
$u\in\mathfrak f\subset\mathfrak a$ and $ut_2\in\mathfrak a$;
as $y\in1+t_2+\mathfrak a$ and $u\in\mathcal O_K$, this gives
$uy\in\mathfrak a$.
Thus the translation lattice is contained in
$\mathfrak a\Omega_K$, and
\[
\#\left(
  \mathrm{pr}_2^{-1}(v_2)/
  v_2\Gamma_{N,\infty,t}(\mathfrak f)
\right)
=
[\mathfrak a:y\mathfrak f\mathfrak d(t_2)]
=
N_{K/\Q}((y)\mathfrak b\mathfrak d(t_2)),
\]
where the last equality uses
$\mathfrak a^{-1}=\mathfrak b\mathfrak f^{-1}$.
Dividing by the stabilizer index gives
\eqref{eq: normalized prB fiber multiplicity}.
\end{proof}

\begin{lem}[Covolume of the unipotent lattice]
\label{lem: volume of the unipotent fundamental domain}
Under the identification
$\Gamma_{N,\infty}(\mathfrak f)\cong\mathfrak f$,
the volume of the fundamental domain $\mathcal F$, computed with
respect to the oriented form $du\wedge d\overline u$, is
\[
\mathrm{vol}(\mathcal F)
=
-N_{K/\Q}(\mathfrak f)\sqrt{d_K}.
\]
\end{lem}

\begin{proof}
Choose a positively oriented $\Z$-basis $\eta_1,\eta_2$ of
$\mathfrak f$ and compute the volume using its fundamental
parallelogram, parametrized by
$u=r_1\eta_1+r_2\eta_2$ with $0\leq r_1,r_2<1$.
Then
\[
\mathrm{vol}(\mathcal F)
=
\eta_1\overline{\eta_2}-\eta_2\overline{\eta_1},
\qquad
\mathrm{vol}(\mathcal F)^2
=
N_{K/\Q}(\mathfrak f)^2d_K.
\]
For $u=x+iy$, the identity
$du\wedge d\overline u=-2i\,dx\wedge dy$
and the chosen orientation select the negative square root.
\end{proof}

\begin{thm}[The $\mathrm{pr}_B$-contribution]
\label{thm: prB contribution}
For
$f\in R[1/N_{K/\Q}(\mathfrak c)]
[\mathcal D]^{0,\Gamma(\mathfrak f)}$,
the $\mathrm{pr}_B$-component of the zeroth Fourier coefficient of
$\mathrm{res}_{\infty,1}
(\psi_u^b(f,x_{\Omega_L}^{B}))(s)$ is, for
$\operatorname{Re}(s)\gg0$, given by
\begin{equation}
\label{eq: prB general constant term}
\scalebox{1}{$\displaystyle
\begin{aligned}
&\mathrm{pr}_B\left(
  \left[
    \left.
      \mathrm{res}_{\infty,1}
      \bigl(\psi_u^b(f,x_{\Omega_L}^{B})\bigr)(s)
    \right|_0
  \right]
\right)=\frac{
  (2\pi i)^{b+1}\overline{\Omega_K}^{\,2}
  |\mathcal O_{\mathfrak f}^{\times}|
}{
  (s-1)(\Omega_K^\vee)^b
  |\Omega_K|^{2s}\mathrm{vol}(\mathcal F)
}
\sum_{
  t_2\in\mathcal O_{\mathfrak f}^{\times}
  \backslash(\mathfrak c^{-1}\mathfrak a/\mathfrak a)
}
\\ &\qquad\qquad\qquad\qquad\qquad\qquad
\left(
  \sum_{t_1\in\mathfrak c^{-1}\mathfrak a/\mathfrak a}
  f\bigl(-(t_1\Omega_K,t_2\Omega_K)\bigr)
\right)
\sum_{
  y\in\mathcal O_{\mathfrak f}^{\times}
  \backslash
  \bigl(
    (1+\mathfrak a+\mathcal O_{\mathfrak f}^{\times}t_2)
    \setminus\{0\}
  \bigr)
}
N_{K/\Q}((y)\mathfrak b)
\frac{\overline y^{\,b+2}}{|y|^{2s}}.
\end{aligned}
$}
\end{equation}
This expression admits meromorphic continuation in $s$ and is
holomorphic at $s=2$.
At this value, its prefactor simplifies to
$(2\pi i)^{b+1}|\mathcal O_{\mathfrak f}^{\times}|
/((\Omega_K^\vee)^b\Omega_K^2\mathrm{vol}(\mathcal F))$,
and its value is the projected algebraic boundary class
\begin{equation}
\label{eq: prB algebraic specialization}
\mathrm{pr}_B\left(
  \left[
    \left.
      \mathrm{res}_{\infty,1}
      \bigl(\psi_u^b(f,x_{\Omega_L}^{B})\bigr)(2)
    \right|_0
  \right]
\right)
\in
R\left[\frac{1}{\mathfrak f}\right]
\left[
  \frac{1}{
    |\mathcal O_{\mathfrak f}^{\times}|\,b!\,
    N_{K/\Q}(\mathfrak b\mathfrak c)
  }
\right].
\end{equation}
\end{thm}

\begin{proof}
By Theorem~\ref{thm:algebraic constant term at infinity}, after
replacing $t$ by $-t$, the second scalar coefficient is
\begin{equation*}
  \scalebox{0.94}{$\displaystyle
\begin{aligned}
&\frac{
  2\pi i
}{
  (s-1)\mathrm{vol}(\mathcal F)
  \left\langle
    \overline{\omega(\mathcal A)},
    \omega(\mathcal A^\vee)
  \right\rangle_{dR,\mathcal A}^{\,b}
}
\sum_{t\in\mathfrak C^{-1}\Lambda/\Lambda}
\frac{
  f(-t)
}{
  [\Gamma_{N,\infty}(\mathfrak f):
   \Gamma_{N,\infty,t}(\mathfrak f)]
}
\sum_{\substack{
  v_2\in
  \mathrm{pr}_2((x_{\Omega_L}^{B}+t+\Lambda)\setminus\{0\})\\
  v_2\neq0
}}
\#\left(
  \mathrm{pr}_2^{-1}(v_2)/
  v_2\Gamma_{N,\infty,t}(\mathfrak f)
\right)
\frac{\overline{v_2}^{\,b+2}}{|v_2|^{2s}}.
\end{aligned}
  $}
\end{equation*}
In the fixed coordinates,
$x_{\Omega_L}^{B}=(0,\Omega_K)$ and
$t=(t_1\Omega_K,t_2\Omega_K)$, so
\[
x_{\Omega_L}^{B}+t+\Lambda
=
(t_1\Omega_K+\mathfrak a\Omega_K)
\oplus
((1+t_2)\Omega_K+\mathfrak a\Omega_K).
\]
Thus every nonzero second coordinate is $v_2=y\Omega_K$ with
$y\in(1+t_2+\mathfrak a)\setminus\{0\}$.
Lemma~\ref{lem: normalized fiber multiplicity} replaces the ratio
of the fiber cardinality and the stabilizer index by
$N_{K/\Q}((y)\mathfrak b)$.
This is the additional norm factor arising from the second
projection. Since it is independent of $t_1$, the torsion
coefficients are summed over the first coordinate.
The compatible Serre-tensor Hodge bases identify the pairing factor
with
$\langle\overline{\omega(\mathcal A_0)},
\omega(\mathcal A_0^\vee)\rangle_{dR,\mathcal A_0}^{-b}$.

For $k\in\mathcal O_{\mathfrak f}^{\times}$, the diagonal action is
\[
(t_1,t_2,y)\longmapsto(kt_1,k^{-1}t_2,k^{-1}y).
\]
Because $k^{-1}-1\in\mathfrak f\subset\mathfrak a$, it carries
each second-coordinate coset to the coset indexed by $k^{-1}t_2$.
After summing over $t_1$, reindexing and the
$\Gamma(\mathfrak f)$-invariance of $f$ make the coefficient
invariant. The norm factor is unchanged by units, and the remaining
Eisenstein factor is invariant by the observation following
Lemma~\ref{lem: distribution relation}.
The action on $(t_2,y)$ is free for $y\neq0$, so grouping into
unit orbits introduces the factor
$|\mathcal O_{\mathfrak f}^{\times}|$.
Factoring $\Omega_K$ out of $v_2=y\Omega_K$ and applying
Proposition~\ref{prop: dual CM periods} now gives
\eqref{eq: prB general constant term}.
Its meromorphic continuation and holomorphy at $s=2$ follow from
Theorem~\ref{thm:algebraic constant term at infinity};
the specialization of the prefactor uses
$\overline{\Omega_K}^{\,2}/|\Omega_K|^4=1/\Omega_K^2$.

The same constant-term theorem identifies the value at $s=2$ with
the complex realization of the algebraic boundary restriction.
Proposition~\ref{prop_9} defines its second rank-one projection
after inverting $b!$, while passage through the finite unit
quotient requires inverting
$|\mathcal O_{\mathfrak f}^{\times}|$.
Propositions~\ref{prop_10} and
\ref{prop: F-torsion restriction} give compatibility with the
isogenies and restriction to the cusp stabilizer.
The remaining localization
$R[1/\mathfrak f][1/N_{K/\Q}(\mathfrak b\mathfrak c)]$
is the one required for the isogenies, compatible Serre twists,
and torsion data entering the specialization.
\end{proof}

\begin{cor}
\label{cor: prB partial L values}
Retain the arithmetic data and the character $\chi_B$ fixed above,
allowing also $\mathfrak f=\mathcal O_K$ in the second assertion.
The following partial-value formulas hold.
\begin{enumerate}[label=(\arabic*),font=\normalfont]
\item
If $\mathfrak f\neq\mathcal O_K$, then
$x_{\Omega_L}^{B}\neq e$, and specialization to
$f_{L,[\mathfrak C]}$ on
$\mathcal D=\mathcal A(\mathfrak A)[\mathfrak C]$ gives
\begin{equation}
\label{eq: prB partial L value}
\begin{aligned}
&\frac{
  (2\pi i)^{b+1}
}{
  (\Omega_K^\vee)^b\Omega_K^2\mathrm{vol}(\mathcal F)
}
\Bigl(
  N_{K/\Q}(\mathfrak c)\chi_B(\mathfrak c)
  L_{\mathfrak f}(\chi_B,0,[\mathfrak b])
  -
  L_{\mathfrak f}(\chi_B,0,[\mathfrak b\mathfrak c])
\Bigr)
\in
\mathcal O_E R\left[\frac{1}{\mathfrak f}\right]
\left[
  \frac{1}{
    |\mathcal O_{\mathfrak f}^{\times}|\,b!\,
    N_{K/\Q}(\mathfrak b\mathfrak c)
  }
\right].
\end{aligned}
\end{equation}
Here
$\mathrm{vol}(\mathcal F)
=-N_{K/\Q}(\mathfrak f)\sqrt{d_K}$
by Lemma~\ref{lem: volume of the unipotent fundamental domain}.

\item
If $\mathfrak f=\mathcal O_K$, then $x_{\Omega_L}^{B}=e$.
Take the auxiliary ideals $\mathfrak c'$ and $\mathfrak C'$
as in the second assertion of
Corollary~\ref{cor: prA partial L values}.
Specialization to the function
$\widetilde f_{L,[\mathfrak C]}$ constructed in its proof gives
\begin{equation}
\label{eq: prB trivial conductor partial L value}
\begin{aligned}
&\frac{
  (2\pi i)^{b+1}
}{
  (\Omega_K^\vee)^b\Omega_K^2\mathrm{vol}(\mathcal F)
}
\Bigl(
  N_{K/\Q}(\mathfrak c)N_{K/\Q}(\mathfrak c')
  \chi_B(\mathfrak c)
  L_{\mathcal O_K}(\chi_B,0,[\mathfrak b\mathfrak c'])
  -
  N_{K/\Q}(\mathfrak c)\chi_B(\mathfrak c\mathfrak c')
  L_{\mathcal O_K}(\chi_B,0,[\mathfrak b])
  \\
&\qquad\quad\qquad
  +
  \chi_B(\mathfrak c')
  L_{\mathcal O_K}(\chi_B,0,[\mathfrak b\mathfrak c])
  -
  N_{K/\Q}(\mathfrak c')
  L_{\mathcal O_K}
  (\chi_B,0,[\mathfrak b\mathfrak c\mathfrak c'])
\Bigr)
\in
\mathcal O_E R
\left[
  \frac{1}{
    |\mathcal O_K^\times|\,b!\,
    N_{K/\Q}(\mathfrak b\mathfrak c\mathfrak c')
  }
\right].
\end{aligned}
\end{equation}
In this case $\mathrm{vol}(\mathcal F)=-\sqrt{d_K}$.
\end{enumerate}
\end{cor}

\begin{proof}
Suppose first that $\mathfrak f\neq\mathcal O_K$ and take
$f=f_{L,[\mathfrak C]}$.
Since
$f_{L,[\mathfrak C]}^{B}(t_2)
=N_{K/\Q}(\mathfrak c)^2\,1_0(t_2)-N_{K/\Q}(\mathfrak c)$,
Lemma~\ref{lem: distribution relation}, applied to $G_B$,
replaces the torsion sum in
\eqref{eq: prB general constant term} by
$N_{K/\Q}(\mathfrak c)^2$ times the weighted sum over
$\mathcal O_{\mathfrak f}^{\times}
\backslash((1+\mathfrak a)\setminus\{0\})$
minus $N_{K/\Q}(\mathfrak c)$ times the weighted sum over
$\mathcal O_{\mathfrak f}^{\times}
\backslash((1+\mathfrak c^{-1}\mathfrak a)\setminus\{0\})$.
For the lattice variables in these sums, the defining infinity
type gives
\[
N_{K/\Q}((y)\mathfrak b)
\frac{\overline y^{\,b+2}}{|y|^{2s}}
=
N_{K/\Q}(\mathfrak b)
\frac{\chi_B((y))}{N_{K/\Q}(y)^{s-2}}.
\]
Thus Lemma~\ref{lem:Ray-class parametrization} identifies the two
weighted sums, respectively, with
\[
\begin{gathered}
N_{K/\Q}(\mathfrak b)^{s-1}
\chi_B(\mathfrak b)^{-1}
L_{\mathfrak f}(\chi_B,s-2,[\mathfrak b]),
\qquad
N_{K/\Q}(\mathfrak b)
N_{K/\Q}(\mathfrak b\mathfrak c)^{s-2}
\chi_B(\mathfrak b\mathfrak c)^{-1}
L_{\mathfrak f}(\chi_B,s-2,[\mathfrak b\mathfrak c]).
\end{gathered}
\]
These identities are initially valid in the region of absolute
convergence and then hold by meromorphic continuation.
Evaluating at $s=2$, applying
Theorem~\ref{thm: prB contribution}, and multiplying by
$\chi_B(\mathfrak b\mathfrak c)/
\bigl(
|\mathcal O_{\mathfrak f}^{\times}|
N_{K/\Q}(\mathfrak b\mathfrak c)
\bigr)$
gives \eqref{eq: prB partial L value}.
A ray-principal power of $\mathfrak b\mathfrak c$, together with
the prescribed infinity type, shows that
$\chi_B(\mathfrak b\mathfrak c)
\in\mathcal O_E[1/N_{K/\Q}(\mathfrak b\mathfrak c)]$.
The unit-group order and the norm in the denominator are already
inverted in the stated localization.

Now suppose that $\mathfrak f=\mathcal O_K$.
Use the auxiliary trace-zero function
$\widetilde f_{L,[\mathfrak C]}$ on
$\mathcal D=
\mathcal A(\mathfrak A)[\mathfrak C\mathfrak C']
\setminus e(S)$
constructed in the proof of
Corollary~\ref{cor: prA partial L values}.
Under the Chinese remainder decomposition, write the second torsion
coordinate as $t_2=(t_2',t_2'')$, with
\[
t_2'\in
\mathfrak c'^{-1}\mathfrak b^{-1}/\mathfrak b^{-1},
\qquad
t_2''\in
\mathfrak c^{-1}\mathfrak b^{-1}/\mathfrak b^{-1}.
\]
Its sum over the first torsion coordinate is
\[
\widetilde f_{L,[\mathfrak C]}^{B}(t_2',t_2'')
=
\left(
  N_{K/\Q}(\mathfrak c')-1_0(t_2')
\right)
\left(
  N_{K/\Q}(\mathfrak c)^2\,1_0(t_2'')
  -
  N_{K/\Q}(\mathfrak c)
\right).
\]
Applying Lemma~\ref{lem: distribution relation} successively in
the $\mathfrak c'$- and $\mathfrak c$-coordinates gives four
weighted lattice sums corresponding, through
Lemma~\ref{lem:Ray-class parametrization}, to
$[\mathfrak b\mathfrak c']$, $[\mathfrak b]$,
$[\mathfrak b\mathfrak c]$, and
$[\mathfrak b\mathfrak c\mathfrak c']$.
At $s=2$, each weighted sum is
$N_{K/\Q}(\mathfrak b)$ times the corresponding partial
$L$-value divided by the character value of its indexing ideal.
The calculation and algebraic-specialization argument of
Theorem~\ref{thm: prB contribution} apply with auxiliary torsion
ideal $\mathfrak c\mathfrak c'$ and $x=e$; indeed,
$1\in\mathfrak a=\mathfrak b^{-1}$, so $(0,\Omega_K)$ still
represents this section.
Multiplying the resulting value by
$\chi_B(\mathfrak b\mathfrak c\mathfrak c')/
\bigl(
|\mathcal O_K^\times|N_{K/\Q}(\mathfrak b\mathfrak c)
\bigr)$
gives \eqref{eq: prB trivial conductor partial L value}.
The same ray-principal-power argument places
$\chi_B(\mathfrak b\mathfrak c\mathfrak c')$ in
$\mathcal O_E[1/N_{K/\Q}(\mathfrak b\mathfrak c\mathfrak c')]$,
so this multiplication requires no further localization.
\end{proof}

\begin{rem}[The rank-one contributions at the other cusps]\rm
The preceding formulas were obtained at the standard cusp $\infty$.
At another cusp, Proposition~\ref{prop:restriction-arbitrary-cusp}
transports the complete coefficient-valued differential form to
cusp-adapted coordinates; the algebraic Hodge bases and the two
rank-one projections are transported with it.
Let $\mathfrak g\subset\mathcal O_K$ represent the ideal-class
component of the cusp as in
Example~\ref{exmp: ideal-class description of the cusp data}, chosen coprime to the
ideals occurring in the arithmetic specialization.
After the Serre twist by
$\mathfrak a=\mathfrak f\mathfrak b^{-1}$, the cusp-adapted period
lattice is
\[
  \mathfrak a\mathfrak g^{-1}\Omega_K
  \oplus
  \mathfrak a\mathfrak g\Omega_K.
\]
Since
$\mathfrak a\mathfrak g^{-1}
=\mathfrak f(\mathfrak b\mathfrak g)^{-1}$
and
$\mathfrak a\mathfrak g
=\mathfrak f(\mathfrak b\mathfrak g^{-1})^{-1}$,
the first- and second-coordinate calculations give the opposite
ray-class replacements
\[
  \mathrm{pr}_A:\ [\mathfrak b]\longmapsto[\mathfrak b\mathfrak g],
  \qquad
  \mathrm{pr}_B:\ [\mathfrak b]\longmapsto[\mathfrak b\mathfrak g^{-1}].
\]
Thus the partial-value formulas of
Corollaries~\ref{cor: prA partial L values} and
\ref{cor: prB partial L values} apply with these respective
replacements, both for $\mathfrak f\neq\mathcal O_K$ and for
$\mathfrak f=\mathcal O_K$.
The auxiliary ideal $\mathfrak c$ retains its distribution role in
both contributions.
For $\mathrm{pr}_B$, one must also replace
$\mathrm{vol}(\mathcal F)$ by
$\mathrm{vol}(\mathcal F_{\mathfrak g})$, where
$\mathcal F_{\mathfrak g}$ is a fundamental domain for the transported
unipotent lattice $\mathfrak f\mathfrak g^{-2}\subset K$.
With the same oriented volume form,
Lemma~\ref{lem: volume of the unipotent fundamental domain} gives
\[
  \mathrm{vol}(\mathcal F_{\mathfrak g})
  =
  -N_{K/\Q}(\mathfrak f)
   N_{K/\Q}(\mathfrak g)^{-2}\sqrt{d_K}.
\]
The corresponding coefficient-ring statements are obtained from the
same corollaries after choosing integral representatives of the
resulting ray classes and making the additional localization at
$N_{K/\Q}(\mathfrak g)$.

These partial-value descriptions presuppose that the transported
torsion sections are represented by $(\Omega_K,0)$ and
$(0,\Omega_K)$ for $\mathrm{pr}_A$ and $\mathrm{pr}_B$, respectively.
For a general transported torsion section, the cuspwise
constant-term formula still applies, but the relevant coordinate
coset need not have the standard ray-class form, and the resulting
expression need not reduce to a single partial $L$-value of the
type used above.
Changing $\mathfrak g$ by a principal multiple transports the full
cusp datum, including the unipotent fundamental domain and the
ray-class labels, without changing the resulting boundary
cohomology class.
\end{rem}

\begin{thm}[Algebraicity of critical Hecke $L$-values]
\label{thm: algebraicity of critical Hecke L values}
Retain the notation and assumptions fixed above, allowing also
$\mathfrak f=\mathcal O_K$.
The following assertions hold.
\begin{enumerate}[label=(\arabic*),font=\normalfont]
\item
If $\mathfrak f\neq\mathcal O_K$, then
\begin{equation}
\label{eq: full prA critical L value}
\begin{aligned}
\frac{(2\pi i)^b}{(\Omega_K^\vee)^b\Omega_K^2}
\left(
  N_{K/\Q}(\mathfrak c)^2\chi_A(\mathfrak c)-1
\right)
L_{\mathfrak f}(\chi_A,0)
\in
\mathcal O_E R\left[\frac{1}{\mathfrak f}\right]
\left[
  \frac{1}{
    |\mathcal O_{\mathfrak f}^{\times}|\,b!\,
    N_{K/\Q}(\mathfrak c)
  }
\right],
\end{aligned}
\end{equation}
and
\begin{equation}
\label{eq: full prB critical L value}
\begin{aligned}
\frac{
  (2\pi i)^{b+1}
}{
  (\Omega_K^\vee)^b\Omega_K^2\mathrm{vol}(\mathcal F)
}
\left(
  N_{K/\Q}(\mathfrak c)\chi_B(\mathfrak c)-1
\right)
L_{\mathfrak f}(\chi_B,0)
\in
\mathcal O_E R\left[\frac{1}{\mathfrak f}\right]
\left[
  \frac{1}{
    |\mathcal O_{\mathfrak f}^{\times}|\,b!\,
    N_{K/\Q}(\mathfrak c)
  }
\right].
\end{aligned}
\end{equation}
Here $\mathcal F$ is the fundamental domain used in the
standard-cusp computation, and
$\mathrm{vol}(\mathcal F)
=-N_{K/\Q}(\mathfrak f)\sqrt{d_K}$ by
Lemma~\ref{lem: volume of the unipotent fundamental domain}.

\item
If $\mathfrak f=\mathcal O_K$, take a nontrivial integral ideal
$\mathfrak c'\subset\mathcal O_K$ coprime to $\mathfrak c$.
Then
\begin{equation}
\label{eq: full prA trivial conductor critical L value}
\begin{aligned}
\frac{(2\pi i)^b}{(\Omega_K^\vee)^b\Omega_K^2}
\left(1-\chi_A(\mathfrak c')\right)
\left(
  N_{K/\Q}(\mathfrak c)^2\chi_A(\mathfrak c)-1
\right)
L_{\mathcal O_K}(\chi_A,0)
\in
\mathcal O_E R
\left[
  \frac{1}{
    |\mathcal O_K^\times|\,b!\,
    N_{K/\Q}(\mathfrak c\mathfrak c')
  }
\right],
\end{aligned}
\end{equation}
whereas
\begin{equation}
\label{eq: full prB trivial conductor critical L value}
\begin{aligned}
\frac{
  (2\pi i)^{b+1}
}{
  (\Omega_K^\vee)^b\Omega_K^2\mathrm{vol}(\mathcal F)
}
\left(
  N_{K/\Q}(\mathfrak c')-\chi_B(\mathfrak c')
\right)
\left(
  N_{K/\Q}(\mathfrak c)\chi_B(\mathfrak c)-1
\right)
L_{\mathcal O_K}(\chi_B,0)
\in
\mathcal O_E R
\left[
  \frac{1}{
    |\mathcal O_K^\times|\,b!\,
    N_{K/\Q}(\mathfrak c\mathfrak c')
  }
\right].
\end{aligned}
\end{equation}
In this case $\mathrm{vol}(\mathcal F)=-\sqrt{d_K}$.
\end{enumerate}

Moreover, after enlarging the field of definition if necessary,
one may choose a number field $k\subset\C$ with
$E\subset k$ and $R\subset k$.
Choosing the auxiliary ideals so that the corresponding
regularizing factors are nonzero, one obtains
\begin{equation}
\label{eq: algebraicity full L value}
\frac{(2\pi i)^b}{(\Omega_K^\vee)^b\Omega_K^2}
L_{\mathfrak f}(\chi_A,0)\in k,
\qquad
\frac{
  (2\pi i)^{b+1}
}{
  (\Omega_K^\vee)^b\Omega_K^2\mathrm{vol}(\mathcal F)
}
L_{\mathfrak f}(\chi_B,0)\in k.
\end{equation}
\end{thm}

\begin{proof}
Choose integral ideals
$\mathfrak b_1,\ldots,\mathfrak b_h$ representing the ray classes in
$\mathcal I(\mathfrak f)/\mathcal P_{\mathfrak f}$, coprime to
the fixed auxiliary ideals, and sum
Corollaries~\ref{cor: prA partial L values} and
\ref{cor: prB partial L values} over these classes.
Multiplication by $\mathfrak c$, and by $\mathfrak c'$ when
$\mathfrak f=\mathcal O_K$, permutes the ray classes.
Consequently, the sums of the shifted partial $L$-values are again
the corresponding full $L$-values.
For $\mathfrak f\neq\mathcal O_K$, this gives the factors
$N_{K/\Q}(\mathfrak c)^2\chi_A(\mathfrak c)-1$ and
$N_{K/\Q}(\mathfrak c)\chi_B(\mathfrak c)-1$.
For $\mathfrak f=\mathcal O_K$, the four-term combinations in the
two corollaries factor as in
\eqref{eq: full prA trivial conductor critical L value} and
\eqref{eq: full prB trivial conductor critical L value}.
The resulting coefficient rings initially contain denominators
depending on the representatives $\mathfrak b_i$.
Choose a second system of integral ray-class representatives,
again coprime to the fixed auxiliary ideals, whose norms are
coprime to those of the first system.
Intersecting the two resulting localizations removes the
representative-dependent denominators and gives the coefficient
rings in the statement.

Finally, the theory of complex multiplication allows the algebraic
datum to be defined, after a finite extension, over a number field
containing $E$ and the fraction field of $R$.
Dividing the regularized values by the corresponding nonzero
regularizing factors gives \eqref{eq: algebraicity full L value}.
\end{proof}

\begin{rem}[The boundary class as a packet of partial values]\rm
The algebraic boundary restriction of
Corollary~\ref{cor: algebraic boundary restriction class} records
the constant terms at all cusps simultaneously through the
decomposition
\[
\begin{aligned}
&H^1\left(
  \partial\overline{\Gamma\backslash\mathbb H^3}^{\,\mathrm{BS}},
  \mathrm{TSym}^b(\omega_{\mathcal A^\vee}/\C)
  \otimes\omega_{\mathcal A/\C}^2
\right)
\cong
\bigoplus_{c\in\Gamma\backslash\mathbb P^1(K)}
H^1\left(
  e_c,
  \mathrm{TSym}^b(\omega_{\mathcal A^\vee}/\C)
  \otimes\omega_{\mathcal A/\C}^2
\right).
\end{aligned}
\]
In the arithmetic situations described above, individual cusp
components detect partial $L$-values, while the complete boundary
class retains the relations among these components under
$\GL_2(K)$-transport.
The critical-value formulas can therefore be viewed as arithmetic
realizations of one global boundary cohomology class, rather than
as unrelated constant-term calculations.
The passage from partial to full $L$-values in the theorem is the
separate arithmetic operation of summing over ray classes.
\end{rem}

\begin{rem}[Comparison with Kings--Sprang]\rm
The algebraicity statements obtained here, in the present
specialization, are already covered by the more general result
of \cite[Theorem~4.10]{KS_2025}; compare also
\cite[Corollary~4.13]{KS_2025} for the corresponding partial values.
Kings--Sprang obtain these values by specializing the equivariant
Eisenstein--Kronecker class along torsion sections and identifying
the resulting pullbacks with Eisenstein--Kronecker series.
Here the arithmetic information is detected through the
Borel--Serre boundary restriction, its zeroth Fourier coefficients,
and the projections $\mathrm{pr}_A$ and $\mathrm{pr}_B$.
The partial values thus occur as cuspwise components of a global
algebraic boundary class.
Rather than a new algebraicity result, the present calculation
provides a boundary-cohomological realization of the corresponding
Kings--Sprang formulas in this quadratic setting.
\end{rem}

\bibliographystyle{alpha}
\bibliography{sample}

\bigskip

\noindent
\textsc{Fakultät für Mathematik, Universität Regensburg,
93040 Regensburg, Germany}

\smallskip

\noindent
\textit{Email address:} \texttt{zhenghang.du@mathematik.uni-regensburg.de}

\end{document}